\documentclass[a4paper,reqno,11pt]{amsart}

\usepackage{amsmath, amsfonts, amssymb, amsthm, amscd}
\usepackage{graphicx}
\usepackage{psfrag}
\usepackage{perpage}
\usepackage{url}
\usepackage{color}
\usepackage{mathrsfs}
\usepackage{mathabx}
\usepackage{scalerel}
\usepackage{tikz}\usepackage{caption,subcaption}
\usetikzlibrary{arrows,decorations.pathmorphing,backgrounds,positioning,fit,petri}
\usepackage{subdepth}

\usepackage{dsfont} 

\usepackage[utf8]{inputenc}
\usepackage[T1]{fontenc}
\usepackage{microtype}

\usepackage[a4paper,scale={0.72,0.74},marginratio={1:1},footskip=7mm,headsep=10mm]{geometry}

\usepackage{hyperref}

\makeatletter
\def\@secnumfont{\bfseries\scshape}

\def\section{\@startsection{section}{1}
  \z@{.9\linespacing\@plus\linespacing}{.5\linespacing}%
  {\normalfont\large\bfseries\scshape\centering}}

\def\subsection{\@startsection{subsection}{2}%
  \z@{.5\linespacing\@plus.7\linespacing}{-.5em}%
  {\normalfont\bfseries\scshape}}

\def\subsubsection{\@startsection{subsubsection}{3}%
  \z@{.5\linespacing\@plus.7\linespacing}{-.5em}%
  {\normalfont\scshape}}

\def\specialsection{\@startsection{section}{1}%
  \z@{\linespacing\@plus\linespacing}{.5\linespacing}%
  {\normalfont\centering\large\bfseries\scshape}}
\makeatother

\makeatletter

\renewenvironment{proof}[1][\proofname]{\par
\pushQED{\qed}%
\normalfont \topsep4\p@\@plus4\p@\relax
\trivlist
\item[\hskip\labelsep
\bfseries
#1\@addpunct{.}]\ignorespaces
}{%
\popQED\endtrivlist\@endpefalse
}
\makeatother

\makeatletter
\newcommand \Dotfill {\leavevmode \leaders \hb@xt@ 6pt{\hss .\hss }\hfill \kern \z@}
\makeatother

\makeatletter
\def\@tocline#1#2#3#4#5#6#7{\relax
  \ifnum #1>\c@tocdepth 
  \else
    \par \addpenalty\@secpenalty\addvspace{#2}%
    \begingroup \hyphenpenalty\@M
    \@ifempty{#4}{%
      \@tempdima\csname r@tocindent\number#1\endcsname\relax
    }{%
      \@tempdima#4\relax
    }%
    \parindent\z@ \leftskip#3\relax \advance\leftskip\@tempdima\relax
    \rightskip\@pnumwidth plus4em \parfillskip-\@pnumwidth
    #5\leavevmode\hskip-\@tempdima
      \ifcase #1
       \or\or \hskip 1.65em \or \hskip 3.3em \else \hskip 4.95em \fi%
      #6\nobreak\relax
    \Dotfill
    \hbox to\@pnumwidth{\@tocpagenum{#7}}\par
    \nobreak
    \endgroup
  \fi}
\makeatother

\makeatletter
\def\l@section{\@tocline{1}{0pt}{1pc}{}{\scshape}}
\renewcommand{\tocsection}[3]{%
\indentlabel{\@ifnotempty{#2}{\ignorespaces#1 #2.\hskip 0.7em}}#3}
\def\l@subsection{\@tocline{2}{0pt}{1pc}{5pc}{}}

\def\l@subsubsection{\@tocline{3}{0pt}{1pc}{7pc}{}}

\makeatother

\numberwithin{equation}{section}

\newtheoremstyle{mytheorem}{.7\linespacing\@plus.3\linespacing}{.7\linespacing\@plus.3\linespacing}%
     {\itshape}
     {}
     {\bfseries}
     {. }
     {0.3ex}
     {\thmname{{\bfseries #1}}\thmnumber{ {\bfseries #2}}\thmnote{ (#3)}}  

\theoremstyle{mytheorem}

\newtheorem{theorem}{Theorem}[section]
\newtheorem{lemma}[theorem]{Lemma}
\newtheorem{proposition}[theorem]{Proposition}

\newtheorem{remark}[theorem]{Remark}
\newtheorem{definition}[theorem]{Definition}

\newtheorem{assumption}[theorem]{Assumption}

\newcommand\overlap{
\begin{tikzpicture}[scale=0.2]
\foreach \i in {0,1,...,4}{
	\draw [fill] (2*\i,0)--(2*\i,0) circle [radius=0.2];
	}
\draw [thick] (0,0)  to [out=45,in=135]  (2,0) to [out=45,in=135]  (4,0) to [out=45,in=135]  (6,0) to [out=45,in=135]  (8,0) ;
\draw [thick, red] (0,0)  to [out=-45,in=-135]  (2,0) to [out=-45,in=-135]  (4,0)  to [out=-45,in=-135]  (6,0) to [out=-45,in=-135]  (8,0);
\end{tikzpicture}
}

\newcommand\overlapbr{
\begin{tikzpicture}[scale=0.2]
\foreach \i in {0,1,...,4}{
	\draw [fill] (2*\i,0)--(2*\i,0) circle [radius=0.2];
	}
\draw [thick, blue] (0,0)  to [out=45,in=135]  (2,0) to [out=45,in=135]  (4,0) to [out=45,in=135]  (6,0) to [out=45,in=135]  (8,0) ;
\draw [thick, red] (0,0)  to [out=-45,in=-135]  (2,0) to [out=-45,in=-135]  (4,0)  to [out=-45,in=-135]  (6,0) to [out=-45,in=-135]  (8,0);
\end{tikzpicture}
}

\newcommand{\bbE}{{\ensuremath{\mathbb E}} }

\newcommand{\bbP}{{\ensuremath{\mathbb P}} }

\newcommand{\bbT}{{\ensuremath{\mathbb T}} }

\newcommand\bsf{\boldsymbol{f}}

\newcommand\bsY{\boldsymbol{Y}}

\newcommand{\cA}{{\ensuremath{\mathcal A}} }
\newcommand{\cB}{{\ensuremath{\mathcal B}} }

\newcommand{\cD}{{\ensuremath{\mathcal D}} }

\newcommand{\cG}{{\ensuremath{\mathcal G}} }
\newcommand{\cH}{{\ensuremath{\mathcal H}} }
\newcommand{\cI}{{\ensuremath{\mathcal I}} }
\newcommand{\cJ}{{\ensuremath{\mathcal J}} }

\newcommand{\cM}{{\ensuremath{\mathcal M}} }

\newcommand{\cQ}{{\ensuremath{\mathcal Q}} }

\newcommand{\cS}{{\ensuremath{\mathcal S}} }
\newcommand{\cT}{{\ensuremath{\mathcal T}} }
\newcommand{\cU}{{\ensuremath{\mathcal U}} }
\newcommand{\cV}{{\ensuremath{\mathcal V}} }
\newcommand{\cW}{{\ensuremath{\mathcal W}} }

\newcommand{\cZ}{{\ensuremath{\mathcal Z}} }

\DeclareMathSymbol{\leqslant}{\mathalpha}{AMSa}{"36} 
\DeclareMathSymbol{\geqslant}{\mathalpha}{AMSa}{"3E} 
\DeclareMathSymbol{\eset}{\mathalpha}{AMSb}{"3F}     

\newcommand{\sumtwo}[2]{\sum_{\substack{#1 \\ #2}}} 

\newcommand{\be}{\begin{equation}}
\newcommand{\ee}{\end{equation}}

\newcommand{\asto}[1]{\underset{{#1}\to\infty}{\longrightarrow}}

\newcommand{\Asto}[1]{\underset{{#1}\to\infty}{\Longrightarrow}}

\newcommand{\R}{\mathbb{R}}

\newcommand{\Z}{\mathbb{Z}}
\newcommand{\N}{\mathbb{N}}

\newcommand{\T}{\mathbb{T}}

\newcommand{\PEfont}{\mathrm}

\newcommand{\p}{\ensuremath{\PEfont P}}

\DeclareMathOperator{\e}{\ensuremath{\PEfont E}}
\newcommand{\I}{\ensuremath{\PEfont I}}

\newcommand{\E}{\ensuremath{\PEfont  E}}
\renewcommand{\P}{\p}

\DeclareMathOperator{\bbvar}{\ensuremath{\mathbb{V}ar}}
\DeclareMathOperator{\bbcov}{\ensuremath{\mathbb{C}ov}}

\newcommand{\ind}{\mathds{1}}

\newcommand{\eps}{\varepsilon}
\renewcommand{\epsilon}{\varepsilon}
\renewcommand{\theta}{\vartheta}
\renewcommand{\rho}{\varrho}

\newenvironment{myenumerate}{%
\renewcommand{\theenumi}{\arabic{enumi}}%
\renewcommand{\labelenumi}{{\rm(\theenumi)}}%
\begin{list}{\labelenumi}
	{%
	\setlength{\itemsep}{0.4em}%
	\setlength{\topsep}{0.5em}%
	\setlength\leftmargin{2.45em}%
	\setlength\labelwidth{2.05em}%
	\setlength{\labelsep}{0.4em}%
	\usecounter{enumi}%
	}%
	}%
{\end{list}
}

{\end{list}
}

{\end{list}
}

{\end{myenumerate}}

\newenvironment{myitemize}{%
\begin{list}{$\bullet$}%
 	{%
	\setlength{\itemsep}{0.4em}%
	\setlength{\topsep}{0.5em}%
	\setlength\leftmargin{2.65em}%
	\setlength\labelwidth{2.65em}%
	\setlength{\labelsep}{0.4em}%
	}%
	}%
{\end{list}}

\renewenvironment{itemize}{
\begin{myitemize}}%
{\end{myitemize}}

\MakePerPage[2]{footnote} 

\date{\today}

\newcommand\dd{\mathrm{d}}

\newcommand\sfC{\mathsf C}

\newcommand\sfM{\mathsf M}
\newcommand\sfP{\mathsf P}

\newcommand\sfU{\mathsf U}
\newcommand\sfV{\mathsf V}
\newcommand\sfW{\mathsf W}

\newcommand\sfa{\mathsf a}
\newcommand\sfb{\mathsf b}

\newcommand\sfi{\mathsf i}
\newcommand\sfI{\mathsf I}
\newcommand\sfj{\mathsf j}

\newcommand\sfm{\mathsf m}
\newcommand\sfn{\mathsf n}

\newcommand\sfx{\mathsf x}

\newcommand\sfz{\mathsf z}

\newcommand\bx{\boldsymbol{x}}
\newcommand\by{\boldsymbol{y}}
\newcommand\bz{\boldsymbol{z}}

\newcommand\dist{\mathrm{dist}}

\newcommand\cg{\mathrm{cg}}
\newcommand\even{\mathrm{even}}
\newcommand\odd{\mathrm{odd}}
\newcommand\notri{\mathrm{no\, triple}}

\newcommand\diff{\mathrm{diff}}
\newcommand\nodiff{\mathrm{superdiff}}

\newcommand\sfc{\mathsf{c}}

\newcommand\bcA{\vec{\boldsymbol{\mathcal{A}}}}

\newcommand{\ev}[1]{[\![ #1 ]\!]}

\title{The critical 2d Stochastic Heat Flow}

\begin{document}

\author[F. Caravenna]{Francesco Caravenna}
\address{Dipartimento di Matematica e Applicazioni\\
 Universit\`a degli Studi di Milano-Bicocca\\
 via Cozzi 55, 20125 Milano, Italy}
\email{francesco.caravenna@unimib.it}

\author[R. Sun]{Rongfeng Sun}
\address{Department of Mathematics\\
National University of Singapore\\
10 Lower Kent Ridge Road, 119076 Singapore
}
\email{matsr@nus.edu.sg}

\author[N. Zygouras]{Nikos Zygouras}
\address{Department of Mathematics\\
University of Warwick\\
Coventry CV4 7AL, UK}
\email{N.Zygouras@warwick.ac.uk}

\begin{abstract}
 We consider directed polymers in random environment in the critical
  dimension $d = 2$, focusing on the intermediate disorder regime when the
  model undergoes a phase transition. We prove that, at criticality, the diffusively rescaled
  random field of partition functions has a \textit{unique scaling limit}:
 a universal process of random measures on
  $\mathbb{R}^2$ with logarithmic correlations, which we call the \textit{Critical 2d Stochastic Heat Flow}.
  It is the natural candidate for the long sought solution of the critical
2d Stochastic Heat Equation with multiplicative space-time white noise.
\end{abstract}

\maketitle

\tableofcontents

\section{Introduction and main results}

\subsection{Overview}

The model of \emph{directed polymer in random environment} (DPRE)
is by now a fundamental model in statistical physics
and probability theory. It is one of the simplest and yet most challenging models for disordered systems,
where the effect of \emph{disorder} --- which is synonymous with random environment ---
can be investigated.
Originally introduced by Huse and Henley \cite{HH85}
in the physics literature to study interfaces of the Ising model with random impurities, over the years,
DPRE has become
an object of mathematical interest and lies at the heart of two areas of intense research in recent years.
On the one
hand, it is one of the canonical examples in the Kardar-Parisi-Zhang (KPZ) universality class of interface
growth models,
which has witnessed tremendous progress over the last two decades in spatial dimension $d=1$
(see e.g.~the surveys
\cite{QS15, Cor12, Cor16}); on the other hand, it provides a discretisation of the Stochastic Heat Equation
(SHE) and (via the
Cole-Hopf transformation) of the Kardar-Parisi-Zhang (KPZ) equation,
for which a robust solution theory in $d=1$ has been
developed only recently in the larger context of singular stochastic partial differential equations (SPDE)
\cite{H13, H14, GIP15, Kup14, GJ14}.

Our goal in this paper is to consider DPRE in the
\emph{critical spatial dimension $d=2$},
for which much
remains unknown.
Our main result shows that, in a \emph{critical window} for the disorder strength,
the family of partition functions of DPRE
converges to a universal limit, which can be interpreted as the solution of the
(classically ill-defined) 2-dimensional SHE.
This is the first example of a singular SPDE for which a solution has been constructed
in the critical dimension and for critical disorder strength.

In the remainder of the introduction,
we first recall the definition of DPRE, its basic properties,
and the works leading up to our current result. We then
present our main results and discuss their connections with singular SPDEs and related research.

\subsection{The model} \label{sec:intro}

The first ingredient in the definition of DPRE is a simple symmetric random walk $(S = (S_n)_{n\ge 0}, \P)$ on $\Z^d$, started at
$S_0 = 0$. To specify a different starting time $m$ and position $z$, we will write $\P(\,\cdot\,|S_m = z)$.
The second ingredient is the disorder or random environment,
encoded by a family of i.i.d.\ random variables
$(\omega = (\omega(n,z))_{n\in\N, z\in\Z^d}, \bbP)$ with
zero mean, unit variance and some finite exponential moments:
\begin{equation}\label{eq:lambda}
\begin{gathered}
	\bbE[\omega]=0 \,, \qquad \bbE[\omega^2]=1 \,, \\
	\exists\, \beta_0>0 \quad \text{such that} \qquad 	\lambda(\beta)
	:= \log \bbE[e^{\beta\omega}] < \infty \qquad \forall \beta \in [0, \beta_0] \,.
\end{gathered}
\end{equation}
Given $N\in\N$, $\beta>0$, and  a realization of $\omega$,
the polymer measure of length $N\in\N$ and disorder strength
(inverse temperature) $\beta$ in the random environment $\omega$ is given by
\begin{align}\label{eq:pathmeasure}
	\dd\P^{\beta,\, \omega}_{N}(S \,|\,S_0=z)
	:=\frac{1}{Z_N^{\beta,\,\omega}(z)}
	e^{\sum_{n=1}^{N-1} \{\beta \omega(n,S_n) - \lambda(\beta)\}} \,
	\dd\P(S \,|\, S_0 = z) \,,
\end{align}
where
\begin{align} \label{eq:paf}
	Z_N^{\beta,\,\omega}(z) &=
	\E\bigg[ e^{\sum_{n=1}^{N-1} \{\beta \omega(n,S_n) - \lambda(\beta)\}}
	\,\bigg|\, S_0 = z \bigg]
\end{align}
is the {\it partition function}. Note that $\lambda(\beta)$ in the exponent ensures that $\bbE[Z_N^{\beta,\,\omega}(z)]=1$.

In the mathematical literature, DPRE was first studied by Imbrie and Spencer \cite{IS88}. There have been
many results since then,
although many fundamental questions remain open. We
briefly recall what is known and refer to the recent monograph
by Comets~\cite{C17} for more details and references.

DPRE exhibits a phase transition between a {\em weak disorder} phase and a {\em strong disorder} phase.
Using the martingale structure of the partition functions $(Z_N^{\beta,\,\omega}(0))_{N\in\N}$,
first identified by Bolthausen in \cite{B89}, DPRE is said to be in the weak disorder (or strong disorder)
phase if the martingale
converges almost surely to a positive limit (or to 0).
It was later shown in \cite{CY06} that there is a critical value $\beta_c\geq 0$
such that strong disorder holds for $\beta>\beta_c$ and weak disorder holds for $0 \le \beta < \beta_c$,
where $\beta_c \in (0,\infty)$ for $d\geq 3$ \cite{IS88, B89}, and $\beta_c=0$ for $d=1,2$
\cite{CH02,CSY03} (see also \cite{L10, BL17, N19}).

In the weak disorder phase, a series of works culminating in \cite{CY06}
showed that the random walk under the polymer
measure converges to a Brownian motion under
diffusive scaling of space and time, as if the disorder is not present.
In the strong disorder phase, it is believed that under the polymer measure, the path should be
super-diffusive, but this has only been proved for special integrable
models in dimension $d=1$, see~\cite{J00,CH}.
Even less is known in $d\geq 2$ due to the lack of integrable models within the same universality class.
We mention that the strong disorder phase can alternatively be characterised by the
fact that two polymer paths sampled independently in the same random environment
have positive overlap, see \cite{CH02, CSY03, V07} and the more recent
results~\cite{Cha19,BC20a,BC20b, Ba21}.

\subsection{The case $d=2$}
\label{sec:d=2}

Henceforth, we will focus on dimension $d=2$.
Surprisingly, even though $\beta_c=0$, there is still a weak to strong disorder transition,
which was identified in \cite{CSZ17b}. More precisely,
if we choose $\beta=\beta_N=\hat\beta/\sqrt{\log N}$, which is called an
\emph{intermediate
disorder regime}, then it was shown in \cite{CSZ17b} that
below the critical point $\hat\beta_c = \sqrt{\pi}$, the partition function $Z_N^{\beta_N, \,\omega}(0)$
converges in distribution to a log-normal random variable, which
is strictly positive, while at and above $\hat\beta_c$, it converges to $0$
(such a transition does not occur in $d=1$).
This raises many interesting questions about the 2-dimensional DPRE.

There are two main perspectives in the study of the partition functions of DPRE.
One is to investigate the fluctuation of a single log-partition function $\log Z_N^{\beta, \,\omega}(0)$
as $N\to\infty$. In $d=1$, this is conjectured
to converge, under suitable rescaling,
to the universal Tracy-Widom distribution whenever
$\beta>0$. Similar universal fluctuations are expected to arise in $d\geq 2$ when $\beta >\beta_c$,
although only numerical results are available so far~\cite{HH12, HH13}. In $d=2$ and in the intermediate
disorder regime $\beta_N= \hat\beta/\sqrt{\log N}$ with
a sub-critical interaction strength $\hat\beta<\hat\beta_c$, \cite{CSZ17b} showed
that $\log Z_N^{\beta, \,\omega}(0)$ converges to a universal normal limit independent of the law of
$\omega$. The super-critical case $\hat\beta\geq \hat\beta_c$ remains a difficult challenge.

Another perspective, which we take in this paper,
is to study the \emph{diffusively rescaled field of partition functions}
indexed by all starting points in space-time:
\begin{equation}\label{eq:field}
	\big(\, \cU_N(t,x) :=
	Z_{Nt}^{\beta_N,\, \omega}(\sqrt{N}x) \big)_{t>0,\, x\in\R^2} \,,
\end{equation}
as well as the diffusively rescaled field of log-partition functions:
\begin{equation}\label{eq:field2}
	\big(\, \cH_N(t,x) :=
	\log Z_{Nt}^{\beta_N,\, \omega}(\sqrt{N}x) \big)_{t>0,\, x\in\R^2} \,.
\end{equation}
The fields $\cU_N$ and $\cH_N$ provide natural discretizations of the solutions
of the two-dimensional \emph{Stochastic Heat Equation (SHE)}
and {\em Kardar-Parisi-Zhang equation (KPZ)} respectively:
\begin{align}
	\partial_t u &= \frac{1}{2}\Delta u + \beta \, \dot{W} \, u,  \label{eq:2dSHE} \\
	\partial_t h &= \frac{1}{2}\Delta h
	+ \frac{1}{2} |\nabla h|^2 + \beta \, \dot{W} \,,  \label{eq:2dKPZ}
\end{align}
where $\dot{W} = \dot{W}(t,x)$ denotes space-time white noise.
These stochastic PDEs are singular and ill-posed:
even the recent breakthrough solution theories of regularity structures \cite{H13, H14}
and paracontrolled distributions~\cite{GIP15, GP17}
only apply in $d=1$ but not in the critical dimension $d=2$.
Therefore, if $\cU_N$ and $\cH_N$ admit non-trivial limits,
then these limits are natural candidates for the long-sought solutions of SHE and KPZ in $d=2$.

The study of the random field $\cU_N$ was initiated in \cite{CSZ17b}, which showed that
in the subcritical regime $\hat\beta<\hat\beta_c$, the centered and rescaled random field
$\sqrt{\log N}\big(\,\cU_N(t,x) -1\big)$ converges to the solution of the so-called Edwards-Wilkinson
equation, which is a Gaussian free field at each time $t$.
The study of the random field $\cH_N$ was first carried out in \cite{CD20},\footnote{More
precisely,  \cite{CD20} and \cite{Gu20} both study the analogue of $\cH_N$ defined
by mollifying the noise $\dot{W}$ in \eqref{eq:2dKPZ} instead of discretizing space and time,
while \cite{CSZ17b} considered both types of regularizations.}
which showed that $\sqrt{\log N}\big(\,\cH_N(t,x) - \bbE[\cH_N(t,x)]\big)$
is tight in $N$ as a family of distribution-valued random variables for $\hat\beta$ sufficiently small;
shortly after, \cite{CSZ20} proved
convergence to the solution of the same Edwards-Wilkinson equation as for $\cU_N$
for all $\hat\beta<\hat\beta_c$ (simultaneously, the same result was proved in \cite{Gu20} for
$\hat\beta$ sufficiently small).

In the much more interesting and delicate critical regime $\hat\beta= \hat\beta_c$
--- there is in fact
a critical window of width $O(1/\log N)$ around $\hat\beta_c$, see \eqref{intro:sigma} below ---
the random field $\cU_N(t,x)$ no longer needs any centering and rescaling.
Its limiting correlation structure was first identified in \cite{BC98}
through a different
regularisation of the $2d$ SHE \eqref{eq:2dSHE} (mollifying the noise $\dot{W}$ instead
of discretizing space and time).
In \cite{CSZ19b}, the third moment of the averaged random field
$\cU_N(t,\varphi) := \int \cU_N(t,x) \, \varphi(x) \, \dd x$, for test functions $\varphi$,
was computed and shown to converge to a finite limit as $N\to\infty$,
which implies that all subsequential limits of $\cU_N$
have the same correlation  structure identified in \cite{BC98}
(tightness is trivial since $\bbE[\cU_N]\equiv1$).
Subsequently, \cite{GQT21} identified
the limit of all moments of $\cU_N(t, \varphi)$ (see also the more recent work \cite{C21}).
However, the \emph{uniqueness of the limit} of $\cU_N$ remained elusive and challenging,
because the limiting moments identified in \cite{GQT21} and \cite{C21}
grow too fast to uniquely determine the
law of the random field.

Our main result settles this question and shows that, in the critical window
around $\hat\beta = \hat\beta_c$, the random field
$\cU_N$ indeed converges to a
\emph{unique universal limit}, which naturally provides a notion of
solution of the $2d$ SHE \eqref{eq:2dSHE}
for disorder strength~$\beta$ in the critical window. Therefore, we
name it the \emph{Critical $2d$ Stochastic Heat Flow}.

\subsection{Main results}\label{sec:main}
To formulate our main results, we generalize the partition functions in \eqref{eq:paf}
by introducing a point-to-point version, where both the starting and ending positions
of the random walk are fixed: for $M \le N \in \N_0 = \{0,1,2,\ldots\}$ and $w,z\in\Z^2$ we set
\begin{equation} \label{eq:Zab}
	Z_{M,N}^{\beta,\, \omega}(w,z) := \E \bigg[
	e^{\sum_{n=M+1}^{N-1} \{\beta \omega(n,S_n) - \lambda(\beta)\}}
	\, \ind_{S_N = z} \,\bigg|\, S_M = w \bigg]  \,,
\end{equation}
with the convention $\sum_{n=M+1}^{N-1} \{\ldots\} := 0$ for $N \le M+1$.

To deal with parity issues, for $x \in \R^2$
we denote by $\ev{x}$ the closest point $z \in \Z^2_{\rm even}
:=\{(z_1, z_2)\in \Z^2: z_1+z_2 \mbox{ even}\}$;
for $s \in \R$ we define the even approximation
$\ev{s} := 2 \, \lfloor s/2 \rfloor \in \Z_{\rm even} := 2\Z$.
We then introduce
the process of \emph{diffusively rescaled partition functions}:\footnote{Note that
$\bbE[Z_{M,N}^{\beta_N,\, \omega}(w,z)]
= \P(S_{N} = z \,|\, S_{M} = w) = O(\frac{1}{N-M}) = O(\frac{1}{N})$
for $M/N \le c < 1$, by the local limit theorem,
which explains the prefactor $N$ in \eqref{eq:rescZmeas}.
The extra factor $\frac{1}{4}$ is due to periodicity.}
\begin{equation} \label{eq:rescZmeas}
	\cZ^{\beta_N}_N = \bigg(\cZ^{\beta_N}_{N;\, s,t}(\dd x, \dd y) :=
	\frac{N}{4} \, Z_{\ev{Ns}, \ev{Nt}}^{\beta_N,\,\omega}(\ev{\sqrt{N} x},
	\ev{\sqrt{N} y}) \, \dd x \, \dd y
	\bigg)_{0 \le s \le t < \infty}
\end{equation}
where $\dd x \, \dd y$ denotes the Lebesgue measure on $\R^2 \times \R^2$,
and $\beta_N$ will be defined shortly.

We regard $\cZ^{\beta_N}_{N;\, s,t}(\dd x, \dd y)$ as a random measure on
$\R^2 \times \R^2$, where we equip the space of
locally finite measures on $\R^2 \times \R^2$ with the topology of vague convergence:
\begin{equation*}
	\mu_N \rightarrow \mu \quad \ \iff \quad \
	\int \phi(x,y) \, \mu_N(\dd x, \dd y) \to \int \phi(x,y) \, \mu(\dd x, \dd y)
	\quad \forall \phi\in C_c(\R^2\times \R^2) \,.
\end{equation*}
Our main result
proves weak convergence of the law of $\cZ^{\beta}_N$ as $N\to\infty$, when
$\beta = \beta_N$ is rescaled in a suitable \emph{critical window}, that we define next.
Let us introduce the sequence
\begin{equation}\label{eq:RN0}
	R_N := \sum_{n=1}^N \sum_{z\in\Z^2} \P(S_n = z)^2
	= \sum_{n=1}^N \P(S_{2n}=0)
	= \sum_{n=1}^N \bigg\{ \frac{1}{2^{2n}} \binom{2n}{n}\bigg\}^2
	\sim \frac{\log N}{\pi} \,,
\end{equation}
which is the expected overlap (number of collisions)
between two independent simple symmetric random walks starting from
the origin in $\Z^2$ up to time $N$.
Recalling that $\lambda(\cdot)$ is the disorder log-moment generating function,
see \eqref{eq:lambda}, the critical window for $\beta = \beta_N$ is
 \begin{equation} \label{intro:sigma}
	e^{\lambda(2\beta_N)-2\lambda(\beta_N)}-1
	= \frac{1}{R_N} \bigg(1 + \frac{\theta + o(1)}{\log N}\bigg) \,,\quad\quad
	\text{for some fixed } \theta\in \R \,.
\end{equation}
Since $\lambda(\beta) \sim \frac{1}{2} \beta^2$ as $\beta \downarrow 0$,
see \eqref{eq:lambda}, we have
$\beta_N \sim \hat\beta_c / \sqrt{\log N}$ with $\hat\beta_c = \sqrt{\pi}$ irrespective of
the parameter $\theta$,
which contributes to the second order asymptotics, see \eqref{eq:betaN}.

\smallskip

We can now state our main result, which will
be proved in Section~\ref{s:Thm1.3}.

 \begin{theorem}[Critical 2d Stochastic Heat Flow]\label{th:main0}
Fix $\beta_N$ in the critical window \eqref{intro:sigma}, for some $\theta\in \R$.
As $N\to\infty$, the family of random measures
$\cZ^{\beta_N}_N =
(\cZ^{\beta_N}_{N;\, s,t}(\dd x, \dd y))_{0\le s \le t < \infty}$ defined in \eqref{eq:rescZmeas}
converges in finite dimensional distributions
to a {\it unique} limit
\begin{equation*}
	\mathscr{Z}^\theta = (\mathscr{Z}_{s,t}^\theta(\dd x , \dd y))_{0 \le s \le t <\infty} \,,
\end{equation*}
which we call the \emph{Critical $2d$ Stochastic Heat Flow}. This limit $\mathscr{Z}^\theta$ is
universal, in that it does not depend on the law of the disorder $\omega$ except
for the assumptions in \eqref{eq:lambda}.
\end{theorem}

We can infer directly from its construction some basic properties of the
Critical $2d$ Stochastic Heat Flow, which we collect in the next result,
also proved in Section~\ref{s:Thm1.3}.

\begin{theorem}\label{th:main1}
The Critical $2d$ Stochastic Heat Flow $\mathscr{Z}^\theta$ is (space-time) translation invariant in law:
\begin{equation*}
	(\mathscr{Z}_{s+\sfa, t+\sfa}^\theta(\dd (x+\sfb) , \dd (y+\sfb)))_{0 \le s \le t <\infty}
	\stackrel{\rm dist}{=}
	(\mathscr{Z}_{s,t}^{\theta}(\dd x , \dd y))_{0 \le s \le t <\infty}
	\quad \ \forall \sfa \ge 0, \ \forall \sfb \in \R^2 \,,
\end{equation*}
and it satisfies the following scaling relation:
\begin{equation}\label{eq:scaling}
	(\mathscr{Z}_{\sfa s, \sfa t}^\theta(\dd (\sqrt{\sfa} x) , \dd (\sqrt{\sfa} y)))_{0 \le s \le t <\infty}
	\stackrel{\rm dist}{=}
	(\sfa\, \mathscr{Z}_{s,t}^{\theta+ \log \sfa}(\dd x , \dd y))_{0 \le s \le t <\infty}
	\quad \ \forall \sfa > 0 \,.
\end{equation}
The first and second moments of $\mathscr{Z}^\theta$ are given by
\begin{equation}
\begin{aligned}
	\label{eq:SHFmean}
	\bbE[\mathscr{Z}^\theta_{s,t}(\dd x, \dd y)]
	&= \tfrac{1}{2} \, g_{\frac{1}{2}(t-s)}(y-x) \, \dd x \, \dd y \,, \\
	\bbcov[\mathscr{Z}^\theta_{s,t}(\dd x, \dd y), \mathscr{Z}^\theta_{s,t}(\dd x', \dd y')]
	&= \tfrac{1}{2} \, K_{t-s}^\theta(x,x'; y, y') \, \dd x \, \dd y \, \dd x' \, \dd y' \,,
\end{aligned}
\end{equation}
where $g$ denotes the heat kernel in $\R^2$, see \eqref{eq:gt},
and $K^\theta$ is an explicit kernel, see \eqref{eq:m2-lim}.
\end{theorem}
\begin{remark}
The covariance kernel $K_{t-s}^\theta(x,x'; y, y')$ was
first identified in \cite{BC98} $($see also \cite{CSZ19b}$)$
and is {\em logarithmically divergent} near the diagonals $x=x'$ or $y=y'$.
\end{remark}

\smallskip

We now briefly explain the proof strategy.  As noted before, the moments of $\mathscr{Z}^\theta$
identified in \cite{GQT21} and \cite{C21} grow too fast to uniquely characterize the law of
$\mathscr{Z}^\theta$.  The bounds given in these works suggest that the $n$-th moment is at most of order
$\exp(\exp(n^2))$, while our recent work \cite{CSZ22} gives a lower bound of $\exp(c n^2)$. Physical arguments on the Delta-Bose gas \cite{R99}
suggest that the growth should be $\exp(\exp(n))$. It may thus be surprising that we are still able to prove Theorem~\ref{th:main0}
and show that the limit is unique, without
criteria to uniquely identify the limit. Another prominent result of this nature, which gave us inspiration,
is the work
of Kozma \cite{Koz07} on the convergence of the three-dimensional loop erased random walk
with dyadic scaling of the lattice $2^{-N}\Z^3$.

The basic strategy is to show that the laws of $(\cZ^{\beta_N}_{N})_{N\in\N}$ form
a \emph{Cauchy sequence}, i.e.\
\begin{equation} \label{eq:close}
	\text{\emph{$\cZ^{\beta_N}_{N}$ and $\cZ^{\beta_M}_{M}$
	are close in distribution for large $N,M \in \N$}} \,.
\end{equation}
To accomplish this, we first construct a \emph{coarse-grained model}
$\mathscr{Z}_{\epsilon}^{(\cg)}(\,\cdot\, |\Theta)$, for each
$\eps\in (0,1)$, which is a function of a family $\Theta$
of coarse-grained disorder variables.
We then perform a coarse-graining approximation of the partition function
on the time-space scale $(\eps N, \sqrt{\eps N})$, which shows that
\emph{$\cZ^{\beta_N}_{N}$ can be approximated by
the coarse-grained model $\mathscr{Z}_{\epsilon}^{(\cg)}(\,\cdot\, |\Theta)$}
for a specific choice of coarse-grained disorder $\Theta=\Theta_{N, \eps}$ that depends
on $N$ and~$\eps$, with an approximation error which is small for small $\eps$ and large~$N$
(shown via second moment bounds). As a consequence, we finally prove \eqref{eq:close}
by showing that the coarse-grained models $\mathscr{Z}_{\epsilon}^{(\cg)}(\,\cdot\, |\Theta)$
with $\Theta=\Theta_{N, \eps}$ and $\Theta=\Theta_{M, \eps}$ are close
in distribution, for small $\eps > 0$ and large $N, M \in \N$ (shown via a Lindeberg principle).

\smallskip

We give a more detailed proof outline in Section~\ref{sec:outline}.
Let us just highlight here the key proof ingredients:
\begin{itemize}
\item[{\bf A.}] \emph{Coarse-Graining}, which leads to a \emph{coarse-grained model}
with the same structure as the original model, demonstrating a degree of self-similarity;
\item[{\bf B.}] \emph{Time-Space Renewal Structure},
which sheds probabilistic light on second moment computations and leads in the continuum limit
to the so-called {\it Dickman subordinator};
\item[{\bf C.}] \emph{Lindeberg Principle} for multilinear polynomials of {\em dependent} random
variables, which controls the effect of changing $\Theta$ in the coarse-grained model
$\mathscr{Z}_{\epsilon}^{(\cg)}(\,\cdot\, |\Theta)$;
\item[{\bf D.}] \emph{Functional Inequalities for Green's Functions} of multiple random walks on $\Z^2$,
which yield sharp higher moment bounds for the coarse-grained model.
\end{itemize}
This framework is robust enough that it can also be used to show convergence of
other approximations of SHE~\eqref{eq:2dSHE} to the Critical 2d Stochastic Heat Flow.

\begin{remark}[Mollified SHE]
The same proof steps \textbf{A}, \textbf{B}, \textbf{C}, \textbf{D} can be carried out
for the solution $u_\delta$ of the \emph{mollified SHE} \eqref{eq:2dSHE}, where
the space-time white noise $\dot{W}$ is mollified spatially on the scale $\delta$
and $\beta=\beta_\delta$ is chosen in the corresponding critical window,
that is $\beta^2_\delta=\tfrac{2\pi}{|\log \delta|} +\tfrac{\theta+o(1)}{(\log \delta)^2}$
$($cf.~\eqref{eq:betaN}$)$. A key point is that coarse-graining $u_\delta$ on the
mesoscopic scale leads to exactly \emph{the same coarse-grained model
$\mathscr{Z}_{\epsilon}^{(\cg)}(\,\cdot\, |\Theta)$ constructed in this paper},
just with a different family of coarse-grained disorder variables $\Theta = \Theta_{\delta,\epsilon}$.
This means that the solution $u_\delta$ of the mollified SHE
would converge as $\delta \downarrow 0$
to \emph{the same universal limit $\mathscr{Z}^\theta$ in Theorem~\ref{th:main0}}.
We will not carry out the details here since the paper is long enough.
\end{remark}

We remark that Clark has proved in~\cite{Cla21} an analogue of Theorem~\ref{th:main0} for
DPRE on the hierarchical diamond lattice, which is particularly useful for
renormalization analysis and can mimic Euclidean lattices of different dimensions as the lattice
parameters vary.  Furthermore, in \cite{Cla19a, Cla19b}, he
also constructed the continuum polymer measures and studied their properties. This raises
interesting questions as to whether similar results can be proved for DPRE on the Euclidean lattice,
where exact renormalization analysis is no longer available.
We point out that our work developed in parallel
to that of Clark, and our proof strategies share some common features, such as coarse-graining
and controlling distributional distances via a Lindeberg principle in our case vs.\
Stein's method in \cite{Cla21}, and showing that the laws of the partition functions form a Cauchy sequence.

\smallskip

Now that we have proved the existence of a unique limit  $\mathscr{Z}^\theta$ ---the
Critical 2d Stochstic Heat Flow ---
the next challenge will be to investigate its properties and characterize its law.

\begin{remark}[Alternative scaling]
The simple random walk on $\Z^2$ is $2$-periodic and
each component has variance $\frac{1}{2}$. As a consequence,
the diffusively rescaled partition functions $\cU_N(t,x)$ in \eqref{eq:field}
provide a discretization of a slightly modified SHE \eqref{eq:2dSHE}, namely
\begin{equation*}
	\partial_t \tilde u = \frac{1}{4}\Delta \tilde u + \sqrt{2} \, \beta \, \dot{W} \, \tilde u
\end{equation*}
(see \cite[Appendix~A.3]{CSZ22} for more details).
The SHE with the usual parameters in \eqref{eq:2dSHE} can be recovered via the change of variable
$\cU_N(t,\frac{x}{\sqrt{2}})$.
Therefore to describe a candidate solution of \eqref{eq:2dSHE}, 
we should consider the rescaled Critical $2d$ Stochastic Heat Flow given by
(recall \eqref{eq:scaling})
\begin{equation*}
	\widehat{\mathscr{Z}}^{\,\theta} \,:=\,
	\big( \mathscr{Z}_{s,t}^\theta\big( \dd\tfrac{ x}{\sqrt{2}} ,
	\dd\tfrac{ y}{\sqrt{2}} \big) \big)_{0 \le s \le t <\infty}
	\,\overset{d}{=}\,
	\big( 2\, \mathscr{Z}_{2s,2t}^{\theta+\log 2}( \dd x ,
	\dd y ) \big)_{0 \le s \le t <\infty} \,,
\end{equation*}
which is also normalized to have mean~$1$ rather than~$\frac{1}{2}$ (see \eqref{eq:SHFmean}).
\end{remark}

\subsection{Related literature} \label{sec:lit}
We next discuss the connection between our work and various results in the literature
and point out some future directions of research.

\subsubsection*{Singular SPDEs}
As explained in Section~\ref{sec:d=2}, the scaling limit $\mathscr{Z}^\theta$ in
Theorem~\ref{th:main0} can be interpreted as the solution of the $2$-dimensional SHE \eqref{eq:2dSHE} in
the critical window. For SHE, dimension $d=2$ marks the {\it critical dimension} in the language of singular
SPDEs and renormalisation group theory. To define a solution for singular SPDEs, such as SHE and KPZ
in \eqref{eq:2dSHE}-\eqref{eq:2dKPZ}, a standard
approach is to mollify the space-time noise $\dot{W}$ in space on the scale of $\eps$,
and then try to identify a scaling limit as $\eps\downarrow 0$. Discretizing
space-time by considering a lattice model, such as the DPRE that we study in this paper,
is just another way of removing the singularity on small scales (also known as ultraviolet cutoff).

All existing solution theories for singular SPDEs, including regularity
structures \cite{H13,H14}, paracontrolled distributions \cite{GIP15, GP17}, the
renormalization group approach \cite{Kup14}, or energy solutions \cite{GJ14}, do not apply at the
critical dimension. The only singular SPDEs
for which progress has been made in defining its solution at the critical dimension are
SHE and KPZ (via the Cole-Hopf transform). The phase
transition identified in \cite{CSZ17b} was unexpected, and
to the best of our knowledge no such transition has been established for
other singular SPDEs in the critical dimension. Theorem~\ref{th:main0} is thus the first result to define
a solution for a singular SPDE at the critical dimension and
for critical disorder strength.

In dimension $d=2$, recently there has also been significant progress in understanding the solution of
the \emph{anisotropic version} of the KPZ equation (aKPZ), which differs from
\eqref{eq:2dKPZ} in that
$|\nabla v|^2= (\partial_{x_1}v)^2+(\partial_{x_2}v)^2$ therein is replaced by
$(\partial_{x_1}v)^2- (\partial_{x_2}v)^2$. This case is also beyond the reach of existing solution
theories, and unlike the isotropic KPZ, it cannot be linearized via the Cole-Hopf transformation.
Cannizzaro, Erhard, and Sch\"onbauer \cite{CES21} regularized the aKPZ via a cutoff in
Fourier space, instead of
discretizing space and time or mollifying the noise on the spatial scale $\eps$ (all are ultraviolet cutoffs).
They showed that if the non-linear term $(\partial_{x_1}v)^2- (\partial_{x_2}v)^2$ is rescaled by a factor
$\lambda/\sqrt{|\log\epsilon|}$, then the solution of the regularized aKPZ is tight with non-trivial
limit points, which is the anisotropic analogue of \cite{CD20}. Very recently, Cannizzaro, Erhard,
and Toninelli \cite{CET21} succeeded in proving that the limit is in fact Gaussian and solves the
Edwards-Wilkinson equation, which is the anisotropic analogue of \cite{CSZ20, Gu20}. In contrast to
the isotropic case \eqref{eq:2dKPZ}, there is no phase transition in $\lambda$ for the aKPZ.
The same authors also
studied the aKPZ without scaling the non-linearity, and in a surprising result \cite{CET20a, CET20b},
they showed that the solution exhibits logarithmic superdiffusive behaviour.

In the supercitical dimensions $d\geq 3$, the transition between the weak and strong disorder phases
for the directed polymer is long known \cite{C17} and
has a natural counterpart for SHE and KPZ. In recent years, there have been many studies on the solutions
of SHE and KPZ via mollification, namely, analogues of the random fields $\cU_N$ and
$\cH_N$ defined in \eqref{eq:field}-\eqref{eq:field2}. These studies are all in the
weak disorder regime and are analgous to results in $d=2$,
see e.g.\ \cite{MU18, MSZ16, CN21, CNN20, CCM20, GRZ18, DGRZ20, LZ20}.

\subsubsection*{Coarse-graining}

The first step in our approach is to construct a coarse-grained model. Coarse-graining
has a long history in statistical mechanics and renormalisation theory.
In the framework of directed polymer models, coarse-graining has played a crucial role
in the studies by Lacoin \cite{L10} and Berger-Lacoin \cite{BL17} on free energy asymptotics,
which extended previous works in the literature of pinning models,
see \cite{Gi11},
from which we single out the fundamental work of Giacomin-Lacoin-Toninelli \cite{GLT10}.

In our analysis, we need a family of coarse-grained models which provide a \emph{sharp approximation}
of the partition function \emph{at the critical point}, while the works mentioned above used
coarse-grained models to provide upper bounds away from
the critical point. The need for a sharper approximation
creates several challenges, which lead to the refined estimates in Sections~\ref{sec:2ndmoment}
and~\ref{Sec:4MomCoarse} and the
development of the enhanced Lindeberg principle in Appendix~\ref{app:Lind}.

\subsubsection*{DPRE on hierarchical lattices}

In a series of papers \cite{Cla21, Cla19a, Cla19b}, Clark successfully
treated the directed polymer model on hierarchical diamond lattices at the ``critical dimension''
and in the critical window of disorder strength,
which contains an analogue of Theorem~\ref{th:main0} and more.
Due to their tree-like structure, hierarchical lattices are especially convenient for performing
exact renormalization group calculations that are typically intractable on the Euclidean lattice. By tuning
suitable parameters (such as the number of branches and
the number of segments along each branch),
hierarchical lattices can mimic Euclidean lattices with different spatial dimensions. When the
branch number equals the segment number, hierarchical lattices mimic $\Z^2$. For DPRE
on these lattices, Clark was able to prove in \cite{Cla21} the analogue of Theorem~\ref{th:main0}.

Exploiting the structure of hierarchical lattices, in \cite{Cla19a},
Clark was able to use the limiting partition functions obtained in \cite{Cla21}
to construct a continuum version of the polymer measure
and study its properties. Furthermore, in \cite{Cla19b}, he identified
an interesting conditional  Gaussian Multiplicative Chaos (GMC) structure among the continuum
polymer measures with different parameter $\theta$ (similar to $\theta$ in Theorem~\ref{th:main0}).
These results raise
interesting questions as to whether similar results can be obtained for DPRE on the Euclidean lattice.
In this respect, Theorem~\ref{th:main0} provides the starting point.

\subsubsection*{Continuum polymer measure}

A continuum version of the DPRE polymer measure in dimension $d=1$
was constructed in \cite{AKQ14a,AKQ14b}, exploiting the continuum limit of the point-to-point partition
functions. The same approach was
applied in \cite{CSZ16} to pinning models with tail exponent $\alpha \in (\frac{1}{2},1)$.
An essential feature of these constructions, as well as the one by Clark \cite{Cla19a}
in the hierarchical setting in the critical regime, is that
the continuum partition functions are random \emph{functions} of the polymer endpoints.
The same holds for DPRE in dimension $d=2$ in the
\emph{subcritical regime} $\beta_N \sim \hat\beta / \sqrt{\log N}$,
with $\hat\beta < \hat\beta_c = \sqrt{\pi}$,
where it was recently shown in \cite{G21+} that the
discrete polymer measure, diffusively rescaled,
converges to the law of Brownian motion.

The situation for DPRE in dimension $d=2$
\emph{in the critical window} is radically different, because the continuum
partition functions $\cZ^\theta_{z,t}(\dd x, \dd y)$ given in Theorem~\ref{th:main0}
are only random \emph{measures} and undefined pointwise. The point-to-plane partition
function $Z^{\beta_N, \omega}_N$ defined in \eqref{eq:paf} in fact converges to $0$ as
$N\to\infty$, as shown in \cite{CSZ17b}. For this reason,
constructing a continuum version of the polymer measure ---
or studying the scaling properties of the discrete polymer measure ---
started from a \emph{fixed} point, remains a significant challenge.
However, if we consider discrete polymer measures with the starting point
chosen uniformly from a ball on the diffusive scale, then the same proof strategy
as that for Theorem \ref{th:main0} should be applicable to show that the measures
converge to a continuum polymer measure starting from a ball, whose finite dimensional
distributions are uniquely determined.

\subsubsection*{Schr\"odinger operators with point interactions}

When the disorder $\omega$ is standard normal, a direct calculation shows that for $k\in\N$,
the $k$-th moment of the polymer partition function in \eqref{eq:paf}
is the exponential moment (with parameter $\beta^2$) of the total pairwise collision local time up to
time $N$ among $k$ independent random walks on $\Z^2$.
When $k=2$, by a classic result of Erd\"os and Taylor~\cite{ET60} (see also \cite{GS09}), the collision
local time rescaled by $1/\log N$ converges to an exponential random variable with parameter~$\pi$.
In the critical window we consider here, we have
$\beta_N=\hat\beta_c/\sqrt{\log N}$ with $\hat\beta_c=\sqrt{\pi}$, and hence the parameter of the
exponential moment matches exactly the parameter of the limiting exponential random variable, making
the moment analysis particularly delicate.

Via the Feynman-Kac formula, it can also be seen that the
$k$-th moment of the partition function is the solution of a discrete space-time parabolic
Schr\"odinger equation with a potential supported
on the diagonal (point interaction). In the continuum setting, there have been a number of studies on
the Schr\"odinger operator with point interactions (also called Delta-Bose gas) in dimension $d=2$
\cite{AGHKH05, AFHKKL92, DFT94, DR04}. Using ideas from these studies, especially the works
of Dell'Antonio-Figari-Teta \cite{DFT94} and of Dimock-Rajeev \cite{DR04},
Gu, Quastel, and Tsai \cite{GQT21} were able to compute asymptotically
all moments of the averaged solution of
the mollified SHE, which are analogues of the averaged polymer partition functions
$\cZ^{\beta_N}_{N;\, s,t}(\varphi, \psi) :=\iint \varphi(x) \, \psi(y) \,
\cZ^{\beta_N}_{N;\, s,t}(\dd x, \dd y)$
in \eqref{eq:rescZmeas}, with $\varphi$ and $\psi$ assumed to be in $L^2$ in \cite{GQT21}.
Previously, only the third moment had been obtained in \cite{CSZ19b}.
When $\varphi$ is a delta function, the moments of $\cZ^{\beta_N}_{N;\, s,t}(\varphi, \psi)$
diverge as $N\to\infty$, and the asymptotics of
the third moment has been investigated in \cite{Feng}. But all mixed moments of the form
$\E\big[\prod_{i=1}^n\cZ^{\beta_N}_{N;\, s,t}(\varphi_i, \psi_i)\big]$ converge if $\varphi_i$
are chosen to be distinct $\delta$ functions, which was shown recently by Chen in~\cite{C21}.

As an input to the Lindeberg principle mentioned in the proof sketch for Theorem~\ref{th:main0},
we need to bound the fourth moment of the coarse-grained model,
which approximates the original
partition function. The results from the Schr\"odinger operator literature and \cite{GQT21}
are not applicable in our setting, because they rely on explicit Fourier
calculations. We therefore develop an alternative and more robust approach based on
\emph{functional inequalities}
for Green's function of multiple random walks on $\Z^2$,
see Lemma~\ref{HLSineq}.
Instead of working with
$\varphi, \psi \in L^2$ as in \cite{GQT21}, we can work with \emph{weighted $L^p$--$L^q$ spaces with
$\frac{1}{p}+\frac{1}{q}=1$}. The choice of a weight
allows us to consider a wider class of boundary conditions, such as $\psi\equiv 1$
and $\varphi$ an approximate delta function,
and also to control spatial decay when the support of $\varphi$
and $\psi$ are far apart, all of which are needed in our proof.
See Section~\ref{Sec:functional} for more details.

\subsubsection*{Lindeberg principle}

A Lindeberg principle is said to hold when the law of a function
$\Phi$ of a family of random variables does not change much if the family of random variables is switched
to another family with some matching moments.
Lindeberg principles have been very powerful tools
in proving universality. The usual formulation such as in \cite{Cha06} requires the
family of random variables to be independent (or exchangeable), and $\Phi$ needs to have bounded
first three derivatives. This is not satisfied when $\Phi$ is a multilinear polynomial,
whose derivatives are unbounded. This case was addressed in \cite{R79, MOO10}
when the arguments are independent random variables (see  also \cite{CSZ17a}).

In the proof of Theorem~\ref{th:main0}, we need to deal with
a multilinear polynomial of \emph{dependent} random variables with a local form of dependence.
We formulate an extension of the Lindeberg principle to this setting in Appendix~\ref{app:Lind}.
Our calculations are inspired by a work of R\"ollin on Stein's method \cite{R13},
which is an analogue
of  \cite{Cha06} for a function $\Phi$ (with bounded first three derivatives)
of dependent random variables.

\subsection{Structure of the paper}
The rest of the paper is organized as follows.
\begin{itemize}
\item In Section~\ref{sec:outline}, we give a detailed proof outline.

\item In Section~\ref{sec:notation}, we introduce some basic notation and tools
that we need for the rest of the paper, which includes in particular
the polynomial chaos expansion and second moment
asymptotics for the partition function.

\item In Section~\ref{sec:CGmodel}, we define
the coarse-grained model $\mathscr{Z}_{\epsilon}^{(\cg)}(\,\cdot\, |\Theta)$
and the coarse-grained disorder $\Theta = \Theta_{N,\epsilon}$. Then in Section~\ref{sec:2ndmoment},
we show that $\mathscr{Z}_{\epsilon}^{(\cg)}(\,\cdot\, |\Theta_{N,\eps})$
provides a good $L^2$ approximation for the diffusively rescaled
partition functions $\cZ_N$ in \eqref{eq:rescZmeas}.

\item In Sections~\ref{Sec:functional}, \ref{Sec:MomCoarse}
and~\ref{Sec:4MomCoarse}, we derive key moment bounds for $\cZ_N$,
$\Theta_{N, \eps}$ and $\mathscr{Z}_{\epsilon}^{(\cg)}(\,\cdot\, |\Theta)$.

\item In Section~\ref{s:Thm1.3}, we wrap up the proof of our main results: Theorems~\ref{th:main0}
and~\ref{th:main1}.

\item In Appendix~\ref{app:Lind}, we formulate
an enhanced Lindeberg principle for multilinear polynomials of dependent random variables.
\end{itemize}

\subsection*{Notation}
We denote by $C_b(\R^d)$, resp.\ $C_c(\R^d)$, the space of bounded,
resp.\ compactly supported functions $\varphi: \R^d \to \R$.
The usual $L^p$ norms will be denoted by $\|\varphi\|_p$ for functions $\varphi: \R^d \to \R$
and by $\|X\|_{L^p}$ for random variables~$X$.
For notational simplicity, we will use $c, C, C', C''$
to denote generic constants, whose values may change from place to place.

\section{Proof outline}\label{sec:outline}

We elaborate in more detail our proof strategy for Theorem~\ref{th:main0},
especially the coarse-graining procedure. After reading the proof strategy,
to see how the pieces fit together more precisely,
we encourage the reader to go directly to Section~\ref{s:Thm1.1}
to read the proof of Theorems~\ref{th:main0} and~\ref{th:main1}. The proof is contingent on some earlier
results, such as Theorems~\ref{th:cg-main} and~\ref{th:cgmom}, but otherwise is mostly self-contained.

Recalling \eqref{eq:rescZmeas},
we just consider a single averaged partition function
\begin{equation*}
	\cZ_N := \cZ_{N, 0, 1}^{\beta_N}(\varphi, \psi) =
	\int_{\R^2\times\R^2} \varphi(x) \, \psi(y) \, \cZ_{N, 0, 1}^{\beta_N}(\dd x , \dd y ) \,,
\end{equation*}
for some $\varphi \in C_c(\R^2)$, $\psi\in C_b(\R^2)$, and $\beta_N=\beta_N(\theta)$
chosen as in \eqref{intro:sigma} for some fixed $\theta\in \R$.
To prove that $\cZ_N$ converges in distribution to a limit as claimed in Theorem~\ref{th:main0},
we will show that  the laws of $(\cZ_N)_{N\in\N}$ form a Cauchy sequence.

The starting point of our analysis is a polynomial chaos expansion for $\cZ_N$, which will be recalled in more detail in Section~\ref{sec:poly}. In short, by introducing the i.i.d.\ random variables
$$
\xi_N(n,z) := e^{\beta_N\omega(n,z)-\lambda(\beta_N)}-1, \qquad (n, z) \in \N\times \Z^2,
$$
which have mean $0$ and variance $\sigma_N^2$ as in \eqref{intro:sigma}, we can expand $\cZ_N$
as a multilinear polynomial in the $\xi_N$'s as follows:
\begin{equation}\label{chaos-intro}
\begin{split}
	& \cZ_{N}  \, = \, q^N_{0,N}(\varphi,\psi)  \,+\, \\
	& \frac{1}{N}
	\sum_{r=1}^{\infty} \!\!\!\!
	\sum_{\substack{ z_1, \ldots, z_r \in \Z^2\\
	0 < n_1 < \ldots < n_r < N}}\!\!\!\!\!\!\!\!
	q^N_{0,n_1}(\varphi ,z_1) \, \xi_N(n_1,z_1)
	\Bigg\{ \prod_{j=2}^r q_{n_{j-1}, n_j}(z_{j-1},z_j) \xi_N(n_j,z_j) \Bigg\}
	q^N_{n_r, N}(z_r, \psi)   \,,
\end{split}
\end{equation}
where $q_{m,n}(x,y):= \P(S_n = y|S_m = x)$ is the random walk transition kernel,
and $q^N_{m,n}(\varphi, z_1)$, $q^N_{m,n}(z_r, \psi)$, $q^N_{m,n}(\varphi, \psi)$
are the averages of $q_{m,n}(x,y)$ w.r.t.\ $\varphi(x/\sqrt{N})$, $\psi(y/\sqrt{N})$, or both
(see \eqref{eq:qNphi}-\eqref{eq:qNphipsi}).

Each term in
the sum in \eqref{chaos-intro} contains a sequence of disorder
variables $(\xi_N(n_j, z_j))_{1\leq j\leq r}$ linked by random walk transition kernels, and
different terms in the sum are $L^2$-orthogonal. We will see that when it comes to second
moment calculations, the sequence of points
$(n_1, z_1), \ldots, (n_r, z_r)$ can be interpreted as a {\em time-space renewal configuration}.
\medskip

Before explaining our proof strategy and ingredients, we first give a heuristic calculation that already shows
universality, namely that as $N\to\infty$, the limiting law of $\cZ_{N}$ in \eqref{chaos-intro} (if a unique limit exists)
does not depend on the law of the i.i.d.\ random variables $\xi_N(\cdot, \cdot)$ provided the first two moments are unchanged.
The heuristic is based on a Lindeberg principle, which will help to illustrate some key ideas in our proof.

\subsection*{A Heuristic Calculation} Let us write $\cZ_N(\xi_N)$ to emphasise the dependence on the i.i.d.\ family $\xi_N(\cdot, \cdot)$,
and let $\cZ_N(\eta_N)$ be defined similarly with $\xi_N$ replaced by an i.i.d.\ family $\eta_N$ with matching first two moments and finite
third moment. To show that $\cZ_N(\xi_N)$ and $\cZ_N(\eta_N)$ are close in law, it suffices to show that for any $f:\R\to\R$ with bounded
first three derivatives,
\begin{equation}\label{eq:flind}
\lim_{N\to\infty}|f(\cZ_N(\xi_N))-f(\cZ_N(\eta_N))|=0.
\end{equation}
This difference can be bounded by a Lindeberg principle. In particular, we can apply Theorem \ref{lem:lind3} to the case of i.i.d.\
random variables (the sums in \eqref{eq:I1bd}-\eqref{eq:I2bd} will only contain indices $k=l=m$ due to the i.i.d.\ assumption) to get the bound
\begin{equation}
|f(\cZ_N(\xi_N))-f(\cZ_N(\eta_N))|\leq C \Vert f'''\Vert_\infty \!\!\!\!\!\!\!\! \sum_{1\leq n\leq N, z\in \Z^2} \int_0^1 \E[|\partial_{(n,z)}\cZ(\xi^{(t)}_N)|^3] {\rm d}t, \label{eq:heur}
\end{equation}
where $\xi_N^{(t)}:= \sqrt{t}\, \xi_N + \sqrt{1-t}\, \eta_N$ interpolates between $\eta_N$ and $\xi_N$, and $\partial_{(n,z)}\cZ(\xi_N)$ denotes partial derivative w.r.t.\ $\xi_N(n,z)$.  Since $\cZ(\xi^{(t)}_N)$ is a multilinear polynomial in $\xi^{(t)}_N(\cdot, \cdot)$, it is easily seen from \eqref{chaos-intro} that
$$
\partial_{(n,z)}\cZ(\xi^{(t)}_N) = \frac{1}{N} \cZ(\varphi, (n,z)) \cZ((n,z), \psi),
$$
where $\cZ((n,z), \psi)$ is the point-to-plane partition function starting from the point $(n,z)$ and terminating at time $N$ with boundary condition $\psi$, and $\cZ(\varphi, (n,z))$ is the plane-to-point partition function with initial boundary condition $\varphi$ and terminating at the point $(n,z)$.
Since $\varphi$ has compact support, only $(n,z)$ on the diffusive scale ($n$ of order $N$ and $z\in \Z^2$ of order $\sqrt{N}$) contribute to the sum in \eqref{eq:heur}, and there are $N^2$ such terms. This sum is more than compensated by the factor $\frac{1}{N^3}$ from $\E[|\partial_{\xi_N(n,z)}\cZ(\xi_N)|^3]= \frac{1}{N^3} \E[|\cZ(\varphi, (n,z))|^3] \E[|\cZ((n,z), \psi)|^3]$, where we used the independence between $\cZ(\varphi, (n,z))$ and $\cZ((n,z), \psi)$. To deduce \eqref{eq:flind}, it suffices to show that the moment of the point-to-plane partition function
$\E[|\cZ((n,z), \psi)|^3] \ll N$ as $N\to\infty$, which holds by Remark \ref{r:2} below. 

This heuristic can be made rigorous using the results we establish in Section \ref{Sec:functional}. But this argument will not show that $\cZ_N(\xi_N)$ has a unique limit in law. For that, we need to define coarse-grained models and compare $\cZ_N(\xi_N)$, for different $N$, with the same coarse-grained model. We outline the proof strategy below, which contains many of the same ideas in the heuristic above, but in a more complicated setting.

\subsection*{A. Coarse-graining}

As a first step, for each $\eps\in (0,1)$, we approximate $\cZ_N$ in $L^2$ by a
coarse-grained model $\mathscr{Z}_{\epsilon}^{(\cg)}(\varphi, \psi |\Theta_{N,\eps})$,
which is a multi-linear polynomial in suitable coarse-grained disorder variables $\Theta_{N,\eps}$
and depends on $N$ only through $\Theta_{N,\eps}$.
The details will be given in Section~\ref{sec:CGmodel}. Here we give a sketch.

We partition  $\N \times \Z^2$ into \emph{mesoscopic} time-space boxes
\begin{equation}
	\cB_{\epsilon N}(\sfi,\sfa)
	\ := \ \underbrace{((\sfi-1)\epsilon N, \sfi \epsilon N]}_{\cT_{\epsilon N}(\sfi)}
	\times \underbrace{((\sfa-(1,1))\sqrt{\epsilon N},
	\sfa\sqrt{\epsilon N}]}_{\cS_{\eps N}(\sfa)}
	 \ \, \cap \ \ \Z^3_\even \,,
\end{equation}
where $(\sfi, \sfa) \in\N\times \Z^2$ is the mesoscopic time-space index of $\cB_{\epsilon N}(\sfi,\sfa)$,
which has temporal width $\eps N$ and
spatial side length $\sqrt{\eps N}$, and
$(\sfa-\sfb, \sfa] = (\sfa_1-\sfb_1,\sfa_1] \times (\sfa_2-\sfb_2,\sfa_2]$ for squares in $\R^2$. We then decompose the sum in \eqref{chaos-intro}
according to the sequence of mesoscopic time intervals
$\cT_{\eps N}(\sfi_1), \ldots, \cT_{\eps N}(\sfi_k)$ visited by the renewal configuration
$(n_1, z_1), \ldots, (n_r, z_r)$. For each $\cT_{\eps N}(\sfi_j)$, we then further decompose according to
the first and last mesoscopic spatial boxes $\cS_{\eps N}(\sfa_j), \cS_{\eps N}(\sfa_j')$ visited in this
time interval. This replaces the microscopic sum over $(n_1, z_1), \ldots, (n_r, z_r)$
in~\eqref{chaos-intro} by a mesoscopic sum over time-space
renewal configurations $(\sfi_1; \sfa_1, \sfa_1'), \ldots, (\sfi_k; \sfa_k, \sfa_k')$, which specify
the sequence of mesoscopic boxes $\cB_{\epsilon N}(\sfi_j,\sfa_j)$
and $\cB_{\epsilon N}(\sfi_j,\sfa_j')$ visited. See Figure~\ref{CG-intro} for an illustration.

Ideally, we would like to replace each random walk kernel $q_{n, m}(x,y)$ in \eqref{chaos-intro}
that connects two consecutive visited mesoscopic boxes
$\cB_{\epsilon N}(\sfi_j,\sfa_j')\ni (n,x)$ and $\cB_{\epsilon N}(\sfi_{j+1},\sfa_{j+1})\ni (m,y)$
by a corresponding heat kernel. Namely, by the local limit theorem \eqref{eq:llt}, replace $q_{n, m}(x,y)$ by
\begin{equation*}
	2 \, g_{\tfrac{1}{2}(\sfi_{j+1}-\sfi_j)\eps N}((\sfa_{j+1}-\sfa_j')\sqrt{\eps N})
	= \frac{2}{\eps N} g_{\tfrac{1}{2}(\sfi_{j+1}-\sfi_j)}(\sfa_{j+1}-\sfa_j') \,,
\end{equation*}
where the factor $2$ is due to periodicity.
With such replacements, given
a mesoscopic renewal configuration $(\sfi_1; \sfa_1, \sfa_1'), \ldots, (\sfi_k; \sfa_k, \sfa_k')$,
as we sum over compatible microscopic renewal configurations $(n_1, z_1), \ldots, (n_r, z_r)$
in \eqref{chaos-intro}, the contributions of $\xi_N(n,z)$ from each interval $\cT_{\eps N}(\sfi_j)$
would decouple, leading to a product of coarse-grained disorder variables of the form
\begin{equation}\label{eq:Thet1}
	\Theta_{N, \eps}(\sfi_j; \sfa_j, \sfa_j') := \frac{2}{\eps N}
	\sum_{r=1}^\infty \sum_{\substack{(n_1,z_1), \ldots, (n_r, z_r)\in \Z^3_{\rm even} \\
	z_1\in \cS_{\eps N}(\sfa_j), z_r\in \cS_{\eps N}(\sfa_j') \\
	n_1<\cdots <n_r, \, n_i \in \cT_{\eps N}(\sfi_j)}} \!\!\!\!\!\! \xi_N(n_1,z_1)
	\prod_{j=2}^r q_{n_{j-1}, n_j}(z_{j-1},z_j) \xi_N(n_j,z_j) \,,
\end{equation}
with consecutive coarse-grained disorder variables $\Theta_{N, \eps}(\sfi_j; \sfa_j, \sfa_j')$ and
$\Theta_{N, \eps}(\sfi_{j+1}; \sfa_{j+1}, \sfa_{j+1}')$ linked by the heat kernel
$g_{\tfrac{1}{2}(\sfi_{j+1}-\sfi_j)}(\sfa_{j+1}-\sfa_j')$
(we absorbed the factor $\frac{2}{\epsilon N}$ into \eqref{eq:Thet1}).
This would give our desired coarse-grained model
$\mathscr{Z}_{\epsilon}^{(\cg)}(\varphi, \psi |\Theta_{N,\eps})$.

\begin{figure}
\hskip -1.5cm
\begin{tikzpicture}[scale=0.4]
\draw[ystep=1cm, xstep=2,gray,very thin] (0,0) grid (2, 10);
\draw[ystep=1cm, xstep=2,gray,very thin] (4,0) grid (6, 10);
\draw[ystep=1cm, xstep=2,gray,very thin] (10,0) grid (12, 10);
\draw[ystep=1cm, xstep=2,gray,very thin] (14,0) grid (16, 10);
\draw[ystep=1cm, xstep=2,gray,very thin] (20,0) grid (22, 10);
\node at (3,5) {{$\cdots$}}; \node at (8,5) {{$\cdots$}}; \node at (18,5) {{$\cdots$}}; \node at (24,5) {{$\cdots$}};
\node at (1,-0.5) {\scalebox{2.0}{$\leftrightarrow$}}; \node at (-0.3, 0.5) {\scalebox{1.5}{$\updownarrow$}};
\node at (1,-1) {\scalebox{0.8}{$\epsilon N$}}; \node at (-1.5, 0.5) {\scalebox{0.8}{$\sqrt{\epsilon N}$}};
\node at (7.5,5.9) {\scalebox{1.5}{$\longleftarrow$}}; \node at (8.5,5.9) {\scalebox{1.5}{$\longrightarrow$}};
\node at (17.5,5.9) {\scalebox{1.5}{$\longleftarrow$}}; \node at (18.5,5.9) {\scalebox{1.5}{$\longrightarrow$}};
\node at (12.7,5.9) {\scalebox{1.5}{$\leftarrow$}}; \node at (13.3,5.9) {\scalebox{1.5}{$\rightarrow$}};
\node at (8.1,6.7) {\scalebox{1.0}{$\geq K_\epsilon$}};
\node at (18.1,6.7) {\scalebox{1.0}{$\geq K_\epsilon$}};
\node at (13.1,6.7) {\scalebox{0.9}{$\leq \!K_\epsilon$}};
\draw (0, 5)  circle [radius=0.1];
\draw (12, 2) circle [radius=0.1]; \draw (6, 5) circle [radius=0.1];
\draw (16, 2) circle [radius=0.1]; \draw (22, 5) circle [radius=0.1];
\draw[thick] (0,5) to [out=30,in=150] (6,5);
\draw[thick] (6, 5) to [out=-40, in=170] (12, 2);
\draw[thick, dashed] (11.6,1.5) to [out=-30,in=-150] (14.2, 1.5);
\draw[thick] (16, 2) to [out=45,in=190] (22, 5);
\draw[black!70] (9.7, -0.3) -- (16.3, -0.3) -- (16.3, 4.2) -- (9.7, 4.2) --(9.7,-0.3);
\begin{pgfonlayer}{background}
 \fill[color=gray!10] (4,4) rectangle (6,7);
 \fill[color=gray!10] (10,0) rectangle (12,4);
  \fill[color=gray!10] (14,0) rectangle (16,4);
 \fill[color=gray!10] (20,3) rectangle (22,6);
 \end{pgfonlayer}
\end{tikzpicture}
\caption{An illustration of the chaos expansion for the coarse-grained model \eqref{CGchaos-intro}.
The solid laces represent heat kernels linking consecutively
visited mesoscopic time-space boxes. The grey blocks represent the regions defining the
coarse-grained disorder variables $\Theta_{N, \eps}$.
The double block in the middle represents a coarse-grained disorder
variable $\Theta_{N, \eps}(\vec \sfi, \vec\sfa)$ visiting two mesoscopic time
intervals $\cT_{\eps N}(\sfi)$ and $\cT_{\eps N}(\sfi')$ with
$|\sfi'-\sfi|\leq K_\eps =(\log \frac{1}{\eps})^6$ and cannot be decoupled.}
\label{CG-intro}
\end{figure}
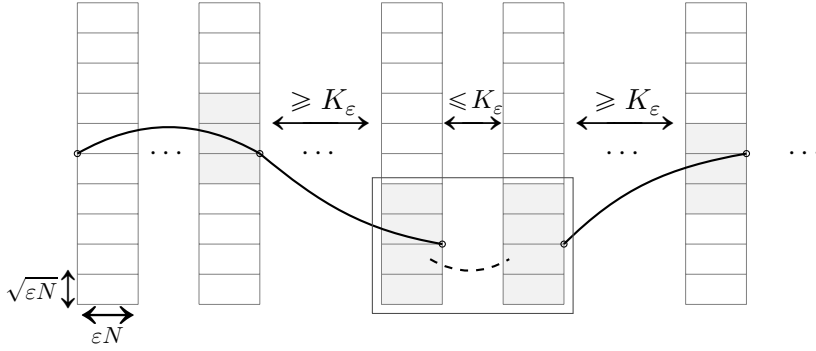

Unfortunately, this ideal procedure \emph{does not produce a sharp approximation of
the partition function $\cZ_N$ in \eqref{chaos-intro}}. Indeed,
the kernel replacement
\begin{equation}\label{eq:kerrep}
	q_{n,m}(x, y) \rightsquigarrow \frac{2}{\eps N} g_{\tfrac{1}{2}(\sfi_{j+1}-\sfi_j)}(\sfa_{j+1}-\sfa_j')
\end{equation}
induces an $L^2$-error, and
this error is small (in the sense that it vanishes as $\eps\downarrow 0$, uniformly in large $N$)
\emph{only if $\sfi_{j+1}-\sfi_j$ is sufficiently large} (we will
choose it to be larger than $K_\eps= (\log \frac{1}{\eps})^6$) and
\emph{$|\sfa_{j+1}-\sfa_j'|$ is not too large on the diffusive scale}
(we will choose it to be smaller than
$M_\eps \sqrt{\sfi_{j+1}-\sfi_j}$ with $M_\eps =\log \log \frac{1}{\eps}$).
We address this issue as follows.

The first crucial observation is that, modulo a small $L^2$ error,
microscopic renewal configurations
$(n_1, z_1)$, \ldots, $(n_r, z_r)$  in \eqref{chaos-intro}
\emph{cannot visit three or more mesoscopic
time intervals $\cT_{\eps N}(\sfi_j)$, $\cT_{\eps N}(\sfi_{j+1})$, and
$\cT_{\eps N}(\sfi_{j+2})$ with both $\sfi_{j+1}-\sfi_j\leq K_\eps$ and $\sfi_{j+2}-\sfi_{j+1}\leq K_\eps$}
(see Lemma~\ref{th:no3} below). Furthermore, with a small $L^2$ error,
we can also enforce a diffusive truncation $|\sfa_{j+1}-\sfa_j'| \le M_\eps \sqrt{\sfi_{j+1}-\sfi_j}$
(see Lemma~\ref{th:diff} below).
We will then make the random walk/heat kernel replacement \eqref{eq:kerrep}
\emph{only between mesoscopic boxes
$\cB_{\epsilon N}(\sfi_j,\sfa_j')\ni (n,x)$ and $\cB_{\epsilon N}(\sfi_{j+1},\sfa_{j+1})\ni (m,y)$
that satisfy the constraint $\sfi_{j+1}-\sfi_j > K_\eps$}.

After such kernel replacements, what are left between the heat kernels decouple and appear as a product of
two types of coarse-grained disorder variables:
\begin{itemize}
\item one type is as given in \eqref{eq:Thet1}, which visits a single
mesoscopic time interval $\cT_{\eps N}(\sfi)$;
\item another type visits two mesoscopic time intervals $\cT_{\eps N}(\sfi)$
and $\cT_{\eps N}(\sfi')$, with $\sfi' - \sfi \le K_\epsilon$: we denote it by
$\Theta_{N, \eps}(\vec \sfi, \vec \sfa)$ with $\vec\sfi=(\sfi, \sfi')$
and $\vec \sfa=(\sfa, \sfa')$, where $\sfa$ identifies the first mesoscopic spatial box
visited in the time interval $\cT_{\eps N}(\sfi)$, while $\sfa'$ identifies the
last mesoscopic spatial box visited in the time interval $\cT_{\eps N}(\sfi')$ (see \eqref{eq:Theta}).
\end{itemize}
This leads to the actual coarse-grained model we will work with:
\begin{equation}\label{CGchaos-intro}
\begin{aligned}
	& \mathscr{Z}_{\epsilon}^{(\cg)}(\varphi,\psi|\Theta)
	 := \tfrac{1}{2} \, g_{\frac{1}{2}}(\varphi, \psi) \,  + \, \\
	 &\ \frac{\epsilon}{2} \!\!
	 \sum_{r=1}^{(\log\frac{1}{\epsilon})^2}
	 \!\!\!\!\!\! \sum_{\substack{(\vec\sfi_1, \ldots, \vec\sfi_r)  \\ (\vec\sfa_1, \ldots, \vec\sfa_r)}}
	\!\!\!\!\!\!   g_{\frac{1}{2}\sfi_1}(\varphi_\epsilon, \sfa_1)  \Theta(\vec\sfi_1, \vec\sfa_1)
	\Bigg\{ \prod_{j=2}^r  g_{\frac{1}{2}(\sfi_j -\sfi_{j-1}')} (\sfa_j -\sfa_{j-1}')
	\Theta(\vec\sfi_j, \vec\sfa_j) \Bigg\} g_{\frac{1}{2}(\frac{1}{\eps}- \sfi_r')}(\sfa_r', \psi_\eps),
\end{aligned}
\end{equation}
where $\varphi_\eps$ and $\psi_\eps$ are averaged versions of $\varphi$ and $\psi$ on the spatial scale
$\sqrt{\eps}$, while $g_{\sfi/2}(\varphi_\eps, \sfa)$, $g_{\sfi/2}(\sfa', \psi_\eps)$,
$g_{\sfi/2}(\varphi_\eps, \psi_\eps)$ are averages of the heat kernel $g_{\sfi/2}(\sfa-\sfa')$
w.r.t.\ $\varphi_\eps$, $\psi_\eps$, or both.

In the sum in \eqref{CGchaos-intro}, we have hidden the various constraints on the
mesoscopic time-space variables for simplicity
(see \eqref{eq:Zcg-gen} for the complete definition).
Also note that in \eqref{CGchaos-intro} we denote by $\Theta = (\Theta(\vec\sfi,\vec\sfa))$
a generic family of coarse-grained disorder variables; in order to approximate
the averaged partition function $\cZ_N$, we simply set $\Theta=\Theta_{N, \eps}$.

\begin{remark}[Self-similarity]
The coarse-grained model $\mathscr{Z}_{\epsilon}^{(\cg)}(\varphi,\psi|\Theta)$
in \eqref{CGchaos-intro}
has the same form as the original partition function $\cZ_N$
in \eqref{chaos-intro}, with $1/\eps$ in place of $N$, $\Theta_{N,\eps}$ in place of $\xi_N$, and the
heat kernel $g_{\sfi/2}$ in place of the random walk kernel $q_n$.
This shows a remarkable degree of self-similarity: coarse-graining
retains the structure of the model.
\end{remark}

\subsection*{B. Time-Space Renewal Structure}

Once we have defined precisely
the coarse-grained model $\mathscr{Z}_{\epsilon}^{(\cg)}(\varphi,\psi |\Theta_{N,\eps})$,
see Section~\ref{sec:CGmodel}, we need to show
that it indeed provides a good $L^2$ approximations of the original
partition function $\cZ_N$, in the following sense:
\begin{equation}\label{eq:cgapp}
	\lim_{\eps\downarrow 0} \ \limsup_{N\to\infty} \
	\big\| \mathscr{Z}_{\epsilon}^{(\cg)}(\varphi,\psi|\Theta_{N,\epsilon}) - \cZ_N \big\|_{L^2}^2
	\,=\, 0 \,.
\end{equation}
This approximations will
be carried out in Section~\ref{sec:2ndmoment}, where we
rely crucially on the time-space renewal interpretation of the sum in \eqref{chaos-intro}, which in
the continuum limit with $N\to\infty$ leads to the so-called Dickman subordinator \cite{CSZ19a}.
This will be reviewed in Section~\ref{sec:Dickman}.

\subsection*{C. Lindeberg Principle}

In view of \eqref{eq:cgapp},
given $\eps>0$ small, we can approximate $\cZ_N$ by
$\mathscr{Z}_{\epsilon}^{(\cg)}(\varphi,\psi|\Theta_{N,\epsilon})$, where the $L^2$ error
is uniform in large $N$ and tends to $0$ as $\eps\downarrow 0$.
To prove that the laws of $(\cZ_N)_{N\in\N}$ form a Cauchy sequence,
it then suffices to show that given $\eps>0$ we can bound the distributional distance
between $\mathscr{Z}_{\epsilon}^{(\cg)}(\varphi,\psi|\Theta_{M,\epsilon})$ and
$\mathscr{Z}_{\epsilon}^{(\cg)}(\varphi,\psi|\Theta_{N,\epsilon})$ uniformly in $M\geq N$ large,
and furthermore, this bound can be made arbitrarily small by choosing $\eps>0$ sufficiently small.
This would then complete the proof that $\cZ_N$ converges in distribution to a unique limit.

The control of the distributional distance is carried out via a Lindeberg principle for
the coarse-grained model $\mathscr{Z}_{\epsilon}^{(\cg)}(\varphi,\psi|\Theta_{N, \eps})$,
which is a multilinear polynomial in the family of coarse-grained disorder variables
$\Theta_{N, \eps}=\{\Theta_{N, \eps}(\vec \sfi, \vec \sfa)\}$. We note that
$\Theta_{N, \eps}((\sfi, \sfi'), (\sfa, \sfa'))$ and $\Theta_{N, \eps}((\sfj, \sfj'), (\sfb, \sfb'))$ have
non-trivial dependence if $(\sfi, \sfa)$ or $(\sfi', \sfa')$ coincides with either $(\sfj, \sfb)$ or
$(\sfj', \sfb')$. We thus need a Lindeberg principle for multilinear polynomials
of \emph{dependent} random variables, which we formulate in Appendix~\ref{app:Lind} and is of
independent interest.

\subsection*{D. Functional Inequalities for Green's Functions}

To successfully apply the Lindeberg principle, we need to control
the second and fourth moments of the coarse-grained disorder variables $\Theta_{N, \eps}$.
We also need to control the {\em influence} of each $\Theta_{N, \eps}$,
which boils down to bounding the fourth moment of the coarse-grained model
$\mathscr{Z}_{\epsilon}^{(\cg)}(\varphi,\psi|\Theta_{N, \eps})$,
with the choice of boundary
conditions $\psi\equiv 1$ and $\varphi(x)=\frac{1}{\eps} \ind_{|x|\leq \sqrt{\eps}}$.

The moment bounds on $\Theta_{N, \eps}$ and
$\mathscr{Z}_{\epsilon}^{(\cg)}(\varphi,\psi|\Theta_{N, \eps})$ are technically the most delicate parts
of the paper, especially since we need to allow $\varphi(x)=\frac{1}{\eps} \ind_{|x|\leq \sqrt{\eps}}$ and
$\psi\equiv 1$. Since the structure of $\Theta_{N, \eps}$
is similar to an averaged partition function,
we will  first derive general moment bounds on the averaged partition function
$\cZ_N$ in Section~\ref{Sec:functional}.
The fourth moment bound on $\Theta_{N, \eps}$ then follows as
a corollary in Section~\ref{Sec:MomCoarse}.

The approach we develop is different from the methods employed in \cite{GQT21}
to bound the moments of the averaged solution of the mollified SHE. Our approach is based on
functional inequalities for the Green's function of random walks
(see Lemma~\ref{HLSineq}) and it is
robust enough to be applied also to the coarse-grained model defined in \eqref{CGchaos-intro}, which
will be carried out Section~\ref{Sec:4MomCoarse}.

\section{Notation and tools}\label{sec:notation}
In this section, we introduce some basic notation and tools, including the polynomial chaos expansion
for the partition function, random walk estimates, the renewal
interpretation for the second moment of partition functions and the Dickman subordinator
that arises in the continuum limit.

\subsection{Random walk and disorder}\label{Intro-notation}
As in Section~\ref{sec:intro},
let $(S = (S_n)_{n\ge 0}, \P)$ be the simple symmetric random walk on $\Z^2$,
whose transition kernel we denote by
\begin{equation} \label{eq:q}
	q_n(z) := \P(S_n = z) \,, \qquad
	q_{m, n}(x,z) := q_{n-m}(z-x) =
	\P(S_n = z \,|\, S_m = x) \,.
\end{equation}
Let $(\omega = (\omega(n,z))_{n\in\N, z\in\Z^2}, \bbP)$ be the disorder,
given by a family of i.i.d.\
random variables with zero mean, unit variance and locally finite exponential moments,
see \eqref{eq:lambda}.

The expected overlap between two independent walks is (see \cite[Proposition~3.2]{CSZ19a})
\begin{equation} \label{eq:RN}
\begin{gathered}
	R_N := \sum_{n=1}^N \sum_{z\in\Z^2} q_n(z)^2
	= \sum_{n=1}^N q_{2n}(0) = \frac{\log N}{\pi} + \frac{\alpha}{\pi} + o(1)  \\
	\textstyle \text{with} \qquad \alpha := \gamma + \log 16 - \pi \simeq 0.208 \,, \qquad
	\gamma := -\int_0^\infty e^{-u} \, \log u \, \dd u \simeq 0.577 \,.
\end{gathered}
\end{equation}
Note that $R_N$ is the expected number of collisions up to time~$N$ between
two independent copies of the random walk $S$ when both
start from the origin.
Also note that $\gamma$ is the Euler-Mascheroni constant.
We further define
\begin{equation} \label{eq:u}
	u(n) := \sum_{x\in\Z^2} q_n(x)^2 = q_{2n}(0) \sim \frac{1}{\pi} \cdot \frac{1}{n}
	\qquad \text{as } n\to\infty \,,
\end{equation}
where the asymptotic behavior follows by the
local limit theorem, see \eqref{eq:llt} below.

In order to deal with the periodicity of simple random walk, we set
\begin{equation}\label{eq:Zdeven}
	\Z^d_\even := \{z = (z_1, \ldots, z_d) \in \Z^d: \ z_1 + \ldots + z_d \text{ is even}\} \,.
\end{equation}
Given $x \in\R^d$ with $d \ge 2$, we denote by  $\ev{x}$
the point in $\Z^d_\even$ closest to~$x$ (fix any convention to break the tie if $\ev{x}$ is not unique). More explicitly, we have
\begin{equation} \label{eq:B}
	\ev{x} = v \in \Z^d_\even \quad \ \iff \quad \
	x \in B(v) := \big\{ x\in\R^d: \ |x_1 - v_1| + \ldots + |x_d - v_d| < 1 \big\} \,.
\end{equation}
For $s \in \R$ it is convenient to define the even approximation $\ev{s} \in 2\Z$ by
\begin{equation}\label{eq:evs}
	\ev{s} := 2 \, \bigg\lfloor \frac{s}{2} \bigg\rfloor \,.
\end{equation}

\subsection{Partition functions at criticality}
The point-to-point partition functions $Z_{M,N}^\beta(w,z)$
were defined in \eqref{eq:Zab}. We mainly consider the case $M=0$, for which we write
\begin{equation} \label{eq:ZN}
\begin{split}
	Z_N^\beta(w,z) &=
	\E\bigg[ e^{\sum_{n=1}^{N-1} \{\beta \omega(n,S_n) - \lambda(\beta)\}}
	\, \ind_{S_N = z} \,\bigg|\, S_0 = w \bigg]  \,.
\end{split}
\end{equation}

The field of diffusively rescaled partition functions
$\mathcal{Z}^{\beta}_{N; s,t}(\dd x, \dd y)$
was introduced in \eqref{eq:rescZmeas}.
In the special case $s=0$ we simply write:
\begin{equation*}
	\mathcal{Z}^{\beta}_{N, t}(\dd x, \dd y) :=
	\frac{N}{4} \, Z_{\ev{Nt}}^{\beta}(\ev{\sqrt{N} x},
	\ev{\sqrt{N} y}) \, \dd x \, \dd y \,,
\end{equation*}
where we recall that $\dd x \, \dd y$ denotes the Lebesgue measure on $\R^2 \times \R^2$.
We next define {\em averaged partition functions}
$\mathcal{Z}^{\beta}_{N, t}(\varphi, \psi)$ for suitable $\varphi, \psi: \R^2 \to \R$:
\begin{equation} \label{eq:ZNav}
\begin{split}
	\mathcal{Z}^{\beta}_{N, t}(\varphi,\psi)
	&:=  \iint_{\R^2 \times \R^2} \varphi(x)
	\, \mathcal{Z}^{\beta}_{N, t}(\dd x , \, \dd y) \, \psi(y)	\\
	& =  \frac{1}{4N} \iint_{\R^2 \times \R^2} \varphi(\tfrac{x}{\sqrt{N}})
	\ Z_{\ev{Nt}}^{\beta}(\ev{x}, \ev{y})
	\ \psi(\tfrac{y}{\sqrt{N}}) \, \dd x \, \dd y  \,.
\end{split}
\end{equation}

We can rewrite the integrals in \eqref{eq:ZNav} as sums.
For a locally integrable function $\varphi: \R^2 \to \R$,
we define $\varphi_N : \Z^2_\even \to \R$ as the average of $\varphi(\frac{\cdot}{\sqrt{N}})$ over cells
$B(v) \subseteq \R^2$, see \eqref{eq:B}:
\begin{equation}\label{eq:phiNpsiN}
	\varphi_N(v) := \frac{1}{|B(v)|} \int\limits_{B(v)}
	\varphi\big(\tfrac{x}{\sqrt{N}}\big) \, \dd x =
	\frac{1}{2} \int\limits_{\{|x_1 - v_1| + |x_2 - v_2| < 1\}}
	\varphi\big(\tfrac{x}{\sqrt{N}}\big) \, \dd x \,.
\end{equation}
If we similarly define $\psi_N: \Z^2_\even \to \R$ given $\psi: \R^2 \to \R$,
we can rewrite the second line
of \eqref{eq:ZNav} as a sum over the points $v=\ev{x}, w= \ev{y} \in \Z^2_\even$ as follows:
\begin{equation} \label{eq:ZNavalt}
\begin{split}
	\mathcal{Z}^{\beta}_{N, t}(\varphi,\psi)
	& \,=\, \frac{1}{N} \, \sum_{v,w \,\in\, \Z^2_\even}
	\varphi_N(v)
	\; Z_{\ev{Nt}}^{\beta} ( v,w )  \; \psi_N(w) \,.
\end{split}
\end{equation}
\begin{remark}[Parity issue]
Let $\Z^d_\odd := \Z^d\backslash \Z^d_{\rm even}$.  If in \eqref{eq:ZNavalt}
we sum over $v,w \in \Z^2_\odd$,
we obtain an alternative ``odd version'' of the averaged partition function, which is
{\em independent of the ``even version''} because two simple random walks started
at even vs.\ odd sites can never meet. This explains why we enforce a parity restriction
in \eqref{eq:ZNavalt}.
\end{remark}

Finally, we recall the {\em critical window} of the disorder strength (inverse temperature)
that was introduced in \eqref{intro:sigma}.
Given the definition \eqref{eq:RN} of $R_N$,
for some fixed $\theta \in \R$,
we choose $\beta = \beta_N=\beta_N(\theta)$ such that
\begin{equation} \label{eq:sigma}
	\sigma_N^2 := \bbvar[e^{\beta_N \omega - \lambda(\beta_N)}]
	= e^{\lambda(2\beta_N)-2\lambda(\beta_N)} - 1
	= \frac{1}{R_N} \bigg(1 + \frac{\theta + o(1)}{\log N}\bigg) \,.
\end{equation}
We can spell out this condition more explicitly in terms of $\beta_N$
(see \cite[Appendix~A.4]{CSZ19a}):
\begin{equation} \label{eq:betaN}
\begin{split}
	\beta_N^2 & =
	\frac{1}{R_N} - \frac{\kappa_3}{(R_N)^{3/2}} +
	\frac{\theta/\pi + (\frac{3}{2} \kappa_3^2 - \frac{7}{12}\kappa_4 - \frac{1}{2})}{(R_N)^2}
	+ o\bigg(\frac{1}{(R_N)^2}\bigg) \\
	& = \frac{\pi}{\log N} - \frac{\kappa_3 \, \pi^{3/2}}{(\log N)^{3/2}}
	+ \frac{\pi(\theta - \alpha)
	+ \pi^2 (\frac{3}{2} \kappa_3^2 - \frac{7}{12}\kappa_4 - \frac{1}{2})}
	{(\log N)^2} +o\bigg(\frac{1}{(\log N)^2}\bigg)  \,,
\end{split}
\end{equation}
where $\kappa_3, \kappa_4$ are the disorder cumulants,
i.e.\ $\lambda(\beta) = \frac{1}{2}\beta^2 +
\frac{\kappa_3}{3!} \beta^3 + \frac{\kappa_4}{4!} \beta^4
+ o(\beta^4)$ as $\beta \downarrow 0$,
and $\alpha \simeq 0.208$ is as in \eqref{eq:RN}.
\emph{Henceforth we always set $\beta = \beta_N$.}

\subsection{Polynomial chaos expansion} \label{sec:poly}

We now recall the \emph{polynomial chaos expansion} of the partition function.
This is based on the following product expansion, valid for any set $A$ and
any family of real numbers $(h_n)_{n\in A}$ labelled by~$A$:
\begin{equation} \label{eq:product-expansion}
\begin{split}
	e^{\sum_{n \in A} h_n}
	& \ = \ \prod_{n \in A} \big(1 + (e^{h_n}-1) \big)
	\ = \ 1 + \sum_{\emptyset \ne B \subseteq A} \
	\prod_{n \in B} (e^{h_n}-1)  \,.
\end{split}
\end{equation}

If we apply \eqref{eq:product-expansion} to the partition function
$Z^{\beta_N}_{d,f}(x,y)$ in \eqref{eq:Zab}, by \eqref{eq:q} we obtain
\begin{equation*}
\begin{split}
	& Z^{\beta_N}_{d,f}(x,y)
	\,-\, q_{d,f}(x,y) \\
	& \quad \ \,=\,
	\E\bigg[ \Big( e^{\sum_{n=d+1}^{f-1}\sum_{z\in\Z^2} (\beta_N \omega(n,z)-\lambda(\beta_N))
	\ind_{S_n=z}} - 1 \Big)
	\, \ind_{S_f = y} \,\bigg|\, S_d=x \bigg] \\
	& \quad \ = \, \sum_{r=1}^{\infty} \ \sum_{\substack{d < n_1 < \ldots < n_r < f \\
	z_1, \ldots, z_r \in \Z^2}} \
	\E\Bigg[ \Bigg\{ \prod_{j=1}^r \Big(
	e^{(\beta_N \omega(n_j,z_j)-\lambda(\beta_N))\ind_{S_{n_j}=z_j}} - 1 \Big)
	\Bigg\}\, \ind_{S_f = y} \,\Bigg|\, S_d=x \Bigg] \,.
\end{split}
\end{equation*}
Recalling \eqref{eq:sigma}, we introduce a family $(\xi_N(n,z))_{(n,z) \in \Z^2}$
of i.i.d.\ random variables by
\begin{equation} \label{eq:xi}
\begin{gathered}
	\xi_N(n,z) := e^{\beta_N\omega(n,z)-\lambda(\beta_N)}-1 \\
	\text{so that} \qquad \bbE[\xi_N(n,z)] = 0 \,, \qquad
	\bbvar[\xi_N(n,z)] = \sigma_N^2 \,.
\end{gathered}
\end{equation}
These variables allow us to write
\begin{equation*}
	e^{(\beta_N \omega(n,z)-\lambda(\beta_N))\ind_{S_n=z}} - 1
	= (e^{\beta_N\omega(n,z)-\lambda(\beta_N)}-1) \, \ind_{S_n=z}
	= \xi_N(n,z) \,  \ind_{S_n=z} \,,
\end{equation*}
hence, by the Markov property for the random walk with kernel $q$,  we get
\begin{equation} \label{eq:Zpoly}
\begin{split}
	Z^{\beta_N}_{d,f}(x,y)  & \, = \, q_{d,f}(x,y) + \sum_{r=1}^{\infty} \
	 \sum_{\substack{d < n_1 < \ldots < n_r < f \\
	z_1, \ldots, z_r \in \Z^2}}  \\
	& \qquad q_{d,n_1}(x,z_1) \, \xi_N(n_1,z_1) \Bigg\{ \prod_{j=2}^r q_{n_{j-1}, n_j}(z_{j-1},z_j)
	\, \xi_N(n_j,z_j) \Bigg\} \, q_{n_j,f}(z_j,y)  \,,
\end{split}
\end{equation}
where $\prod_{j=2}^r (\ldots) := 1$ if $r=1$.
We have expressed the point-to-point partition function as a
multilinear polynomial (\emph{polynomial chaos}) in the independent random variables
$\xi_N(n,z)$.
\smallskip

A similar polynomial chaos representation holds for the \emph{averaged}
partition function $\mathcal{Z}^{\beta_N}_{N, t}(\varphi,\psi)$ given in \eqref{eq:ZNavalt}.
To simplify notation, it is convenient to define
an averaged version of the random walk transition kernel $q_{m,n}(x,y)$.
Given suitable $\varphi, \psi: \R^2 \to \R$, a time horizon $M \in (0,\infty)$,
and two points $(m, w), (n, z) \in \Z^3_{\rm even}$,
recalling $\varphi_N$ and $\psi_N$ from \eqref{eq:phiNpsiN},
we define
\begin{align}
	\label{eq:qNphi}
	q^N_{0,m}(\varphi, w) &\,:=\,
	\sum_{v\in\Z^2_\even} \varphi_N(v) \,
	q_{0,m}(v,w) \,, \\
	\label{eq:qNpsi}
	q^N_{n,  M}(z, \psi) &\,:=\,
	\sum_{w\in\Z^2_\even} q_{n, \ev{M}}(z,w)
	\, \psi_N(w)  \,, \\
	q^N_{0, M}(\varphi, \psi)
	& \,=\   \frac{1}{N} \sum_{v,w\in\Z^2_\even} \varphi_N(v) \, q_{0, \ev{M}}(v,w)
	\, \psi_N(w)   \,. \label{eq:qNphipsi}
\end{align}
Then \eqref{eq:Zpoly} yields the following polynomial chaos expansion for
$\mathcal{Z}^{\beta_N}_{N, t}(\varphi,\psi)$ in \eqref{eq:ZNavalt}:
\begin{equation} \label{eq:Zpolyav}
\begin{split}
	& \mathcal{Z}^{\beta_N}_{N, t}(\varphi,\psi)  \, = \,
	q^N_{0, Nt}(\varphi,\psi)  \ + \ \frac{1}{N} \sum_{r=1}^{\infty} \ \sum_{\substack{0 < n_1 < \ldots < n_r < Nt \\
	z_1, \ldots, z_r \in \Z^2}} \\
	& \qquad\quad  q^N_{0,n_1}(\varphi ,z_1) \, \xi_N(n_1,z_1) \cdot  \Bigg\{ \prod_{j=2}^r q_{n_{j-1}, n_j}(z_{j-1},z_j)
	\, \xi_N(n_j,z_j) \Bigg\} \cdot  q^N_{n_r, \ev{Nt}}(z_r, \psi) \,.
\end{split}
\end{equation}
As will be explained later,
when it comes to second moment calculations,
the time-space points $(n_1, z_1), \ldots, (n_r, z_r)$
in the sum can be interpreted as a time-space renewal configuration.

\subsection{Random walk estimates}
Let $g_t: \R^2 \to (0,\infty)$ denote the heat kernel on $\R^2$:
\begin{equation}\label{eq:gt}
	g_t(x) := \frac{1}{2\pi t} \, e^{-\frac{|x|^2}{2t}} \,,
	\qquad g_t(x,y) := g_t(y-x) \,,
\end{equation}
where, unless otherwise specified, we denote by $|\cdot|$ the Euclidean norm on $\R^d$.

The asymptotic behavior of the random walk transition kernel
$q_n(x) = \P(S_n = x)$ is given by the local central limit theorem:
as $n\to\infty$ we have, uniformly for $x\in\Z^2$,
\begin{equation}\label{eq:llt}
	\begin{aligned}
	q_n(x)
	& \,=\, \big(g_{\frac{n}{2}}(x) + O\big(\tfrac{1}{n^2}\big) \big)
	\, 2\,\ind_{(n, x)\in \Z^3_{\rm even}} \\
	&
	\,=\, g_{\frac{n}{2}}(x) \, e^{O\big(\tfrac{1}{n}\big) + O\big(\tfrac{|x|^4}{n^3}\big)}
	\, 2 \, \ind_{(n, x)\in \Z^3_{\rm even}} \,,
\end{aligned}
\end{equation}
where the two lines are two different variants of the local central limit theorem for the simple
symmetric random walk on $\Z^2$ given by Theorems 2.3.5 and 2.3.11  in \cite{LaLi10}.
We recall that $\Z^d_\even$ is defined in \eqref{eq:Zdeven},
the multiplicative factor $2$ comes from the periodicity of the simple random walk
$S_n = (S^{(1)}_n, S^{(2)}_n)$ on $\Z^2$,
while the factor $\frac{1}{2}$ in the time argument of the heat kernel
comes from the fact that $\E[S_n^{(i)} S_n^{(j)}] = \frac{n}{2} \, \ind_{i=j}$.
We also note that
\begin{equation} \label{eq:gtrans}
	g_{\frac{n}{2}}(x) = \tfrac{1}{N} \, g_{\frac{1}{2} \frac{n}{N}}(\tfrac{x}{\sqrt{N}})
	\qquad \forall n, N \in \N, \ \forall x \in \Z^2 \,.
\end{equation}

Similar to the averaged random walk kernels $q^N_{\cdot, \cdot}$
defined in \eqref{eq:qNphi}-\eqref{eq:qNphipsi}, given $\varphi \in L^1(\R^2)$,
$\psi \in L^\infty(\R^2)$, $t > 0$, and $a,b \in \R^2$, we define the averaged heat kernels
\begin{align} \label{eq:gstphi}
	g_t(\varphi,a) &\,:=\, \int_{\R^2} \varphi(x) \, g_{t}(a-x) \, \dd x \,, \\
	\label{eq:gstpsi}
	g_{t}(b, \psi) &\,:=\, \int_{\R^2} g_{t}(y-b) \, \psi(y) \, \dd y \,, \\
	\label{eq:gstphipsi}
	g_{t}(\varphi, \psi)
	&\,:=\, \int_{\R^2\times\R^2} \varphi(x) \, g_{t}(y-x) \, \psi(y) \, \dd x \, \dd y \,.
\end{align}
Recall $q^N_{0, Nt}(\varphi, \psi)$ from \eqref{eq:qNphipsi}.
By the local limit theorem \eqref{eq:llt}, recalling \eqref{eq:phiNpsiN}
and \eqref{eq:gtrans}, we have
\begin{align}
	\forall t > 0: \qquad
	\label{eq:gstphipsiconv}
	\lim_{N\to\infty} \, q^N_{0, Nt}(\varphi, \psi) & \,=\,
	\tfrac{1}{2} \, g_{\frac{t}{2}}(\varphi, \psi) \,,
\end{align}
where the prefactor $\frac{1}{2}$ is due to periodicity.

We will also need the following lemma, which allows us to replace a random walk transition kernel by a heat
kernel even if the time-space increments are perturbed.

\begin{lemma}\label{lem:ker}
Let $q_n(\cdot)$ be the transition kernel of the simple symmetric random walk on $\Z^2$,
see \eqref{eq:q}, and
let $g_t(\cdot)$ be  the heat kernel on $\R^2$, see \eqref{eq:gt}.
Then there exists $C\in (0, \infty)$ such that,
for all $n \in \N$ and for all $x \in \Z^2$ with $|x|\leq n^{\frac34}$, we have
\begin{equation}\label{eq:ker1}
	q_{n}(x)  \leq C \, g_{\frac{n}{2}} (x).
\end{equation}
Let $\rho_1, \rho_2>0$ and set $C := 2 e \, \rho_1 \, \rho_2$.
Then, given an arbitrary $m\in\N$, for
all $n_1, n_2 \in \N$ with $n_1 \ge m$
and $\frac{n_2}{n_1}\in [1/\rho_1, \rho_2]$, and for all $x_1, x_2\in \R^2$ with
$|x_1-x_2|\leq \sqrt{m}$, we have
\begin{equation}\label{eq:ker2}
	g_{\frac{n_1}{2}} (x_1) \leq  C g_{\rho_1 n_2} (x_2) =  \frac{C}{m} g_{\frac{\rho_1n_2}{m}}
	\Big(\frac{x_2}{\sqrt m}\Big).
\end{equation}
\end{lemma}

\begin{proof}
Let us prove \eqref{eq:ker1}: by the second variant of the local limit theorem in \eqref{eq:llt},
\begin{align*}
	q_{n_1}(x_1) & = \ind_{\{(n_1, x_1)\in \Z^3_{\rm even}\}}
	\,2\, g_{\frac{n_1}{2}}(x_1)
	\, \exp\Big\{ O\Big(\frac{1}{n_1}\Big) + O\Big( \frac{|x_1|^4}{n_1^3} \Big) \Big\}
	\leq C \, g_{\frac{n_1}{2}}(x_1) \,.
\end{align*}
We next prove \eqref{eq:ker2}: by the assumption $\frac{n_2}{n_1}\in [1/\rho_1, \rho_2]$, we have
\begin{align*}
	\frac{g_{\frac{n_1}{2}}(x_1)}{g_{\rho_1 n_2}(x_2)}
	= \frac{2\rho_1n_2}{n_1} \exp\Big\{\frac{|x_2|^2}{2\rho_1n_2} - \frac{|x_1|^2}{n_1}\Big\}
	\leq 2\rho_1\rho_2 \, \exp\Big\{\frac{|x_2|^2}{2n_1} - \frac{|x_1|^2}{n_1}\Big\} \leq
	2\rho_1\rho_2 \, e \,,
\end{align*}
where the last inequality holds because $|x_2|^2
\le 2 (|x_1|^2 + |x_2-x_1|^2) \le 2 |x_1|^2 + 2 m$ and $n_1 \ge m$ by assumption.
\end{proof}

\subsection{Renewal estimates and Dickman subordinator}\label{sec:Dickman}
We next present the time-space renewal process underlying the second moment calculations
for the partition function.
Under diffusive scaling, this leads to the so-called Dickman subordinator in the continuum limit.
This approach was developed in \cite{CSZ19a, CSZ19b}.

We first define a slight modification of the partition function $Z_{d,f}^\beta(x,y)$
in \eqref{eq:Zab}, where we ``attach'' disorder variables $\xi_N(n,z)$,
see \eqref{eq:xi},
at the boundary points $(d,x)$ and $(f,y)$ (which may coincide, if $d=f$):
\begin{equation} \label{eq:X}
	X^{\beta_N}_{d,f}(x,y) := \begin{cases}
	\xi_N(d,x) \, \ind_{\{y=x\}} & \text{if } f=d \\
	\rule{0pt}{1.2em}\xi_N(d,x) \,
	Z^{\beta_N}_{d,f}(x,y) \, \xi_N(f,y) & \text{if }f \ge d+1
	\end{cases} \,.
\end{equation}
Such quantities will appear as basic building blocks in our proofs. Note that
$\bbE[X^{\beta_N}_{d,f}(x,y)] = 0$. The second moment of $X^{\beta_N}_{d,f}(x,y)$
can be computed explicitly by the polynomial chaos expansion \eqref{eq:Zpoly}
and it can be expressed as follows:
\begin{equation}\label{eq:secondX}
	\bbE\big[ X^{\beta_N}_{d,f}(x,y)^2 \big] = \sigma_N^2 \, U_N(f-d,y-x) \,,
\end{equation}
where we recall that $\sigma_N^2=\mathbb{V}{\rm ar}(\xi_N(a,x))$,
and for $n\in\N_0 = \{0,1,2,\ldots\}$ and $x \in \Z^2$ we define
\begin{equation}\label{eq:U}
\begin{split}
	U_N(n,x) := \begin{cases}
	\displaystyle
	\ind_{\{x=0\}}
	& \text{ if } n = 0, \\
	\rule{0pt}{4.3em}
	\displaystyle
\begin{split}
	\sigma_N^2 \, q(n,x)^2 \,+\,
	& \sum_{r=1}^\infty
	 (\sigma_N^2)^{r+1} \sum_{\substack{0 < n_1 < \ldots < n_r < n\\ z_1,\ldots, z_r \in \Z^2}}
	\, q_{0,n_1}(0,z_1)^2 \, \times \\
	& \ \times \bigg\{ \prod_{j=2}^r q_{n_{j-1},n_j}(z_{j-1},z_j)^2 \bigg\}
	\, q_{n_r,n}(z_r,x)^2
\end{split}
	& \text{ if } n \ge 1.
	\end{cases}
\end{split}
\end{equation}

The quantity $U_N(n,x)$, which plays an important role throughout this paper,
admits a probabilistic interpretation as a \emph{renewal function}.
More precisely, let $(\tau^{(N)}_r, S^{(N)}_r)_{r\ge 0}$
denote the random walk (time-space renewal process) on $\N_0 \times \Z^2$
starting at $(0,0)$ and with one-step distribution
\begin{equation}\label{eq:tauS}
	\P\big(\tau_1^{(N)}=n, \, S_1^{(N)}=x\big) = \frac{q_n(x)^2}{R_N} \,
	\ind_{\{1,\ldots, N\}}(n) \,,
\end{equation}
where $R_N$ is the random walk overlap defined in \eqref{eq:RN}. Then we can write,
recalling \eqref{eq:sigma},
\begin{equation} \label{eq:Uren}
\begin{gathered}
	U_N(n,x) = \sum_{r = 1}^\infty (\lambda_N)^r \,
	\P\big(\tau_r^{(N)}=n, \, S_r^{(N)}=x\big) \\
	\text{where} \quad \ \lambda_N := \sigma_N^2 \, R_N
	= 1 + \frac{\theta + o(1)}{\log N}\,.
\end{gathered}
\end{equation}
When $\lambda_N=1$, we see that $U_N(n,x)$ is just the renewal function of
$(\tau^{(N)}_r, S^{(N)}_r)_{r\ge 0}$.
When $\lambda_N\neq1$, we can think of $U_N(n,x)$ as an exponentially weighted renewal function,
weighted according to the number of renewals.
Note that the first component $\tau^{(N)} = (\tau^{(N)}_r)_{r\ge 0}$ is a renewal process
with one-step distribution
\begin{equation}\label{eq:tau}
	\P\big(\tau_1^{(N)}=n\big) = \frac{u(n)}{R_N} \,
	\ind_{\{1,\ldots, N\}}(n) \,,
\end{equation}
where $u(n)=\sum_x q_n(x)^2$ is defined in \eqref{eq:u}.
Correspondingly, we can define
\begin{equation}\label{eq:Usingle}
	U_N(n) := \sum_{x\in\Z^2} U_N(n,x)
	= \sum_{r = 1}^\infty (\lambda_N)^r \,
	\P\big(\tau_r^{(N)}=n\big) \,.
\end{equation}

The asymptotic behaviors of $U_N(n,x)$ and $U_N(n)$
were obtained in \cite{CSZ19a}, exploiting the fact that $\tau^{(N)}$
is in the domain of attraction of the so-called \emph{Dickman subordinator},
defined as
the pure jump L\'evy process with L\'evy measure $\frac{1}{x} \, \ind_{(0,1)}(x) \, \dd x$.
More precisely, we have the following convergence result, which
is an extension of \cite[Proposition 2.2]{CSZ19a} from finite dimensional distribution
convergence to process level convergence.

\begin{lemma}\label{lem:Levyconv}
Let $(\tau^{(N)}_r, S^{(N)}_r)_{r\ge 0}$ be the
space-time random walk defined in \eqref{eq:tauS}.
Let $(Y_s)_{s\geq 0}$ be the so-called Dickman subordinator \cite{CSZ19a},
i.e.\ the pure jump L\'evy process with L\'evy measure $\frac{1}{t}\ind_{(0,1)}(t){\rm d}t$,
and let $V_s:= \frac{1}{2} W_{Y_s}$ where $W$ is an independent Brownian motion.
Then we have
the convergence in distribution
\begin{equation}\label{eq:YV}
	\Bigg(\frac{\tau^{(N)}_{\lfloor s\log N\rfloor}}{N},
	\frac{S^{(N)}_{\lfloor s\log N\rfloor }}{\sqrt{N}} \Bigg)_{s\geq 0}
	\ \Asto{N} \ (\bsY_s)_{s\geq 0} := (Y_s, V_s)_{s\geq 0},
\end{equation}
on the space of c\`adl\`ag paths equipped with the Skorohod topology.
\end{lemma}
\begin{proof} Denote
$\bsY^{(N)}_s=(Y^{(N)}_s, V^{(N)}_s):= \Big(\frac{\tau^{(N)}_{\lfloor s\log N\rfloor}}{N},
\frac{S^{(N)}_{\lfloor s\log N\rfloor }}{\sqrt{N}} \Big)$. The convergence of finite dimensional distributions
was already proved in \cite[Proposition 2.2]{CSZ19a}. We prove tightness
by verifying Aldous' tightness criterion \cite[Theorem 14.11]{Kal97}, namely that for
any bounded sequence of stopping times $\tau_N$ with respect to $(\bsY^{(N)}_s)_{s\geq 0}$
and any positive constants $h_N\downarrow 0$, we have
$\bsY^{(N)}_{\tau_N+h_N}-\bsY^{(N)}_{\tau_N}\to 0$ in probability as $N\to\infty$.
This follows immediately from the fact that the increments of $\bsY^{(N)}$
are i.i.d.\ and
$\bsY^{(N)}_{h_N}\to (0,0)$ in probability as $N\to\infty$.
\end{proof}

\smallskip

For $\theta \in (0,\infty)$, we define the exponentially weighted Green's function for
$\bsY = (\bsY_s)_{s\geq 0}$:
\begin{equation}\label{eq:G3}
\widehat G_\theta(t,x)  = \int_0^\infty e^{\theta s} \bsf_s(t,x) {\rm d}s,
\end{equation}
where $\bsf_s(\cdot, \cdot)$ is the density of the law of
$\bsY_s$ on $[0,\infty)\times\R^2$,
given that $\bsY_0=(0,0)$
(we take notation from \eqref{eq:YV}). It was shown in ~\cite{CSZ19a} that
\begin{equation}\label{eq:G2}
	\widehat G_\theta(t,x) := \widehat G_\theta(t) \, g_{\frac{t}{4}}(x) \,,
\end{equation}
where $g_\cdot(\cdot)$ is the heat kernel, see \eqref{eq:gt}, and
$\widehat G_\theta(t) :=\int_{\R^2} \widehat G_\theta(t,x) \dd x$ is closely related to the so-called Dickman
function in number theory. For $t\leq 1$, it can be computed explicitly as
\begin{equation}\label{eq:G1}
	\widehat G_\theta(t) = G_\theta(t)
	 := \int_0^\infty \frac{e^{(\theta-\gamma)s} \, s \, t^{s-1}}{\Gamma(s+1)} \, \dd s \,,
\end{equation}
with $\gamma$ as in \eqref{eq:RN} (see \cite{CSZ19a}\footnote{In \cite{CSZ19a}, there was no separate notation $\widehat G_\theta$ for the weighted Green's function, which might cause some confusion.}).
We will also denote $G_\theta(t,x) := G_\theta(t) \, g_{\frac{t}{4}}(x)$. Note that for $t\leq 1$, $G_\theta(t,x)$ and $G_\theta(t)$ are the continuum
analogues of $U_N(n,x)$ and $U_N(n)$, respectively. It is therefore no surprise that the asymptotics of $U_N$
will be expressed in terms of $G_\theta$, which we record below for later use.

In light of \eqref{eq:secondX}, it is convenient to define
\begin{equation}\label{eq:barU}
	\overline{U}_N(n,x) := \sigma_N^2 \, U_N(n,x) \,, \qquad
	\overline{U}_N(n) := \sigma_N^2 \, U_N(n) = \sum_{x\in\Z^2} \overline{U}_N(n,x) \,.
\end{equation}
Recalling \eqref{eq:U}, we can give a graphical representation for
$\overline{U}_N(b-a,y-x)$ as follows:
\begin{align}\label{U-diagram}
\overline{U}_N(b-a,y-x)
&\equiv \quad
\begin{tikzpicture}[baseline={([yshift=1.3ex]current bounding box.center)},vertex/.style={anchor=base,
    circle,fill=black!25,minimum size=18pt,inner sep=2pt}, scale=0.45]
\draw  [fill] (0, 0)  circle [radius=0.2];  \draw  [fill] ( 4,0)  circle [radius=0.2];
\draw [-,thick, decorate, decoration={snake,amplitude=.4mm,segment length=2mm}] (0,0) -- (4,0);
\node at (0,-1) {\scalebox{0.8}{$(a,x)$}}; \node at (4,-1) {\scalebox{0.8}{$(b,y)$}};
\end{tikzpicture} \\
&:= \sum_{k\geq 1 }\sumtwo{n_1,...,n_k}{x_1,...,x_k}
\begin{tikzpicture}[baseline={([yshift=1.3ex]current bounding box.center)},vertex/.style={anchor=base,
    circle,fill=black!25,minimum size=18pt,inner sep=2pt}, scale=0.45]
\draw  [fill] (0, 0)  circle [radius=0.2];  \draw  [fill] ( 2,0)  circle [radius=0.2]; \draw  [fill] (4,0)  circle [radius=0.2];
\draw  [fill] (8, 0)  circle [radius=0.2]; \draw  [fill] (10, 0)  circle [radius=0.2];
\draw [thick] (0,0)  to [out=45,in=135]  (2,0) to [out=45,in=135]  (4,0) to [out=45,in=180]  (4.5,0.3);
\draw [thick] (0,0)  to [out=-45,in=-135]  (2,0) to [out=-45,in=-135]  (4,0)  to [out=-45,in=180]  (4.5,-0.3);
\draw [thick] (7.5,0.3)  to [out=0,in=135]  (8,0) to [out=45,in=135]  (10,0);
\draw [thick] (7.5,-0.3) to [out=0,in=-135] (8,0)  to [out=-45,in=-135]  (10,0);
\node at (0,-1) {\scalebox{0.7}{$(a,x)$}}; \node at (2,-1) {\scalebox{0.7}{$(n_1,x_1)$}};
\node at (4,-1) {\scalebox{0.7}{$(n_2,x_2)$}}; \node at (8,-1) {\scalebox{0.7}{$(n_k,x_k)$}};
\node at (10,-1) {\scalebox{0.7}{$(b,y)$}};
\node at (6,0) {$\cdots$};
\end{tikzpicture} \notag
\end{align}
where in the second line we assign weights $q_{n'-n}(x'-x)$ to any solid line going from
$(n,x)$ to $(n',x')$ and we assign weight
$\sigma_N^2$ to every solid dot.

Recall that $\sigma_N^2 \sim \frac{\pi}{\log N}$, see \eqref{eq:sigma} and \eqref{eq:RN}.
We now rephrase some results from~\cite{CSZ19a}. Fix $T>0$.

\begin{itemize}
\item By \cite[Theorem~1.4]{CSZ19a}, for any fixed $\delta > 0$, as $N\to\infty$ we have
\begin{equation}\label{eq:asU1}
	\overline{U}_N(n) = \frac{\pi}{N} \, \big( G_\theta\big(\tfrac{n}{N}\big)
	+ o(1)\big) \qquad
	{\rm uniformly for } \ \ \delta N \le n \le T N  \,,
\end{equation}
and moreover there is $C < \infty$ such that
\begin{equation}\label{eq:asU1ub}
	\overline{U}_N(n) \le \frac{C}{N} \, G_\theta\big(\tfrac{n}{N}\big)
	\qquad \forall 0 < n \le T N \,.
\end{equation}

\item By \cite[Theorem~2.3 and~3.7]{CSZ19a},
for any fixed $\delta > 0$, as $N\to\infty$ we have
\begin{equation}\label{eq:asU2}
\begin{gathered}
	\overline{U}_N(n,x) = \frac{\pi}{N^2} \,
	\big( G_\theta\big(\tfrac{n}{N}, \tfrac{x}{\sqrt{N}}\big) + o(1)\big)
	\, 2 \, \ind_{(n,x) \in \Z^3_\even} \\
	\text{uniformly for} \ \ \delta N \le n \le T N \ \ \text{and} \ \
	|x| \le \tfrac{1}{\delta} \sqrt{N} \,.
\end{gathered}
\end{equation}
The prefactor $2$ is due to periodicity and,
moreover, there is $C < \infty$ such that
\begin{equation}\label{eq:asU2ub}
	\overline{U}_N(n,x) \le \frac{C}{N} \, \frac{1}{n} \, G_\theta\big(\tfrac{n}{N}\big)
	\qquad \forall 0 < n \le T N  \,, \ \forall x \in \Z^2 \,.
\end{equation}

\item
By \cite[Proposition~1.6]{CSZ19a}, for $t\in (0,1]$
the function $G_\theta(t)$ is $C^\infty$ and strictly positive,
and as $t \downarrow 0$ it has the following asymptotic behavior:
\begin{equation}\label{eq:Gas}
	G_\theta(t) = \frac{1}{t(\log\frac{1}{t})^2} \bigg\{ 1 + \frac{2\theta}{\log\frac{1}{t}}
	+ O\bigg(\frac{1}{(\log\frac{1}{t})^2}\bigg) \bigg\} \,,
\end{equation}
hence as $t \downarrow 0$
\begin{equation}\label{eq:G5}
\int_0^t G_\theta(s) {\rm d}s = \frac{1}{\log \frac{1}{t}} \bigg\{ 1 + \frac{\theta}{\log\frac{1}{t}}
	+ O\bigg(\frac{1}{(\log\frac{1}{t})^2}\bigg) \bigg\} .
\end{equation}
\end{itemize}

\begin{remark}\label{R:NtildeN}
In the proof of \eqref{eq:asU1}-\eqref{eq:asU2ub}, the case $T>1$ has to be treated differently from $T=1$. In \cite{CSZ19a}, the case $T>1$ was reduced to $T=1$ through a renewal decomposition and recursion (see \cite[Section 7]{CSZ19a}). Alternatively, we can reduce the case $T>1$ to $T=1$ by first setting $\widetilde N:=TN$, $\tilde \theta:= \theta +\log T+o(1)$ so that $\sigma_N^2=\sigma_N^2(\theta) = \sigma_{\widetilde N}^2(\tilde \theta)$ by their definitions in \eqref{eq:sigma}, and then applying \eqref{eq:asU1}-\eqref{eq:asU2ub} with $N$ replaced by $\widetilde N$, using the observation that $\frac{1}{T}G_{\theta+\log T}(\tfrac{t}{T}) = G_\theta(t)$.
\end{remark}

We will also need the following bound to complement \eqref{eq:asU2}.

\begin{lemma}\label{lem:UnGauss}
There exists $c\in (0,\infty)$ such that for all $\lambda \geq 0$ and $0\leq n\leq N$,
\begin{equation}\label{eq:asU3ub}
	\sum_{x\in \Z^2} \overline{U}_N(n,x) \, e^{\lambda |x|}
	\,\leq\, c\, e^{c\lambda^2 n} \, \overline{U}_N(n)   \,.
\end{equation}
\end{lemma}
\noindent
Note that by the Markov inequality and optimisation over $\lambda >0$, \eqref{eq:asU3ub} implies that
the probability kernel $\overline{U}_N(n,\cdot)/\overline{U}_N(n)$ has
Gaussian decay on the spatial scale $\sqrt{n}$.

\begin{proof}
Recall the definition of $U_N(n,x)$ from \eqref{eq:Uren}.
Conditioned on $\tau^{(N)}_1, \ldots, \tau^{(N)}_r$ with $\tau^{(N)}_r=n$, we can
write $S^{(N)}_r=\zeta_1+\cdots +\zeta_r$ for independent $\zeta_i$'s with
\begin{equation*}
	\p(\zeta_i=x) =\frac{q_{n_i}(x)^2}{\sum_{y\in \Z^2} q_{n_i}(y)^2} \,,
\end{equation*}
where  $n_i:=\tau^{(N)}_i - \tau^{(N)}_{i-1}$ and $x\in \Z^2$. For each $i$, denote by
$\zeta_{i,1}$ and $\zeta_{i,2}$ the two components of $\zeta_i\in \Z^2$. Then we note that
there exists $c>0$ such that for any $\lambda\geq 0$, $n_i \in \N$,
\begin{equation}\label{eq:qiGauss}
	\e[e^{\pm \lambda \zeta_{i, j}}] \leq e^{c\lambda^2 n_i}, \qquad
	 j=1,2 \,.
\end{equation}
This can be seen by Taylor expanding the exponential and using that
$\e[\zeta_{i,\cdot}]=0$ by symmetry,  $|\e[\zeta^{2k+1}_{i,\cdot}]|
\leq \frac{1}{2}(\e[\xi^{2k}_{i,\cdot}] + \e[\zeta^{2k+2}_{i,\cdot}])$
by Young's inequality,
as well as $\e[\zeta^{2k}_{i,\cdot}] \leq (Cn_i)^k (2k-1)!!$ for some $C>0$ uniformly in $n_i, k\in \N$. The bound
on $\e[\zeta^{2k}_{i,\cdot}]$ holds because by \eqref{eq:llt},
\begin{equation*}
	\p(\zeta_i=x) = \frac{q_{n_i}(x)^2}{q_{2n_i}(0)} \leq
	\bigg\{ \frac{\sup_{x\in\Z^2} q_{n_i}(x)}{q_{2n_i}(0)}\bigg\} \, q_{n_i}(x)
	\le C' \, q_{n_i}(x) \,,
\end{equation*}
where $q_{n_i}$ has the same Gaussian tail decay as the heat kernel $g_{n_i/2}$. Using  $e^{|x|} \le e^x+e^{-x}$,
this then implies
$$
\begin{aligned}
	\quad \qquad \e\big[e^{\lambda |S^{(N)}_r|} \big| \tau^{(N)}_\cdot \big]
	= \e\Big[e^{\lambda |\sum_{i=1}^r \zeta_i |} \Big] & \leq
	\e\Big[e^{2\lambda |\sum_{i=1}^r \zeta_{i,1} |} \Big]^{\frac12}
	\e\Big[e^{2\lambda |\sum_{i=1}^r \zeta_{i,2} |} \Big]^{\frac12} \\
	& \leq \prod_{j=1,2}\Big(\e\Big[e^{2\lambda \sum_{i=1}^r \zeta_{i,j}}\Big]
	+ \e\Big[e^{-2\lambda \sum_{i=1}^r \zeta_{i,j}} \Big] \Big)^{\frac12} \\
	& = \prod_{j=1,2} \Big(\prod_{i=1}^r \e\Big[e^{2\lambda \zeta_{i,j}}\Big]
	+ \prod_{i=1}^r \e\Big[e^{-2\lambda \zeta_{i,j}} \Big]\Big)^{\frac12} \\
	& \leq 2 e^{4c\lambda^2 n}.
\end{aligned}
$$
The bound \eqref{eq:asU3ub} then follows readily from the definitions of $U_N(n,x)$
and $U_N(n)$ in \eqref{eq:Uren} and \eqref{eq:Usingle}, recalling that $\overline{U}_N(n,x)$ and $\overline{U}_N(n)$
are defined in \eqref{eq:barU}.
\end{proof}

\subsection{Second moment of averaged partition function}
Using $X^{\beta_N}_{d,f}(x,y)$ as introduced in \eqref{eq:X}, and recalling \eqref{eq:Zpoly},
we can now rewrite the chaos expansion for the averaged partition
function $\mathcal{Z}^{\beta_N}_{N, t}(\varphi,\psi)$ in \eqref{eq:Zpolyav} as follows:
\begin{equation}\label{eq:Zpolysimple}
	\mathcal{Z}^{\beta_N}_{N, t}(\varphi,\psi)  \, = \,
	q^N_{0, Nt}(\varphi,\psi)  + \frac{1}{N} \sum_{\substack{0<d\leq f<Nt\\ x,y \in \Z^2}}
	q^N_{0,d}(\varphi, x) \, X^{\beta_N}_{d,f}(x,y) \, q^N_{f, Nt}(y, \psi) \,,
\end{equation}
so that by \eqref{eq:secondX} and the fact that $\overline{U}_N:=\sigma_N^2 U_N$, we have
\begin{equation}\label{eq:m2-repr}
	 \bbE\big[\mathcal{Z}^{\beta_N}_{N, t}(\varphi,\psi)^2 \big]
	\,=\,  q^N_{0, Nt}(\varphi,\psi)^2 + \frac{1}{N^2} \!\!\!\!\!\!
	\sum_{\substack{x,y \in \Z^2 \\ 0<d\leq f<Nt}} \!\!\!\!\!\!
	q^N_{0,d}(\varphi, x)^2 \, \overline{U}_N(f-d, y-x) \, q^N_{f, Nt}(y, \psi)^2.
\end{equation}

We now compute the limit of $\bbE\big[\mathcal{Z}^{\beta_N}_{N, t}(\varphi,\psi)^2 \big]$
as $N\to\infty$. This was first obtained
for the Stochastic Heat Equation in~\cite{BC98}
in the special case $\psi \equiv 1$;
see also \cite[Theorems~1.2 and~1.7]{CSZ19b}
for an alternative derivation, that also includes directed polymers.

\begin{proposition}[First and second moments]\label{th:m2}
Recall $G_\theta(t)$ from \eqref{eq:G1} for all $t>0$. For $\varphi: \R^2 \to \R$, define
\begin{equation}\label{eq:phinorm}
	\Vert \varphi \Vert_{\cG_t}^2 :=
	\iint_{\R^2\times\R^2} \varphi(z) \, \cG_t(z'-z) \, \varphi(z')
	\,\dd z \, \dd z',\footnote{The positivity
	of $\Vert \varphi\Vert_{\cG_t}^2$ can be seen via Fourier transform.}
	\qquad \mbox{where}\quad \cG_t(x) \,:=\, \int_0^t g_s(x) \, \dd s \,.
\end{equation}
Then for all $\varphi$ with $\Vert \varphi \Vert_{\cG_t}<\infty$ and all $\psi \in L^\infty(\R^2)$, we have
\begin{align}
	\label{eq:m1-lim0}
	\lim_{N\to\infty} \bbE\big[\mathcal{Z}^{\beta_N}_{N, t}(\varphi, \psi) \big]
	& \,=\, \tfrac{1}{2} \, g_{\frac{t}{2}}(\varphi,\psi) \,, \\
	\label{eq:m2-lim0}
	\lim_{N\to\infty} \bbE\big[\mathcal{Z}^{\beta_N}_{N, t}(\varphi, \psi)^2 \big]
	& \,=\, \tfrac{1}{4} \, g_{\frac{t}{2}}(\varphi,\psi)^2
	\,+\, \tfrac{1}{2} \mathscr{V}_{t}^{\theta}(\varphi,\psi) \,,
\end{align}
where
\begin{equation}
\label{eq:m2-lim1}
\begin{split}
	\mathscr{V}_{t}^{\theta}(\varphi,\psi)
	& \,=\, \iiiint\limits_{(\R^2)^4}
	\varphi(z) \, \varphi(z') \, K_{t}^{\theta}(z,z'; w,w') \,
	\psi(w) \, \psi(w') \, \dd z \, \dd z' \, \dd w \, \dd w' \\
	& \, \leq \, \pi
	\, \Vert \psi\Vert_\infty^2 \,  \Vert \varphi\Vert_{\cG_t}^2 \, \int_0^t G_\theta(u) \, \dd u \,,
\end{split}
\end{equation}
and the kernel $K_{t}^{\theta}$ is defined by
\begin{equation}
\label{eq:m2-lim}
\begin{split}
	K_{t}^{\theta}(z,z'; w,w')
	&\,:=\, \pi \, \, g_{\frac{t}{4}}\big(\tfrac{w+w'}{2} - \tfrac{z+z'}{2}\big) \\
	& \qquad\quad \times \iint\limits_{0<s<u<t} g_s(z'-z) \,
	G_\theta(u-s) \, g_{t-u}(w'-w) \, \dd s \, \dd u \,.
\end{split}
\end{equation}
\end{proposition}

\begin{proof}
The first moment convergence \eqref{eq:m1-lim0} holds because
by $\bbE\big[\mathcal{Z}^{\beta_N}_{N, t}(\varphi, \psi) \big] =
q^{N}_{0,Nt}(\varphi,\psi)$, see \eqref{eq:Zpolysimple}, in view
of the asymptotic relation \eqref{eq:gstphipsiconv}.

For the second moment computation \eqref{eq:m2-lim0}
we exploit \eqref{eq:m2-repr}, where the first term
in the r.h.s.\ converges to $\frac{1}{4}
\, g_{t/2}(\varphi,\psi)^2$
by \eqref{eq:gstphipsiconv}, which matches the first term in the r.h.s.\ of \eqref{eq:m2-lim0}.
It remains to show that the sum in \eqref{eq:m2-repr} converges to
the term $\frac{1}{2} \mathscr{V}_{t}^{\theta}(\varphi,\psi)$ in \eqref{eq:m2-lim0}.

Recall the definition of $q^N_\cdot$ in \eqref{eq:qNphi}-\eqref{eq:qNpsi}.
By the local limit theorem \eqref{eq:llt} and in view of \eqref{eq:gtrans},
we see that for any $\epsilon > 0$, uniformly for $m > \epsilon N$ and $w \in \Z^2$,
we have as $N\to\infty$
\begin{equation*}
	q^N_{0, m}(\varphi, w)  \,=\,
	\Big( g_{\frac{1}{2} \frac{m}{N}}\big(\varphi, \tfrac{w}{\sqrt{N}} \big) \,+\,
	o(1) \Big)
	\, \ind_{(m,w) \in \Z^3_\even} \,,
\end{equation*}
and similarly, uniformly for $n \le (1- \epsilon) Nt$ and $z \in \Z^2$,
\begin{equation*}
	q^N_{n, Nt}(z,\psi) \,=\,
	\Big( g_{\frac{1}{2}(t - \frac{n}{N})}\big(\tfrac{z}{\sqrt{N}}, \psi \big) \,+\,
	o(1) \Big) \, \ind_{(n,z) \in \Z^3_\even} \,.
\end{equation*}
Applying the asymptotic relation \eqref{eq:asU2} for
$\overline{U}_N(f-d,y-x)$, we see that the sum in \eqref{eq:m2-repr}
is a Riemann sum that converges as $N\to\infty$
to the multiple integral\footnote{The contributions to the sum in \eqref{eq:m2-repr}
given by $m \le \epsilon N$ and $n > (1-\epsilon)Nt$ are small
when $\epsilon > 0$ is small, uniformly in large $N$, as can be checked using the
uniform bound \eqref{eq:ker1}.}
\begin{equation}\label{eq:m2-lim3}
	\frac{1}{2} \mathscr{V}_{t}^{\theta}(\varphi,\psi)
	\,:=\, \frac{\pi}{2} \ \iiiint\limits_{\substack{0 < s < u < t\\ a,b \in \R^2}}
	g_{\frac{s}{2}}(\varphi, a)^2 \, G_\theta(u-s,b-a) \, g_{\frac{t-u}{2}}(b,\psi)^2
	\dd s \, \dd u \, \dd a \, \dd b \,,
\end{equation}
where the prefactor $\frac{1}{2}$ results from combining the
periodicity factor $2$ in \eqref{eq:asU2}
with the volume factor $\frac{1}{2} \cdot \frac{1}{2}$ which originates
from the restrictions $(d,x), (f,y)\in \Z^3_{\rm even}$ in \eqref{eq:m2-repr}.
Then it follows by \eqref{eq:gstphi}, \eqref{eq:gstpsi} and \eqref{eq:G2}
that the equality in \eqref{eq:m2-lim1} holds with
\begin{equation*}
\begin{split}
	K_{t}^{\theta}(z,z'; w,w')
	&\,=\, \pi \, \iiiint\limits_{\substack{0 < s < u < t \\ a,b \in \R^2}}
	\big\{ g_{\frac{s}{2}}(a-z) \, g_{\frac{s}{2}}(a-z') \big\}
	 \, G_\theta(u-s) \, g_{\frac{u-s}{4}}(b-a)   \\
	& \qquad\qquad\qquad\qquad
	\times \big\{ g_{\frac{t-u}{2}}(w-b)\, g_{\frac{t-u}{2}}(w'-b) \big\} \,
	\dd s \, \dd u \, \dd a \, \dd b \,.
\end{split}
\end{equation*}
We can simplify both brackets via the identity $g_t(x) \, g_t(y) = g_{2t}(x-y)
\, g_{\frac{t}{2}}(\frac{x+y}{2})$, see \eqref{eq:gt}.
Performing the integrals over $a,b \in \R^2$ we then obtain \eqref{eq:m2-lim}.

The bound in \eqref{eq:m2-lim1} follows by bounding $\psi$ with $\Vert \psi\Vert_\infty$
and then successively integrating out $w, w'$, followed by $u$ and $s$ in \eqref{eq:m2-lim}.
\end{proof}

\begin{remark}[Point-to-plane partition function]\label{rem:paf}
For $\psi(w) = \ind(w) \equiv 1$, we can view
$\mathcal{Z}^{\beta_N}_{N, t}(\varphi, \ind)$
as the point-to-plane partition function $Z_N^{\beta_N}(z)$ in \eqref{eq:paf}
averaged over its starting point~$z$.
By \eqref{eq:m1-lim0}-\eqref{eq:m2-lim},
\begin{align*}
	\lim_{N\to\infty} \bbE\big[\mathcal{Z}^{\beta_N}_{N, t}(\varphi, \ind) \big]
	& = \frac{1}{2} \, g_{\frac{t}{2}}(\varphi,1)
	= \frac{1}{2} \, \int\limits_{\R^2} \varphi(z) \, \dd z \,, \\
	\lim_{N\to\infty} \bbvar\big[\mathcal{Z}^{\beta_N}_{N, t}(\varphi, \ind) \big]
	& = \frac{1}{2} \mathscr{V}_{t}^{\theta}(\varphi, 1)
	= \frac{1}{2} \iint\limits_{(\R^2)^2} \varphi(z) \, \varphi(z') \,
	K_{t}^{\theta}(z-z') \dd z\dd z',
\end{align*}
where we set
\begin{gather*}
	K_{t}^{\theta}(x)
	:= \pi \, \iint_{0 < s < u < t} g_s(x) \,
	G_\theta(u-s) \, \dd s \, \dd u \,.
\end{gather*}
We note that both the asymptotic mean and the
asymptotic variance of $\mathcal{Z}^{\beta_N}_{N, t}(\varphi, \ind)$ are
half of those obtained in \cite[eq.~(1.19)-(1.20)]{CSZ19b}.
This is because here we have defined $\mathcal{Z}^{\beta_N}_{N, t}(\varphi,\psi)$
as a sum over $\Z^2_\even$, see~\eqref{eq:ZNavalt}, while in \cite{CSZ19b},
the sum is over both $\Z^2_\odd$ and $\Z^2_\even$, which
give rise to two i.i.d.\ limits as $N\to\infty$ by the parity of the simple random walk on $\Z^2$.
\end{remark}

\section{Coarse-graining}\label{sec:CGmodel}

In this section, we give the details of how to coarse-grain the averaged partition
function and what is the precise definition
of the coarse-grained model, which were outlined in Section~\ref{sec:outline}.
The main result is Theorem~\ref{th:cg-main},
which shows that the averaged partition function
$\mathcal{Z}^{\beta_N}_{N, t}(\varphi,\psi)$, see \eqref{eq:ZNavalt},
can be approximated in $L^2$ by the coarse-grained model.

\subsection{Preparation}

The starting point is the polynomial
chaos expansion \eqref{eq:Zpolyav} for the averaged partition function
$\mathcal{Z}^{\beta_N}_{N, t}(\varphi,\psi)$, which is a multilinear polynomial in
the disorder variables $\xi_N(n,z)$. We will call the sequence
of time-space points $(n_1, z_1)$, \ldots, $(n_r, z_r)\in \N\times \Z^2$
in the sum in \eqref{eq:Zpolyav} a {\em microscopic (time-space) renewal configuration}.
We assume that the disorder strength is chosen
to be $\beta_N=\beta_N(\theta)$ as defined in  \eqref{eq:sigma}-\eqref{eq:betaN}.
For simplicity, we assume the time horizon to be $tN$ with $t=1$.

Given $\epsilon \in (0,1)$ and $N\in\N$, we partition
discrete time-space $\{1,\ldots, N\} \times \Z^2$
into \emph{mesoscopic boxes}
\begin{equation} \label{eq:cB}
	\cB_{\epsilon N}(\sfi,\sfa)
	\ := \ \underbrace{((\sfi-1)\epsilon N, \sfi \epsilon N]}_{\cT_{\epsilon N}(\sfi)}
	\times \underbrace{((\sfa-(1,1))\sqrt{\epsilon N},
	\sfa\sqrt{\epsilon N}]}_{\cS_{\eps N}(\sfa)}
	 \ \, \cap \ \ \Z^3_\even \,,
\end{equation}
where $\cT_{\epsilon N}(\sfi)$ is mesoscopic time interval and $\cS_{\eps N}(\sfa)$
a mesoscopic spatial square.\footnote{We use the
notation $(\sfa-\sfb, \sfa] = (\sfa_1-\sfb_1,\sfa_1] \times (\sfa_2-\sfb_2,\sfa_2]$ for squares in $\R^2$.}
These boxes
are indexed by \emph{mesoscopic variables}
\begin{equation*}
	(\sfi,\sfa) \in \{1,\ldots, \lfloor\tfrac{1}{\epsilon}\rfloor\} \times \Z^2 \,.
\end{equation*}

Recall from Section~\ref{sec:outline} that to carry out the coarse-graining,
we need to organize the chaos expansion
\eqref{eq:Zpolyav} according to which mesoscopic boxes $\cB_{\eps N}$ are visited by the microscopic
renewal configuration $(n_1, z_1)$, \ldots, $(n_r, z_r)$.
To perform the kernel replacement \eqref{eq:kerrep}, which
allows each summand in the chaos expansion \eqref{eq:Zpolyav} to factorize into
a product of coarse-grained disorder variables $\Theta_{N, \eps}$
connected by heat kernels, we will impose some constraints on the set of visited
mesoscopic time intervals
$\cT_{\eps N}(\cdot)$ and spatial boxes $\cS_{\eps N}(\cdot)$, which will be shown
to have negligible costs in $L^2$.
We first introduce the necessary notation.

Let us fix two thresholds
\begin{equation}\label{eq:KM}
	K_\epsilon := (\log\tfrac{1}{\epsilon})^6 \,, \qquad
	M_\epsilon := \log \log\tfrac{1}{\epsilon}  \,.
\end{equation}
We will require that the visited mesoscopic time intervals $\cT_{\eps N}(\sfi_1)$, \ldots, $\cT_{\eps N}(\sfi_k)$ belong to
\begin{equation}\label{eq:cA0}
\begin{aligned}
	\cA_{\epsilon}^{(\notri)} \,:=\,
	\bigcup_{k\in\N}
	\Big\{(\sfi_1, \ldots, \sfi_k) \in \N^k : \
	& K_\eps \le
	\sfi_1 < \sfi_2 < \ldots < \sfi_k
	 \leq \lfloor\tfrac{1}{\epsilon}\rfloor -K_\eps
	\ \ \text{such that}  \\
	& \ \ \text{if} \ \ \sfi_{j+1}-\sfi_j < K_\epsilon \,, \ \text{then}
	\ \ \sfi_{j+2} - \sfi_{j+1} \geq K_\epsilon \ \Big\} \,.
\end{aligned}
\end{equation}
We call this the {\em no-triple condition}, since it forbids three consecutive mesoscopic time indices $\sfi_j, \sfi_{j+1}, \sfi_{j+2}$ with
both $\sfi_{j+1} -\sfi_j< K_\eps$ and $\sfi_{j+2} -\sfi_{j+1}< K_\eps$. We can then partition $(\sfi_1, \ldots, \sfi_k)$ into time blocks such that
$\sfi_j, \sfi_{j+1}$ belong to the same block whenever $\sfi_{j+1}-\sfi_j<K_\eps$.
\begin{definition}[Time block]\label{def:time-block}
We call a \emph{time block} any pair
$\vec\sfi = (\sfi, \sfi') \in \N \times \N$ with $\sfi \le \sfi'$.
The width of a time block is
\begin{equation*}
	|\vec\sfi| := \sfi' - \sfi +1\,.
\end{equation*}
The (non symmetric) ``distance'' between two time blocks $\vec\sfi, \vec\sfm$ is defined by
\begin{equation*}
	\dist(\vec\sfi \,,\, \vec\sfm) := \sfm - \sfi' \qquad \text{for } \
	\vec\sfi = (\sfi,\sfi') \, \text{ and } \ \vec\sfm = (\sfm,\sfm') \,,
\end{equation*}
and we write ``$\,\vec\sfi < \vec\sfm$''
to mean that ``$\,\vec\sfi$ precedes $\vec\sfm$''$:$
\begin{equation*}
	\vec\sfi < \vec\sfm \qquad \iff \qquad
	\dist(\vec\sfi \,,\, \vec\sfm) > 0 \ \text{ i.e. } \ \sfi' < \sfm \,.
\end{equation*}
\end{definition}
With the partitioning of the indices $(\sfi_1, \ldots, \sfi_k)$ of the visited mesoscopic
time intervals into consecutive time blocks as defined
above, which we denote by $\vec\sfi_1=(\sfi_1, \sfi_1')$, \ldots, $\vec \sfi_r=(\sfi_r, \sfi_r')$ with possibly $\sfi_\ell=\sfi_\ell'$,
the constraint $\cA_{\epsilon}^{(\notri)}$ then
becomes the following:
\begin{equation}\label{eq:bcA}
\begin{split}
	\bcA_{\epsilon}^{(\notri)} :=
	\bigcup_{r\in\N} \ \Big\{\,
	\text{time blocks} \quad
	K_\epsilon \,\le\, \vec\sfi_1 \,<\, \ldots \,<\, \vec\sfi_r \,\le\,
	\lfloor\tfrac{1}{\epsilon}\rfloor - K_\epsilon
	 \quad \text{such that}  \quad \ & \\
	|\vec\sfi_j|\leq K_\eps \ \ \forall j=1, \ldots, r \,, \quad
	\dist(\vec\sfi_{j-1} \,, \, \vec\sfi_{j}) \geq K_\epsilon \ \
	\forall j=2, \ldots, r \, & \Big\} \,.
\end{split}
\end{equation}
If the time horizon is $Nt$ with $t\neq 1$, then
in \eqref{eq:cA0} and \eqref{eq:bcA} we just replace the upper bound
$\vec\sfi_r \le \lfloor\tfrac{1}{\epsilon}\rfloor - K_\epsilon$ by
$\vec\sfi_r \le \lfloor\tfrac{t}{\epsilon}\rfloor - K_\epsilon$.

Given a time block $\vec\sfi=(\sfi, \sfi')$ with $\sfi'-\sfi+1\leq K_\eps$ (possibly $\sfi=\sfi'$),
which identifies two mesoscopic
time intervals $\cT_{\eps N}(\sfi)$ and $\cT_{\eps N}(\sfi')$ visited by the microscopic
renewal configuration
$(n_1, z_1), \ldots, (n_r, z_r)$ from \eqref{eq:Zpolyav} and no intervals in-between is visited, we can
identify the
first and last mesoscopic spatial boxes visited in the time intervals $\cT_{\eps N}(\sfi)$ and $\cT_{\eps N}(\sfi')$, respectively.
We call this pair of mesoscopic spatial indices a space block.

\begin{definition}[Space block]\label{def:space-block}
We call a \emph{space block} any pair $\vec\sfa = (\sfa,\sfa') \in \Z^2 \times \Z^2$.
The width of a space block is
\begin{equation*}
	|\vec\sfa| := |\sfa' - \sfa| \,,
\end{equation*}
with $|\cdot|$ being the Euclidean norm.
The (non symmetric) ``distance'' between two space blocks $\vec\sfa, \vec\sfb$ is
\begin{equation*}
	\dist(\vec\sfa \,,\, \vec\sfb) := |\sfb - \sfa'| \qquad \text{for } \
	\vec\sfa = (\sfa,\sfa') \, \text{ and } \ \vec\sfb = (\sfb,\sfb') \,.
\end{equation*}
\end{definition}
Putting the time block and space block together, we have the following.

\begin{definition}[Time-space block]
We call a \emph{time-space block} any pair $(\vec\sfi, \vec\sfa)$ where $\vec\sfi$
is a time block and $\vec\sfa$ is a space block. We also define
\begin{equation}\label{eq:iacond}
\bbT_\eps:=\Big\{ \text{ time-space blocks } \ (\vec\sfi, \vec\sfa) \ \text{ with } \
 	|\vec\sfi| \leq K_\eps \  \text{ and  }   \
	|\vec\sfa| \le M_\epsilon \sqrt{|\vec\sfi|} \Big\}.
\end{equation}
\end{definition}
In \eqref{eq:Zpolyav}, we will restrict to $(n_1, z_1), \ldots, (n_r, z_r)$ (interpreted as a time-space renewal configuration)
that satisfy condition \eqref{eq:cA0}, so that they determine a sequence of mesoscopic time-space blocks $\{(\vec\sfi_1, \vec \sfa_1)$, \ldots,
$(\vec\sfi_r, \vec \sfa_r)\} \in \bcA_{\epsilon}^{(\notri)}$. This would give the main contribution in \eqref{eq:Zpolyav}.
We now impose further constraints on the spatial components that still capture the main contribution.

Given two ``boundary variables'' $\sfb, \sfc \in \Z^2$ and a sequence of time blocks $(\vec\sfi_1, \ldots, \vec\sfi_r)$,
we denote by $\bcA_{\epsilon; \,\sfb, \sfc}^{(\diff)} = \bcA_{\epsilon; \, \sfb, \sfc}^{(\diff)}(\vec\sfi_1, \ldots, \vec\sfi_r)$
the following subset of space blocks $(\vec\sfa_1, \ldots, \vec\sfa_r)$, where we impose \emph{diffusive constraints}
on their widths and distances:
\begin{equation}\label{eq:bcA2}
\begin{split}
	\bcA_{\epsilon; \,\sfb, \sfc}^{(\diff)} :=
	\bigg\{ \text{space blocks} \ \ \vec\sfa_1, \ldots, \vec\sfa_r \quad \text{such that} \ \
	|\vec\sfa_j| \le  M_\epsilon \sqrt{|\vec\sfi_j|} \quad \forall j=1,\ldots, r \,, & \\
	\dist(\vec\sfa_{j-1}, \vec\sfa_{j})  \le
	M_\epsilon \, \sqrt{\dist(\vec\sfi_{j-1}\,,\, \vec\sfi_{j})} \quad \forall j=2,\ldots, r \,, & \\
	|\sfa_1 - \sfb| \le M_\epsilon \sqrt{\sfi_1} \quad \text{and} \quad
	|\sfc - \sfa'_r| \le M_\epsilon \sqrt{\lfloor\tfrac{1}{\epsilon}\rfloor - \sfi'_r}  &
	\, \bigg\} \,.
\end{split}
\end{equation}
Given a sequence of mesoscopic time-space blocks $(\vec\sfi_1, \vec \sfa_1)$, \ldots,
$(\vec\sfi_r, \vec \sfa_r)$ determined by the microscopic renewal configuration
$(n_1, z_1), \ldots, (n_r, z_r)$ from \eqref{eq:Zpolyav}, which
satisfies the constraints $\bcA_{\epsilon}^{(\notri)}$ and
$\bcA_{\epsilon; \,\sfb, \sfc}^{(\diff)}$,
we will perform the kernel replacement \eqref{eq:kerrep},
which leads to a factorization of each summand
in \eqref{eq:Zpolyav} as the product of coarse-grained disorder variables
$\Theta_{N, \eps}(\vec \sfi_j, \vec \sfa_r)$, $1\leq j\leq r$,
connected by the heat kernels $g_{\tfrac{1}{2}(\sfi_{j+1}-\sfi_j)}(\sfa_{j+1}-\sfa_j')$.
See Figure~\ref{CG-fig2}.

\begin{figure}
\hskip -1.5cm
\begin{tikzpicture}[scale=0.45]
\draw[ystep=1cm, xstep=2,gray,very thin] (0,0) grid (2, 10);
\draw[ystep=1cm, xstep=2,gray,very thin] (4,0) grid (6, 10);
\draw[ystep=1cm, xstep=2,gray,very thin] (10,0) grid (12, 10);
\draw[ystep=1cm, xstep=2,gray,very thin] (14,0) grid (16, 10);
\draw[ystep=1cm, xstep=2,gray,very thin] (20,0) grid (22, 10);
\node at (3,5) {{$\cdots$}}; \node at (7.7,5) {{$\cdots$}}; \node at (18,5) {{$\cdots$}}; \node at (23.5,5) {{$\cdots$}};
\node at (1,-0.5) {\scalebox{2.0}{$\leftrightarrow$}}; \node at (-0.3, 0.5) {\scalebox{1.5}{$\updownarrow$}};
\node at (1,-1) {\scalebox{0.8}{$\epsilon N$}}; \node at (-1.5, 0.5) {\scalebox{0.8}{$\sqrt{\epsilon N}$}};
\node at (7.5,6.3) {\scalebox{1.5}{$\longleftarrow$}}; \node at (8.5,6.3) {\scalebox{1.5}{$\longrightarrow$}};
\node at (17.5,6.3) {\scalebox{1.5}{$\longleftarrow$}}; \node at (18.5,6.3) {\scalebox{1.5}{$\longrightarrow$}};
\node at (12.7,6.2) {\scalebox{1.5}{$\leftarrow$}}; \node at (13.3,6.2) {\scalebox{1.5}{$\rightarrow$}};
\node at (8.1,6.9) {\scalebox{1.0}{$\geq K_\epsilon$}};
\node at (18.1,6.9) {\scalebox{1.0}{$\geq K_\epsilon$}};
\node at (13.1,6.9) {\scalebox{0.9}{$\leq \!K_\epsilon$}};
\draw (0, 5)  circle [radius=0.1];  \draw [fill] (4.2, 5.5)  circle [radius=0.08];  \draw [fill] (5.8, 5.6)  circle [radius=0.08];
 \draw [fill] (10.2, 2.3)  circle [radius=0.08];   \draw [fill] (11.6,1.5) circle [radius=0.08]; \draw [fill] (14.2, 1.5)  circle [radius=0.08];
 \draw [fill] (15.8, 1.6)  circle [radius=0.08];  \draw [fill] (20.2, 4.5)  circle [radius=0.08];
\draw [fill] (21.8, 4.6) circle [radius=0.08]; \draw (6, 6) circle [radius=0.1];
\draw (12, 3) circle [radius=0.1];
\draw (16, 2) circle [radius=0.1]; \draw (22, 5) circle [radius=0.1];
\draw[thick] (0,5) to [out=45,in=150] (6,6);
\draw[thick, dashed] (4.2, 5.5) to [out=20,in=180] (4.4,5.8) to [out=-30,in=120] (4.8,5.2) to [out=0, in=180] (5.8, 5.6);
\draw[thick] (6, 6) to [out=0, in=135] (12, 3);
\draw[thick, dashed] (10.2, 2.3) to [out=80,in=180] (11,3.3) to [out=-50,in=90] (11.6,1.5) ;
\draw[thick, dashed] (14.2, 1.5) to [out=20,in=180] (14.4,1.8) to [out=-30,in=120] (14.8,1.2) to [out=0, in=180] (15.8, 1.6);
\draw[thick, dashed] (11.6,1.5) to [out=30,in=150] (14.2, 1.5);
\draw[thick, dashed] (20.2, 4.5) to [out=20,in=180] (20.4,4.8) to [out=-30,in=120] (20.8,4.2) to [out=0, in=180] (21.8, 4.6);
\draw[thick] (16, 2) to [out=70,in=160] (22, 5);
\draw[thick] (22, 5) to [out=50,in=180] (24, 6);
\draw[lightgray] (9.7, -0.3) -- (16.3, -0.3) -- (16.3, 4.2) -- (9.7, 4.2) --(9.7,-0.3);
\begin{pgfonlayer}{background}
 \fill[color=gray!10] (4,4) rectangle (6,7);
 \fill[color=gray!10] (10,0) rectangle (12,4);
  \fill[color=gray!10] (14,0) rectangle (16,4);
 \fill[color=gray!10] (20,3) rectangle (22,6);
 \end{pgfonlayer}
\end{tikzpicture}
\caption{
An illustration of the coarse-graining procedure. The solid lines represent the heat kernels
after the kernel replacement \eqref{eq:kerrep},
which connect adjacent coarse-grained disorder
variables $\Theta^{(\cg)}_{N, \epsilon}(\vec\sfi, \vec\sfa)$ consisting of
sums over the dashed lines in each visited time-space block $\cB_{\epsilon N}(\vec\sfi,\vec\sfa)$
(see \eqref{eq:Theta}). The solid and the dashed
lines satisfy the diffusive constraint given in $\bcA_{\epsilon; \,\sfb, \sfc}^{(\diff)}$
and \eqref{eq:Theta}, respectively.
\label{CG-fig2}}
\end{figure}
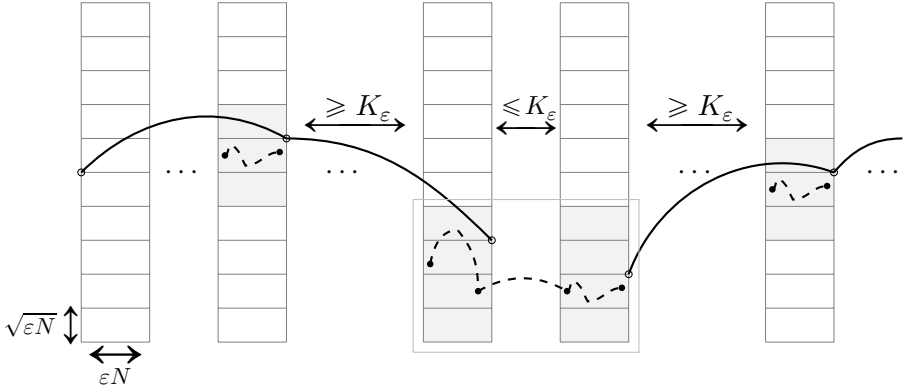

\subsection{Coarse-grained model}

We are now ready to give the precise definition of the \emph{coarse-grained model}
$\mathscr{Z}_{\epsilon}^{(\cg)}(\varphi,\psi|\Theta)$ given earlier in \eqref{CGchaos-intro}, which
depends on $\eps\in (0,1)$ and is a multilinear polynomial of a
given family of random variables $\Theta=\{\Theta(\vec\sfi, \vec \sfa)\}$
indexed by time-space blocks $(\vec\sfi,\vec\sfa)$.

\begin{definition}[Coarse-grained model]\label{def:cg}
Fix $\epsilon \in (0,1)$ and a family of random variables
$\Theta = (\Theta(\vec\sfi,\vec\sfa))_{(\vec\sfi,\vec\sfa) \in \bbT_\epsilon}$
indexed by the set $\bbT_\epsilon$ of time-space blocks defined in \eqref{eq:iacond}.
Fix two locally integrable functions $\varphi, \psi: \R^2 \to\R$ and define
$\varphi_\epsilon, \psi_\epsilon: \Z^2 \to \R$ as follows:
\begin{equation}\label{eq:phieps-psieps}
	\varphi_\epsilon(\sfb) := \int\limits_{(\sfb-(1,1), \sfb]} \varphi(\sqrt{\epsilon} x)
	\, \dd x \,, \qquad
	\psi_\epsilon(\sfc) := \int\limits_{(\sfc-(1,1), \sfc]} \psi(\sqrt{\epsilon} y) \, \dd y
	\qquad \text{for } \sfb, \sfc\in\Z^2 \,.
\end{equation}
Recall the heat kernel $g_t(\cdot)$ from \eqref{eq:gt}.
Then the \emph{coarse-grained model}
$\mathscr{Z}_{\epsilon}^{(\cg)}(\varphi,\psi|\Theta)$ is
\begin{equation}\label{eq:Zcg-gen}
\begin{aligned}
	\mathscr{Z}_{\epsilon}^{(\cg)}(\varphi,\psi|\Theta)
	 & :=  \frac{1}{2} \, g_{\frac{1}{2}}(\varphi, \psi) \,  + \, \frac{\epsilon}{2}
	 \sum_{r=1}^{(\log\frac{1}{\epsilon})^2} \sum_{{\sfb}, {\sfc} \in \Z^2} 	
	\sum_{\substack{(\vec\sfi_1, \ldots, \vec\sfi_r) \in \bcA_{\epsilon}^{(\notri)} \\
	(\vec\sfa_1, \ldots, \vec\sfa_r) \in \bcA_{\epsilon; \, b,c}^{(\diff)}}}
	\!\!\!\!\!\!\!\!\!\!\! \varphi_\epsilon({\sfb})  g_{\frac{1}{2}\sfi_1}(\sfa_1 - \sfb)
	\Theta(\vec\sfi_1, \vec\sfa_1)  \\
	& \times
	\Bigg\{ \prod_{j=2}^r  g_{\frac{1}{2}(\sfi_j -\sfi_{j-1}')} (\sfa_j -\sfa_{j-1}')
	\, \Theta(\vec\sfi_j, \vec\sfa_j) \Bigg\}
	\, g_{\frac{1}{2}(\lfloor\frac{1}{\epsilon}\rfloor- \sfi_r')}({\sfc}-\sfa_r')
	 \psi_\epsilon({\sfc}) \,.
\end{aligned}
\end{equation}
\end{definition}
Note that in \eqref{eq:Zcg-gen}, for technical reasons that will become clear later
 (to control the error induced by kernel replacements - see Section \ref{sec:5.3}), we also imposed
the constraint that the number of time-space blocks
cannot exceed $(\log \frac{1}{\eps})^2$. This coarse-grained model has the same structure
as the original averaged partition function
$\mathcal{Z}^{\beta_N}_{N, 1}(\varphi,\psi)$ in \eqref{eq:Zpolyav}, with $1/\eps$ replacing $N$,
$\Theta$ replacing $\xi_N$, and  the heat kernels replacing
the random walk kernels. Note that when $\varphi$ has compact support, \eqref{eq:Zcg-gen} is a sum
over finitely many terms.

\begin{remark}
To approximate the averaged partition function $\mathcal{Z}^{\beta_N}_{N, t}(\varphi,\psi)$
with $t\neq 1$, we define a corresponding coarse-grained model
$\mathscr{Z}_{\epsilon, t}^{(\cg)}(\varphi,\psi|\Theta)$ which is obtained from
\eqref{eq:Zcg-gen} simply replacing $g_{\frac{1}{2}}(\varphi, \psi)$ by $g_{\frac{t}{2}}(\varphi, \psi)$
and $g_{\frac{1}{2}(\lfloor\tfrac{1}{\epsilon}\rfloor- \sfi_r')}$
by $g_{\frac{1}{2}(\lfloor\tfrac{t}{\epsilon}\rfloor- \sfi_r')}$, as well as
modifying accordingly $\bcA_{\epsilon}^{(\notri)}$ and
$\bcA_{\epsilon; \, b,c}^{(\diff)}$ (replacing $\lfloor \frac{1}{\epsilon}\rfloor$ by
$\lfloor \frac{t}{\epsilon}\rfloor$ therein).
\end{remark}

\subsection{Coarse-grained disorder variables}

We now identify the coarse-grained
disorder variables $\Theta_{N, \eps}^{\rm (cg)}$ so that
the averaged partition function $\mathcal{Z}^{\beta_N}_{N, t}(\varphi, \psi)$
can be approximated in $L^2$ by the coarse-grained model
$\mathscr{Z}_{\epsilon, t}^{(\cg)}(\varphi,\psi|\Theta)$ with $\Theta=\Theta_{N, \eps}^{\rm (cg)}$.

Recall the point-to-point partition function $Z_{d,f}^{\beta_N}(x,y)$
with its chaos expansion as in \eqref{eq:Zpoly}.
Assuming $f-d\leq \eps N$, we introduce a diffusive truncation as follows, the effect of which will be
negligible in $L^2$, but it ensures that the coarse-grained disorder variable $\Theta_{N, \eps}^{\rm (cg)}$
will only depend on $\xi_N(n,z)$ in a localized time-space window.  In \eqref{eq:Zpoly}, let
$\sfa = \sfa(x) \in \Z^2$ be such that $x \in \cS_{\eps N}(\sfa)$ (recall \eqref{eq:cB}). We then restrict $y$
and all space
variables $z_j$ in \eqref{eq:Zpoly} to those mesoscopic boxes
$\cS_{\eps N}(\tilde\sfa)$ with $|\tilde\sfa - \sfa| \le M_\epsilon= \log \log \frac{1}{\eps}$
as in \eqref{eq:KM}, and define
\begin{equation}\label{eq:Zpolydiff}
	Z^{(\diff)}_{d,f}(x,y)  \, := \,
	\begin{cases}
	0 & \text{if } y \not\in \bigcup\limits_{|\tilde\sfa-\sfa| \le M_\epsilon} \cS_{\eps N}(\tilde\sfa), \\
	\begin{aligned}
	& q_{d,f}(x,y)  + \, \sum_{r=1}^{\infty} \
	\sum_{\substack{d < n_1 < \ldots < n_r < f  \\
	z_1, \ldots, z_r \,\in\, \bigcup_{|\tilde\sfa-\sfa| \le M_\epsilon} \cS_{\eps N}(\tilde\sfa)}} \\
	& \ \ \ q_{d,n_1}(x,z_1)  \, \xi_N(n_1,z_1) \times \\
	& \ \ \ \Bigg\{ \prod_{j=2}^r q_{n_{j-1}, n_j}(z_{j-1},z_j)
	\, \xi_N(n_j,z_j) \Bigg\} \, q_{n_j,f}(z_j,y)
	\end{aligned}
	& \text{if } y \in \bigcup\limits_{|\tilde\sfa-\sfa| \le M_\epsilon} \cS_{\eps N}(\tilde\sfa).
	\end{cases}
\end{equation}
Similar to the definition of $X_{d,f}^{\beta_N}(x,y)$ in \eqref{eq:X}, we define
\begin{equation}\label{eq:Xdiff}
	X^{(\diff)}_{d,f}(x,y) := \begin{cases}
	\xi_N(d,x) & \text{if } f=d \\
	\rule{0pt}{1.2em}\xi_N(d,x) \,
	Z^{(\diff)}_{d,f}(x,y) \, \xi_N(f,y) & \text{if }f \ge d+1
	\end{cases} \,.
\end{equation}
Note that we omit the dependence of $Z^{(\diff)}_{d,f}(x,y)$ and $X^{(\diff)}_{d,f}(x,y)$
on $N, \epsilon$.

The coarse-grained disorder
variables $\Theta^{(\cg)}_{N, \epsilon}(\vec \sfi, \vec\sfa)$ are defined as follows
(see Figure~\ref{CG-fig2}).

\begin{definition}[Coarse-grained disorder variable]\label{def:cg-charge}
Given $N\in\N$, $\epsilon \in (0,1)$ and
a time-space block $(\vec\sfi,\vec\sfa)$, with $\vec\sfi = (\sfi,\sfi')$ and $\vec\sfa = (\sfa,\sfa')$,
the associated \emph{coarse-grained disorder variable} $\Theta^{(\cg)}_{N, \epsilon}(\vec\sfi, \vec\sfa)$
is defined by
\begin{equation}\label{eq:Theta}
	\Theta^{(\cg)}_{N, \epsilon}(\vec\sfi, \vec\sfa) :=
	\begin{cases}
	\displaystyle \frac{2}{\epsilon N} \,
	\sum_{\substack{(d,x) \in \cB_{\epsilon N}(\sfi, \sfa)\\
	(f,y) \in \cB_{\epsilon N}(\sfi, \sfa')\\\text{with }d \le f}}
	X_{d,f}^{(\diff)}(x,y) & \text{ if } \ |\vec \sfi| = 1 \,, \\
	\displaystyle
	\rule{0pt}{5.5em}\begin{split}
	\frac{2}{\epsilon N} \,
	& \sum_{\substack{(d,x) \in \cB_{\epsilon N}(\sfi, \sfa)\\
	(f',y') \in \cB_{\epsilon N}(\sfi', \sfa')}}
	\sum_{\substack{\sfb: \, |\sfb - \sfa| \le M_\epsilon \\
	\sfb': \, |\sfb' - \sfa'| \le M_\epsilon \\
	\text{such that} \\ |\sfb' - \sfb| \le M_\epsilon \sqrt{\sfi'-\sfi}}}
	\sum_{\substack{(f,y) \in \cB_{\epsilon N}(\sfi, \sfb)\\
	(d',x') \in \cB_{\epsilon N}(\sfi', \sfb')\\
	\text{such that} \\ d \le f, \, d' \le f'}}  \\
	& \qquad \qquad \quad X^{(\diff)}_{d, f}(x, y) \; q_{f, d'}(y, x') \; X^{(\diff)}_{d',f'}(x',y')
	\end{split}
	& \text{ if } \ |\vec \sfi| >1 \,.
	\end{cases}
\end{equation}
In the special case $|\vec \sfi|=1$, i.e., $\sfi=\sfi'$,
we will also write $\Theta^{(\cg)}_{N, \epsilon}(\sfi; \sfa, \sfa')$
in place of $\Theta^{(\cg)}_{N, \epsilon}(\vec\sfi, \vec\sfa)$.
\end{definition}

\noindent
We point out that
the prefactor $2$ in \eqref{eq:Theta} is due to periodicity, because the sums are restricted to
$\cB_{\epsilon N}(\sfi,\sfa) \subseteq \Z^3_\even$, see \eqref{eq:cB}.

\subsection{Coarse-graining approximation}

We can finally state the key result of this section,
which approximates the averaged partition function $\mathcal{Z}^{\beta_N}_{N, t}(\varphi, \psi)$
in $L^2$ by the coarse-grained model
$\mathscr{Z}_{\epsilon, t}^{(\cg)}(\varphi,\psi|\Theta_{N, \eps}^{\rm (cg)})$,
with an error which is much smaller than $\mathbb{V}{\rm ar}(\mathcal{Z}^{\beta_N}_{N, t}(\varphi,\psi))$
in \eqref{eq:m2-lim1} when $N$ is
large and $\eps$ is small. Recall $\Vert \cdot \Vert_{\cG_t}$ from \eqref{eq:phinorm}.

\begin{theorem}[Coarse-graining]\label{th:cg-main}
Let  $\cZ_{N,t}^{\beta_{N}}(\varphi, \psi)$ be
the averaged partition function in \eqref{eq:ZNav},
where $\beta_{N}=\beta_{N}(\theta)$ satisfies \eqref{eq:sigma}
for some fixed $\vartheta\in \R$.
Let $\mathscr{Z}_{\epsilon,t}^{(\cg)}(\varphi,\psi|\Theta)$ be the coarse-grained model
from \eqref{eq:Zcg-gen}, with
$K_\epsilon=(\log \frac{1}{\eps})^6$ and $M_\epsilon=\log \log\frac{1}{\eps}$ as in \eqref{eq:KM},
and let $\Theta(\vec\sfi,\vec\sfa) = \Theta_{N,\epsilon}^{(\cg)}(\vec\sfi,\vec\sfa)$ be
the coarse-grained disorder variables from Definition~\ref{def:cg-charge}.
Then, for any $T \in (0,\infty)$, there exists $\sfC = \sfC(T) < \infty$
such that, for $\epsilon > 0$ small enough, we have
\begin{equation}\label{eq:cg-approx}
\begin{gathered}
	\limsup_{N\to\infty} \
	\big\| \mathcal{Z}^{\beta_N}_{N, t}(\varphi,\psi) \,-\,
	\mathscr{Z}_{\epsilon,t}^{(\cg)}(\varphi,\psi|\Theta_{N,\epsilon}^{(\cg)}) \big\|_{L^2}^2 \,\le\,
	\sfC \, \Bigg(\Vert \varphi\Vert_{\cG_{\scriptscriptstyle K_\eps \eps }}^2 +
	\frac{\Vert \varphi\Vert_{\cG_T}^2}{\sqrt{\log\frac{1}{\epsilon}}}  \Bigg)\, \Vert \psi\Vert_\infty^2 \,,
\end{gathered}
\end{equation}
uniformly in $t\in [0,T]$, $\psi \in L^\infty(\R^2)$ and $\varphi:\R^2\to\R$ with
$\Vert \varphi\Vert_{\cG_T}<\infty$.
\end{theorem}

Note that the r.h.s.\ of \eqref{eq:cg-approx} tends to $0$ as $\eps\downarrow 0$,
because $K_\epsilon \, \epsilon \to 0$. The whole of Section~\ref{sec:2ndmoment} 
is devoted to the proof of  Theorem~\ref{th:cg-main}.


\section{Second moment bounds for averaged partition functions}
\label{sec:2ndmoment}

This section is devoted mainly to the proof of Theorem~\ref{th:cg-main}, which approximates the averaged
partition function
$\mathcal{Z}^{\beta_N}_{N, t}(\varphi,\psi)$ from \eqref{eq:ZNavalt} by the coarse-grained
model $\mathscr{Z}_{\epsilon,t}^{(\cg)}(\varphi,\psi|\Theta_{N,\epsilon}^{(\cg)})$
from \eqref{eq:Zcg-gen}. We may assume $t=T=1$ without loss of generality. The uniformity in $t\leq T$
will be clear from the proof.
\emph{Throughout this section we simply write $\mathcal{Z}_{N}(\varphi,\psi)$,
omitting the dependence on $t=1$ and on $\beta_N$}.

The starting point of our proof of Theorem~\ref{th:cg-main}
is the polynomial chaos expansion \eqref{eq:Zpolyav}.
In the second moment calculations, the time-space
renewal representation and the limiting Dickman subordinator presented in
Section~\ref{sec:Dickman} play  a crucial role.
The proof will be carried out in three steps, presented in
Subsections~\ref{sec:5.1}-\ref{sec:5.3} below: given $\epsilon > 0$,
we introduce two intermediate approximations $\cZ_{N,\epsilon}^{(\notri)}(\varphi,\psi)$
and $\cZ_{N,\epsilon}^{(\diff)}(\varphi,\psi)$
of the averaged partition function $\mathcal{Z}_{N}(\varphi,\psi)$, and bound the following in $L^2$.
\begin{itemize}
\item Step 1. We bound $\cZ_N(\varphi,\psi) - \cZ_{N,\epsilon}^{(\notri)}(\varphi,\psi)$,
see Lemma~\ref{th:no3} in Section~\ref{sec:5.1}: this
is the cost of imposing the constraints $\cA_{\epsilon}^{(\notri)}$ and $\bcA_{\epsilon}^{(\notri)}$,
see \eqref{eq:cA0}  and \eqref{eq:bcA}.

\smallskip
\item Step 2. We bound $\cZ_{N,\epsilon}^{(\notri)}(\varphi,\psi)
- \cZ_{N,\epsilon}^{(\diff)}(\varphi,\psi)$, see
Lemma~\ref{th:diff} in Section~\ref{sec:5.2}: this is
the cost of imposing diffusive constraints,
including $\bcA_{\epsilon; \, b,c}^{(\diff)}$ in \eqref{eq:bcA2} and the diffusive truncation
in the definition of $\Theta^{(\cg)}_{N, \epsilon}$ in \eqref{eq:Theta}.

\smallskip
\item Step 3. We bound $\cZ_{N,\epsilon}^{(\diff)}(\varphi,\psi)
- \mathscr{Z}_{\epsilon}^{(\cg)}(\varphi,\psi|\Theta_{N,\epsilon})$,
see Lemma~\ref{th:cg} in Section~\ref{sec:5.3}: this is
the cost of the kernel replacement \eqref{eq:kerrep}.
\end{itemize}
Combining Lemmas~\ref{th:no3}, \ref{th:diff} and~\ref{th:cg} then gives Theorem~\ref{th:cg-main}.

In the last Subsection~\ref{sec:5.4}, we
will prove a separate second moment estimate for the coarse-grained model,
which is needed later in Section~\ref{Sec:4MomCoarse} for higher moment bounds.

\smallskip

The proof details in this section are technically a bit heavy and could be skipped in a first reading.

\subsection{Step 1: Constraints on mesoscopic time variables}\label{sec:5.1}
In this step, we introduce our first approximation $\cZ_{N,\epsilon}^{(\notri)}(\varphi,\psi)$
and show that it is close to  $\cZ_N(\varphi,\psi)$.

Recall the mesoscopic time intervals $\cT_{\epsilon N}(\sfi) : = ((\sfi-1)\epsilon N, \, \sfi \epsilon N]$
introduced in \eqref{eq:cB}, to which we associate the mesoscopic time index $\sfi \in \{1, \ldots, \frac{1}{\eps}\}$.
In the chaos expansion for $\cZ_N^{\beta_N}(\varphi, \psi)$ in \eqref{eq:Zpolyav}, each time index $n_j$ belongs to
$\cT_{\epsilon N}(\sfi)$ for some $\sfi\in \{1, \ldots, \frac{1}{\epsilon}\}$. The first step of coarse-graining is
to group the terms in the expansion in \eqref{eq:Zpolyav} in terms of the mesoscopic time intervals
$\cT_{\epsilon N}(\cdot)$ visited by the sequence of time indices $n_1, \ldots, n_r$. Namely, we can rewrite \eqref{eq:Zpolyav} as
(omitting $\beta_N$ from $\cZ^{\beta_N}_N(\varphi,\psi)$, and expanding $q^N_{0,n_1}(\varphi, z_1)$ and $q^N_{n_r, N}(z_r, \psi)$ according to their definitions in
\eqref{eq:qNphi}-\eqref{eq:qNpsi}):
\begin{equation} \label{eq:Zpoly3}
\begin{split}
	\cZ_N(\varphi,\psi)
	\ & = \, q_{0,N}^N(\varphi,\psi) \\
	& \ \ + \ \frac{1}{N}
	\sum_{v,w \in \Z^2_\even} \varphi_N(v) \
	\sum_{k=1}^{\infty} \ \sum_{0 < \sfi_1 < \ldots < \sfi_k \le \frac{1}{\epsilon}} \
	\sum_{\substack{d_1 \le f_1 \,\in\,
	\cT_{\epsilon N}(\sfi_1), \, \ldots\, , \, d_k \le f_k \,\in\, \cT_{\epsilon N}(\sfi_k) \\
	x_1,\, y_1, \, \ldots \,, x_k, \, y_k \,\in\, \Z^2}} \\
	 q_{0,d_1}(v,x_1)
	& \, X_{d_1,f_1}(x_1,y_1) \, \Bigg\{ \prod_{j=2}^k q_{f_{j-1},d_j}(y_{j-1},x_j) \,
	X_{d_j,f_j}(x_j,y_j) \Bigg\} \,
	q_{f_k, N}(y_k, w)  \,  \psi_N(w)  \,,
\end{split}
\end{equation}
where $\varphi_N$, $\psi_N$ were defined in \eqref{eq:phiNpsiN}, and $X_{d,f}(x,y)$ was defined in \eqref{eq:X}.

Recall from \eqref{eq:KM} that $K_\epsilon= (\log \frac{1}{\eps})^6$.
We will show that in \eqref{eq:Zpoly3}, the
dominant contribution (in $L^2$) comes from mesoscopic time variables $(\sfi_1, \ldots, \sfi_k)$
which contains
\emph{no consecutive triples $\sfi_j, \sfi_{j+1}, \sfi_{j+2}$
with both $\sfi_{j+1}- \sfi_j < K_\epsilon$ and $\sfi_{j+2}-\sfi_{j+1} < K_\epsilon$}.
This is encoded in $\cA_{\epsilon}^{(\notri)}$ from \eqref{eq:cA0}, which we recall here
\begin{equation}\label{eq:cA}
\begin{aligned}
	\cA_{\epsilon}^{(\notri)} \,:=\,
	\bigcup_{k\in\N}
	\Big\{(\sfi_1, \ldots, \sfi_k) \in \N^k : \
	& K_\eps \le
	\sfi_1 < \sfi_2 < \ldots < \sfi_k
	 \leq \lfloor\tfrac{1}{\epsilon}\rfloor -K_\eps
	\ \ \text{such that}  \\
	& \ \ \text{if} \ \ \sfi_{j+1}-\sfi_j < K_\epsilon \,, \ \text{then}
	\ \ \sfi_{j+2} - \sfi_{j+1} \geq K_\epsilon \ \Big\} \,.
\end{aligned}
\end{equation}

We will further restrict the sum in \eqref{eq:Zpoly3} to $(\sfi_1, \ldots, \sfi_k)$ with $k \le (\log\frac{1}{\epsilon})^2$,
which leads to the following first approximation of $\cZ_N(\varphi,\psi)$:
\begin{equation} \label{eq:Zno3}
\begin{split}
	& \cZ_{N,\epsilon}^{(\notri)}(\varphi,\psi)
	\, := \, q_{0,N}^N(\varphi,\psi) \ + \ \frac{1}{N}
	\sum_{v,w \in \Z^2_\even} \varphi_N(v) \\
	& \qquad\qquad\qquad\qquad\qquad\qquad\sum_{k=1}^{(\log \frac{1}{\epsilon})^2} \
	\sum_{(\sfi_1, \ldots, \sfi_k) \in \cA_\epsilon^{(\notri)}} \
	\sum_{\substack{d_1 \le f_1 \,\in\,
	\cT_{\epsilon N}(\sfi_1), \, \ldots\, , \, d_k \le f_k \,\in\, \cT_{\epsilon N}(\sfi_k) \\
	x_1,\, y_1, \, \ldots \,, x_k, \, y_k \,\in\, \Z^2}} \\
	& \ q_{0,d_1}(v,x_1)
	 \, X_{d_1,f_1}(x_1,y_1)
	\Bigg\{ \prod_{j=2}^k q_{f_{j-1},d_j}(y_{j-1},x_j) \,
	X_{d_j,f_j}(x_j,y_j) \Bigg\} \,
	 q_{f_k, N}(y_k, w)  \,  \psi_N(w)  \,.
\end{split}
\end{equation}

The main result of this subsection is the following approximation,
which constitutes part of the bound in \eqref{eq:cg-approx}.
The proof is a bit lengthy, but it contains many important ingredients,
including a key renewal interpretation of second moment bounds.

\begin{lemma}[No close triples]\label{th:no3}
Recall from \eqref{eq:KM} that $K_\eps=(\log \frac{1}{\eps})^6$ and recall $\Vert \cdot\Vert_{\cG_t}$ from \eqref{eq:phinorm}.
There exists  $\sfC \in (0, \infty)$ such that for $\epsilon > 0$ small enough, we have:
for all $\varphi$ with $\Vert \varphi\Vert_{\cG_1}^2<\infty$ and $\psi \in L^\infty(\R^2)$,
\begin{equation}\label{eq:appr1no3}
\begin{gathered}
	\limsup_{N\to\infty \text{ with } N\in2\N}  \big\|\big(
	\cZ_{N,\epsilon}^{(\notri)} - \cZ_N\big)(\varphi, \psi) \big\|_{L^2}^2
	 \ \le \
	 \sfC\, \bigg(\Vert \varphi\Vert_{\cG_{\scriptscriptstyle K_\eps \eps }}^2
	+ \frac{(\log K_\epsilon)^2}{\log\frac{1}{\epsilon}} \Vert \varphi\Vert_{\cG_1}^2 \bigg)
	\, \Vert \psi\Vert_\infty^2 \,.
\end{gathered}
\end{equation}
\end{lemma}

\begin{proof}
The random variables $X_{d,f}(x,y)$ depend on the disorder variables $\xi_N(n,x)$
for $d \le n \le f$, see \eqref{eq:X}. They are centered and orthogonal in $L^2$ and,
by \eqref{eq:secondX} and \eqref{eq:barU},
\begin{equation*}
	\bbE[X_{d,f}(x,y) \, X_{d',f'}(x',y')] = \ind_{\{(d,f,x,y) = (d',f',x',y')\}}
	\, \overline{U}_N(f-d,y-x) \,.
\end{equation*}
Since the sum which defines $\cZ_{N,\epsilon}^{(\notri)}(\varphi,\psi)$
\emph{is a subset of that of $\cZ_{N}(\varphi,\psi)$},
cf.\ \eqref{eq:Zpoly3} and \eqref{eq:Zno3}, it follows that we can write
\begin{equation} \label{eq:interm}
\begin{split}
	& \big\|\big(
	 \cZ_{N,\epsilon}^{(\notri)} - \cZ_N\big)(\varphi, \psi)
	\big\|_{L^2}^2  \\
	& \quad \,=\, \frac{1}{N^2} \
	\Bigg\{ \underbrace{\sum_{k > (\log\frac{1}{\epsilon})^2} \
	\sum_{0 < \sfi_1 < \ldots < \sfi_k \le \frac{1}{\epsilon}}}_{\mathrm{I}_{N,\epsilon}} \ + \
	\underbrace{\sum_{k=1}^{(\log\frac{1}{\epsilon})^2} \
	\sum_{(\sfi_1, \ldots, \sfi_k) \in (\cA_\epsilon^{(\notri)})^c}}_{\mathrm{II}_{N,\epsilon}} \Bigg\} \\
	& \quad \times \!\!\!\! \sum_{\substack{d_1 \le f_1 \,\in\, \cT_{\epsilon N}(\sfi_1),
	\, \ldots\, , \, d_k \le f_k \,\in\, \cT_{\epsilon N}(\sfi_k) \\
	x_1,\, y_1, \, \ldots \,, x_k, \, y_k \,\in\, \Z^2}}
	\,  \bigg( \sum_{v \in \Z^2_\even} \varphi_N(v) \,
	q_{0,d_1}(v,x_1) \bigg)^2
	\ \overline{U}_N(f_1-d_1,y_1-x_1)  \\
	& \quad \times \Bigg\{ \prod_{j=2}^k q_{f_{j-1},d_j}(y_{j-1},x_j)^2 \
	\overline{U}_N(f_j-d_j,y_j-x_j) \Bigg\}
	 \bigg( \sum_{w \in \Z^2_\even} q_{f_k, N}(y_k, w)
	\, \psi_N(w) \bigg)^2,
\end{split}
\end{equation}
where $\mathrm{I}_{N,\epsilon}$ and $\mathrm{II}_{N,\epsilon}$ are
the contributions of $\{k > (\log\frac{1}{\epsilon})^2\}$
and $\{(\sfi_1, \ldots, \sfi_k) \in (\cA_\epsilon^{(\notri)})^c\}$. We split the proof in two parts,
where we show that for some $\sfC<\infty$ we have:
\begin{align}\label{eq:toshIII0a}
	\limsup_{N\to\infty \text{ with } N\in2\N} \ \mathrm{I}_{N,\epsilon}
	\ &\le \ \sfC  \, \frac{1}{\log\frac{1}{\epsilon}} \, \Vert \varphi\Vert_{\cG_1}^2 \,
	\Vert \psi\Vert_\infty^2  \,, \\
	\label{eq:toshIII0b}
	\limsup_{N\to\infty \text{ with } N\in2\N} \ \mathrm{II}_{N,\epsilon}
	\ &\le \
	 \sfC \, \Bigg(\Vert \varphi\Vert_{\cG_{\scriptscriptstyle K_\eps \eps }}^2
	+ \frac{(\log K_\epsilon)^2}{\log\frac{1}{\epsilon}} \Vert \varphi\Vert_{\cG_1}^2 \Bigg)
	\, \Vert \psi\Vert_\infty^2 \,.
\end{align}

\smallskip

\begin{remark}\label{rem:interm}
Let us sketch a probabilistic interpretation of \eqref{eq:interm}.
From \eqref{eq:Uren}, we recall that the expansion for
$\overline{U}_N(f-d, y-x)$ has a time-space renewal
interpretation, and from  \eqref{eq:m2-repr}
the expansion of  $\bbE\big[\mathcal{Z}_{N}(\varphi,\psi)^2 \big]$ consists of a mixture of
$\overline{U}_N(f-d, y-x)$,
with weight $\frac{1}{N^2}q^N_{0,d}(\varphi, x)^2 q^N_{f, Nt}(y, \psi)^2$. We can therefore write
\begin{equation*}
	\bbE\big[\mathcal{Z}_{N}(\varphi,\psi)^2 \big]
	= \sum_{\substack{\cS =\{(n_1, z_1), \ldots, (n_r, z_r)\} \\
	1\leq n_1<\cdots n_r \leq N, \, z_1, \ldots, z_r\,\in \Z^2}} \sfM_N^{\varphi,\psi}(\cS),
\end{equation*}
where denoting $\cS=\{(n_1, z_1), \ldots, (n_r, z_r)\}$ with $r=|\cS|$, we define
\begin{equation}
\begin{aligned}
	\sfM_N^{\varphi,\psi}(\cS) & := \frac{\sigma_N^{2r}R_N^{2(r-1)}}{N^2}
	\bigg( \sum_{v \in \Z^2_\even} \varphi_N(v) \, q_{0,n_1}(v, z_1) \bigg)^2
	\bigg( \sum_{w \in \Z^2_\even} q_{n_r, N}(z_r, w) \, \psi_N(w) \bigg)^2 \\
	& \qquad \times \P\Big(\big(\tau_i^{(N)}, S_i^{(N)}\big)
	= (n_i, z_i) \, \forall\, 1\leq i\leq r \, \Big|\, (\tau_1^{(N)}, S_1^{(N)}) = (n_1, z_1)\Big).
\end{aligned}
\end{equation}
The measure $\sfM_N^{\varphi,\psi}(\cdot)$ is called a {\em spectral measure} since
$\sfM_N^{\varphi,\psi}(\cS)$ equals the square of the coefficient of
$\prod_{(n_i, z_i)\in \cS} \xi_N(n_i, z_i)$ in the chaos expansion \eqref{eq:Zpolyav}
for $\mathcal{Z}_{N}(\varphi,\psi)$, where different terms in the expansion
are orthogonal in $L^2$, similar to
a Fourier decomposition. For more on spectral measure, see e.g.~\cite{GS12}.

The r.h.s.\ of \eqref{eq:interm} can then be written as
\begin{equation}\label{eq:interm-bis}
	\big\|\big(
	 Z_{N,\epsilon}^{(\notri)}
	 - Z_N\big)(\varphi, \psi)
	\big\|_{L^2}^2 \,=\, \sfM_N^{\varphi,\psi}
	\Big( |\cI(\cS)| > (\log\tfrac{1}{\epsilon})^2
	\, \text{ or } \, \cI(\cS) \in (\cA_\epsilon^{(\notri)})^c \Big) \,,
\end{equation}
where given $\cS \subseteq \{1,\ldots, N\} \times \Z^2$,
\begin{equation*}
	\cI(\cS) := \big\{\sfi \in \{1,\ldots, \tfrac{1}{\epsilon}\}: \
	\cS \cap \big( \cT_{\epsilon N}(\sfi) \times \Z^2 \big) \ne \emptyset \big\} \,.
\end{equation*}
Thanks to Lemma~\ref{lem:Levyconv}, it can be shown that as $N\to\infty$,  $\sfM_N^{\varphi,\psi}$
converges to a similarly re-weighted measure for the continuum time-space renewal process introduced in
Lemma~\ref{lem:Levyconv}, whose time component is the Dickman subordinator $(Y_s)_{s\geq 0}$ with
exponentially weighted Green's function~$G_\theta$, see \eqref{eq:G3} and~\eqref{eq:G1}.
\end{remark}

\medskip

\noindent
\textbf{Second Moment Bound via Renewal.}
We first explain the common steps in bounding $\mathrm{I}_{N,\epsilon}$ and $\mathrm{II}_{N,\epsilon}$
from \eqref{eq:interm}, which also applies to the variance of $\cZ_N(\varphi, \psi)$ and the mean
squared error of later approximations. The common feature is that they all have the same expansion
as in \eqref{eq:interm}, except the summation constraints are different.

Consider $\mathrm{II}_{N,\epsilon}$ from \eqref{eq:interm}. We first sum out the space variables
in \eqref{eq:interm}. Recall \eqref{eq:phiNpsiN} and note that
\begin{equation} \label{eq:normsphipsi}
	\|\psi_N\|_{\ell^\infty} := \sup_{w\in\Z^2} |\psi_N(w)|
	\le \|\psi\|_\infty \,,
\end{equation}
so that in \eqref{eq:interm} we can bound
\begin{equation} \label{eq:firstsquare}
	\bigg( \sum_{w \in \Z^2_\even} q_{f_k, N}(y_k, w)
	\, \psi_N(w) \bigg)^2
	\le \|\psi_N\|_{\ell^\infty}^2
	\le  \|\psi\|_\infty^2 \,.
\end{equation}
We can plug this bound into \eqref{eq:interm} and sum over the space variables
in reverse order, from
$y_k, x_k$ until $y_2, x_2, y_1$, thus replacing
$\overline{U}_N(f_j-d_j,y_j-x_j)$ by $\overline{U}_N(f_j-d_j)$
and $q_{f_{j-1},d_j}(y_{j-1},x_j)^2$ by $u(d_j - f_{j-1})$,
see \eqref{eq:barU} and \eqref{eq:u}.
Finally, we sum over $x_1$ and observe that
\begin{equation} \label{eq:secondsquare}
\begin{split}
	\sum_{x_1\in\Z^2}
	\bigg( \sum_{v \in \Z^2_\even} \varphi_N(v) \,
	q_{0,d_1}(v,x_1) \bigg)^2
	=  \!\!\!\!\!\sum_{v, v' \in \Z^2_\even} \!\!\!\!\! \varphi_N(v)  \varphi_N(v') \,
	q_{2d_1}(v - v')
	=: N\, \Phi_{N}\Big(\frac{d_1}{N}\Big) \,,
\end{split}
\end{equation}
where we introduced the function $\Phi_N$.
Substituting these bounds into \eqref{eq:interm} then gives
\begin{equation} \label{eq:anana}
\begin{split}
	\mathrm{II}_{N,\epsilon} \le
	\frac{C \, \|\psi\|_\infty^2}{N}
	\sum_{k=1}^{(\log\frac{1}{\epsilon})^2}
	\sum_{(\sfi_1, \ldots, \sfi_k) \in (\cA_\epsilon^{(\notri)})^c} \
	\sum_{d_1 \le f_1 \,\in\, \cT_{\epsilon N}(\sfi_1) , \,
	\ldots, \, d_k \le f_k \,\in\, \cT_{\epsilon N}(\sfi_k)} & \\
	\Phi_N\Big(\frac{d_1}{N}\Big)
	\, \overline{U}_N(f_1-d_1) \, \Bigg\{ \prod_{j=2}^k u(d_j-f_{j-1}) \,
	\overline{U}_N(f_j-d_j) \Bigg\} &  \,.
\end{split}
\end{equation}
A similar estimate can be derived for $\mathrm{I}_{N,\epsilon}$, with a corresponding
summation constraint.

We now compute the limit as $N\to\infty$ of the r.h.s.\ of \eqref{eq:anana}.
Recalling $\varphi_N$ from \eqref{eq:phiNpsiN}, the local limit theorem \eqref{eq:llt},
and \eqref{eq:gtrans},
if $d_1/N\to s\in (0,1)$, then for $\Phi_N$ in \eqref{eq:secondsquare} we have
\begin{equation}\label{eq:PhiNlim}
\begin{aligned}
	\Phi_{N}\Big(\frac{d_1}{N}\Big) :=\ \ & \frac{1}{N} \sum_{v, v' \in \Z^2_\even}
	\varphi_N(v) \, \varphi_N(v') \, q_{2d_1}(v - v') \\
	\asto{N} & \iint_{\R^2\times\R^2} \varphi(z) \varphi(z')  g_s(z'-z)   \, \dd z \dd z' =: \Phi(s),
\end{aligned}
\end{equation}
where we note that, by the definition of $\Vert\varphi\Vert_{\cG_t}$ from \eqref{eq:phinorm},
\begin{equation}\label{eq:intf}
\int_0^t \Phi(s) \dd s = \Vert \varphi\Vert_{\cG_t}^2.
\end{equation}

We will use the following result, which says that as $N\to\infty$,
for each $\sfI:=(\sfi_1, \ldots, \sfi_k)\subset \{1, \ldots, \frac{1}{\eps}\}$,
the term in \eqref{eq:anana} converges to a limit that can be interpreted in terms
of the Dickman subordinator, as mentioned in Remark~\ref{rem:interm}.
\begin{lemma}\label{th:asycom}
Let $\Phi_N$ and $\Phi$ be defined as in \eqref{eq:PhiNlim}. For any fixed $\epsilon > 0$, $k\in\N$
and $\sfI:=\{\sfi_1, \ldots, \sfi_k\}\subset \{1, \ldots, \frac{1}{\eps}\}$
with $\sfi_1 < \sfi_2 < \ldots < \sfi_k$, we have
\begin{equation}\label{eq:usefullimit}
\begin{split}
	\lim_{N\to\infty}
	& \ \frac{1}{N}
	\sum_{d_1 \le f_1 \,\in\, \cT_{\epsilon N}(\sfi_1) , \,
	\ldots, \, d_k \le f_k \,\in\, \cT_{\epsilon N}(\sfi_k)} \\
	& \qquad\quad
	\Phi_N\big(\tfrac{d_1}{N}\big) \, \overline{U}_N(f_1-d_1)
	\, \Bigg\{ \prod_{j=2}^k u(d_j-f_{j-1}) \,
	 \overline{U}_N(f_j-d_j) \Bigg\}  \, = \, \cI_\epsilon^{\Phi}(\sfI) \,,
\end{split}
\end{equation}
with
\begin{equation}\label{eq:cIeps}
	\begin{split}
	\cI_\epsilon^{\Phi}(\sfI) :=  \ \
	\idotsint\limits_{a_1 \le b_1 \,\in\, \cT_{\epsilon}(\sfi_1) \,,
	\, \ldots\, , \, a_k \le b_k \,\in\, \cT_{\epsilon}(\sfi_k)}
	\dd a_1 \, \dd b_1 \, \cdots \, \dd a_k \, \dd b_k
	\qquad\qquad\qquad\qquad  & \\
	{\Phi(a_1)} \, G_\theta(b_1-a_1)
	\Bigg\{ \prod_{j=2}^k \frac{1}{a_j-b_{j-1}} \, G_\theta(b_j-a_j) \Bigg\} & \,,
\end{split}
\end{equation}
where $\cT_{\epsilon}(\sfi):= (\epsilon (\sfi-1), \epsilon \sfi]$, and $G_\theta$ $($see \eqref{eq:G1}
and \eqref{eq:G3}$)$
is the weighted Green's function for the Dickman subordinator
with L\'evy measure $\frac{1}{t}\ind_{(0,1)}(t){\rm d}t$ introduced in Lemma~\ref{lem:Levyconv}.
\end{lemma}
\begin{proof}
If we introduce the macroscopic variables $a_i := d_i/N$, $b_i := f_i/N$
in \eqref{eq:usefullimit}, the sums  converge to corresponding integrals
as $N\to\infty$ (for fixed $\epsilon > 0$),
by the asymptotic expansions  \eqref{eq:u}, \eqref{eq:sigma}, \eqref{eq:asU1}
and \eqref{eq:asU1ub} for $u(\cdot)$, $\sigma_N^2$,  $\overline{U}_N(\cdot)$
(also recall \eqref{eq:RN}), as well as the local limit theorem \eqref{eq:llt}.
This gives \eqref{eq:usefullimit}-\eqref{eq:cIeps}.
\end{proof}

We can interpret $\cI_\epsilon^{\Phi}(\sfI)$ in \eqref{eq:cIeps} as the weight associated to a Dickman subordinator. More precisely, recall that $G_\theta$ is the weighted Green's function of the Dickman subordinator $Y$ introduced in Lemma~\ref{lem:Levyconv}, and satisfies the following \emph{renewal property}
\cite[eq.~(6.14)]{CSZ19a}:
\begin{equation}\label{eq:renewal-Dickman}
	\forall s < \bar t < t: \qquad
	G_\theta(t-s) = \iint_{u \in (s,\bar t), \, v \in (\bar t, t)} G_\theta(u-s) \,
	\frac{1}{v-u} \, G_\theta(t-v) \, \dd u \, \dd v \,.
\end{equation}
In \eqref{eq:cIeps}, let us denote by $\cI_\epsilon^{s,t}(\sfi_1, \ldots, \sfi_k)$
the integral where the extreme variables $a_1$ and $b_k$
are not integrated out but rather fixed to be $s$ and $t$ respectively, namely,
\begin{equation}\label{eq:cIepsst}
	\begin{split}
	\cI_\epsilon^{s,t}(\sfI) :=  \ \
	\idotsint\limits_{s \le b_1 \,\in\, \cT_{\epsilon}(\sfi_1) \,,
	\, \ldots\, , \, a_k \le t \,\in\, \cT_{\epsilon}(\sfi_k)}
	\dd b_1 \, \dd a_2 \, \dd b_2 \, \cdots \, \dd a_{k-1} \, \dd b_{k-1} \, \dd a_k
	\qquad\quad  & \\
	G_\theta(b_1-s)
	\Bigg\{ \prod_{j=2}^{k-1} \frac{1}{a_j-b_{j-1}} \, G_\theta(b_j-a_j) \Bigg\}
	\frac{1}{a_k-b_{k-1}} \, G_\theta(t-a_k)
	& \,.
\end{split}
\end{equation}
This is the weight of renewal configurations that only visit the intervals $\cT_\epsilon(\sfi_1)$, \ldots, $\cT_\eps(\sfi_k)$, and $a_j, b_j$ are the first and last renewal points in $\cT_\eps(\sfi_j)$, while $\frac{1}{a_{j}-b_{j-1}}$ comes from the L\'evy measure of the Dickman subordinator.  An iterative application of \eqref{eq:renewal-Dickman} then shows that
\begin{equation}\label{eq:renewal-Dickman2}
\begin{gathered}
	\forall \sfj \le \sfj', \ \forall s \in \cT_\epsilon(\sfj),
	\ \forall t \in \cT_\epsilon(\sfj'): \\
	\sum_{k=1}^\infty \ \sum_{\sfj =: \sfi_1 < \sfi_2 < \ldots
	< \sfi_{k-1} < \sfi_k := \sfj'} \ \cI_\epsilon^{s,t}(\sfi_1, \ldots, \sfi_k)
	\,=\, G_\theta(t-s) \,,
\end{gathered}
\end{equation}
which is just a renewal decomposition by summing over the set of possible intervals $\cT_\eps(\sfi)$,
$\sfj \leq \sfi \leq\sfj'$, visited by the Dickman subordinator $Y$, given that $s,t$ are in the range of $Y$.

Applying Lemma~\ref{th:asycom} to \eqref{eq:anana} then gives
\begin{equation} \label{eq:cIeps0}
	\limsup_{N\to\infty} \,
	\mathrm{II}_{N,\epsilon}
	\,\le \, C  \, \|\psi\|_\infty^2  \,
	\sum_{k=1}^{(\log \frac{1}{\epsilon})^2} \,
	\sum_{(\sfi_1, \ldots, \sfi_k) \in (\cA_\epsilon^{(\notri)})^c} \
	\cI_\epsilon^{\Phi}(\sfi_1, \ldots, \sfi_k) \,,
\end{equation}
With the same arguments, we obtain a corresponding bound for $\mathrm{I}_{N,\epsilon}$:
\begin{equation} \label{eq:cIeps02}
	\limsup_{N\to\infty} \,
	\mathrm{I}_{N,\epsilon}
	\,\le \, C  \, \|\psi\|_\infty^2
	\, \sum_{k > (\log\frac{1}{\epsilon})^2} \,
	\sum_{0 < \sfi_1 < \ldots < \sfi_k \le \frac{1}{\epsilon}} \
	\cI_\epsilon^{\Phi}(\sfi_1, \ldots, \sfi_k) \,.
\end{equation}
To complete the proof of Lemma~\ref{th:no3},
it remains to derive \eqref{eq:toshIII0a} and \eqref{eq:toshIII0b}
from these bounds. We start with $\mathrm{II}_{N,\epsilon}$, which is more involved,
but we first make a remark.

\begin{remark}[Variance bound]\label{rem:variance}
If we remove any constraint on $k$ and $(\sfi_1, \ldots, \sfi_k)$
from  formula \eqref{eq:interm}, summing over all $k \in \N$
and $0<\sfi_1<\cdots <\sfi_k\leq \frac{1}{\eps}$,
we obtain $\mathbb{V}{\rm ar}(\cZ_N(\varphi, \psi))$
(recall \eqref{eq:Zpolyav}).
We thus have a simpler analogue of \eqref{eq:cIeps0} and \eqref{eq:cIeps02}:
\begin{equation}\label{eq:VarZbdd}
\begin{aligned}
\limsup_{N\to\infty} \,
	\mathbb{V}{\rm ar}(\cZ_N(\varphi, \psi))
	&\,\le \, C  \, \|\psi\|_\infty^2
	\sum_{k=1}^\infty \,
	\sum_{0<\sfi_1<\cdots <\sfi_k\leq \frac{1}{\eps}} \
	\cI_\epsilon^{\Phi}(\sfi_1,\ldots,\sfi_k)  \\
	&\, \le \, C  \, \|\psi\|_\infty^2  \iint\limits_{0<s<t<1} \Phi(s) G_\theta(t-s) \, \dd t \, \dd s \\
	&\, \le\, C  \, \|\psi\|_\infty^2
	\bigg(\int_0^1 G_\theta(t) \dd t \bigg) \bigg( \int_0^1 \Phi(s) \, \dd s \bigg)
	= C c_\theta \|\psi\|_\infty^2 \Vert \varphi\Vert_{\cG_1}^2,
\end{aligned}
\end{equation}
where in the second inequality we applied the renewal decomposition \eqref{eq:renewal-Dickman2},
with $s$ and $t$
being the first and last renewal points, we denoted $c_\theta:=\int_0^1 G_\theta(t)\dd t$,
see \eqref{eq:Gas},
and recalled \eqref{eq:intf}. Note that this bound is the same as the one in \eqref{eq:m2-lim1}
and does not depend on $\eps$.
\end{remark}

\medskip

\noindent
{\bf Bound for $\mathrm{II}_{N,\epsilon}$: proof of \eqref{eq:toshIII0b}.}
We start with \eqref{eq:cIeps0}. The constraint $(\cA_\epsilon^{(\notri)})^c$ contains
$\sfI=(\sfi_1,\ldots, \sfi_k)$ with either $1\leq \sfi_1 <K_\eps$, or $\sfi_k > \frac{1}{\eps} -K_\eps$,
or $\sfi_{j+1}-\sfi_j, \sfi_{j+2}-\sfi_{j+1} < K_\epsilon$ for some $j$ (hence $k\ge 3$). We will treat the three cases
one by one.

For the first case $1\leq \sfi_1 <K_\eps$, omitting the factor $\Vert\psi\Vert_\infty^2$, its contribution
is bounded by
\begin{equation} \label{eq:IIa}
	\sum_{k=1}^{(\log\frac{1}{\epsilon})^2} \!\!\!\!
	\sum_{\substack{0 < \sfi_1 <K_\eps \\ \sfi_1< \ldots < \sfi_k \leq \frac{1}{\epsilon}}}
	\cI_\epsilon^{\Phi}(\sfi_1, \ldots, \sfi_k)
	\leq
	\iint_{\substack{0<s<K_\eps \eps \\ s<t<1}} \Phi(s) \, G_\theta(t-s) \, \dd s\, \dd t
	\leq c_\theta \Vert \varphi\Vert_{\cG_{\scriptscriptstyle K_\eps \eps }}^2,
\end{equation}
where we applied the renewal decomposition \eqref{eq:renewal-Dickman2} as in \eqref{eq:VarZbdd}, and recalled from \eqref{eq:intf} that $\Vert \varphi\Vert_{\cG_{\scriptscriptstyle K_\eps \eps}}^2= \int_0^{K_\eps \eps} \Phi(s) \dd s$. This gives the first term in \eqref{eq:toshIII0b}.

For the second case $\sfi_k > \frac{1}{\eps} -K_\eps$, omitting $\Vert\psi\Vert_\infty^2$, its contribution
is bounded by
\begin{align*}
	\sum_{k=1}^{(\log\frac{1}{\epsilon})^2}
	\sum_{0 < \sfi_1 < \ldots < \sfi_{k-1} < \sfi_k \atop
	                          \frac{1}{\eps} -K_\eps < \sfi_k \leq \frac{1}{\epsilon}}
	\cI_\epsilon^{\Phi}(\sfi_1, \ldots, \sfi_k)
	\leq\ &
	\iint_{0<s<t \atop 1-K_\eps \eps <t<1} \Phi(s) \, G_\theta(t-s) \, \dd s\, \dd t  \\
	\leq\ & C \bigg(\int_0^{K_\eps \eps} G(t) \dd t\bigg) \bigg(\int_0^1 \Phi(s) \dd s \bigg)
	\leq  \frac{C \Vert \varphi\Vert_{\cG_1}^2}{\log\frac{1}{\epsilon}} \,,
\end{align*}
where in the second inequality, we used that  $\int_{a}^{a+K_\eps \epsilon} G_\theta(t) \, \dd t
\le C \int_0^{K_\eps \eps} G_\theta(t) \dd t\leq C/ \log \frac{1}{\epsilon}$ uniformly in $\eps$
small enough and $a\in (0,1)$ (by \eqref{eq:Gas}-\eqref{eq:G5} and the choice
\eqref{eq:KM} of $K_\epsilon$). This bound is much smaller than
the second term in \eqref{eq:toshIII0b}.

For the third case with $\sfi_{j+1}-\sfi_j, \sfi_{j+2}-\sfi_{j+1} < K_\epsilon$ for some $j$, we need to bound
\begin{align*}
	\cW_\epsilon
	:=\ &  \sum_{k=3}^{(\log \frac{1}{\epsilon})^2} \ \sum_{j=1}^{k-2} \
	\sum_{\substack{0 < \sfi_1 < \ldots < \sfi_k \le \frac{1}{\epsilon} \\
	\sfi_{j+1}-\sfi_j < K_\epsilon \text{ and }
	\sfi_{j+2}-\sfi_{j+1} < K_\epsilon}}
	\!\!\!\!\! \cI_\epsilon^{\Phi}(\sfi_1, \ldots, \sfi_k)  \\
	\leq \ & \sum_{\substack{0 < \sfi < \sfi' < \sfi'' \le \frac{1}{\epsilon}\\
	|\sfi'-\sfi| < K_\epsilon \\ |\sfi'' - \sfi'| < K_\epsilon}}
	\ \iiiint\limits_{\substack{b \in \cT_\epsilon(\sfi)\\ a' \le b' \in \cT_\epsilon(\sfi') \\
	a'' \in \cT_\epsilon(\sfi'')}} \,
	\dd b \, \dd a' \, \dd b' \, \dd a'' \
	\int_0^b \dd s \ \int_{a''}^1 \dd t \\
	& \hspace{2cm} {\Phi(s) \, }
	G_\theta(b-s) \, \frac{1}{a'-b} \, G_\theta(b'-a') \, \frac{1}{a''-b'} \,
	G_\theta(t-a'') \,.
\end{align*}
where we again applied the renewal decomposition \eqref{eq:renewal-Dickman2}. Bounding the integral of
$G_\theta(t-a'')$ over $t$
by $c_\theta = \int_0^1 G_\theta(u) \, \dd u$, we obtain
\begin{equation}\label{eq:progbound1}
\begin{split}
	\cW_\epsilon \,\le\,
	\sum_{\sfi = 1}^{\frac{1}{\epsilon}}
	\sum_{\sfi'=\sfi+1}^{\sfi+K_\epsilon}
	\sum_{\sfi''=\sfi'+1}^{\sfi'+K_\epsilon}
	\iiiint\limits_{\substack{b \in \cT_\epsilon(\sfi)\\ a' \le b' \in \cT_\epsilon(\sfi') \\
	a'' \in \cT_\epsilon(\sfi'')}} & \!\!\!\! \dd b \, \dd a' \, \dd b' \, \dd a'' \,
	\int_0^b c_\theta  \Phi(s) \frac{G_\theta(b-s)}{a'-b}\frac{G_\theta(b'-a')}{a''-b'} \, \dd s  \,.
\end{split}
\end{equation}
Note that if we restrict the sum to $2\leq \sfi'-\sfi, \sfi''-\sfi' \le K_\eps$, then using \eqref{eq:G5}, it is not difficult to see that the integrals can be bounded by $C \frac{(\log K_\eps)^2}{\log \frac{1}{\eps}}\int_0^1\Phi(s) \dd s$. Complications only arise when $\sfi'=\sfi+1$ or $\sfi''=\sfi'+1$.

We will proceed in three steps.
The following bound will be used repeatedly:
\begin{equation} \label{eq:boint}
	\forall \delta \in (0,\tfrac{1}{2})\,, \ \forall z \in [\delta,\infty): \qquad
	\int_0^\delta \, G_\theta(x) \,
	\log\bigg(1+\frac{z}{\delta -x}\bigg) \, \dd x
	\,\le\, C \, \frac{\log(1+\frac{z}{\delta})}{\log\frac{1}{\delta}} \,.
\end{equation}
Indeed, splitting the integral over $(0,\frac{\delta}{2})$ and $(\frac{\delta}{2}, \delta)$
and exploiting \eqref{eq:Gas}, we note that:
\begin{itemize}
\item for $x < \frac{\delta}{2}$ we have
$\log(1+\frac{z}{\delta -x}) \le
\log(1+\frac{z}{\delta/2}) \le C \, \log(1+\frac{z}{\delta})$ and
$\int_0^{\delta/2} G_\theta(x) \, \dd x \le \frac{C'}{\log\frac{1}{\delta}}$;

\item for $x > \frac{\delta}{2}$ we can bound
$G_\theta(x) \le \frac{C}{\delta (\log\frac{1}{\delta})^2}$ and,
by the change of variable $t := \frac{z}{\delta-x}$,
\begin{equation} \label{eq:bai}
	\int_{0}^\delta \log\bigg(1+\frac{z}{\delta -x}\bigg) \, \dd x
	= z \int_{\frac{z}{\delta}}^\infty \frac{\log(1+t)}{t^2} \, \dd t
	\le C' \, \delta \, \log\bigg( 1+\frac{z}{\delta} \bigg)  \,.
\end{equation}
\end{itemize}
We now continue to bound the r.h.s.\ of \eqref{eq:progbound1}.

\medskip

\noindent
\emph{Step 1.}
Given $\sfi' \in \N$ and $a' \in \cT_\epsilon(\sfi')$, we bound the integrals
over $a''$ and $b'$ in \eqref{eq:progbound1}:
\begin{equation*}
\begin{split}
	& \sum_{\sfi''=\sfi'+1}^{\sfi'+K_\epsilon} \
	\iint\limits_{\substack{b' \in \cT_\epsilon(\sfi'): \, b' > a'\\
	a'' \in \cT_\epsilon(\sfi'')}} \!\!\! \dd a'' \, \dd b' \,
	G_\theta(b'-a') \, \frac{1}{a''-b'}
	\,=\, \int_{a'}^{\epsilon\sfi'}
	\dd b' \, G_\theta(b'-a') \, \log\frac{\epsilon(\sfi'+K_\epsilon)-b'}{\epsilon\sfi' - b'} \\
	&\qquad\qquad \,=\, \bigg[ \int_0^\delta \dd x \, G_\theta(x) \,
	\log\bigg(1+\frac{\epsilon K_\epsilon}{\delta -x}\bigg) \bigg]
	\bigg|_{\delta := \epsilon\sfi' - a'}
	\,\le\, C \, \frac{\log(1+\frac{\epsilon K_\epsilon}{\epsilon\sfi'-a'})}{\log\frac{1}{\epsilon\sfi'-a'}}   \,,
\end{split}
\end{equation*}
where we used \eqref{eq:boint} and changed variable $x := b'-a'$.
Plugging it into \eqref{eq:progbound1}, we obtain
\begin{equation}\label{eq:progbound2}
	\cW_\epsilon \,\le\, C \,
	\sum_{\sfi = 1}^{\frac{1}{\epsilon}} \
	\sum_{\sfi'=\sfi+1}^{\sfi+K_\epsilon} \
	\iint\limits_{\substack{b \in \cT_\epsilon(\sfi)\\ a' \in \cT_\epsilon(\sfi')}}
	\dd b \, \dd a' \, \int_0^b \dd s \
	\Phi(s) \, G_\theta(b-s) \, \frac{1}{a'-b} \,
	\frac{\log(1+\frac{\epsilon K_\epsilon}{\epsilon\sfi' - a'})}{\log\frac{1}{\epsilon\sfi' - a'}} \,.
\end{equation}

\medskip
\noindent
\emph{Step 2.}
Given $\sfi \in \N$ and $b \in \cT_\epsilon(\sfi)$, we
focus on the integral over $a'$ in \eqref{eq:progbound2}:
\begin{equation}\label{eq:overa}
	\sum_{\sfi'=\sfi+1}^{\sfi+K_\epsilon} \
	\int_{\epsilon(\sfi'-1)}^{\epsilon\sfi'} \dd a' \,
	\frac{1}{a'-b} \,
	\frac{\log(1+\frac{\epsilon K_\epsilon}{\epsilon\sfi' - a'})}{\log\frac{1}{\epsilon\sfi' - a'}}
	= \sum_{\sfi'=\sfi+1}^{\sfi+K_\epsilon} \
	\int_0^\epsilon
	\frac{1}{\epsilon\sfi' -b - x} \,
	\frac{\log(1+\frac{\epsilon K_\epsilon}{x})}{\log\frac{1}{x}} \, \dd x \,,
\end{equation}
by the change of variables $x := \epsilon\sfi'- a'$. We first bound the sum from $\sfi'=\sfi+2$ onward, for which we note that for $x \in (0,\epsilon)$,
\begin{equation*}
\begin{split}
	\sum_{\sfi'=\sfi+2}^{\sfi+K_\epsilon} \frac{1}{\epsilon\sfi' - b - x}
	\,\le\, \frac{1}{\epsilon} \log\frac{\epsilon(\sfi + K_\epsilon + 1) -b-x}{\epsilon(\sfi+1)-b-x}
	& \,=\, \frac{1}{\epsilon} \log\bigg(1+\frac{\epsilon K_\epsilon}{\epsilon(\sfi+1)-b-x}\bigg) \\
	& \,\le\, \frac{1}{\epsilon} \log\bigg(1+\frac{\epsilon K_\epsilon}{\epsilon\sfi -b}\bigg) \,.
\end{split}
\end{equation*}
Moreover, by the change of variables $x = \epsilon \, e^{-t}$,
\begin{equation} \label{eq:boint2}
	\int_0^\epsilon
	\frac{\log(1+\frac{\epsilon K_\epsilon}{x})}{\log\frac{1}{x}} \, \dd x
	\,=\, \epsilon \int_0^\infty e^{-t} \,
	\frac{\log(1+K_\epsilon e^t)}{t + \log\frac{1}{\epsilon}} \,  \dd t
	\,\le\, \frac{\epsilon}{\log\frac{1}{\epsilon}} \,
	\int_0^\infty e^{-t} (t+ \log(1+K_\epsilon)) \,  \dd t \,,
\end{equation}
because $t + \log\frac{1}{\epsilon} \ge \log\frac{1}{\epsilon}$
and $1+K_\epsilon e^t \le (1+ K_\epsilon) e^t$. Therefore
$$
	\sum_{\sfi'=\sfi+2}^{\sfi+K_\epsilon} \
	\int_0^\epsilon
	\frac{1}{\epsilon\sfi' -b - x} \,
	\frac{\log(1+\frac{\epsilon K_\epsilon}{x})}{\log\frac{1}{x}} \, \dd x
	\leq C\,\frac{\log(1+K_\epsilon)}{\log \frac{1}{\eps}}
	\log\bigg(1+\frac{\epsilon K_\epsilon}{\epsilon\sfi -b}\bigg).
$$
Now for the case $\sfi'=\sfi+1$, we have
\begin{align*}
	& \int_0^\epsilon
	\frac{1}{\epsilon(\sfi+1) -b - x} \,
	\frac{\log(1+\frac{\epsilon K_\epsilon}{x})}{\log\frac{1}{x}} \, \dd x \\
	\leq\ & \frac{2}{\eps} \int_0^{\eps/2}  \frac{\log(1+\frac{\epsilon K_\epsilon}{x})}{\log\frac{1}{x}}
	\dd x + C\, \frac{\log (1+K_\eps)}{\log \frac{1}{\eps}}\int_{\eps/2}^\eps  \frac{1}{\epsilon(\sfi+1)
	-b - x} \,\dd x \\
	\leq\ & C\,\frac{\log(1+K_\epsilon)}{\log \frac{1}{\eps}}\bigg(1 +
	\log\bigg(1+\frac{\epsilon/2}{\epsilon\sfi -b}\bigg)\bigg)
	\leq C\,\frac{\log(1+K_\epsilon)}{\log \frac{1}{\eps}}
	\log\bigg(1+\frac{\epsilon K_\epsilon}{\epsilon\sfi -b}\bigg).
\end{align*}
Substituting the above bounds into
\eqref{eq:progbound2} then gives
\begin{equation}\label{eq:progbound3}
\begin{split}
	\cW_\epsilon & \,\le\, C' \, \frac{\log(1+K_\epsilon)}{\log\frac{1}{\epsilon}} \,
	\sum_{\sfi = 1}^{\frac{1}{\epsilon}} \
	\int\limits_{b \in \cT_\epsilon(\sfi)}
	\dd b \, \int_0^b \dd s \
	\Phi(s) \, G_\theta(b-s) \, \log\bigg(1+\frac{\epsilon K_\epsilon}{\epsilon\sfi -b}\bigg) \\
	& \,=\, C' \, \frac{\log(1+K_\epsilon)}{\log\frac{1}{\epsilon}} \,
	\sum_{\sfi = 1}^{\frac{1}{\epsilon}} \
	\int_0^{\epsilon\sfi} \dd s \, \Phi(s) \,
	\int_{\max\{s, \epsilon(\sfi-1)\}}^{\epsilon\sfi}
	\!\dd b \,
	G_\theta(b-s) \, \log\bigg(1+\frac{\epsilon K_\epsilon}{\epsilon\sfi -b}\bigg)\,.
\end{split}
\end{equation}

\medskip
\noindent
\emph{Step 3.} Given $s \in (0,1)$, we bound the integral over $b$ in \eqref{eq:progbound3}.
\begin{itemize}
\item For $s \in (\epsilon(\sfi-2),\epsilon\sfi)$ we can bound, by \eqref{eq:boint}
with $\delta = \epsilon\sfi-s$,
\begin{equation} \label{eq:alogous}
\begin{split}
	\int_{s}^{\epsilon\sfi}
	G_\theta(b-s) \, \log\bigg(1+\frac{\epsilon K_\epsilon}{\epsilon\sfi -b}\bigg) \, \dd b
	&\,\le\, C \, \frac{\log(1+\frac{\epsilon K_\epsilon}{\epsilon\sfi-s})}{\log\frac{1}{\epsilon\sfi-s}}
	\,\le\, C' \, \log(1+K_\epsilon) \,,
\end{split}
\end{equation}
where the last inequality (which is very rough) holds, say for $\epsilon \in (0,\frac{1}{4})$,
\emph{uniformly for $s \in (\epsilon(\sfi-2),\epsilon\sfi)$ and $K_\epsilon \ge 1$},
because $x := \frac{1}{\epsilon\sfi-s} \ge \frac{1}{2\epsilon} \ge 2$ and
\begin{equation*}
	\sup_{x\ge 2}\frac{\log(1+(\epsilon K_\epsilon) x)}{\log x} \le
	\sup_{x\ge 2} \frac{\log((1+ K_\epsilon) x)}{\log x}
	= \frac{\log((1+ K_\epsilon) 2)}{\log 2}
	\le \frac{2}{\log 2} \, \log(1+K_\epsilon) \,.
\end{equation*}

\item For $s \leq \epsilon(\sfi-2)$
we can bound $G_\theta(b-s) \le C \, G_\theta(\epsilon(\sfi-1)-s)$, see \eqref{eq:Gas}, and
\begin{equation*}
\begin{split}
	\int_{\epsilon(\sfi-1)}^{\epsilon\sfi}
	\log\bigg(1+\frac{\epsilon K_\epsilon}{\epsilon\sfi -b}\bigg) \, \dd b
	\,=\, \int_0^\epsilon \log\bigg(1+\frac{\epsilon K_\epsilon}{x}\bigg) \, \dd x
	\,\le\, C' \, \epsilon \, \log(1 + K_\epsilon) \,,
\end{split}
\end{equation*}
by the change of variables $x := \epsilon\sfi-b$ and the estimate
\eqref{eq:bai} with $\delta = \epsilon$ and $z =\epsilon K_\epsilon$.
\end{itemize}
Substituting these bounds into \eqref{eq:progbound3} then gives
\begin{equation} \label{eq:bouint}
\begin{split}
	\cW_\epsilon \le\ & C'' \, \frac{(\log(1+K_\epsilon))^2}{\log\frac{1}{\epsilon}}
	\sum_{\sfi = 1}^{\frac{1}{\epsilon}} \
	\bigg\{ \int_{\epsilon(\sfi-2)^+}^{\epsilon\sfi} \!\!\!\! \Phi(s) \, \dd s
	\,+\, \epsilon \, \int_0^{\epsilon(\sfi-2)^+} \!\!\!\!\!\!\!\!\!
	\Phi(s) \, G_\theta(\epsilon(\sfi-1)-s)  \,  \dd s  \bigg\} \\
	\leq \ &  C'' \, \frac{(\log(1+K_\epsilon))^2}{\log\frac{1}{\epsilon}}
	\int_0^1 \Phi(s) \, \Bigg\{ 2 \,+\,
	\sum_{\sfi=2 + \lceil \frac{s}{\epsilon}\rceil}^{\frac{1}{\epsilon}}
	\epsilon \, G_\theta(\epsilon(\sfi-1)-s) \Bigg\} \, \dd s \\
	\le\ & C \, \frac{(\log(1+K_\epsilon))^2}{\log\frac{1}{\epsilon}}  \int_0^1 \Phi(s) \, \dd s =
	C \, \frac{(\log(1+K_\epsilon))^2}{\log\frac{1}{\epsilon}}  \Vert \varphi\Vert_{\cG_1}^2, 	
\end{split}
\end{equation}
where the last sum can be seen as a Riemann sum and bounded by a multiple of $c_\theta = \int_0^1 G_\theta(x) \, \dd x$. This bound gives the second term in
\eqref{eq:toshIII0b} and completes its proof.
\bigskip

\noindent
{\bf Bound for $\mathrm{I}_{N,\epsilon}$: proof of \eqref{eq:toshIII0a}.}
In view of \eqref{eq:cIeps02} and \eqref{eq:intf}, we need to show that
\begin{equation} \label{eq:bao2}
	\cV_\epsilon \,:=\,
	\sum_{\sfI \subset \{1, \ldots, \frac{1}{\eps}\}, \, |\sfI|>(\log\frac{1}{\epsilon})^2} \
	\cI_\epsilon^{\Phi}(\sfI) \, \le \,
	C \,  \bigg(\int_0^1 \Phi(t) \, \dd t \bigg) \,
	\frac{1}{\log\frac{1}{\epsilon}} \,.
\end{equation}
By Markov's inequality, we can bound
\begin{equation*}
\begin{split}
	\cV_\epsilon  & \,\le\,
	\frac{1}{(\log \frac{1}{\epsilon})^2} \,
	\sum_{\sfI \subset \{1, \ldots, \frac{1}{\eps}\},\, |\sfI|>(\log\frac{1}{\epsilon})^2} \!\!\!\!\!\!\!\!
	|\sfI|\, \cI_\epsilon^{\Phi}(\sfI)
	\,\le\, \frac{1}{(\log \frac{1}{\epsilon})^2} \
	\sum_{\sfj=1}^{\frac{1}{\epsilon}}
	\sum_{\sfI \subset \{1, \ldots, \frac{1}{\eps}\}, \sfI\ni \sfj} \cI_\epsilon^{\Phi}(\sfI) \,.
\end{split}
\end{equation*}
Recalling the renewal interpretation of $\cI_\epsilon^{\Phi}(\sfI)$ after Lemma~\ref{th:asycom} and
the renewal decomposition
\eqref{eq:renewal-Dickman2}, we can integrate over the first renewal visit $s$, the last visited point
$u\in \cT_{\eps}(\sfj)$,
the first visited point $v$ after $\cT_{\eps}(\sfj)$, and the last renewal visit $t\leq 1$, to obtain the bound
\begin{align*}
	\cV_\epsilon
	& \,\le\,
	\frac{1}{(\log \frac{1}{\epsilon})^2} \, \sum_{1\le \sfj \le \frac{1}{\epsilon}}
	\ \ \iiiint\limits_{0<s<u<v<t\le1 \atop u\in \cT_\eps(\sfj), v>\eps \sfj} {\Phi(s)\,} G_\theta(u-s) \,
	\frac{1}{v-u} \, G_\theta(t-v) \, \dd s \, \dd u \, \dd v \, \dd t \\
	& \,\le\, \frac{c_\theta}{(\log \frac{1}{\epsilon})^2} \,
	\sum_{1 \le \sfj \le \frac{1}{\epsilon}}
	\ \ \iint\limits_{0<s <u, \, u \in \cT_\epsilon(\sfj)}
	{\Phi(s)\,} G_\theta(u-s) \,
	\log\frac{1}{\epsilon\sfj-u} \, \dd s \, \dd t \,.
\end{align*}
Observe that the sum of the integrals is exactly the same as in the r.h.s. of \eqref{eq:progbound3} with $K_\eps$ replaced by $\frac{1}{\eps}$.
Therefore the bounds leading to \eqref{eq:bouint} also applies, which gives
$$
\cV_\epsilon \leq \frac{c_\theta}{(\log \frac{1}{\epsilon})^2}  C'' \log\Big(1+\frac{1}{\eps}\Big) \int_0^1 \Phi(s) \dd s
\leq \frac{C}{\log\frac{1}{\epsilon}} \, \int_0^1 \Phi(s) \, \dd s \,.
$$
This matches our goal \eqref{eq:bao2} and completes the proof of \eqref{eq:toshIII0a}.
\end{proof}

\subsection{Step 2: Diffusive truncation in space}\label{sec:5.2}
In this step, we introduce our second approximation $\cZ_{N,\epsilon}^{(\diff)}(\varphi,\psi)$
and show that it is close to $\cZ_{N,\epsilon}^{(\notri)}(\varphi,\psi)$.

For $\sfa = (\sfa^{(1)}, \sfa^{(2)}) \in \Z^2$, recall from \eqref{eq:cB} the mesoscopic spatial square
\begin{equation}\label{eq:cI2}
\begin{split}
	\cS_{\eps N}(\sfa) := & \big((\sfa-(1,1))\sqrt{\epsilon N}, \, \sfa\sqrt{\epsilon N}\,\big] \\
	=\, & \big((\sfa^{(1)}-1)\sqrt{\epsilon N}, \, \sfa^{(1)}\sqrt{\epsilon N} \,\big] \times
	\big((\sfa^{(2)}-1)\sqrt{\epsilon N}, \, \sfa^{(2)}\sqrt{\epsilon N}\,\big] \,,
\end{split}
\end{equation}
to which we associate the mesoscopic space variable $\sfa$. We now perform coarse-graining
in space by grouping
terms in the expansion of $\cZ_{N,\epsilon}^{(\notri)}(\varphi,\psi)$ in \eqref{eq:Zno3}
according to the mesoscopic
spatial boxes visited by the space variables $v, x_i, y_i, w$ in \eqref{eq:Zno3}. Namely, recall the definition
\eqref{eq:cB} of the mesoscopic time-space box
\begin{equation} \label{eq:cB2}
	\cB_{\epsilon N}(\sfi,\sfa) =
	\big( \cT_{\epsilon N}(\sfi) \times \cS_{\epsilon N}(\sfa) \big) \ \cap \ \Z^3_\even \,.
\end{equation}
We can rewrite \eqref{eq:Zno3} by introducing the \emph{mesocopic space variables} 
$\sfb_0, \sfa_1, \sfb_1, \ldots, \sfa_k, \sfb_k, \sfa_{k+1}$:
\begin{align}
	& \cZ_{N,\epsilon}^{(\notri)}(\varphi,\psi)
	\,\,\, = \,\,\, q_{0,N}^N(\varphi,\psi) \, \nonumber \\
	+  \, & \frac{1}{N}
	\sum_{k=1}^{(\log \frac{1}{\epsilon})^2}\!\!\!
	\sum_{(\sfi_1, \ldots, \sfi_k) \in \cA_\epsilon^{(\notri)}}
	 \sum_{{\sfb_0} \in \Z^2}
	 \sum_{\sfa_1, \, \sfb_1, \,\ldots\,, \sfa_k, \, \sfb_k  \,\in\, \Z^2}
	\sum_{\substack{(d_1,x_1) \in \cB_{\epsilon N}(\sfi_1, \sfa_1) \nonumber \\
	(f_1,y_1) \in \cB_{\epsilon N}(\sfi_1, \sfb_1)\\\text{with }d_1 \le f_1}}
	\cdots
	\sum_{\substack{(d_k,x_k) \in \cB_{\epsilon N}(\sfi_k, \sfa_k) \\
	(f_k,y_k) \in \cB_{\epsilon N}(\sfi_k, \sfb_k)\\\text{with }d_k \le f_k}} \nonumber \\
	& \bigg(\sum_{v \in \cS_{\eps N}({\sfb})\cap\Z^2_\even} \!\!\!\!\!\! \varphi_N(v) \, q_{0, d_1}(v,x_1) \bigg)
	\, X_{d_1,f_1}(x_1,y_1)
	\Bigg\{ \prod_{j=2}^k  q_{f_{j-1},d_j}(y_{j-1},x_j) \,
	X_{d_j,f_j}(x_j,y_j) \Bigg\} \nonumber \\
	& \times \, \bigg( \sum_{\sfa_{k+1} \in\Z^2} \
	\sum_{w \in \cS_{\eps N}({\sfa_{k+1}})\cap \Z^2_\even} \!\!\!\!\!\!\!\!\!\!\!\!
	q_{f_k, N}(y_k,w) \, \psi_N(w) \bigg) \,. \label{eq:Zno3a}
\end{align}

We now perform a diffusive scale truncation by replacing each $X_{d_i,f_i}(x_i,y_i)$ in the above expansion by its truncated
version $X^{(\diff)}_{d_i,f_i}(x_i,y_i)$ defined in \eqref{eq:Xdiff}. Let us stress that
\begin{equation} \label{eq:Xdiff0}
	X^{(\diff)}_{d,f}(x,y) = 0 \qquad \text{for } x \in \cS_{\eps N}(\sfa), \ y \in \cS_{\eps N}(\sfb) \
	\text{with} \ |b-a| > M_\epsilon \,.
\end{equation}
Furthermore, we restrict the mesoscopic space variables $(\sfa_1, \, \sfb_1, \,\ldots\,, \sfa_k, \, \sfb_k)$ in \eqref{eq:Zno3a} to
a ``diffusive set'' that depends on the initial and final space variables $\sfb_0$ and $\sfa_{k+1}$, and time variables $(\sfi_0:=0, \sfi_1, \ldots, \sfi_k, \sfi_{k+1}:=\frac{1}{\eps})$:
\begin{align}
	\cA_{\epsilon; {\sfb_0, \sfa_{k+1}}}^{(\diff)}  & :=
	\Big\{ (\sfa_1, \, \sfb_1, \,\ldots\,, \sfa_k, \, \sfb_k) \, \in (\Z^2)^{2k} \
	\ \text{s.t.}\ \forall\, 1\leq j\leq k, \ |\sfb_j - \sfa_j| \le M_\epsilon, \notag\\
	&\hspace{3cm} \text{and} \quad \forall\, 1\leq j\leq k+1, \
	|\sfa_{j} - \sfb_{j-1}| \le M_\epsilon \sqrt{\sfi_{j} - \sfi_{j-1}} \ \Big\}. \label{eq:cAdiff}
\end{align}
\begin{remark}
Once $(\sfi_1, \ldots, \sfi_k)\in \cA_{\epsilon}^{(\notri)}$ are grouped into time blocks,
see Definition~\ref{def:time-block},
we can then group
$(\sfa_1, \, \sfb_1, \,\ldots\,, \sfa_k, \, \sfb_k) \in \cA_{\epsilon; \sfb, \sfc}^{(\diff)}$ into space-blocks,
see Definition~\ref{def:space-block}.
The constraint $\cA_{\epsilon; \sfb, \sfc}^{(\diff)}$
then maps to the constraint $\bcA_{\epsilon; \sfb, \sfc}^{(\diff)}$ defined in \eqref{eq:bcA2}.
\end{remark}
More explicitly, we can approximate $\cZ_{N,\epsilon}^{(\notri)}(\varphi,\psi)$ from \eqref{eq:Zno3a} by
\begin{align}
	&\cZ_{N,\epsilon}^{(\diff)}(\varphi,\psi)
	:=  q_{0,N}^N(\varphi,\psi)  +
	\frac{1}{N} \sum_{k=1}^{(\log \frac{1}{\epsilon})^2}
	\!\!\!\! \sum_{(\sfi_1, \ldots, \sfi_k) \in \cA_\epsilon^{(\notri)}}
	\sum_{{\sfb_0, \sfa_{k+1}}\, \in \Z^2}
	 \sum_{(\sfa_1, \, \sfb_1, \,\ldots\,, \sfa_k, \, \sfb_k)
	\in \cA_{\epsilon; {\sfb_0, \sfa_{k+1}}}^{(\diff)}}   \notag \\
	&  \sum_{\substack{(d_1,x_1) \in \cB_{\epsilon N}(\sfi_1, \sfa_1)\\
	(f_1,y_1) \in \cB_{\epsilon N}(\sfi_1, \sfb_1)\\\text{with }d_1 \le f_1}}
	\!\!\!\cdots \!\!\! \sum_{\substack{(d_k,x_k) \in \cB_{\epsilon N}(\sfi_k, \sfa_k) \\
	(f_k,y_k) \in \cB_{\epsilon N}(\sfi_k, \sfb_k)\\\text{with }d_k \le f_k}}
	\ \sum_{y_0 \in \cS_{\eps N}({\sfb_0})\cap \Z^2_\even} \varphi_N(y_0)  \label{eq:Zdiffa} \\
	& \quad \times \Bigg\{ \prod_{j=1}^k  q_{f_{j-1},d_j}(y_{j-1},x_j) \,
	X_{d_j,f_j}^{(\diff)}(x_j,y_j) \Bigg\}
	\sum_{x_{k+1} \in \cS_{\eps N}(\sfa_{k+1})\cap \Z^2_\even} \!\!\!\!\!\!\!\!\!\!
	q_{f_k, N}(y_k,x_{k+1}) \, \psi_N(x_{k+1}) \,, \notag
\end{align}
where $f_0:=0$. The main result of this subsection is the following, where the approximation
error is much smaller than
${\mathbb V}{\rm ar}(\cZ^{\beta_N}_{N,1}(\varphi, \psi))$ in \eqref{eq:m2-lim1}.

\begin{lemma}[Diffusive bound]\label{th:diff}
Recall from \eqref{eq:KM} that $M_\eps = \log \log \frac{1}{\eps}$ and recall $\Vert \cdot\Vert_{\cG_t}$ from \eqref{eq:phinorm}.
There exist $\sfc, \sfC \in(0, \infty)$ such that for $\epsilon > 0$ small enough, we have: for all $\varphi$ with $\Vert \varphi\Vert_{\cG_1}^2<\infty$
and $\psi \in L^\infty(\R^2)$,
\begin{equation}\label{eq:appr2diff}
\begin{gathered}
	\limsup_{N\to\infty \text{ with } N\in2\N} \
	\big\|\big(  \cZ_{N,\epsilon}^{(\diff)}
	- \cZ_{N,\epsilon}^{(\notri)} \big)(\varphi, \psi)
	\big\|_{L^2}^2
	\, \le \,
	\sfC \, e^{-\sfc M_\eps^2}\Vert \varphi\Vert_{\cG_1}^2 \Vert \psi\Vert_\infty^2 .
\end{gathered}
\end{equation}
\end{lemma}

\begin{proof}
We argue as in the proof of Lemma~\ref{th:no3}. Note that the chaos expansion for
\emph{$\cZ_{N,\epsilon}^{(\diff)}$ is a restricted sum of terms in the expansion for $\cZ_{N,\epsilon}^{(\notri)}$}
due to the following two effects:
\begin{itemize}
	\item [(I)] the replacement of $X_{d_j,f_j}(x_j,y_j)$ by $X^{(\diff)}_{d_j,f_j}(x_j,y_j)$
	(cf.\ \eqref{eq:Xdiff}-\eqref{eq:Zpolydiff} and~\eqref{eq:X});
	\item [(II)] the restriction of
	$(\sfa_1, \, \sfb_1, \,\ldots\,, \sfa_k, \, \sfb_k)$ to $\cA_{\epsilon; \sfb_0, \sfa_{k+1}}^{(\diff)}$.
\end{itemize}
Since terms in the chaos expansion are mutually $L^2$ orthogonal, we can write
\begin{equation} \label{eq:splitdiff}
\begin{split}
	\big\|\big(  \cZ_{N,\epsilon}^{(\diff)}
	- \cZ_{N,\epsilon}^{(\notri)} \big)(\varphi, \psi)
	\big\|_{L^2}^2 \,=\,
	\mathrm{I}_{N,\epsilon} \,+\, \mathrm{II}_{N,\epsilon} \,,
\end{split}
\end{equation}
where $\mathrm{I}_{N,\epsilon}$ and $\mathrm{II}_{N,\epsilon}$ are the squared $L^2$ error
as we first make the replacement in (I) and then impose the restriction in (II).

To be more precise, we can define $X^{(\nodiff)}_{d,f}(x,y)$ by the equality
\begin{equation} \label{eq:diff-no-diff}
	X_{d,f}(x,y) = X^{(\diff)}_{d,f}(x,y) + X^{(\nodiff)}_{d,f}(x,y) \,.
\end{equation}
In view of \eqref{eq:secondX} and \eqref{eq:barU}, we define
\begin{gather}
	\label{eq:Udif}
	\overline{U}_{N ,\, d,f}^{(\diff)}(x,y) := \bbE\big[ X^{(\diff)}_{d,f}(x,y)^2 \big] \,, \\
	\label{eq:Unodif}
	\overline{U}_{N ,\, d,f}^{(\nodiff)}(x,y) := \bbE\big[ X^{(\nodiff)}_{d,f}(x,y)^2 \big] \,.
\end{gather}
Note that $X^{(\nodiff)}_{d,f}(x,y)$ and $X^{(\diff)}_{d,f}(x,y)$ are orthogonal in $L^2$,
see \eqref{eq:Xdiff}-\eqref{eq:Zpolydiff} and~\eqref{eq:X}.
As a consequence, if we plug \eqref{eq:diff-no-diff} into \eqref{eq:Zno3a} and expand the product,
we obtain
\begin{equation} \label{eq:intermI}
\begin{split}
	\mathrm{I}_{N,\epsilon} \ = \
	& \frac{1}{N^2} \!\! \sum_{k=1}^{(\log \frac{1}{\epsilon})^2}
	\!\!\!\! \sum_{(\sfi_1, \ldots, \sfi_k) \in \cA_\epsilon^{(\notri)}}
	\!\!\!\!\!\! \sum_{\substack{\quad \ \ {\sfb_0}, {\sfb_0}' \\ \sfa_1, \, \sfb_1, \,\ldots\,, \sfa_k, \, \sfb_k} \in\, \Z^2}
	\ \sum_{\substack{(d_1,x_1) \in \cB_{\epsilon N}(\sfi_1, \sfa_1)\\
	(f_1,y_1) \in \cB_{\epsilon N}(\sfi_1, \sfb_1)\\\text{with }d_1 \le f_1}}
	\!\! \cdots \!\!
	\sum_{\substack{(d_k,x_k) \in \cB_{\epsilon N}(\sfi_k, \sfa_k)\\
	(f_k,y_k) \in \cB_{\epsilon N}(\sfi_k, \sfb_k)\\\text{with }d_k \le f_k}} \\
	&\rule{0pt}{1.8em} \times \bigg(\sum_{\substack{v \in \cS_{\eps N}({\sfb_0}) \cap \Z^2_\even \\
	v' \in \cS_{\eps N}({\sfb_0}') \cap \Z^2_\even}} \varphi_N(v) \, \varphi_N(v') \,
	\, q_{0,d_1}(v,x_1) \, q_{0,d_1}(v',x_1)\bigg)  \\
	& \rule{0pt}{1.8em} \times \Bigg\{\prod_{j=2}^k q_{f_{j-1},d_j}(y_{j-1},x_j)^2\Bigg\} \
	  \bigg( \sum_{w\in\Z^2_\even} q_{f_k, N}(y_k,w) \, \psi_N(w) \bigg)^2  \\
	&\rule{0pt}{1.8em} \times \sum_{J \subseteq \{1,\ldots, k\}: \, |J| \ge 1 } \
	\prod_{j \in J} \overline{U}^{(\nodiff)}_{N,\, d_j, f_j}(x_j, y_j)  \
	\prod_{j \in \{1,\ldots,k\} \setminus J} \overline{U}^{(\diff)}_{N,\, d_j, f_j}(x_j, y_j) \,.
\end{split}
\end{equation}

The term $\mathrm{II}_{N,\epsilon}$ in \eqref{eq:splitdiff} accounts for the further restriction
to $(\sfa_1, \, \sfb_1, \,\ldots\,, \sfa_k, \, \sfb_k) \in \cA_{\epsilon; \sfb_0, \sfa_{k+1}}^{(\diff)} \cap \cA_{\epsilon; \sfb'_0, \sfa'_{k+1}}^{(\diff)}$ for some ${\sfb_0}, {\sfb_0}', \sfa_{k+1}, \sfa_{k+1}' \in \Z^2$,  and hence
\begin{align}
	\mathrm{II}_{N,\epsilon} \,=\, &
	\frac{1}{N^2}  \!\! \sum_{k=1}^{(\log \frac{1}{\epsilon})^2}
	\!\!\!\! \sum_{(\sfi_1, \ldots, \sfi_k) \in \cA_\epsilon^{(\notri)}}
	\hspace{-1.7cm}
	\sum_{\substack{{\sfb_0}, {\sfb_0}', \sfa_{k+1}, \sfa_{k+1}'\in \Z^2 \\
	(\sfa_1, \, \sfb_1, \,\ldots\,, \sfa_k, \, \sfb_k) \in
	(\cA_{\epsilon; {\sfb_0}, \sfa_{k+1}}^{(\diff)})^c \cap
	(\cA_{\epsilon;  {\sfb_0}', \sfa_{k+1}'}^{(\diff)})^c }}
	\sum_{\substack{(d_1,x_1) \in \cB_{\epsilon N}(\sfi_1, \sfa_1)\\
	(f_1,y_1) \in \cB_{\epsilon N}(\sfi_1, \sfb_1)\\\text{with }d_1 \le f_1}}
	\!\! \cdots \!\!
	\sum_{\substack{(d_k,x_k) \in \cB_{\epsilon N}(\sfi_k, \sfa_k)\\
	(f_k,y_k) \in \cB_{\epsilon N}(\sfi_k, \sfb_k)\\\text{with }d_k \le f_k}}  \notag \\
	& \quad  \bigg( \sum_{\substack{v \in \cS_{\eps N}({\sfb_0}) \cap \Z^2_\even \\
	v' \in \cS_{\eps N}({\sfb_0}') \cap \Z^2_\even}} \varphi_N(v) \, \varphi_N(v') \,
	\, q_{0,d_1}(v,x_1) \, q_{0,d_1}(v',x_1) \bigg)   \notag \\
	& \times  \quad \overline{U}^{(\diff)}_{N,\, d_1, f_1}(x_1, y_1) \Bigg\{ \prod_{j=2}^k q_{f_{j-1},d_j}(y_{j-1},x_j)^2 \,
	\overline{U}^{(\diff)}_{N,\, d_j, f_j}(x_j, y_j) \Bigg\} \notag \\
	& \times \bigg( \sum_{\substack{x_{k+1} \in \cS_{\eps N}({\sfa_{k+1}}) \cap \Z^2_\even \\
	x_{k+1}' \in \cS_{\eps N}({\sfa'_{k+1}}) \cap \Z^2_\even}}
	\!\!\!\!\!\!\!\! q_{f_k, N}(y_k,x_{k+1})q_{f_k, N}(y_k,x_{k+1}')  \psi_N(x_{k+1}) \psi_N(x_{k+1}')\bigg)  \,. \label{eq:intermII}
\end{align}
To prove \eqref{eq:appr2diff}, it suffices to show that for some $\sfc, \sfC \in(0, \infty)$,
\begin{equation}\label{eq:toshIII}
	\limsup_{N\to\infty \text{ with } N\in2\N} \!\!\!\!
	\mathrm{I}_{N,\epsilon} \le
	\sfC e^{-\sfc M_\eps^2}\Vert \varphi\Vert_{\cG_1}^2 \Vert \psi\Vert_\infty^2
	\quad \text{and} \
	\limsup_{N\to\infty \text{ with } N\in2\N} \!\!\!\!
	\mathrm{II}_{N,\epsilon}  \le
	\sfC e^{-\sfc M_\eps^2}\Vert \varphi\Vert_{\cG_1}^2 \Vert \psi\Vert_\infty^2.
\end{equation}

We need the following bound, which follows easily from Lemma~\ref{lem:UnGauss}. Recall  $U_N(n)$ and $\overline{U}_N(n)$ from \eqref{eq:Usingle} and \eqref{eq:barU}.
\begin{equation}\label{eq:claimUsuperdiff}
\begin{aligned}
& \exists\, c, C\in (0,\infty) \ \ s.t. \ \ \forall\, N\in \N, \  \eps>0, \  d\le f\in \N \mbox{ with } |f-d| \le \eps N: \\
& \hspace{2cm} \sum_{y\in\Z^2} \overline{U}^{(\nodiff)}_{N,\, d, f}(x, y) \le C \, e^{-c M_\epsilon^2} \, \overline{U}_N(f-d) \,.
\end{aligned}
\end{equation}

We are ready to bound $\mathrm{I}_{N,\epsilon}$ in \eqref{eq:intermI}. As in the proof of Lemma~\ref{th:no3},
the last term $\big( \sum_{w\in\Z^2} q_{f_k, N}(y_k,w) \, \psi_N(w) \big)^2$ can be bounded as in \eqref{eq:firstsquare}.
We then sum over all space variables in reverse order from $y_k, x_k$ until $y_1$. We will use $\sum_{x \in \Z^2} q_{f,d}(y,x)^2 = u(d - f)$
by \eqref{eq:u}, apply \eqref{eq:claimUsuperdiff} together with the fact that
\begin{equation} \label{eq:Udiffsum}
	\sum_{y\in\Z^2} \overline{U}^{(\diff)}_{N,\, d, f}(x, y)
	\le \sum_{y\in\Z^2} \overline{U}_{N}(f-d, y-x)
	= \overline{U}_N(f-d)  \,,
\end{equation}
and finally sum over $x_1$, noting that, by \eqref{eq:secondsquare},
\begin{equation}\label{eq:secondsquarebis}
\begin{split}
	& \sum_{\sfa_1 \in \Z^2}
	\sum_{x_1 \in \cS_{\epsilon N}(\sfa_1)}  \sum_{{\sfb_0}, {\sfb_0}' \in \Z^2}
	 \sum_{\substack{v \in \cS_{\eps N}({\sfb_0}) \cap \Z^2_\even \\
	v' \in \cS_{\eps N}({\sfb_0}') \cap \Z^2_\even}}  \, \varphi_N(v) \, \varphi_N(v') \,
	\, q_{0,d_1}(v,x_1) \, q_{0,d_1}(v',x_1) \\
	=\ & \sum_{v, v' \in \Z^2_\even} \varphi_N(v) \, \varphi_N(v') \,
	q_{2d_1}(v - v')
	\,=\, N \, \Phi_N\Big(\frac{d_1}{N}\Big).
\end{split}
\end{equation}
We then obtain an analogue of \eqref{eq:anana}:
\begin{equation} \label{eq:intermIbound}
\begin{aligned}
	& \mathrm{I}_{N,\epsilon} \  \le \
	\frac{C \, \|\psi\|_\infty^2}{N}
	\sum_{k=1}^{(\log\frac{1}{\epsilon})^2} \,
	\sum_{(\sfi_1, \ldots, \sfi_k) \in \cA_\epsilon^{(\notri)}} \
	\sum_{d_1 \le f_1 \,\in\, \cT_{\epsilon N}(\sfi_1) \, , \,
	\ldots\,, \, d_k \le f_k \,\in\, \cT_{\epsilon N}(\sfi_k)} \\
	&\ \  \Bigg[ \sum_{J \subseteq \{1,\ldots, k\}: \, |J| \ge 1 }\!\!\!\!\!\!
	C \, e^{-c M_\epsilon^2 |J|} \Bigg]
	\Phi_N\Big(\frac{d_1}{N}\Big) \, \overline{U}_N(f_1-d_1)
	\Bigg\{ \prod_{j=2}^k u(d_j-f_{j-1}) \, \overline{U}_N(f_j-d_j) \Bigg\}.
\end{aligned}
\end{equation}
For $k \le (\log\frac{1}{\epsilon})^2$, recalling that $M_\eps = \log \log \frac{1}{\eps}$, we can bound
for $\epsilon > 0$ small enough
\begin{equation*}
	\Bigg[\sum_{J \subseteq \{1,\ldots, k\}: \ |J| \ge 1 } \,
	C \, e^{-c M_\epsilon^2 |J|} \Bigg]
	= C\, \big\{ (1 +  e^{-cM_\epsilon^2})^k - 1\big\}
	\le  2C k \, e^{-cM_\epsilon^2} \leq 2C e^{-\frac{c}{2} M_\epsilon^2} \,.
\end{equation*}
We now plug this bound into \eqref{eq:intermIbound} and sum freely over
all $0 < \sfi_1 < \ldots < \sfi_k \le \frac{1}{\epsilon}$.
As $N\to\infty$, by Lemma~\ref{th:asycom} and similar to \eqref{eq:cIeps0}, we obtain
(with $\sfc = \frac{c}{2}$)
\begin{equation} \label{eq:argu1}
	\limsup_{N\to\infty} \ \mathrm{I}_{N,\epsilon} \, \le \,
	C \, \|\psi\|_\infty^2 \, e^{-\sfc M_\epsilon^2}
	\, \sum_{k=1}^{(\log\frac{1}{\epsilon})^2} \,
	\sum_{0 < \sfi_1 < \ldots < \sfi_k \le \frac{1}{\epsilon}} \
	\cI_\epsilon^{\Phi}(\sfi_1, \ldots, \sfi_k) \,,
\end{equation}
The renewal decomposition \eqref{eq:renewal-Dickman2}, together with $c_\theta =\int_0^1 G_\theta(x) \, \dd x < \infty$, yields
\begin{equation} \label{eq:argu2}
	\sum_{k=1}^{\infty} \,
	\sum_{0 < \sfi_1 < \ldots < \sfi_k \le \frac{1}{\epsilon}} \
	\cI_\epsilon^{\Phi}(\sfi_1, \ldots, \sfi_k)
	= \iint_{0 < s < t < 1} {\Phi(s) \,} G_\theta(t-s) \, \dd s \, \dd t
	\le c_\theta \int_0^1 \Phi(s) \, \dd s \,.
\end{equation}
By \eqref{eq:intf}, this proves
the first bound in \eqref{eq:toshIII}.
\smallskip

We now prove the bound on $\mathrm{II}_{N,\epsilon}$ in \eqref{eq:toshIII}.
Note from \eqref{eq:cAdiff} that
\begin{equation} \label{eq:deco}
	(\cA_{\epsilon; \sfb_0, \sfa_{k+1}}^{(\diff)})^c =
	\bigcup_{l=1}^k \big\{|\sfb_l - \sfa_l| > M_\epsilon\big\} \ \cup \ \,
	\bigcup_{j=1}^{k+1} \Big\{|\sfa_{j}-\sfb_{j-1}| > M_\epsilon
	\sqrt{\sfi_{j}-\sfi_{j-1}} \, \Big\} \,.
\end{equation}
The first union in \eqref{eq:deco} gives no contribution to \eqref{eq:intermII} since
$\overline{U}^{(\diff)}_{N,d,f}(x,y) = 0$ when $x \in \cS_{\epsilon N}(\sfa)$ and $y \in \cS_{\epsilon N}(\sfb)$
with $|\sfb - \sfa| > M_\epsilon$, by \eqref{eq:Xdiff0} and \eqref{eq:Udif}.
It remains to consider the contribution to \eqref{eq:intermII} from the
second union in \eqref{eq:deco}, namely,  $|\sfa_{j} - \sfb_{j-1}| > M_\epsilon \sqrt{\sfi_j - \sfi_{j-1}}$
for some $j \in \{1,\ldots, k, k+1\}$.

In contrast to the bound on $\mathrm{I}_{N,\epsilon}$ where a small factor $Ce^{- c M_\eps^2}$ comes from the bound
on $\sum_{y_j\in\Z^2} \overline{U}^{(\nodiff)}_{N,\, d_j, f_j}(x_j, y_j)$ in \eqref{eq:claimUsuperdiff}, in the bound for
$\mathrm{II}_{N,\epsilon}$, the same small factor now comes from the following estimates: there exists $c\in (0,\infty)$ such that
for any $\sfi<\sfi'$, $f\in \cT_{\eps N}(\sfi)$, $d'\in \cT_{\eps N}(\sfi')$, $\sfb\in \Z^2$, and $y\in  \cS_{\epsilon N}(\sfb)$, we have
\begin{align} \label{eq:usetobound1}
	\sum_{\sfa' \in \Z^2: \, |\sfa'-\sfb| > M_\epsilon \sqrt{\sfi' - \sfi}}
	\ \sum_{x' \in \cS_{\epsilon N}(\sfa')} \ q_{f,d'}(y,x')
	& \,\le\, C \, e^{-c M_\epsilon^2}, \\
	\label{eq:usetobound2}
	\sum_{\sfa' \in \Z^2: \, |\sfa'-\sfb| > M_\epsilon \sqrt{\sfi' - \sfi}}
	\ \sum_{x' \in \cS_{\epsilon N}(\sfa')} \ q_{f,d'}(y,x')^2
	& \,\le\, C \, e^{-c M_\epsilon^2} \, u(d'-f) \,,
\end{align}
where recall from \eqref{eq:u} that $u(n):=\sum_{z\in \Z^2} q_n(z)^2 = q_{2n}(0)$.
The first bound follows from the fact that
$q_n(\cdot)$ has Gaussian tail decay. The second bound is a consequence of the
first bound, because
$\sup_z q_n(z) \leq  \frac{C'}{n}\leq C u(n)$
by the local limit theorem~\eqref{eq:llt} and \eqref{eq:u}.

The bound  on $\mathrm{II}_{N,\epsilon}$ then follows the
same steps as that for $\mathrm{I}_{N,\epsilon}$,
where we take a union bound over all $1\leq j\leq k+1\leq (\log \frac{1}{\eps})^2+1$ with
$|\sfa_j - \sfb_{j-1}| > M_\epsilon \sqrt{\sfi_j - \sfi_{j-1}}$.
\begin{itemize}
\item For $j=k+1$, bounding $\psi$ by $\Vert\psi\Vert_\infty$ and applying
\eqref{eq:usetobound1}, the sum over $w,w'$ and $\sfa_{k+1}, \sfa_{k+1}'$
in \eqref{eq:intermII} under the super-diffusive constraint
$|\sfa_{k+1} - \sfb_k| > M_\epsilon \sqrt{\frac{1}{\eps} - \sfi_k}$ leads to an extra factor of
$C e^{-c M_\epsilon^2}$ compared with the bound when this constraint is not
present (see~\eqref{eq:firstsquare}).
\item For $2\leq j\leq k$, by \eqref{eq:usetobound2},
the sum of $q_{f_{j-1}, d_j}(y_{j-1}, x_j)^2$ over $x_j$ in \eqref{eq:intermII} under this
super-diffusive constraint gives an extra factor of $C e^{- c M_\eps^2}$ compared to the case when
this constraint is not present.
\item For $j=1$, given $\sfb_0, \sfb'_0\in \Z^2$, we could have either
$|\sfa_1 - \sfb_{0}| > M_\epsilon \sqrt{\sfi_1}$ or $|\sfa_1 - \sfb'_{0}| > M_\epsilon \sqrt{\sfi_1}$;
either way, given $v\in \cS_{\eps N}(\sfb_0)$, $v'\in \cS_{\eps N}(\sfb'_0)$ and
$d_1\in \cT_{\eps N}(\sfi_1)$,  the sum of $q_{0, d_1}(v, x_1)q_{0, d_1}(v', x_1)$
in \eqref{eq:intermII} under this super-diffusive constraint gives a factor $C e^{-c M_\eps^2}$ compared
to the case when this constraint is not present.
\end{itemize}
Since there are at most $(\log \frac{1}{\eps})^2+1$ choices
of such $j$, this leads to the same bound we had for  $\mathrm{I}_{N,\epsilon}$ in \eqref{eq:argu1},
which establishes the second inequality in \eqref{eq:toshIII}.

This completes the proof of Lemma~\ref{th:diff}.
\end{proof}

\subsection{Step 3: Kernel replacement}\label{sec:5.3}
In this  step, we introduce a
last approximation $\cZ_{N,\epsilon}^{(\cg)}(\varphi,\psi)$
and show that it is close both to $\cZ_{N,\epsilon}^{(\diff)}(\varphi,\psi)$
and to the coarse-grained model
$\mathscr{Z}_{\epsilon}^{(\cg)}(\varphi,\psi|\Theta_{N,\epsilon}^{(\cg)})$.
This completes the proof of Theorem~\ref{th:cg-main}.

\smallskip

Let us first summarize the previous steps.
So far, we have performed coarse-graining by grouping terms in the
chaos expansion for $\cZ_N(\varphi, \psi)$ in \eqref{eq:Zpolyav}
according to the mesoscopic time-space boxes $\cB_{\eps N}(\sfi, \sfa)$ visited by the microscopic
time-space renewal configuration $(n_1, z_1), \ldots, (n_r, z_r)$ in  \eqref{eq:Zpolyav}.
Imposing suitable restrictions, we have defined the approximations
$\cZ_{N,\epsilon}^{(\notri)}(\varphi,\psi)$ in \eqref{eq:Zno3}
and $\cZ_{N,\epsilon}^{(\diff)}(\varphi,\psi)$ in \eqref{eq:Zdiffa}.
The next step is to replace the relevant random
walk transition kernels in the expansion \eqref{eq:Zdiffa} for $\cZ_{N,\epsilon}^{(\diff)}(\varphi,\psi)$
by heat kernels as in \eqref{eq:kerrep}, i.e., replace the random walk
transition kernels $q_{f_{j-1}, d_j}(y_{j-1}, x_j)$ connecting the microscopic points
$(f_{j-1}, y_{j-1})\in \cB_{\eps N}(\sfi_{j-1}, \sfb_{j-1})$ and $(d_j, x_j) \in \cB_{\eps N}(\sfi_j, \sfa_j)$
by a heat kernel that depends on the mesoscopic time-space variables $(\sfi_{j-1}, \sfb_{j-1})$ and
$(\sfi_j, \sfa_j)$. However, such an approximation is only good if $\sfi_j-\sfi_{j-1}$ is sufficiently large, say
at least $K_\eps=(\log \frac{1}{\eps})^6$ as in \eqref{eq:KM}.

This naturally leads to the decomposition of
$(\sfi_1, \ldots, \sfi_k)$ into time blocks, where consecutive $\sfi_{j-1}$, $\sfi_j$ with distance less
than $K_\eps$ are grouped into a single block. The constraint $\cA_{\epsilon}^{(\notri)}$
in \eqref{eq:cA} ensures that each time block consists of either a single $\sfi_j$, or two consecutive
$\sfi_{j-1}$, $\sfi_j$, leading to the definition of \emph{time blocks $\vec\sfi$}
in Definition~\ref{def:time-block}, while
$\cA_{\epsilon}^{(\notri)}$ is mapped to the constraint $\bcA_\epsilon^{(\notri)}$ introduced
in \eqref{eq:bcA} for a sequence of time blocks.
Given a sequence of time blocks $(\vec\sfi_1, \ldots, \vec\sfi_r)$ visited by the microscopic
time-space renewal configuration, for each time block $\vec \sfi_l=(\sfi_l, \sfi_l')$, we can identify the first
mesoscopic box $\cS_{\eps N}(\sfa_l)$ visited by the renewal configuration in the time interval $\cT_{\eps N}(\sfi_l)$,
as well as the last mesoscopic box $\cS_{\eps N}(\sfa_l')$ visited by the renewal configuration in the time interval
$\cT_{\eps N}(\sfi_l')$. This produces a \emph{space block $\vec \sfa_l:=(\sfa_l, \sfa_l')$}
as in Definition~\ref{def:space-block}.

Summarizing: we can \emph{rewrite the expansion for
$\cZ_{N,\epsilon}^{(\diff)}(\varphi,\psi)$  in \eqref{eq:Zdiffa} according to the sequence of visited
time-space blocks $(\vec\sfi_1, \vec \sfa_1)$, \ldots,  $(\vec\sfi_r, \vec \sfa_r)$}, where the
diffusive constraint $\cA_{\epsilon; \sfb, \sfc}^{(\diff)}$ in \eqref{eq:cAdiff} is mapped to
$\bcA_{\epsilon; \sfb, \sfc}^{(\diff)}$ defined in \eqref{eq:bcA2} for the sequence of space blocks.
See Figure~\ref{CG-fig2}.

\smallskip

We are ready to define our last
approximation $\cZ_{N,\epsilon}^{(\cg)}(\varphi,\psi)$: having rewritten
the expansion \eqref{eq:Zdiffa} for $\cZ_{N,\epsilon}^{(\diff)}(\varphi,\psi)$
in terms of the visited time-space blocks $(\vec\sfi_1, \vec \sfa_1)$,
\ldots,  $(\vec\sfi_r, \vec \sfa_r)$, we replace \emph{each random walk transition kernel connecting
two consecutive time-space blocks} by a heat kernel depending only on the mesoscopic variables
$(\vec \sfi_\cdot, \vec \sfa_\cdot)$. More precisely, given
$(f_{j-1}, y_{j-1})\in \cB_{\eps N}(\sfi_{j-1}', \sfa'_{j-1})$ and $(d_j, x_j) \in \cB_{\eps N}(\sfi_j, \sfa_j)$,
we make within \eqref{eq:Zdiffa} the replacement
\begin{equation} \label{eq:shorten2}
	q_{f_{j-1},d_j}(y_{j-1},x_j) \ \ \rightsquigarrow \  \
	\frac{2}{\epsilon N} \,
	g_{\frac{1}{2}(\sfi_{j} - \sfi'_{j-1})} ( \sfa_{j} - \sfa'_{j-1} ),	
\end{equation}
where the prefactor $2$ is due to periodicity
(note that $\sfi_j-\sfi'_{j-1}\geq K_\eps$ by the constraint
$\bcA_\epsilon^{(\notri)}$ from \eqref{eq:bcA}).
We similarly replace the ``boundary kernels'' in \eqref{eq:Zdiffa}, namely
\begin{align}\label{eq:replacement1}
	q^N_{0,N}(\varphi,\psi) \ &\rightsquigarrow \
	\frac{1}{2} \, g_{\frac{1}{2}}(\varphi, \psi) \,,\\
	\label{eq:replacement2}
	\rule{0pt}{2.2em} q_{0, d_1}(y_0,x_1) \ &\rightsquigarrow \
	\frac{2}{\epsilon N} \, g_{\frac{1}{2} \sfi_1} ( \sfa_1 - {\sfb_0} ) \qquad
	\text{for} \quad \begin{array}{l}y_0\in \cS_{\eps N}(\sfb_0)\,, \\
	\rule{0pt}{1.1em} (d_1, x_1)\in \cB_{\eps N}(\sfi_1, \sfa_1) \,,
	\end{array} \\
	\label{eq:replacement3}
	\rule{0pt}{2.2em}q_{f_k, N}(y_k, x_{k+1}) \ &\rightsquigarrow \
	\frac{2}{\epsilon N} \, g_{\frac{1}{2} (\frac{1}{\epsilon}-\sfi_k)} ( \sfa_{k+1} -  \sfb_k)
	\qquad \text{for} \quad
	\begin{array}{l}(f_k, y_k)\in \cB_{\eps N}(\sfi_k, \sfb_k) \,, \\
	\rule{0pt}{1.1em} x_{k+1}\in \cS_{\eps N}(\sfa_{k+1}) \,,
	\end{array}
\end{align}
where the constraint $\cA_\epsilon^{(\notri)}$ (which maps to $\bcA_\epsilon^{(\notri)}$)
ensures $\sfi_1\geq K_\eps$ and $\frac{1}{\epsilon}-\sfi_k\geq K_\eps$.
We thus define $\cZ_{N,\epsilon}^{(\cg)}(\varphi,\psi)$
as the expression obtained from \eqref{eq:Zdiffa}
via the replacements \eqref{eq:shorten2} and
\eqref{eq:replacement1}-\eqref{eq:replacement3}
(this description is useful to compare $\cZ_{N,\epsilon}^{(\cg)}(\varphi,\psi)$ with
$\cZ_{N,\epsilon}^{(\diff)}(\varphi,\psi)$).

An alternative, equivalent description of $\cZ_{N,\epsilon}^{(\cg)}(\varphi, \psi)$
is through the following formula:
\begin{align}
	&  \cZ_{N,\epsilon}^{(\cg)}(\varphi, \psi) =  \frac{1}{2} g_{\frac{1}{2}}(\varphi, \psi)
	  +  \frac{1}{N} \! \sum_{r=1}^{(\log \frac{1}{\eps})^2} \!\!\!\!\!\!
	\sum_{\substack{(\vec\sfi_1, \ldots, \vec\sfi_r) \in \bcA_\epsilon^{(\notri)} \\
	\sfa'_0, \sfa_{r+1} \in \Z^2, (\vec\sfa_1, \ldots, \vec\sfa_r) \in
	\bcA_{\epsilon; \sfa'_0, \sfa_{r+1}}^{(\diff)}	 }}
	\!\!\!\!\! \bigg( \!\!\! \sum_{v \in \cS_{\eps N}({\sfa'_0})\cap \Z^2_\even}
	\!\!\!\!\!\!\!\!\!\!\! \varphi_N(v) \!\bigg)   \notag \\
	& \quad \times  \prod_{l=1}^r
	g_{\frac{1}{2}(\sfi_l -\sfi_{l-1}')} (\sfa_l -\sfa_{l-1}')
	 \Theta^{(\cg)}_{N, \epsilon}(\vec\sfi_l, \vec\sfa_l)
	\times  \bigg( g_{\frac{1}{2} (\frac{1}{\epsilon}-\sfi'_r)} (\sfa_{r+1} -  \sfa'_r ) \, \frac{2}{\epsilon N}
	 \!\!\!\!\!\!\!\!\!\!\!\!\!\!\!
	 \sum_{w \in \cS_{\epsilon N}({\sfa_{r+1}})\cap \Z^2_\even}  \!\!\!\!\!\!\!\!\!\!\!\!\!
	 \psi_N(w)\bigg), \label{eq:Zcg}
\end{align}
where $\sfi_0':=0$ and $\Theta^{(\cg)}_{N, \epsilon}(\vec\sfi_j, \vec\sfa_j)$ is the
\emph{coarse-grained disorder variable} defined in \eqref{eq:Theta}, which
collects the contributions in \eqref{eq:Zdiffa} from a given visited time-space block
$(\vec \sfi_j, \vec \sfa_j)$, and it arises thanks to the factorisations induced by the replacements in \eqref{eq:shorten2}
and \eqref{eq:replacement1}-\eqref{eq:replacement3}.

\begin{remark}
Only the prefactor $\frac{2}{\epsilon N}$ arising from the last replacement \eqref{eq:replacement3} appears
explicitly in \eqref{eq:Zcg}: all other factors of $\frac{2}{\epsilon N}$ arising from \eqref{eq:shorten2}
and \eqref{eq:replacement2}
have been absorbed in the coarse-grained disorder variable
$\Theta^{(\cg)}_{N, \epsilon}(\vec\sfi_j, \vec\sfa_j)$ following the
replaced kernel $q$, see \eqref{eq:Theta}.
\end{remark}

Finally, to compare $\cZ_{N,\epsilon}^{(\cg)}(\varphi, \psi)$ with the coarse-grained model
$\mathscr{Z}_{\epsilon}^{(\cg)}(\varphi,\psi|\Theta_{N,\epsilon}^{(\cg)})$
defined in \eqref{eq:Zcg-gen}, we introduce the notation
\begin{equation}\label{eq:phiNeps-psiNeps}
	\varphi^{(N)}_\epsilon(\sfa'_0) := \frac{2}{\epsilon N} \bigg(\!\!\!\!
	\sum_{v \in \cS_{\eps N}({\sfa'_0})\cap \Z^2_\even} \!\!\!\!\!\!\!\!\!\! \varphi_N(v)
	\bigg) \,, \qquad
	\psi^{(N)}_\epsilon(\sfa_{r+1}) := \frac{2}{\epsilon N} \bigg(\!\!\!\!\!
	\sum_{w \in \cS_{\eps N}({\sfa_{r+1}})\cap \Z^2_\even} \!\!\!\!\!\!\!\!\!\!\!\!\! \psi_N(v) \bigg) \,,
\end{equation}
which allows us to rewrite \eqref{eq:Zcg} more compactly as
\begin{equation}\label{eq:Zcg+}
\begin{split}
	 & \cZ_{N,\epsilon}^{(\cg)}(\varphi, \psi) \, = \, \frac{1}{2} \, g_{\frac{1}{2}}(\varphi, \psi)
	+ \frac{\epsilon}{2} \sum_{r=1}^{(\log \frac{1}{\eps})^2}
	\!\!\!\! \sum_{{\sfa'_0}, {\sfa_{r+1}}\, \in \Z^2}
	\sum_{\substack{(\vec\sfi_1, \ldots, \vec\sfi_r) \in \bcA_\epsilon^{(\notri)} \\
	(\vec\sfa_1, \ldots, \vec\sfa_r) \in \bcA_{\epsilon; \sfa'_0, \sfa_{r+1}}^{(\diff)}}} \\
	&\ \varphi^{(N)}_\epsilon(\sfa'_0) \times \Bigg\{ \prod_{l=1}^r
	g_{\frac{1}{2}(\sfi_l -\sfi_{l-1}')} (\sfa_l -\sfa_{l-1}')
	\, \Theta^{(\cg)}_{N, \epsilon}(\vec\sfi_l, \vec\sfa_l) \Bigg\} \,
	g_{\frac{1}{2} (\frac{1}{\epsilon}-\sfi'_r)} ( {\sfa_{r+1}} -  \sfa'_r )
	\, \psi^{(N)}_\epsilon(\sfa_{r+1}) \,.
\end{split}
\end{equation}
Compare this with $\mathscr{Z}_{\epsilon}^{(\cg)}(\varphi,\psi|\Theta_{N,\epsilon}^{(\cg)})$
in \eqref{eq:Zcg-gen}, the only difference is that
$\varphi_\eps$ and $\psi_\eps$ in \eqref{eq:Zcg-gen} are now replaced by $\varphi^{(N)}_\epsilon$
and $\psi^{(N)}_\epsilon$.

\smallskip

The main result of this subsection is the following $L^2$ approximation, which completes the proof of Theorem~\ref{th:cg-main}.

\begin{lemma}[Coarse graining]\label{th:cg}
Recall $\cZ_{N,\epsilon}^{(\diff)}(\varphi, \psi)$ from \eqref{eq:Zcg},
$\mathscr{Z}_{\epsilon}^{(\cg)}(\varphi,\psi|\Theta_{N,\epsilon}^{(\cg)})$ from \eqref{eq:Zcg-gen}
and $\Vert \cdot\Vert_{\cG_t}$ from \eqref{eq:phinorm}.  There exists $\sfC \in (0, \infty)$
such that for $\epsilon > 0$ small enough, we have: for all $\varphi$
with $\Vert \varphi\Vert_{\cG_1}^2<\infty$
and $\psi \in L^\infty(\R^2)$,
\begin{align}\label{eq:appr3cg}
	& \limsup_{N\to\infty \text{ with } N\in2\N} \
	\big\|\big(\cZ_{N,\epsilon}^{(\cg)}(\varphi, \psi) -
	 \cZ_{N,\epsilon}^{(\diff)}\big)(\varphi, \psi)
	\big\|_{L^2}^2
	\ \le \  \frac{\sfC}{\log\frac{1}{\epsilon}}\,
	\Vert \varphi\Vert_{\cG_1}^2 \Vert \psi\Vert_\infty^2 \,, \\
	\label{eq:appr3cg2}
	& \lim_{N\to\infty \text{ with } N\in2\N} \
	\big\| \mathscr{Z}_{\epsilon}^{(\cg)}(\varphi,\psi|\Theta^{\rm (cg)}_{N,\epsilon})
	- \cZ_{N,\epsilon}^{(\cg)}(\varphi,\psi) \big\|_{L^2}^2
	\ = \ 0.
\end{align}
\end{lemma}
\begin{proof}

We first prove \eqref{eq:appr3cg}.
To define $\cZ_{N,\epsilon}^{(\cg)}(\varphi,\psi)$ from $\cZ_{N,\epsilon}^{(\diff)}(\varphi,\psi)$
in \eqref{eq:Zdiffa}, we first replaced the summation
constraint $1\leq k\leq (\log \frac{1}{\eps})^2$ in \eqref{eq:Zdiffa} (on the number of visited
mesoscopic time intervals) with the
constraint $1\leq r\leq (\log \frac{1}{\eps})^2$ in \eqref{eq:Zcg} (on the number of visited
coarse-grained disorder variables), where each coarse-grained
disorder variable can visit either one or two mesoscopic time intervals. This amounts to adding some terms
with $(\log \frac{1}{\eps})^2<k \leq 2(\log \frac{1}{\eps})^2$ in \eqref{eq:Zdiffa}.
The error from such additions is bounded as in \eqref{eq:bao2} and
agrees with the bound
in \eqref{eq:appr3cg}. We then replaced the random walk kernels by heat
kernels as in \eqref{eq:shorten2} and \eqref{eq:replacement1}-\eqref{eq:replacement3}.
We will make these replacements sequentially and control the error in each step,
showing that it is bounded by the r.h.s.\ of \eqref{eq:appr3cg}.

\smallskip

First note that the replacement \eqref{eq:replacement1} simply changes the first term
 in \eqref{eq:Zdiffa}, which is a deterministic constant.
Since the l.h.s.\ of \eqref{eq:replacement1} converges to the r.h.s. as $N\to\infty$,
see \eqref{eq:gstphipsiconv}, the $L^2$ cost of
the replacement \eqref{eq:replacement1} vanishes as $N\to\infty$,
which does not contribute to the bound in \eqref{eq:appr3cg}.

Next note that thanks to the diffusive constraint $\cA_{\epsilon; {\sfb_0, \sfa_{k+1}}}^{(\diff)}$ in \eqref{eq:Zdiffa}, which maps to $\bcA_{\epsilon; \sfa'_0, \sfa_{r+1}}^{(\diff)}$ in \eqref{eq:Zcg}, the replacements  in \eqref{eq:shorten2} and \eqref{eq:replacement2}-\eqref{eq:replacement3} are all of the form
\begin{equation*}
	q_{s, t}(x, y) \quad \rightsquigarrow \quad
	\frac{2}{\epsilon N} \,
	g_{\frac{1}{2}(\sfi_{j} - \sfi)} ( \sfb - \sfa)	
\end{equation*}
for some $(s, x) \in \cB_{\eps N}(\sfi, \sfa)$ and $(t,y) \in \cB_{\eps N}(\sfj, \sfb)$, with
$(t-s, y-x)\in \Z^3_{\rm even}$, $\sfj-\sfi\geq K_\eps$, and $|\sfb-\sfa|\leq M_\eps\sqrt{\sfj-\sfi}$
(recall from \eqref{eq:KM} that $M_\eps=\log \log \frac{1}{\eps}$). We can then apply the local
limit theorem \eqref{eq:llt} and refine the bounds in Lemma~\ref{lem:ker} as follows:
\begin{align*}
	q_{s, t}(x, y) & =  2 g_{\frac{t-s}{2}}(y-x) \exp\Big\{ O\Big(\tfrac{1}{t-s}\Big)
	+ O\Big(\tfrac{|y-x|^4}{(t-s)^3}\Big)\Big\} = 2 g_{\frac{t-s}{2}}(y-x) \,
	e^{O\big(\tfrac{M_\eps^4}{\eps K_\eps N}\big)} \\
	& = \frac{2}{\eps N} g_{\frac{\sfj-\sfi}{2}}(\sfb-\sfa) \cdot
	\frac{\eps N g_{\frac{t-s}{2}}(y-x)}{g_{\frac{\sfj-\sfi}{2}}(\sfb-\sfa)} \,
	e^{O\big(\tfrac{M_\eps^4}{\eps K_\eps N}\big)} \\
	& = \frac{2}{\eps N} g_{\frac{\sfj-\sfi}{2}}(\sfb-\sfa)\,
	e^{O\big(\tfrac{M_\eps^4}{\eps K_\eps N}\big)} \cdot \frac{\eps N(\sfj-\sfi)}{t-s}
	\exp\Big\{-\frac{|y-x|^2}{t-s} + \frac{|\sfb-\sfa|^2}{\sfj-\sfi}\Big\} \,,
\end{align*}
where since
$|(t-s) - \epsilon N(\sfj-\sfi)| \le \epsilon N$ and
$|(y-x)-\sqrt{\epsilon N}(\sfb-\sfa)| \le \sqrt{2\epsilon N}$ we can bound
\begin{equation*}
	\frac{\eps N(\sfj-\sfi)}{t-s} = 1+ O\big(\tfrac{1}{K_\eps}\big) \,, \qquad
	-\frac{|y-x|^2}{t-s} + \frac{|\sfb-\sfa|^2}{\sfj-\sfi}
	= O\Big(\tfrac{M_\eps}{\sqrt{K_\eps}} \Big) \,,
\end{equation*}
so that for some $c>0$, uniformly in $\eps>0$ small enough and $N$ large, we have
\begin{equation}\label{eq:qgratio}
	e^{-c M_\eps/\sqrt{K_\eps}} \leq \frac{q_{s, t}(x, y)}{ \frac{2}{\eps N}
	g_{\frac{\sfj-\sfi}{2}}(\sfb-\sfa)} \leq e^{c M_\eps/\sqrt{K_\eps}}.
\end{equation}
Namely, every time we replace a random walk kernel by the corresponding heat kernel,
we introduce an error factor of $e^{\pm c M_\eps/\sqrt{K_\eps}}$.

We first estimate the cost of the bulk replacements \eqref{eq:replacement3}.
Consider each term in the sum in \eqref{eq:Zdiffa}, which we abbreviate by $\cZ_i$
for simplicity, where $i$ gathers the indices $k$ and $\sfi_j, \sfa_j, \sfb_j, (d_j, x_j), (f_j, y_j)$
for $1 \le j \le k$ (excluding $\sfb_0, y_0$ and $\sfa_{k+1}, x_{k+1}$).
Note that within each term $\cZ_i$,
we replace at most $(\log \frac{1}{\eps})^2$ random walk kernels,
which amounts to replacing $\cZ_i$ by $\gamma_i \cZ_i$ with
$e^{-c (\log \frac{1}{\eps})^2 M_\eps/\sqrt{K_\eps}}\leq \gamma_i
\leq  e^{c (\log \frac{1}{\eps})^2 M_\eps/\sqrt{K_\eps}}$. We then have
\begin{equation}\label{eq:qgratio2}
	\bbE[(\gamma_i \cZ_i- \cZ_i)^2]
	= (\gamma_i-1)^2 \, \bbE[\cZ_i^2]
	\leq C \, \big(\log \tfrac{1}{\eps}\big)^4 \, \frac{M^2_\eps}{K_\eps}  \, \bbE[\cZ_i^2]
	\leq \frac{C}{\log \frac{1}{\eps}} \bbE[\cZ_i^2],
\end{equation}
since $M_\eps=\log\log\frac{1}{\eps}$ and $K_\eps=\big(\log \frac{1}{\eps}\big)^6$
by \eqref{eq:KM}.
Since the terms $\cZ_i$ in the sum in \eqref{eq:Zdiffa} are \emph{mutually orthogonal},
we can sum the bound above over $i$ and we see that the contribution of the
bulk replacements to \eqref{eq:appr3cg} is at most
\begin{align} \label{eq:asino}
	\frac{C}{\log \frac{1}{\eps}}  \limsup_{N\to\infty \text{ with } N\in2\N}
	\mathbb{V}{\rm ar}(\cZ_{N,\epsilon}^{(\diff)}(\varphi, \psi))
	\ \le \ \frac{C}{\log \frac{1}{\eps}} \Vert \psi\Vert_\infty^2 \Vert \varphi\Vert_{\cG_1}^2,
\end{align}
where the last bound follows from \eqref{eq:VarZbdd}.
This agrees with \eqref{eq:appr3cg}.

We next consider the boundary replacements \eqref{eq:replacement1}
and \eqref{eq:replacement3}.
Replacing the leftmost random walk kernel $q_{0,d_1}(y_0, x_1)$ in \eqref{eq:Zdiffa}
by the corresponding heat kernel
introduces an error factor $e^{-c M_\eps/\sqrt{K_\eps}} \leq
\gamma_{y_0, (d_1, x_1)} \leq e^{+c M_\eps/\sqrt{K_\eps}}$, see \eqref{eq:qgratio},
which affects the $L^2$ norm by
$\{\sum_{y_0} \varphi_N(y_0) (\gamma_{y_0, (d_1, x_1)}-1)\}^2
\le C \frac{M_\epsilon^2}{K_\epsilon} \{\sum_{y_0} |\varphi_N(y_0)|\}^2$
(no disorder variable is attached to~$y_0$). Thus, as in \eqref{eq:asino},
the left boundary replacement contributes to \eqref{eq:appr3cg} at most by
\begin{equation*}
	\frac{M_\epsilon^2}{K_\epsilon}
	\limsup_{N\to\infty \text{ with } N\in2\N}
	\mathbb{V}{\rm ar}(\cZ_{N,\epsilon}^{(\diff)}(|\varphi|, \psi))
	\ \le \ \frac{C}{(\log \frac{1}{\eps})^{5}} \Vert \psi\Vert_\infty^2 \Vert \varphi\Vert_{\cG_1}^2,
\end{equation*}
which is a stronger bound than \eqref{eq:appr3cg}.
The right boundary replacement \eqref{eq:replacement3} is
controlled in a similar fashion, which completes the proof of \eqref{eq:appr3cg}.

\smallskip

Lastly, we prove \eqref{eq:appr3cg2}.
As noted before, the only difference between
$\cZ_{N,\epsilon}^{(\cg)}(\varphi,\psi)$ and
$\mathscr{Z}_{\epsilon}^{(\cg)}(\varphi,\psi|\Theta^{\rm (cg)}_{N,\epsilon})$ is that
$\varphi_\epsilon, \psi_\epsilon$ in \eqref{eq:Zcg-gen} are replaced by
$\varphi^{(N)}_\epsilon, \psi^{(N)}_\epsilon$ in \eqref{eq:Zcg+}.
For $\sfa\in \Z^2$, $\eps\varphi_\eps(\sfa)$ is the integral of
$\varphi$ over the square
$\cS_{\eps}(\sfa)=(\sqrt{\eps}\sfa-(\sqrt{\eps},\sqrt{\eps}), \sqrt{\eps}\sfa]$, by the definition of
$\varphi_\eps$ in \eqref{eq:phieps-psieps}.
On the other hand, by the definition of $\varphi_\eps^{(N)}$ in \eqref{eq:phiNeps-psiNeps} and
$\varphi_N$ in \eqref{eq:phiNpsiN}, $\eps \varphi_\eps^{(N)}$ is the integral of $\varphi$ over
the region $\tilde \cS_\eps(\sfa):= \bigcup_{v\in \cS_\eps(\sfa) \cap (\Z^2_{\rm even}/\sqrt{N}) }
\,\{x\in \R^2: |x-v|_1\leq \frac{1}{\sqrt N}\}$. The difference between $\cS_\eps(\sfa)$ and
$\tilde \cS_{\eps}(\sfa)$ is contained in a shell of thickness $1/\sqrt{N}$ around the boundary of
$\cS_\eps(\sfa)$. Therefore, if $\varphi: \R^2 \to \R$ is locally
integrable, then $\varphi_\eps^{(N)}$ converges pointwise to $\varphi_\eps$ as functions on
$\Z^2$, while if $\psi: \R^2 \to \R$ is also bounded,
then $\psi_\eps^{(N)}$ converges uniformly to $\psi_\eps$. If $\varphi$ has bounded
support, then \eqref{eq:appr3cg2} is easily seen to hold since we already have control over
$\mathbb{V}{\rm ar}(\cZ_{N,\epsilon}^{(\cg)}(\varphi,\psi))$ that is uniform in $N$. General $\varphi$
can then be handled by truncating its support. This concludes the proof of Lemma~\ref{th:cg}.
\end{proof}

Combining Lemmas~\ref{th:no3}, \ref{th:diff} and~\ref{th:cg} then gives Theorem~\ref{th:cg-main}.

\subsection{A second moment bound for the coarse-grained model}\label{sec:5.4}
In this subsection, we prove a second moment bound for the coarse-grained model, which is loosely speaking the analogue of Lemma~\ref{lem:UnGauss}. This is needed when we bound the fourth moment of the coarse-grained model in Section~\ref{Sec:4MomCoarse}.

First we introduce some notation. Let us define the following variants of $\bcA_{\epsilon}^{(\notri)}$ and $\bcA_{\epsilon; \,\sfb, \sfc}^{(\diff)}$ (see
\eqref{eq:bcA} and \eqref{eq:bcA2}), without dependence on the boundary conditions:
\begin{equation}\label{eq:bcAbis}
\begin{split}
	\bcA_{\epsilon}^{(\notri)} :=
	\Big\{\,
	\text{time blocks} \ \
	\vec\sfi_1 \,<\, \ldots \,<\, \vec\sfi_r
	 \quad \text{such that}  \ \
	|\vec\sfi_j|\leq K_\eps \ \ \forall j=1, \ldots, r & \\
	\text{and} \ \
	\dist(\vec\sfi_{j-1} \,, \, \vec\sfi_{j}) \geq K_\epsilon \ \
	\forall j=2, \ldots, r  & \, \Big\} \,,
\end{split}
\end{equation}
\begin{equation}\label{eq:bcA2bis}
\begin{split}
	\bcA_{\epsilon}^{(\diff)} :=
	\bigg\{ \text{space blocks} \ \ \vec\sfa_1, \ldots, \vec\sfa_r \quad \text{such that} \ \
	|\vec\sfa_j| \le  M_\epsilon \sqrt{|\vec\sfi_j|} \quad \forall j=1,\ldots, r \,, & \\
	\dist(\vec\sfa_{j-1}, \vec\sfa_{j})  \le
	M_\epsilon \, \sqrt{\dist(\vec\sfi_{j-1}\,,\, \vec\sfi_{j})} \quad \forall j=2,\ldots, r &
	\, \bigg\} \,.
\end{split}
\end{equation}

Recall the definition \eqref{eq:barU}-\eqref{U-diagram} of
$\overline{U}(n-m,x-y)$. We introduce an analogous quantity
for the coarse-grained model defined in \eqref{eq:Zcg-gen} (illustrated in Figure~\ref{CG-fig2}).
Given $\sfn \in \N_0$ and $\sfx \in \Z^2$, we define a
coarse-grained analogue of $X_{d,f}(x,y)$ in \eqref{eq:X}:
\begin{equation}\label{eq:Zcgnx}
	\mathscr{X}^{(\cg)}_{N, \eps}(\sfn, \sfx)
	:=  \begin{cases}
	\rule{0pt}{1.3em}
	 \Theta^{\rm(cg)}_{N,\eps}\big( (0,\sfn), (0,\sfx)\big)
	& \text{if } \sfn < K_\epsilon \,, \\
	\rule{0pt}{6.5em}\begin{split}
	\sum_{r =2}^{(\log \frac{1}{\eps})^2}
	\!\!\! \sum_{\substack{
	(\vec\sfi_1, \ldots, \vec\sfi_r) \in  \bcA_\epsilon^{(\notri)} \\
	\sfi_1=0,\ \sfi_r'=\sfn }} \,
	\sum_{\substack{
	(\vec\sfa_1, \ldots, \vec\sfa_r) \in  \bcA_\epsilon^{(\diff)} \\
	\sfa_1=0, \ \sfa_r'=\sfx }} \, \Theta^{\rm(cg)}_{N,\eps}(\vec\sfi_1, \vec\sfa_1)
	\ \ & \\
       \times  \prod_{j=2}^r g_{\frac{1}{2}(\sfi_j -\sfi_{j-1}')} (\sfa_j -\sfa_{j-1}') \,
	\Theta^{\rm(cg)}_{N,\eps}(\vec\sfi_j, \vec\sfa_j) &
	\end{split}
	& \text{if } \sfn \ge K_\epsilon \,,
	\end{cases}
\end{equation}
and define
\begin{equation} \label{eq:Ucgnx}
	\overline{U}^{\rm (cg)}_{N, \eps}(\sfn, \sfx)
	:= \bbE[(\mathscr{X}^{(\cg)}_{N, \eps}(\sfn, \sfx))^2].
\end{equation}

We prove the following analogue of Lemma~\ref{lem:UnGauss} (with an extra sum in the time index).

\begin{lemma}\label{th:ren-cg}
For every $\sfc \in (0,\infty)$, there exist $\sfC = \sfC(\sfc) \in (0, \infty)$
and $\hat\lambda_0=\hat \lambda_0(\sfc) \in (0,\infty)$
such that:  there exists $\epsilon_0 >0$ such that for all $\epsilon \in (0, \epsilon_0)$
and $\hat\lambda \in (\hat\lambda_0,\infty)$, we have
\begin{equation}\label{eq:ren-cg}
	\limsup_{N\to\infty} \,
	\sum_{\sfn=0}^{2/\epsilon} e^{-\hat\lambda \epsilon \sfn} \, \sum_{\sfx \in \Z^2}
	e^{\sfc \sqrt{\epsilon} |\sfx|} \,
	\overline{U}^{\rm (cg)}_{N, \eps}(\sfn, \sfx)
	\,\le\, \frac{\sfC}{\log\hat\lambda} \,.
\end{equation}
\end{lemma}

\begin{proof}
The basic strategy is to first undo the replacement of the random walk kernels by heat kernels in the definition of the
coarse-grained model $\cZ_{N,\epsilon}^{(\cg)}(\varphi,\psi)$ in Section~\ref{sec:5.3}. We then undo the summation
constraints imposed in Sections \ref{sec:5.1} and \ref{sec:5.2}, which allows us to bound $\overline{U}^{\rm (cg)}_{N, \eps}(\sfn, \sfx)$, the second moment of $\mathscr{X}^{(\cg)}_{N, \eps}(\sfn, \sfx)$, in terms of the second moment of
the original partition function, so that Lemma~\ref{lem:UnGauss} can be applied. The details are as follows.

Let us recall how the coarse-grained model $\cZ_{N,\epsilon}^{(\cg)}(\varphi, \psi)$ in \eqref{eq:Zcg} was
defined from $\cZ_{N,\epsilon}^{(\diff)}(\varphi,\psi)$ in \eqref{eq:Zdiffa} by replacing the random walk
kernels with heat kernels in the chaos expansion (see \eqref{eq:shorten2}, \eqref{eq:replacement2}-\eqref{eq:replacement3}).
It was shown in the proof of Lemma~\ref{th:cg}, in particular, in \eqref{eq:qgratio} and
\eqref{eq:qgratio2}, that the aggregate effect of such replacements is to change the second moment by a factor that is bounded
between $1-C/\log \frac{1}{\eps}$ and $1+ C/\log\frac{1}{\eps}$. We can therefore undo these replacements, which only changes the second moment by a factor that is bounded between $1-C'/\log \frac{1}{\eps}$ and $1+ C'/\log\frac{1}{\eps}$ ($\leq 2$ for $\eps$ small). More precisely,
define
$$
	\mathcal{X}_{N, \eps}(\sfn, \sfx) 	
	:=  \frac{2}{\eps N}\, \sum_{r=1}^{\infty}
	\!\!\!\!\!\!\!\!\!\!\!\!\!\!\!\! \sum_{\substack{n_1 < \ldots < n_r  \\
	z_1, \ldots, z_r \,\in\,\Z^2 \\
	(n_1, z_1)\in \cB_{\eps N}(0,0), \, (n_r, z_r)\in \cB_{\eps N}(\sfn, \sfx)}}
	\!\!\!\!\!\!\!\!\!\!\!\!\!\!\!\! \xi_N(n_1,z_1) \prod_{j=2}^r q_{n_{j-1}, n_j}(z_{j-1},z_j)
	\, \xi_N(n_j,z_j),
$$
which is obtained by reversing the replacements of the random walk transition kernels by heat kernels in the definition of
$\mathscr{X}^{(\cg)}_{N, \eps}(\sfn, \sfx)$ in \eqref{eq:Zcgnx}, plugging in the chaos expansion for
the coarse-grained disorder variables $\Theta^{(\rm cg)}_{N,\eps}$ from \eqref{eq:Theta}, \eqref{eq:Xdiff} and \eqref{eq:Zpolydiff},
and then relaxing the constraints on the time-space summation indices. The pre-factor $2/\eps N$ comes from the first $\Theta^{(\rm cg)}_{N, \eps}$ in \eqref{eq:Zcgnx} and is a normalising factor in its definition in \eqref{eq:Theta}. Since relaxing the summation
constraint only increases the second moment because the terms of the chaos expansion are
$L^2$ orthogonal, we have that for $\eps$ sufficiently small and all $N$ large,
\begin{equation}
\begin{aligned}
	\overline{U}^{\rm (cg)}_{N, \eps}(\sfn, \sfx)
	= \bbE[(\mathscr{X}^{(\cg)}_{N, \eps}(\sfn, \sfx))^2]
	& \leq 2 \, \bbE[(\mathcal{X}_{N, \eps}(\sfn, \sfx))^2]  \\
	& = \frac{8}{(\eps N)^2} \sum_{(m,y) \in
	\cB_{\eps N}(0,0), \, (n, x)\in \cB_{\eps N}(\sfn, \sfx)} \!\!\!\!\!\!\!\!\!\! \overline{U}_N(n-m,x-y) \,,
\end{aligned}
\end{equation}
where we have used the definition of $\overline{U}_N(n-m,x-y)$ from \eqref{eq:barU}-\eqref{U-diagram}.
Substituting this bound into the l.h.s.\ of \eqref{eq:ren-cg} then gives
\begin{align*}
	& \sum_{\sfn=0}^{2/\epsilon} e^{-\hat \lambda \epsilon \sfn} \,
	\sum_{\sfx \in \Z^2}
	e^{\sfc \sqrt{\epsilon} |\sfx|} \, \overline{U}^{\rm (cg)}_{N, \eps}(\sfn, \sfx) \\
	& \leq\ \frac{8}{(\eps N)^2} \sum_{(m,y) \in \cB_{\eps N}(0,0)}
	\sum_{\sfn=0}^{2/\epsilon} e^{-\hat \lambda \epsilon \sfn} \, \sum_{\sfx \in \Z^2}
	e^{\sfc \sqrt{\epsilon} |\sfx|} \!\!\!\!\!  \sum_{(n, x)\in \cB_{\eps N}(\sfn, \sfx)}
	\!\!\!\!\! \overline{U}_N(n-m,x-y) \,.
\end{align*}
We now observe that for $(m,y) \in \cB_{\eps N}(0,0)$ and
$(n,x)\in \cB_{\eps N}(\sfn, \sfx)$ we have
$\frac{n-m}{\eps N} \in [\sfn-1, \sfn+1]$ and $\frac{|x-y|}{\sqrt{\eps N}}
\in [|\sfx|-\sqrt{2}, |\sfx|+\sqrt{2}]$, hence $\epsilon \sfn = \frac{n-m}{N} + O(\epsilon)$ and
$\sqrt{\epsilon} |\sfx| =\frac{x-y}{\sqrt{N}} + O(\sqrt{\epsilon})$.
Recalling that $|\cB_{\eps N}(0,0)| = O((\epsilon N)^2)$,
the change of variables $(n-m, x-y) = (l, z)$ then yields
\begin{align*}
	& \limsup_{N\to\infty} \ \sum_{\sfn=0}^{2/\epsilon} e^{-\hat \lambda \epsilon \sfn} \,
	\sum_{\sfx \in \Z^2}
	e^{\sfc \sqrt{\epsilon} |\sfx|} \, \overline{U}^{\rm (cg)}_{N, \eps}(\sfn, \sfx)
	\ \leq \ C \, \limsup_{N\to\infty}\
	\sum_{l=0}^{3N} e^{-\hat \lambda \frac{l}{N}} \, \sum_{z \in \Z^2}
	e^{\sfc \frac{|z|}{\sqrt N}} \ \overline{U}_N(l,z) \\
	& \ \leq\ C \limsup_{N\to\infty} \ \sum_{l=0}^{3N} e^{-(\hat \lambda -c \sfc^2)
	\frac{l}{N}}  \overline{U}_N(l)
	\ \leq \ C \limsup_{N\to\infty} \ \sum_{l=1}^{3N} e^{-(\hat \lambda -c \sfc^2)\frac{l}{N}}
	\frac{1}{N} G_\theta\Big(\frac{l}{N}\Big) \\
	&\  =\ C \int_0^3 e^{-(\hat\lambda -c \sfc^2)s } G_{\theta}(s) \dd s
	\leq C \int_0^{\frac{1}{\sqrt{\hat \lambda}}} e^{-\frac{\hat\lambda}{2} s } G_{\theta}(s) \dd s
	+ C \int^3_{\frac{1}{\sqrt{\hat \lambda}}} e^{-\frac{\hat\lambda}{2} s } G_{\theta}(s) \dd s
	\ \leq\ \frac{\sfC}{\log \hat\lambda},
\end{align*}
where we applied Lemma~\ref{lem:UnGauss} and~\eqref{eq:asU1ub} in the second
and third inequalities, then we chose $\hat\lambda \geq 2 c \sfc^2=:\hat\lambda_0$ and
applied \eqref{eq:G5} in the last line. This concludes the proof of Lemma~\ref{th:ren-cg}.
\end{proof}

\section{Higher moment bounds for averaged partition functions}
\label{Sec:functional}

In this section, we bound higher moments of the averaged partition function $\cZ_N^\beta(\varphi, \psi)$
(see \eqref{eq:ZNav} and \eqref{eq:Zpolyav}) in the critical window as specified in Theorem~\ref{th:main0}
and \eqref{eq:sigma}-\eqref{eq:betaN}. As noted in Section~\ref{sec:lit} and the discussion therein on Schr\"odinger operators with
point interactions, these bounds are very delicate in the critical window. Unlike in the sub-critical regime
considered in \cite{CSZ17b, CSZ20}, where the chaos expansion of $\cZ_N^\beta(\varphi, \psi)$ is
supported (up to a small $L^2$ error)
on chaoses of finite order independent of $N$, for $\beta=\beta_N$ in the critical window, the
expansion is supported on chaoses of order $\log N$, so hypercontractivity can no longer be used to bound
higher moments in terms of the second moment. Instead, the expansion has to be controlled with much
greater care. Bounds on the third moment were first obtained in \cite{CSZ19b}. Bounds on higher moments
of the averaged solution of the mollified stochastic heat equation (continuum analogues of
$\cZ_N^\beta(\varphi, \psi)$), for $\varphi, \psi \in L^2$, were then obtained in \cite{GQT21} using
techniques from the study of Schr\"odinger operator with point interactions (also called Delta-Bose gas)
\cite{DFT94, DR04}. The recent work \cite{C21} studied the semigroup associated with the Schr\"odinger
operator and allowed $\varphi$ to be delta functions.

Our goal is to develop similar moment bounds as in \cite{GQT21} for the averaged polymer partition
function $\cZ_N^\beta(\varphi, \psi)$. The approach of
\cite{GQT21} used explicit Fourier calculations and the underlying space-time white noise, which cannot be
easily adapted to lattice models with general disorder. We develop an alternative approach,
where the key
ingredient is a functional inequality for the Green's function of multiple random walks on $\Z^2$
(see Lemma~\ref{HLSineq}).
This leads to Theorem~\ref{th:mom}, which is the main result of this section, where
instead of working with $\varphi, \psi \in L^2$ as in \cite{GQT21}, we
will work with weighted $L^p$--$L^q$ spaces with $\frac{1}{p}+\frac{1}{q}=1$, which allows
$\psi(y)\equiv 1$ and $\varphi(x) = \epsilon^{-1} \ind_{|x|\le \sqrt{\epsilon}}$
to be an approximate delta function on the scale $\sqrt{\eps}$, and it also
gives spatial decay if the support of $\varphi$ and $\psi$ are far apart. Our
approach is robust enough that
it can be applied the coarse-grained disorder variables $\Theta_{N,\epsilon}^{(\cg)}$,
which can be seen as an averaged partition functions (see Lemma~\ref{Theta4th}), and
it can also be adapted to the coarse-grained model
$\mathscr{Z}_{\epsilon,t}^{(\cg)}(\varphi,\psi|\Theta_{N,\epsilon}^{(\cg)})$,
as we will show in Theorem~\ref{th:cgmom}.

\subsection{Statement and proof}

Given a countable set $\bbT$ and a function $f: \bbT \to \R$, we use the standard notation
\begin{equation} \label{eq:ell}
	\|f\|_{\ell^p}
	= \|f\|_{\ell^p(\bbT)} := \bigg( \sum_{z\in\bbT} |f(z)|^p \bigg)^{1/p}
	\quad \text{for } p \in [1,\infty) \,, \qquad
	\|f\|_{\ell^\infty} := \sup_{z\in\bbT} |f(z)| \,,
\end{equation}
while we let $\|g\|_p$ denote the usual $L^p$ norm of $g: \R^2 \to \R$.
We will ignore parity issues, since this only affects the bounds by a constant multiple:
for a locally integrable function $\varphi: \R^2 \to \R$, we consider its discretization
$\varphi_N: \Z^2 \to \R$ in \eqref{eq:phiNpsiN} to be defined on the whole $\Z^2$
(rather than just on $\Z^2_\even$).
Here is the main result of this section.

\begin{theorem}[Higher moments]\label{th:mom}
For $N \le \tilde N \in \N$,
let $\cZ_{N}^{\beta_{\tilde N}}(\varphi, \psi) = \cZ_{N,1}^{\beta_{\tilde N}}(\varphi, \psi)$ be
the averaged partition function in \eqref{eq:ZNav},
where $\beta = \beta_{\tilde N}=\beta_{\tilde N}(\theta)$ satisfies \eqref{eq:sigma}
for some $\vartheta\in \R$.
Fix $p, q \in (1,\infty)$ with $\frac{1}{p}+\frac{1}{q}=1$, an integer $h\ge 3$, and
a weight function $w: \R^2 \to (0,\infty)$ such that $\log w$
is Lipschitz continuous.
Then there exist $\sfC, \sfC' < \infty$ such that, uniformly
in large $N \le \tilde N\in\N$
and locally integrable $\varphi, \psi: \R^2 \to \R$, we have
\begin{equation}\label{eq:mombd}
\begin{split}
	\Big|\bbE\Big[\Big(\cZ_{N}^{\beta_{\tilde N}}(\varphi, \psi) -
	\bbE[\cZ_{N}^{\beta_{\tilde N}}(\varphi, \psi)]\Big)^h\Big]\Big|
	& \,\leq\, \frac{\sfC}{\log (1+\frac{\tilde N}{N})}
	\, \frac{1}{N^{h}} \, \Big\Vert \frac{\varphi_N}{w_N}
	\Big\Vert_{\ell^p}^h \, \Vert \psi_N\Vert_{\ell^\infty}^h
	\, \Vert w_N \ind_{B_N} \Vert_{\ell^q}^h \\
	& \,\leq\, \frac{\sfC'}{\log (1+\frac{\tilde N}{N})} \, \Big\Vert \frac{\varphi}{w}
	\Big\Vert_{p}^h \, \Vert \psi\Vert_{\infty}^h
	\, \Vert w \ind_{B} \Vert_{q}^h \,,
\end{split}
\end{equation}
where $\varphi_N, \psi_N, w_N: \Z^2 \to \R$
are defined from $\varphi, \psi, w: \R^2 \to \R$ by \eqref{eq:phiNpsiN},
we denote by $B \subseteq \R^2$ a ball on which $\psi$ is supported (possibly $B=\R^2$),
and we set $B_N := B \sqrt{N}$.
\end{theorem}

\smallskip

Theorem~\ref{th:mom} will be needed later in the proof of Lemma~\ref{Theta4th},
where we consider $N= \epsilon \tilde N$ with $\epsilon \in (0,1)$;
this is why we allow for $\beta = \beta_{\tilde N}$ with $\tilde N \ge N$.

\begin{remark}\label{rem:aswe}
The second line
of \eqref{eq:mombd} follows from the first line by Riemann sum approximation
(note that $w(\frac{x+y}{\sqrt{N}}) = (1+O(\frac{|y|}{\sqrt{N}})) \, w(\frac{x}{\sqrt{N}})$
by the Lipschitz continuity of  $\log w$).
\end{remark}

\begin{remark}\label{r:1}
Typically we will let $w(x) = e^{-|x|}$, which allows $\psi\equiv 1$ provided $\varphi$
decays sufficiently fast at $\infty$. If $\psi$ is bounded with compact support, \eqref{eq:mombd} also gives
exponential spatial decay as the support of $\psi$ moves to infinity. This answers the conjecture in
\cite[Remark 1.2]{GQT21} in our lattice setting, which can be extended to their continuum setting.
\end{remark}

\begin{remark}
Similar to \cite[Theorem 1.1]{GQT21}, we can show that the moments
in the l.h.s.\ of \eqref{eq:mombd} converge as $N\to\infty$. However, the limits are
expressed as series of iterated integrals and are not very informative, so we will not state them here.
\end{remark}

\begin{remark}\label{r:2}
In the bound \eqref{eq:mombd},
we could first assign $\varphi_N, \psi_N : \Z^2 \to \R$
and then define the corresponding $\varphi, \psi: \R^2 \to \R$, e.g.\ by piecewise constant extension
$\varphi(x) := \varphi_N(\ev{\sqrt{N}x})$ and
$\psi(y) := \psi_N(\ev{\sqrt{N}x})$, because
 $\cZ_N^\beta(\varphi, \psi)$
depends on the functions $\varphi, \psi: \R^2 \to \R$ only through their discretizations
$\varphi_N, \psi_N$ in \eqref{eq:phiNpsiN}, see \eqref{eq:ZNavalt}.

In particular, we can apply the bound \eqref{eq:mombd}
to the point-to-plane partition function $Z_N^{\beta_N}(0)$ defined in \eqref{eq:paf}.
More precisely, we can write
$Z_N^{\beta_N}(0) = \sum_{y\in \Z^2}Z^{\beta_{N}}_N(0, y)=:\cZ^{\beta_{N}}_N(\varphi, \ind)$
with $\varphi_N(w)= N \ind_{\{w=0\}}$
and $\psi_N(z) = \ind(z) \equiv 1$,
cf.~\eqref{eq:ZN} and \eqref{eq:ZNavalt},
which correspond to $\varphi(x) := \varphi_N(\ev{\sqrt{N}x})
= N \, \ind_{\{|x_1|+|x_2|\le1/\sqrt{N}\}}$ and $\psi(y) \equiv 1$.
Note that $\|\varphi\|_p = O(N^{1-\frac{1}{p}})$. Then,
applying \eqref{eq:mombd} with $w(x)= e^{-| x |}$
implies that for any integer $h\ge 3$ and for any $p>1$,
there exists $C_{p,h} < \infty$ such that for all $N\in\N$,
\begin{equation}
	\Big|\bbE\Big[\Big(
	Z_N^{\beta_N}(0) -
	\bbE\big[Z_N^{\beta_N}(0)\big]\Big)^h\Big]\Big|
	\,=\, \Big|\bbE\Big[\big(
	Z_N^{\beta_N}(0) - 1 \big)^h\Big]\Big|
	\, \leq \, C_{p,h} \, N^{h (1-\frac{1}{p})} \,.
\end{equation}
Since we can take any $p>1$, this shows that centered moments of any order $h\ge 3$
of the point-to-plane partition function $Z_N^{\beta_N}(0)$ diverge as $N\to\infty$ more slowly
than any polynomial.
\end{remark}

\begin{proof}[Proof of Theorem~\ref{th:mom}]
Our starting point is the polynomial chaos expansion of $\cZ_N^{\beta_{\tilde N}}(\varphi, \psi)$ as in \eqref{eq:Zpolyav}, which gives
\begin{align}
	&M_{N, \tilde N, h}^{\varphi, \psi}:=
	\bbE\Big[\Big(\cZ_N^{\beta_{\tilde N}}(\varphi, \psi)
	- \bbE[\cZ_N^{\beta_{\tilde N}}(\varphi, \psi)]\Big)^h\Big] \label{eq:MNh} \\
	= & \frac{1}{N^h} \bbE\Big[ \Big(\sum_{r=1}^{\infty} \!\!\!\!\!\!
	\sum_{\substack{z_1, \ldots, z_r \in \Z^2 \\ 0 < n_1 < \ldots < n_r \le N }}  \!\!\!\!\!\!\!\!\!\!
	q^N_{0,n_1}(\varphi,z_1) \, \xi(n_1,z_1)
	 \Big\{ \prod_{j=2}^r q_{n_{j-1}, n_j}(z_{j-1},z_j)
	\, \xi(n_j,z_j) \Big\} \, q^N_{n_j,N}(z_j,\psi)\Big)^h \Big],  \nonumber
\end{align}
where $\xi(n,z) = \xi_{\tilde N}(n,z)$ is as defined in \eqref{eq:xi} with $\beta_N$
therein replaced by $\beta_{\tilde N}(\theta)$ so that $\mathbb{V}{\rm ar}(\xi)=\sigma_{\tilde N}^2$.
We will expand the $h$-fold product above, which gives a sum over $h$ {\em microscopic
time-space renewals}
$(n^i_1, z^i_1), \ldots, (n^i_{r_i}, z^i_{r_i})$, $1\leq i\leq h$. Given these $h$ renewals, each lattice point
$(m,x)$ will contribute a factor of $\bbE[\xi(m,x)^{\#}]$, where $\#$ is the number of times $(m,x)$
appears among the $h$ time-space renewals. Recall $\xi$ and $\sigma_{\tilde N}$  from \eqref{eq:xi}
and \eqref{eq:sigma}, we have
\begin{equation}\label{eq:ximom}
	\bbE[\xi(m,x)]=0, \quad \bbE[\xi(m,x)^2] = \sigma_{\tilde N}^2\sim \frac{\pi}{\log \tilde N},
	\quad \big|\bbE[\xi(m,x)^l]\big| \leq C \sigma_{\tilde N}^l \quad \mbox{for } l\geq 3.
\end{equation}
Therefore a given configuration of $h$ time-space renewals will give a non-zero contribution to the
expansion in \eqref{eq:MNh}
if each $(m,x)$ is visited by none or by at least two of the $h$ renewals. We will rewrite the expansion by
first summing over all possible choices of the set of time coordinates
$\bigcup_{i=1}^h \{n^i_1, \ldots , n^i_{r_i}\}$, then for each time $n$ in this set, sum over the locations
$x\in\Z^2$ such that $(n,x)$ is visited by (at least two) of the $h$ renewals
$(n^i_1, z^i_1), \ldots, (n^i_{r_i}, z^i_{r_i})$, and lastly, determine which of the $h$ renewals visits $(n,x)$.
Note that in the expansion \eqref{eq:MNh}, for each of the $h$ renewal sequences that visits $(n,x)$, there is
a random walk transition kernel $q$ entering $(n,x)$ and another one exiting $(n,x)$, while for each renewal
that does not visit the time plane $\{(n,y): y \in \Z^2\}$,  there is a transition kernel $q_{a,b}(x, z)$ with
$a<n<b$, for which we will use Chapman-Kolmogorov to rewrite it as
$q_{a,b}(x,z) = \sum_{y\in \Z^2} q_{a,n}(x,y) q_{n,b}(y, z)$.\footnote{This is the key difference
between the expansions in \cite{CSZ19b} and \cite{GQT21}. This decomposition was used in \cite{GQT21},
which allows a functional analytic interpretation of the iterated sums and helps bypass the combinatorial
complexity encountered in \cite{CSZ19b}, which the authors could control for the third moment but seemed
intractable for higher moments.}

To expand the centred moment $M_{N,\tilde N, h}^{\varphi, \psi}$ in \eqref{eq:MNh} as described
above, we first introduce some notation. Given $h\geq 2$, let $I \vdash\{1,...,h\}$ denote a partition
$I=I(1)\sqcup\cdots\sqcup I(m)$ of $\{1,...,h\}$ into disjoint subsets $I(1),...,I(m)$ with cardinality $|I|=m$.
Write $k\stackrel{I}{\sim} l$ if $k$ and $l$ belong to the same partition element of $I$.
The interpretation is that, for a given time $n$,
we have $k\stackrel{I}{\sim} l$ if the $k$-th and $l$-th time-space renewals visit
the same time-space point $(n,x)$ for some $x\in\Z^2$,
which leads to a power of the disorder variable $\xi(n,x)$. Given $I \vdash\{1, \ldots, h\}$, denote
\begin{equation}\label{eq:Exi}
\bbE[\xi^I] := \prod_{1\leq j \leq |I|, |I(j)|\geq 2}\bbE[\xi^{|I(j)|}].
\end{equation}
For $\bx \in (\Z^2)^h$, we denote
\begin{equation}\label{eq:xsimI}
\bx \sim I \quad \mbox{if  }  x_k=x_l \ \forall\,  k \stackrel{I}{\sim} l.
\end{equation}
For $\bx, \tilde \bx \in (\Z^2)^h$, denote the $h$-component random walk transition probabilities by
\begin{align}\label{eq:Qdef}
	Q_{t}(\bx, \tilde \bx):= \prod_{i=1}^h q_{t}(\tilde x_i-x_i), \ \, Q^N_t(\varphi, \bx)
	:=\prod_{i=1}^h q^N_{0,t}(\varphi, x_i),  \  \, Q^N_t(\bx, \psi):=\prod_{i=1}^h q^N_{0,t}(x_i, \psi),
\end{align}
where $q^N_{0,t}(\varphi, x_i)$ and $q^N_{0,t}(x_i, \psi)$
are defined in \eqref{eq:qNphi}--\eqref{eq:qNpsi}, and for $I, J\vdash \{1, \ldots, h\}$, denote
\begin{equation}\label{eq:QIJt}
	Q^{I, J}_{t}(\bx, \tilde \bx):= \ind_{\{\bx\sim I, \tilde\bx\sim J\}}Q_{t}(\bx, \tilde \bx).
\end{equation}
We can then write
\begin{equation}\label{eq:MNh2}
\begin{aligned}
M_{N,\tilde N, h}^{\varphi, \psi} = \frac{1}{N^h} & \sum_{r=1}^\infty \sum_{\substack{1\leq n_1<\cdots <n_r\leq n_{r+1}:=N \\ I_1, \ldots, I_r \vdash\{1, \ldots, h\},\, m_i:=|I_i|<h \\
 \by_1, \ldots, \by_r \in (\Z^2)^h}}  Q^N_{n_1}(\varphi, \by_1) \ind_{\{\by_1\sim I_1\}} \bbE[\xi^{I_1}] \\
& \qquad \times \prod_{i=2}^r Q^{I_{i-1}, I_i}_{n_i-n_{i-1}}(\by_{i-1}, \by_i) \bbE[\xi^{I_i}] \times \ind_{\{\by_r\sim I_r\}}  Q^N_{n_{r+1}-n_r}(\by_r, \psi).
\end{aligned}
\end{equation}

\begin{figure}
\hskip -0.2cm
\begin{tikzpicture}[scale=0.4]
\draw[thick] (-6,-5)--(-6,5); \draw[thick] (-4,-5)--(-4,5); \draw[thick] (0,-5)--(0,5); \draw[thick] (2,-5)--(2,5); \draw[thick] (5,-5)--(5,5);
\draw[thick] (7,-5)--(7,5); \draw[thick] (9,-5)--(9,5);  \draw[thick] (7,-5)--(7,5); \draw[thick] (14,-5)--(14,5);
\draw   (-8,3.5) circle [radius=0.2];  \draw  (-8,-1)  circle [radius=0.2];   \draw   (-8,2)  circle [radius=0.2];  \draw   (-8,-3)  circle [radius=0.2];
\draw  (-6, -3.55)  circle [radius=0.2]; \draw  (-6, 3.8)  circle [radius=0.2];  \draw  [fill] (-6,0.5)  circle [radius=0.2];
\draw   (-4, -3)  circle [radius=0.2];  \draw  (-4, 4.0)  circle [radius=0.2]; \draw  [fill] (-4,0.5)  circle [radius=0.2];
\draw   (0, 2.7)  circle [radius=0.2]; \draw  [fill] (0, -0.5)  circle [radius=0.2]; \draw   (0, 4.0)  circle [radius=0.2];
\draw  (2, 2.7)  circle [radius=0.2];  \draw  [fill] (2, -0.5)  circle [radius=0.2]; \draw   (2, 3.9)  circle [radius=0.2];
\draw  [fill] (5, 1)  circle [radius=0.2]; \draw   (5, 3.5)  circle [radius=0.2];
\draw  [fill] (7, 1)  circle [radius=0.2];
\draw  (9,-1)  circle [radius=0.2];   \draw  [fill] (9, 3)  circle [radius=0.2]; \draw  (9, -3)  circle [radius=0.2];
\draw   (14,3) circle [radius=0.2]; \draw   (14,1) circle [radius=0.2]; \draw   (14,-1) circle [radius=0.2]; \draw   (14,-3) circle [radius=0.2];
\draw[thick] (-8,3.5) to [out=10,in=170] (5, 3.5); \draw[thick] (-8,-1) to [out=200,in=260] (-6,0.5);  \draw[thick] (-8,2) to [out=-30,in=100] (-6,0.5);
\draw[thick] (-8,-3) to [out=-30,in=210] (0,-0.5);
\draw[thick] (-4,0.5) to [out=50,in=130] (5,1); \draw[thick] (-4,0.5) to [out=-30,in=210] (0,-0.5);
\draw[thick] (2,-0.5) to [out=50,in=180] (5,1); \draw[thick] (2,-0.5) to [out=0,in=220] (5,1);
\draw [-,thick] (5,1) -- (7,1);  \draw [thick] (5,1) to [out=45,in=135] (7,1); \draw [thick] (5,1) to [out=-45,in=-135] (7,1); \draw [thick] (5,3.5) to [out=-20,in=100] (7,1);
\draw[thick] (7,1) to [out=-60,in=150] (9,-1);  \draw[thick] (7,1) to [out=60,in=200] (9, 3);  \draw[thick] (7,1) to [out=30,in=250] (9, 3);  \draw[thick] (7,1) to [out=-80,in=145] (9, -3);
\draw [thick] (12,4) to [out=0,in=135] (14,3); \draw [thick] (12,1.5) to [out=0,in=145] (14,1);
\draw [thick] (12,-1.5) to [out=0,in=-145] (14,-1); \draw [thick] (12,-4) to [out=0,in=-135] (14,-3);
\draw [-,thick, decorate, decoration={snake,amplitude=.4mm,segment length=2mm}] (0,-0.5) -- (2,-0.5);
\draw [-,thick, decorate, decoration={snake,amplitude=.4mm,segment length=2mm}] (-6,0.5) -- (-4,0.5);
\node at (-8.2,3.0) {\scalebox{0.6}{$(0,z_1)$}};
\node at (-8.2,1.5) {\scalebox{0.6}{$(0,z_2)$}};
\node at (-8.2,-1.5) {\scalebox{0.6}{$(0,z_3)$}};
\node at (-8.2,-3.5) {\scalebox{0.6}{$(0,z_4)$}};
\node at (11,0) {$\cdots\cdots$};
\end{tikzpicture}
\caption{
An illustration of the expansion for the fourth moment in \eqref{eq:MNh2}.
Solid dots are assigned weight $\bbE[\xi^{\hash}]$ with $\hash$ being the number of renewal
sequences visiting the lattice site, while circles are assigned weight $1$ and arise from the Chapman-
Kolmogorov decomposition
of $q_{a,b}(x,y)$ at an intermediate time. Curly lines
between sites $(a,x), (b,y)$ (together with the solid dots at both ends) represent
$U_{\tilde N}(b-a,y-y)$ as in \eqref{U-op} and \eqref{U-diagram}, while solid lines between
sites (either solid dots or circles) $(a,x), (b,y)$ are assigned weight $q_{a, b}(x, y)$.
As an illustration of the expansion in
\eqref{eq:Mbd1}, we see the sequence
of operators $\sfP^{*,I_1}=Q^{*,I_1} U^{I_1}, \sfP^{I_1,I_2}
=Q^{I_1,I_2}U^{I_2}, \sfP^{I_2,I_3}=Q^{I_2,I_3},
\sfP^{I_3,I_4}=Q^{I_3,I_4}, \sfP^{I_4,I_5}=Q^{I_4,I_5}$, with $|I_1|=3, |I_2|=3, |I_3|=2, |I_4|=1, |I_5|=3$.
\label{figure-fourth-GQT}}
\end{figure}
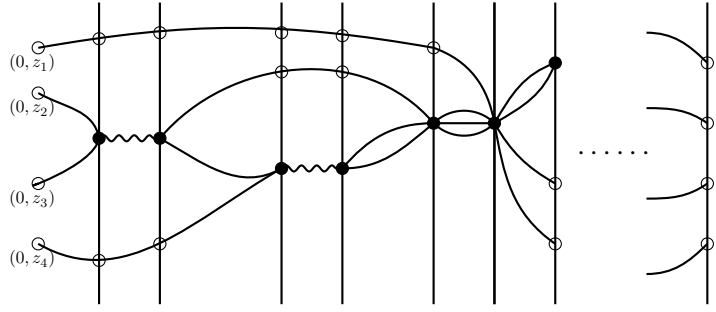


First note that we can bound $|M_{N,\tilde N,h}^{\varphi, \psi}|$ from above by replacing
$\bbE[\xi^{I_i}]$, $\varphi$ and $\psi$ with their absolute values. To simplify notation,
we assume from now on $\bbE[\xi^{I_i}]$, $\varphi$ and $\psi$ are all non-negative.

Next, we bound $\psi$ by $\Vert \psi\Vert_\infty \ind_B$, where $B$ is the ball of radius
$\rho\in [1,\infty]$ containing the support of $\psi$.
We then note that, uniformly in $1\leq n_r\leq N\leq n_{r+1}\leq 2N$ and $\by_r$,
we can find $C>0$ such that
\begin{equation}\label{eq:QB}
	Q^N_{N-n_r}(\by_r, \Vert \psi\Vert_\infty \ind_B)
	\leq C \, Q^N_{n_{r+1}-n_r}(\by_r, \Vert \psi\Vert_\infty \ind_B).
\end{equation}
Recalling the definition of $q_{0,t}^N(y, \Vert \psi\Vert_\infty \ind_B)$ from \eqref{eq:qNpsi}, this bound
follows readily from the observation that, given that a random walk starting from $y$
reaches $\sqrt{N}B$ at time $N-n_r$, the probability of being inside $\sqrt{N}B$ at time
$n_{r+1}-n_r \in [0, 2N]$ is uniformly bounded away from $0$.\footnote{For $3N/2\leq n_{r+1}\leq 2N$,
the inequality \eqref{eq:QB} holds for much more general $B$ than balls of radius $\rho\geq 1$, and
Theorem~\ref{th:mom} can be extended accordingly. But we use balls for simplicity.}
Therefore we can sum the r.h.s.\ of \eqref{eq:MNh2} over $N\leq n_{r+1}\leq 2N$
and then divide by $N$ to get
\begin{equation}\label{eq:MNh3}
\begin{aligned}
	|M_{N,\tilde N,h}^{\varphi, \psi}| \leq  \frac{C\Vert \psi\Vert_\infty^h }{N^{h+1}}
	& \sum_{r=1}^\infty \sum_{\substack{1\leq n_1<\cdots <n_r\leq n_{r+1}\leq 2 N \\
	I_1, \ldots, I_r \vdash\{1, \ldots, h\}, \, m_i:=|I_i|<h \\
	 \by_1, \ldots, \by_r \in (\Z^2)^h}}  Q^N_{n_1}(\varphi, \by_1)
	 \ind_{\{ \by_1\sim I_1\}} \bbE[\xi^{I_1}] \\
	& \quad \times \prod_{i=2}^r Q^{I_{i-1}, I_i}_{n_i-n_{i-1}}(\by_{i-1}, \by_i) \bbE[\xi^{I_i}]
	\times \ind_{\{ \by_r\sim I_r\}} Q^N_{n_{r+1}-n_r}(\by_r, \ind_B).
\end{aligned}
\end{equation}

We first single out consecutive appearances of the same $I$ among $I_1, \ldots, I_r$ with $|I|=h-1$
(that is, $I$ consists of all singletons except for a pair $\{k,l\}$).
Given $I\vdash \{1, \ldots, h\}$ with $|I|=h-1$, for $1\leq s_1 < s_2\leq N$
and $\bz_1, \bz_2 \in (\Z^2)^h$,  define
\begin{equation}\label{eq:sfU}
	\sfU^I_{s_2-s_1,\tilde N}(\bz_1, \bz_2)
	:= \ind_{\{\bz_1, \bz_2\sim I \}}\sum_{r=1}^\infty \bbE[\xi^2]^{r} \!\!\!\!\!\!\!\!\!\! 	
	\sum_{\substack{n_0:=s_1<n_1<\dots <n_r:= s_2\\ \by_i \in (\Z^2)^h, \, \by_0
	:=\bz_1, \, \by_r:=\bz_2}} \prod_{i=1}^r Q^{I, I}_{n_i-n_{i-1}}(\by_{i-1}, \by_i)
\end{equation}
(recall that the moments of $\xi$ depend on $\tilde N$,
see \eqref{eq:ximom})
and define $\sfU^I_{0,\tilde N}(\bz_1, \bz_2) := \ind_{\{\bz_1=\bz_2\sim I\}} $.
If $\{k, l\}$ is the unique partition element of $I$ with cardinality two,
then we can write
\begin{equation}\label{U-op}
	\sfU^I_{s_2-s_1,\tilde N}(\bz_1, \bz_2)
	= U_{\tilde N} (s_2-s_1, \bz_{2,k}-\bz_{1, k})
	\prod_{i \in \{1, \ldots, h\}\backslash\{k, l\}} q_{s_2-s_1}(\bz_{2, i}-\bz_{1, i}),
\end{equation}
where $U_{\tilde N}(n,x)$
was defined in \eqref{eq:U} (with $\sigma_N^2$ therein replaced by $\sigma_{\tilde N}^2$).
In \eqref{eq:MNh3}, we can then contract the consecutive appearances of $I_i = I_{i+1} =\cdots I_j=I$ with
$|I|=h-1$ into a single kernel $\sfU^I_\cdot(\cdot, \cdot)$, so that each $I_i$ with
$|I_i|=h-1$ does not appear twice in a row in the summation in \eqref{eq:MNh3}.

With a slight overload of notation in order to avoid extra symbols, for $\lambda > 0$ we set
\begin{equation}\label{eq:QU}
\begin{aligned}
	& Q_{\lambda,N}^{I, J} (\by, \bz)
	:= \sum_{n=1}^{2N} e^{-\lambda n} \, Q^{I, J}_n(\by, \bz),   \qquad  &\by, \bz \in (\Z^2)^h , \\
	& \sfU_{\lambda, N, \tilde N}^{J}(\by, \bz)
	:= \sum_{n=0}^{2N} e^{-\lambda n} \,
	\sfU^J_{n,\tilde N} (\by, \bz),  &\by, \bz \in (\Z^2)^h,
\end{aligned}
\end{equation}
for partitions $I, J \vdash \{1, \ldots, h\}$, and we finally define, with operator notation,
\begin{equation}\label{eq:P1}
\begin{aligned}
	\sfP^{I, J}_{\lambda, N,\tilde{N}}  :=
	\begin{cases}
	Q_{\lambda,N}^{I, J}  &  \text{if } |J| <h-1 \,, \\
	\rule{0pt}{1.4em} Q_{\lambda,N}^{I, J} \, \sfU_{\lambda,N,\tilde N}^{J}
	& \text{if } |J| = h-1 \,.
	\end{cases}
\end{aligned}
\end{equation}
To lighten notation, we will often omit the dependence of these operators on $N, \tilde N$.

The introduction of the parameter $\lambda$, especially the choice $\lambda =\hat\lambda/N$ we will take later, will be crucial in
decoupling the sum over $n_1, n_2-n_1, \ldots, n_{r+1}-n_r=N-n_r$ in \eqref{eq:MNh2}. Together with the replacement of $n_{r+1}=N$
by averaging $n_{r+1}$ over $[N, 2N]$ in \eqref{eq:MNh3}, this allows us to take Laplace transform and bound the r.h.s.\ of \eqref{eq:MNh2}
in terms of the operators defined in \eqref{eq:QU}-\eqref{eq:P1}. Furthermore, by taking $\hat\lambda$ large, we can extract a logarithmic decay in $\hat\lambda$ from $\sfU_{\lambda, N, \tilde N}^{J}$, see \eqref{eq:norm3}. These ideas were used in \cite{GQT21} in a continuum setting.

To proceed with estimates along these lines, we first obtain an upper bound on  \eqref{eq:MNh3} by inserting
the factor $e^{2\lambda N} e^{-\lambda \sum_{i=1}^{r+1} (n_i-n_{i-1})}\geq 1$,
for $\lambda>0$ to be determined later, and enlarging
the range of summation for each $n_i-n_{i-1}$ to $[1, 2N]$.
Denoting by $I=*$ the partition consisting of singletons (namely, $I=\{1\}\sqcup\{2\}\cdots \sqcup\{h\}$),
we can then rewrite the sum in \eqref{eq:MNh3} and obtain the bound
\begin{equation}\label{eq:Mbd1}
	\big|M_{N,\tilde N, h}^{\varphi, \psi}\big| \leq \frac{Ce^{2\lambda N}
	\Vert \psi \Vert_\infty^h}{N^{h+1}}  \sum_{r=1}^\infty \sum_{I_1, \ldots, I_r} \!\!\!\!
	\big\langle \varphi_N^{\otimes h}, \,  \sfP^{*,I_1}_{\lambda} \,
	\sfP^{I_1, I_2}_{\lambda}
	\cdots \sfP^{I_{r-1}, I_r}_{\lambda} \,
	Q^{I_r, *}_{\lambda} \, \ind^{\otimes h}_{B_N} \big\rangle
	\prod_{i=1}^r \bbE[\xi^{I_i}],
\end{equation}
where the sum is over $r$ partitions $I_1, \cdots, I_r \vdash \{1, \cdots, h\}$ such that $|I_i|\leq h-1$ for
all $1\leq i\leq r$ and there is no consecutive $I_{i-1}=I_i$ with $|I_i|=h-1$; we also applied the
definition of $q^N_{0,n}(\varphi, z)$ and $q^N_{0,n}(z, \psi)$ from \eqref{eq:qNphi}--\eqref{eq:qNpsi},
we set $B_N:=\sqrt{N}B$
and, given $f: \Z^2 \to \R$, for $\by \in (\Z^2)^h$
we define $f^{\otimes h}(\by) := \prod_{i=1}^h f(\by_i)$.

Our bounds will be in terms of the norms of operators acting on the function space $\ell^q((\Z^2)^h)$ for some $q>1$. To allow for $\psi = \ind^{\otimes h}_{B_N}$ in \eqref{eq:Mbd1}, it is necessary to introduce spatial weights, which incidentally will also give bounds on the spatial decay if $\psi$ has compact support and we shift its support toward $\infty$. More precisely, for a function
$w: \R^2\to (0,\infty)$ such that $\log w$ is Lipschitz,  we
define its discretized version $w_N: \Z^2 \to \R$ by \eqref{eq:phiNpsiN}
and we introduce the weighted operators
\begin{equation}\label{eq:hatQU}
\begin{aligned}
	\widehat Q_{\lambda,N}^{I, J} (\by, \bz)
	& :=  \frac{w_N^{\otimes h}(\by)}{w^{\otimes h}_N(\bz)} \, Q_{\lambda,N}^{I, J} (\by, \bz) , \\
	\widehat \sfU_{\lambda, N, \tilde N}^{J}(\by, \bz)
	& :=  \frac{w_N^{\otimes h}(\by)}{w^{\otimes h}_N(\bz)} \,
	\sfU_{\lambda, N,\tilde N}^{J}(\by, \bz) ,
\end{aligned}
\end{equation}
with $\widehat \sfP^{I, J}_{\lambda, N, \tilde N}$  defined
from $\widehat Q$ and $\widehat U$ as in \eqref{eq:P1}.
Given a partition $I\vdash\{1, \ldots, h\}$, denote
\begin{equation}\label{ZI-space}
(\Z^2)^h_I:=\{ \bx\in (\Z^2)^h: \bx \sim I\},
\end{equation}
which is just a copy of $(\Z^2)^{|I|}$ embedded in $(\Z^2)^h$.
Due to the delta function constraints in its definition ($\ind_{\{\bx \sim I, \tilde \bx \sim J\}}$
in \eqref{eq:QIJt}), \emph{we will regard $\widehat Q_{\lambda, N}^{I, J}(\bx, \tilde\bx)$ as an operator
mapping from $\ell^q((\Z^2)^h_J)$ to $\ell^q((\Z^2)^h_I)$ for some $q>1$, and similarly for
$\widehat \sfU_{\lambda, N, \tilde N}^{J}$ and
$\widehat P_{\lambda, N, \tilde N}^{I, J}$}. For $p, q>1$
with $\frac{1}{p}+\frac{1}{q}=1$, by H\"older's inequality,
we can then rewrite the bound \eqref{eq:Mbd1} as
\begin{align}
	\big|M_{N,\tilde N,h}^{\varphi, \psi}\big| & \leq \frac{C e^{2\lambda N}
	\Vert \psi\Vert_\infty^h}{N^{h+1}}  \sum_{r=1}^\infty \sum_{I_1, \ldots, I_r}
	\langle \frac{\varphi_N^{\otimes h}}{w_N^{\otimes h}},
	\widehat \sfP^{*, I_1}_{\lambda} \,
	\widehat \sfP^{I_1, I_2}_{\lambda} \cdots
	\widehat \sfP^{I_{r-1}, I_r}_{\lambda} \, \widehat Q^{I_r, *}_{\lambda}
	\, \ind_{B_N}^{\otimes h}w_N^{\otimes h}\rangle\,  \prod_{i=1}^r \bbE[\xi^{I_i}] \notag \\
	& \leq \frac{C e^{2\lambda N} \Vert \psi\Vert_\infty^h}{N^{h+1}} \sum_{r=1}^\infty
	\sum_{I_1, \ldots, I_r}  \Big\Vert
	\frac{\varphi_N^{\otimes h}}{w_N^{\otimes h}} \Big\Vert_{\ell^p}
	\, \Big\Vert \widehat \sfP^{*, I_1}_{\lambda} \Big\Vert_{\ell^q\to \ell^q} \,
	\Big\Vert \widehat \sfP^{I_1, I_2}_{\lambda}
	\Big\Vert_{\ell^q\to \ell^q} \cdots  \label{eq:Mbd2} \\
	& \qquad\qquad\qquad\qquad\qquad \cdots\, \Big\Vert \widehat \sfP^{I_{r-1}, I_r}_{\lambda}
	\Big\Vert_{\ell^q\to \ell^q} \, \Big\Vert \widehat Q^{I_r, *}_{\lambda} \,
	\Big\Vert_{\ell^q\to \ell^q}
	\, \Big\Vert \ind_{B_N}^{\otimes h}w_N^{\otimes h}\Big\Vert_{\ell^q} \prod_{i=1}^r
	\bbE[\xi^{I_i}], \nonumber
\end{align}
where $\Vert \cdot \Vert_{\ell^p(\bbT)}$
is defined in \eqref{eq:ell}, and given an operator $\mathsf{A}: \ell^q(\bbT) \to \ell^q(\bbT')$,
we set
\begin{equation}
	\Vert \mathsf{A} \Vert_{\ell^q\to \ell^q} \,:= \,
	\sup_{g \nequiv 0} \, \frac{\Vert \mathsf{A} \, g \Vert_{\ell^q(\bbT')}}{\Vert g \Vert_{\ell^q(\bbT)}}
	\,=\,  \sup_{ \Vert f\Vert_{\ell^p(\bbT')}\leq 1,
	\, \Vert g\Vert_{\ell^q(\bbT)}\leq 1} \langle f, \mathsf{A} \, g\rangle.
\end{equation}
In our case $\widehat \sfP^{I, J}_{\lambda},
\widehat Q^{I, J}_{\lambda} : \ell^q((\Z^2)^h_{J}) \to \ell^q((\Z^2)^h_{I})$
(note that for $I=*$ we have $(\Z^2)^h_{I} = (\Z^2)^h$).

We will choose $\lambda:= \hat\lambda/N$ with $\hat\lambda$
large but fixed so that $e^{\lambda N}$ remains bounded. We will show the following.

\begin{proposition}\label{prop:opnorm}
Fix $p, q>1$ with $\frac{1}{p}+\frac{1}{q}=1$,
an integer $h\ge 2$ and $\hat\lambda > 0$.
Then there exists $c = c_{p,q,h,\hat\lambda} < \infty$ such that,
uniformly for partitions $I, J \vdash \{1, \ldots, h\}$ with $1\leq |I|, |J| \leq h-1$
and $I\neq J$ when $|I|=|J|=h-1$,  for large $N \le \tilde N \in\N$
and $\lambda= \frac{\hat\lambda}{N}$ we have
\begin{gather}
	\Big\Vert \widehat Q^{I, J}_{\lambda, N} \Big\Vert_{\ell^q\to \ell^q}
	\leq c \,;\label{eq:norm1} \\
	\Big\Vert \widehat Q^{*,I}_{\lambda,N} \Big\Vert_{\ell^q\to \ell^q}
	\leq c\, N^{\frac{1}{q}} \,,
	\qquad \Big\Vert \widehat Q^{I,*}_{\lambda,N} \Big\Vert_{\ell^q\to \ell^q} \leq
	c\, N^{\frac{1}{p}} \,;
	\label{eq:norm2} \\
	\text{furthermore, for } |I|=h-1, \quad \quad \Big\Vert \widehat \sfU^{I}_{\lambda,
	N,\tilde N}
	\Big\Vert_{\ell^q\to \ell^q} \leq \frac{c}{(\log \hat\lambda \frac{\tilde N}{N})
	\, \sigma_{\tilde N}^2} .
	\label{eq:norm3}
\end{gather}
\end{proposition}
Recall by \eqref{eq:ximom} that
$|\bbE[\xi^{I_i}]|=\sigma_{\tilde N}^2$ if $|I_i|=h-1$,
while $|\bbE[\xi^{I_i}]| \le c' \, \sigma_{\tilde N}^3
= O ((\log \tilde N)^{-\frac32})$ if $|I_i|<h-1$.
Then, by the definition of $\widehat\sfP$ analogous to \eqref{eq:P1},
Proposition~\ref{prop:opnorm} implies that in \eqref{eq:Mbd2}, for each $2\leq i\leq r$, we have
for $N$ sufficiently large
$$
	\bbE\big[\xi^{I_i}\big] \Big\Vert \widehat \sfP^{I_{i-1}, I_i}_{\lambda}
	\Big\Vert_{\ell^q\to \ell^q} \leq \ind_{\{|I_i|=h-1\}}
	\frac{c^2}{(\log \hat\lambda \frac{\tilde N}{N})}
	+ \ind_{\{|I_i|<h-1\}} c \, c' \sigma_{\tilde N}^3
	\leq \frac{c''}{\log \hat\lambda \frac{\tilde N}{N}} \,,
$$
where for $|I_i|=h-1$ with $\widehat\sfP_\lambda^{I_{i-1}, I_i}=\widehat Q_{\lambda}^{I_{i-1}, I_{i}} \, \widehat\sfU_{\lambda}^{I_i}$, we used the fact that $\Vert \widehat Q_{\lambda}^{I_{i-1}, I_{i}} \, \widehat\sfU_{\lambda}^{I_i}\Vert_{\ell^q\to \ell^q} \leq \Vert \widehat Q_{\lambda}^{I_{i-1}, I_{i}} \Vert_{\ell^q\to \ell^q} \Vert \widehat\sfU_{\lambda}^{I_i} \Vert_{\ell^q\to \ell^q}$. Similarly,
$$
	\bbE\big[\xi^{I_1}\big]
	\Big\Vert \widehat \sfP^{*, I_1}_{\lambda} \Big\Vert_{\ell^q\to \ell^q}
	\leq c N^{\frac{1}{q}}
	\bigg(  \frac{c \, \ind_{\{|I_1|=h-1\}}}{\log \hat\lambda \frac{\tilde N}{N}}
	+ \ind_{\{|I_1|<h-1\}} c'
	\sigma_{\tilde N}^3\bigg)
	\leq \frac{c'' }{\log \hat\lambda \frac{\tilde N}{N}}  \, N^{\frac{1}{q}} \,.
$$
Substituting these bounds into \eqref{eq:Mbd2}, bounding the number of choices
for each $I_i \vdash\{1, \ldots, h\}$ by a suitable constant $\mathfrak{c}_h$,
choosing $\hat\lambda$ large such that $\mathfrak{c}_h \frac{c''}{\log \hat \lambda} \le \frac{1}{2}$,
and using the fact that $|I_1|, |I_r|\leq h-1$, we then obtain that, for all $N$ sufficiently large,
\begin{equation}\label{eq:Mbd3}
\begin{aligned}
	\big|M_{N,\tilde N, h}^{\varphi, \psi}\big| & \leq
	\frac{C e^{2\hat\lambda} \Vert \psi\Vert_\infty^h}{N^{h}} \, c \,
	\Big\Vert \frac{\varphi_N}{w_N}\Big\Vert_{\ell^p}^h  \,
	\Vert w_N \ind_{B_N}\Vert^h_{\ell^q} \sum_{r=1}^\infty
	\Big(\frac{\mathfrak{c}_h \, c''}{\log \hat\lambda \frac{\tilde N}{N}}\Big)^r \\
	& \leq \frac{(2 \, C \, c \, \mathfrak{c}_h \, c'')}{N^h \, \log \hat\lambda \frac{\tilde N}{N}}
	\Vert \psi\Vert_\infty^h  \,
	\Big\Vert \frac{\varphi_N}{w_N}
	\, \Big\Vert_{\ell^p}^h \,  \Vert w_N \ind_{B_N}\Vert^h_{\ell^q}\\
	& \leq \frac{\sfC}{N^{h}\log (1+\frac{\tilde N}{N})}
	\, \Vert \psi\Vert_\infty^h  \,
	\Big\Vert \frac{\varphi_N}{w_N}
	\, \Big\Vert_{\ell^p}^h \,  \Vert w_N \ind_{B_N}\Vert^h_{\ell^q}\,,
\end{aligned}
\end{equation}
where the last inequality holds for $\hat\lambda \ge 2$.
This concludes the proof of Theorem~\ref{th:mom}.
\end{proof}

\subsection{Functional Inequalities}
It only remains to prove the bounds in Proposition~\ref{prop:opnorm}. The key ingredient is a Hardy-Littlewood-Sobolev type inequality. First we need the following bound on the Green's function of a random walk on $(\Z^2)^h$.

\begin{lemma}\label{lem:Green}
Given $N\in\N$, $\lambda \geq 0$, an integer $h\geq 2$ and  $\bx, \by \in (\Z^2)^h$, denote
$Q_{\lambda,N}(\bx, \by):= \sum_{n=1}^{2N} e^{-\lambda n} \prod_{i=1}^h q_n(y_i-x_i)$.
Then for some $C\in (0,\infty)$
uniformly in $\lambda$, $N$ and $\bx, \by$,
\begin{equation}\label{eq:QlambN}
Q_{\lambda, N}(\bx, \by) \leq
\left\{
\begin{aligned}
\frac{C}{(1+|\by-\bx|^2)^{h-1}} \qquad & \mbox{for all } \bx, \by \in (\Z^2)^h, \\
\frac{C}{N^{h-1}} e^{-\frac{|\by-\bx|^2}{CN}} \qquad & \mbox{for   } |\bx-\by|> \sqrt{N}.
\end{aligned}
\right.
\end{equation}
\end{lemma}
\begin{proof}
We may assume $\lambda=0$. Note that $\prod_{i=1}^h q_n(\cdot)$ is the transition kernel of a
random walk on $\Z^{2h}$. By the local central limit theorem \cite[Theorem 2.3.11]{LaLi10} and a
Gaussian concentration bound, we have
$$
	\prod_{i=1}^h q_n(y_i-x_i) \leq \frac{C_1}{n^h} e^{-C_2\frac{|\by-\bx|^2}{n}}
$$
for some $C_1, C_2\in (0,\infty)$ uniformly in $n\in\N$ and $\bx, \by\in (\Z^2)^h$. We then have
$$
\begin{aligned}
	Q_{\lambda}(\bx, \by) & \leq C_1 \sum_{n=1}^{2N} n^{-h} e^{-C_2\frac{|\by-\bx|^2}{n}}
	\leq C_1 \min\Bigg\{2, \ \frac{1}{N^{h-1}} \int_0^2 t^{-h} e^{-C_2
	\frac{|\bz_N|^2}{t}} {\rm d}t\Bigg\},
\end{aligned}
$$
where we used a Riemann sum approximation and we set $\bz_N:= (\by-\bx)/\sqrt{N}$.
When $\bz_N=0$, we just use the constant upper bound $Q_{\lambda, N}(\bx, \by)  \leq 2C_1$.
When $\bz_N\neq 0$, we write
\begin{align*}
Q_{\lambda, N}(\bx, \by) \leq \frac{C_1}{N^{h-1}} \int_0^2 t^{-h} e^{-C_2 \frac{|\bz_N|^2}{t}} {\rm d}t & = \frac{C_1}{(C_2N|\bz_N|^2)^{h-1}} \int_{\frac{1}{2}C_2|\bz_N|^2}^\infty \tau^{h-2} e^{-\tau} {\rm d}\tau,
\end{align*}
where $N|\bz_N|^2= |\by-\bx|^2$, while the integral is bounded uniformly in $\bz_N$ and can be bounded by $C_3 e^{-C_4|\bz_N|^2}$ when $|\bz_N|\geq 1$. The bound \eqref{eq:QlambN} then follows.
\end{proof}
\medskip

The following crucial lemma proves a
Hardy-Littlewood-Sobolev type inequality. This generalizes
an inequality  of Dell'Antonio-Figari-Teta in \cite{DFT94} (see Lemma 3.1
and inequalities (3.1) and (3.5) therein)
which played an important role  in \cite{GQT21}
for moment bounds with $L^2$ test functions and initial conditions.

\begin{lemma}\label{HLSineq}
Fix $p, q>1$ with $\frac{1}{p}+\frac{1}{q}=1$ and an integer $h\ge 2$.
Consider partitions $I, J \vdash \{1, \ldots, h\}$ with $1\leq |I|, |J| \leq h-1$,
and $I\neq J$ if $|I|=|J|=h-1$. Recall $(\Z^2)^h_I$ from \eqref{ZI-space} and the
associated function space $\ell^p\big((\Z^2)^h_I\big)$. Let $f\in\ell^p\big((\Z^2)^h_I\big)$ and
$g \in \ell^q\big((\Z^2)^h_J\big)$. Then there exists $C = C_{p,q,h} < \infty$,
independent of $f$ and $g$, such that
\begin{align}\label{ItoJ}
	\sum_{\bx \in(\Z^2)^h_{I}\,,\,\by\in(\Z^2)^h_{J} }
	\frac{f(\bx) g(\by)}{\big(1+\sum_{i=1}^h  |x_i-y_i |^2\big)^{h-1}}
	\leq C \, \big\|f \big\|_{\ell^p} \big\|g \big\|_{\ell^{q}}.
\end{align}
\end{lemma}

\noindent
In \cite{DFT94}, an analogue of \eqref{ItoJ} was proved in the continuum
for the special case $p=q=2$ (i.e., $L^2$ test functions) and
$|I|=|J|=h-1$ with $I\ne J$.
They presented their inequality in Fourier space, but in the $L^2$ case, it is equivalent  to \eqref{ItoJ}
by the Plancherel theorem. Here we work on the lattice, which requires us to also consider partitions
with $|I|<h-1$ or $|J|<h-1$, which cases are not present in the continuum. We also consider test functions in general $\ell^p$-$\ell^q$ spaces
(our proof steps can also be carried out in the continuum to extend
the inequality of \cite{DFT94} to $L^p$-$L^q$ spaces). Instead of working in Fourier space as
in \cite{DFT94}, we will work directly in real space, which allows the extensions mentioned above.

\begin{remark}
The inequality \eqref{ItoJ} is not expected to hold for $|I|=|J|=h-1$ with $I=J$,
because it is exactly the borderline case of the Hardy-Littlewood-Sobolev inequality when it fails,
see \cite[Theorem 4.3]{LiLo01}.
\end{remark}

\begin{proof}[Proof of Lemma \ref{HLSineq}]
We first consider the case $|I|=|J|=h-1$, and $I\neq J$.
Then $I$ and $J$ each contains a partition element
with cardinality 2, say $\{k, l\}$ and $\{m, n\}$ respectively, and $\{k, l\}\neq \{m, n\}$. In particular,
$x_k=x_l$ and $y_m=y_n$ for $\bx\in (\Z^2)^h_{I}, \by\in (\Z^2)^h_{J}$.

Fix any $0<a<\frac{1}{p\vee q}$. We then apply H\"older to bound the left hand side of \eqref{ItoJ} by
\begin{equation}\label{ItoJf}
\begin{aligned}
	& \Bigg( \sum_{\bx\in (\Z^2)^h_{I}, \by\in (\Z^2)^h_{J}}
	\frac{f(\bx)^p}{\big(1+\sum_{i=1}^h  |x_i-y_i |^2\big)^{h-1}}
\cdot \frac{(1+|x_m-x_n|^{2a})^p}{(1+|y_k-y_l|^{2a})^p}
\Bigg)^{1/p} \\
	& \qquad \times \Bigg( \sum_{\bx\in (\Z^2)^h_{I}, \by\in (\Z^2)^h_{J}}
	\frac{g(\by)^q}{\big(1+\sum_{i=1}^h |x_i-y_i|^2\big)^{h-1}}
	\cdot \frac{(1+|y_k-y_l|^{2a})^q}{(1+|x_m-x_n|^{2a})^q} \Bigg)^{1/q}.
\end{aligned}
\end{equation}

We now bound the first factor in \eqref{ItoJf}. Note that since $y_m=y_n$, by the triangle inequality,
\begin{equation}\label{eq:3or4}
	|x_m-y_m|^2+ |x_n-y_n|^2 \geq \frac{|x_m-x_n|^2 + |x_n-y_{n}|^2}{3}.
\end{equation}
Substituting this inequality into $\sum_{i=1}^h  |x_i-y_i|^2$ then bounds the first factor in \eqref{ItoJf} by
\begin{equation}\label{ItoJf2}
	C \Bigg( \sum_{\bx} f(\bx)^p (1+|x_m-x_n|^{2a})^p \sum_{\by}
	\frac{1}{\big(1+|x_m-x_n|^2
	+ \sum_{i\neq m} |x_i-y_i |^2\big)^{h-1}(1+|y_k-y_l|^{2a})^p} \Bigg)^{1/p}.
\end{equation}
Note that for $r > 1$ we can bound
$\sum_{y \in \Z^2} \frac{1}{(s + |y - x|^2)^r} \le \frac{C }{s^{r-1}}$
uniformly for $s \ge 0$ and $x\in\Z^2$.
We can then successively sum over the variables $y_j$ with $j\neq k, l$:
there are $|J|-2$ such variables $y_j$ and they are all present in the sum
$\sum_{i\neq m} |y_i-x_i |^2$ (recall that we sum over $\by\in (\Z^2)^h_J$,
hence $y_m = y_n$ is a single variable), hence we get
\begin{align*}
	&\sum_{\by\in (\Z^2)^h_J} \frac{1}{\big(1+|x_m-x_n|^2
	+ \sum_{i\neq m} |y_i-x_i |^2\big)^{h-1}(1+|y_k-y_l|^{2a})^p} \\
	\leq\ &  C \sum_{y_k, y_l} \frac{1}{\big(1+|x_m-x_n|^2+ |y_k-x_k|^2 + |y_l-x_l|^2\big)^{h+1-|J|}
	(1+|y_k-y_l|^{2a})^p}.
\end{align*}
Note that $x_k=x_l$. Via the change of variables $\tilde y_1 = y_k-y_l$ and $\tilde y_2=y_k+y_l-2x_k$, and the observation that
$|y_k-x_k|^2 + |y_l-x_l|^2= (\tilde y_1^2 +\tilde y_2^2)/2$, we can bound the above sum by
\begin{align*}
	& C \sum_{\tilde y_1, \tilde y_2} \frac{1}{\big(1+|x_m-x_n|^2+ \tilde y_1^2
	+ \tilde y_2^2\big)^{h+1-|J|} (1+|\tilde y_1|^{2a})^p} \\
	\leq\ & C \sum_{\tilde y_1} \frac{1}{\big(1+|x_m-x_n|^2+ \tilde y_1^2\big)^{h-|J|}
	(1+|\tilde y_1|^{2a})^p} \leq \frac{C}{1+|x_n-x_m|^{2(h-1-|J|)+2ap}},
\end{align*}
where the last inequality is obtained by summing separately over  $|\tilde y_1|\leq |x_n-x_m|$ and
$|\tilde y_1|>|x_n-x_m|$, plus the assumption that $ap<1$. Substituting this bound into \eqref{ItoJf2},
since we assume $|J| = h-1$,
we obtain that the first factor in \eqref{ItoJf} is bounded by $C \|f \|_{\ell^p}$.
The second factor in \eqref{ItoJf} can similarly be bounded by $C\|g\|_{\ell^q}$.
This concludes the proof of \eqref{ItoJ}, and hence also \eqref{eq:norm1},
for the case $|I|=|J|=h-1$ and $I\neq J$.

We can adapt the proof to the case $\min\{|I|, |J|\}<h-1$ as follows. If $|I|, |J|<h-1$, then there is no need to introduce the factor $\frac{1+|x_m-x_n|^{2a}}{1+|y_k-y_l|^{2a}}$ and its reciprocal in \eqref{ItoJf} because we already have $\sum_{\by \in (\Z^2)^h_J} \frac{1}{(1+ \sum_{i=1}^h |x_i-y_i|^2)^{h-1}}<\infty$. If $|I|<h-1$ and $|J|=h-1$, then we can still find $k,l$ in the same partition element of $I$, but not the same partition element of $J$. We should then replace the factor $\frac{1+|x_m-x_n|^{2a}}{1+|y_k-y_l|^{2a}}$ and its reciprocal in \eqref{ItoJf} by $\frac{1}{1+|y_k-y_l|^{2a}}$. The rest of the proof is essentially the same.
\end{proof}

\subsection{Proof of Proposition~\ref{prop:opnorm}}
We now prove \eqref{eq:norm1}--\eqref{eq:norm3}.

\begin{proof}[Proof of \eqref{eq:norm1}]
Note that \eqref{eq:norm1} is equivalent to showing (recall \eqref{eq:hatQU})
\begin{equation}\label{eq:hatQnormbd}
	\sum_{\bx\in (\Z^2)^h_{I}, \by\in (\Z^2)^h_{J}}  f(\bx)
	Q^{I,J}_{\lambda} (\bx, \by) \frac{w_N^{\otimes h}(\bx)} {w_N^{\otimes h}(\by)} \, g(\by)
	\leq c \, \Vert f\Vert_{\ell^p} \, \Vert g\Vert_{\ell^q}
\end{equation}
uniformly for all $f\in \ell^p((\Z^2)^h_{I})$ and $g\in \ell^q((\Z^2)^h_{J})$.
To control the effect of the weight $w_N^{\otimes h}$, we split the summation into the regions
$$
	A_N:=\big\{(\bx,\by) \colon |\bx-\by|\leq C_0\sqrt N\big\}
$$
and $A_N^c$ for some $C_0$ to be chosen later.
Note that $\log w(\frac{y}{\sqrt{N}}) - \log w(\frac{x}{\sqrt{N}}) = O(\frac{|x-y|}{\sqrt{N}})$,
because $\log w$ is assumed to be Lipschitz.
Since $w_N: \Z^2 \to \R$ is obtained from
$w: \R^2 \to \R$  by \eqref{eq:phiNpsiN}, we have for all
$\bx, \by \in (\Z^2)^h$
\begin{equation}\label{eq:wLip}
	\frac{w_N^{\otimes h}(\bx)}{w_N^{\otimes h}(\by)}\leq e^{C |\bx - \by|/\sqrt{N}},
\end{equation}
which is bounded by $e^{CC_0}$ in $A_N$. Therefore,
the contribution of this region to the l.h.s.\ of \eqref{eq:hatQnormbd}
is controlled by the following uniform bound, that we prove below:
\begin{equation}\label{eq:AN}
	\sum_{\bx\in (\Z^2)^h_{I}, \by\in (\Z^2)^h_{J}} f(\bx) \,  Q^{I, J}_{\lambda, N}(\bx, \by) \, g(\by)
	\leq c_1 \, \|f\|_{\ell^p} \, \|g\|_{\ell^q}.
\end{equation}
In the region $A_N^c$, since $Q_{\lambda, N}^{I,J} \le Q_{\lambda, N}$
(recall \eqref{eq:Qdef}-\eqref{eq:QIJt}),
we can apply Lemma~\ref{lem:Green} to bound
\begin{align} \label{eq:babao}
	\frac{w_N^{\otimes h}(\bx)}{w_N^{\otimes h}(\by)}  Q_{\lambda, N}(\bx, \by)
	&\leq \frac{C}{N^{h-1}}  \exp \Big\{ -\tfrac{|\bx-\by|^2}{CN} + \tfrac{C|\bx- \by|}{\sqrt N} \Big\}
	\leq \frac{C}{N^{h-1}}  \exp \Big\{ -\tfrac{|\bx- \by|}{\sqrt N} \Big\},
\end{align}
where the last inequality holds for $|\bx- \by|>C_0\sqrt{N}$
with $C_0 := C(C+1)$. Thus
\begin{equation}\label{eq:ANc}
\begin{aligned}
&\sum_{(\bx,\by)\in A_N^c}
f(\bx) \cdot \frac{w_N^{\otimes h}(\bx)}{w_N^{\otimes h}(\by)}\,  Q_{\lambda, N}^{I, J}(\bx, \by) \cdot g(\by)\\
\leq & \quad \frac{C}{N^{h-1}}\,
\sum_{\bx,\by} f(\bx) \, e^{ -\frac{|\bx-\by|}{\sqrt N} } \, g(\by) \\
\leq & \quad  \frac{C}{N^{h-1}}
\Big( \sum_{\bx\in (\Z^2)^h_{I}, \by\in (\Z^2)^h_{J}}  |f(\bx)|^p \, e^{ -\frac{|\bx-\by|}{\sqrt N} } \Big)^{1/p}
\Big( \sum_{\bx\in (\Z^2)^h_{I}, \by\in (\Z^2)^h_{J}}  |g(\by)|^q \, e^{ -\frac{|\bx-\by|}{\sqrt N} } \Big)^{1/q} \\
\leq & \quad  C N^{\tfrac{|J|}{p} + \tfrac{|I|}{q}-(h-1)} \,\,  \|f\|_{\ell^p} \, \|g\|_{\ell^q},
\end{aligned}
\end{equation}
where the prefactor is bounded if $|I|, |J|\leq h-1$. Combined with \eqref{eq:AN},
this implies \eqref{eq:hatQnormbd}.
\medskip
It only remains to prove \eqref{eq:AN}, which follows from Lemmas~\ref{lem:Green}
and \ref{HLSineq}  above.
\end{proof}
\medskip

\begin{proof}[Proof of \eqref{eq:norm2}] It suffices to show that for $p, q>1$
with $\frac{1}{p}+\frac{1}{q}=1$ and for $|I|\le h-1$
\begin{equation}\label{eq:hatQf}
	\sum_{\bx\in (\Z^2)^h_{I}, \by\in (\Z^2)^{h}} f(\bx) \,  Q^{I,*}_{\lambda, N} (\bx, \by)
	\frac{w_N^{\otimes h}(\bx)}{w_N^{\otimes h}(\by)} \, g(\by)
	\leq c N^{\frac{1}{p}} \Vert f\Vert_{\ell^p} \Vert g\Vert_{\ell^q}
\end{equation}
uniformly in $f\in \ell^p((\Z^2)^h_{I})$
and $g\in \ell^q((\Z^2)^h)$, which proves the second relation in \eqref{eq:norm2};
the first relation follows by interchanging $f$ and $g$.
(We recall that $J=*$ denotes the partition of $\{1,\ldots, h\}$ consisting of
$h$ singletons, i.e.\ $J = \{1\}, \{2\}, \ldots, \{h\}$.)

The proof is similar to that of \eqref{eq:hatQnormbd}. When the sum in \eqref{eq:hatQf}
is restricted to $A_N^c$ with $A_N:=\big\{(\bx,\by) \colon |\bx-\by|\leq C_0\sqrt N\big\}$,
the same bound in \eqref{eq:babao}-\eqref{eq:ANc} holds, which gives an upper bound of
\begin{equation}
	C N^{\frac{h}{p} + \frac{|I|}{q} -(h-1)} \Vert f\Vert_{\ell^p} \Vert g \Vert_{\ell^q}
	= C N^{1-\frac{h-|I|}{q}} \Vert f\Vert_{\ell^p} \Vert g\Vert_{\ell^q} \leq C N^{\frac{1}{p}}
	\Vert f\Vert_{\ell^p} \Vert g\Vert_{\ell^q} \,.
\end{equation}

It only remains to bound the sum in \eqref{eq:hatQf} restricted to $A_N$
and show the following analogue of \eqref{ItoJ}:
\begin{align}\label{eq:ItoJ*}
	\sum_{\bx \in(\Z^2)^h_{I}\,,\,\by\in(\Z^2)^{h} \atop |\bx- \by| \leq C_0\sqrt{N}}
	\frac{f(\bx) g(\by)}{\big(1+\sum_{i=1}^h  |x_i-y_i |^2\big)^{h-1}}
	\leq C N^{\frac{1}{p}} \big\|f \big\|_{\ell^p} \big\|g \big\|_{\ell^q} \,.
\end{align}
W.l.o.g., we may assume that $1$ and $2$ belong to the same partition element of $I$,
so that $x_1=x_2$. By H\"older's inequality, we can bound the l.h.s.\ by
\begin{equation}\label{eq:ItoJf*}
\begin{aligned}
	& \Bigg(\sum_{\bx \in(\Z^2)^h_{I}\,,\,\by\in(\Z^2)^{h} \atop |\bx- \by| \leq C_0\sqrt{N}}
	\frac{f^p(\bx)  \big(\log \big(1+\frac{C_0^2 N}{1+|y_1-y_2|^2}\big)\big)^{\frac{p}{q}}}
	{\big(1+\sum_{i=1}^h  |x_i-y_i |^2\big)^{h-1}}\Bigg)^{\frac{1}{p}} \\
	& \quad \times \Bigg(\sum_{\bx \in(\Z^2)^h_{I}\,,\,\by\in(\Z^2)^{h} \atop |\bx- \by|
	\leq C_0\sqrt{N}} \frac{g(\by)^q}{\big(1+\sum_{i=1}^h  |x_i-y_i |^2\big)^{h-1}
	\log\big(1+ \frac{C_0^2 N}{1+|y_1-y_2|^2}\big)}\Bigg)^{\frac{1}{q}}.
\end{aligned}
\end{equation}
In the second factor, since $x_1=x_2$,
we can bound $|x_2-y_2|^2 + |x_1-y_1|^2 \ge \frac{|y_1-y_2|^2 + |x_1-y_1|^2}{3}$
as in \eqref{eq:3or4} to replace $|x_2-y_2|^2$ by $|y_1-y_2|^2$
inside $\sum_{i=1}^h |x_i-y_i |^2$. By the same argument as that following \eqref{ItoJf2}, we can
sum out the variables $x_i$ for $i \ge 3$. Since there are $|I|-1$ such variables in $(\Z^2)^h_{I}$, for $|I|\leq h-1$ we get
\begin{equation} \label{eq:gNh}
\begin{aligned}
	& \sum_{\bx \in (\Z^2)^h_{I}} \frac{1}{\big(1+|y_1-y_2|^2
	+ \sum_{i\neq 2}  |x_i-y_i |^2\big)^{h-1} \log \big(1+\frac{C_0^2 N}{1+|y_1-y_2|^2}\big)}  \\
	\leq\ &  \frac{1}{\log \big(1+\frac{C_0^2 N}{1+|y_1-y_2|^2}\big)}
	\sum_{x_1\in \Z^2\atop |x_1-y_1|\leq C_0\sqrt{N}}
	\frac{1}{(1+|y_1-y_2|^2 + |x_1-y_1|^2)^{h-|I|}} \leq C \,,
\end{aligned}
\end{equation}
where the last bound holds because
$\sum_{x\in\Z^2:\, |x| \le k} \frac{1}{s+|x|^2} \le C \log(1 + \frac{k}{\sqrt{s}})$
uniformly in $k, s \ge 1$, and furthermore $|y_1-y_2|\leq
|y_1-x_1| + |x_2-y_2| \le 2C_0\sqrt{N}$ by $|\bx -\by|\leq C_0\sqrt{N}$.
This implies that the second factor in \eqref{eq:ItoJf*} can be bounded by $C\Vert g\Vert_{\ell^q}$.

For the first factor in \eqref{eq:ItoJf*}, we can first sum over $\by \in (\Z^2)^h$ to bound
\begin{align*}
	\sum_{\by\in (\Z^2)^h} \frac{\big(\log \big(1+\frac{C_0^2 N}{1+|y_1-y_2|^2}
	\big)\big)^{\frac{p}{q}}}{\big(1+\sum_{i=1}^h |x_i-y_i |^2\big)^{h-1}}
	& \leq C \sum_{y_1, y_2\in \Z^2 \atop |y_1-x_1|, |y_2-x_2|\leq C_0\sqrt{N}}
	\frac{\big(\log \big(1+\frac{C_0^2 N}{1+|y_1-y_2|^2}\big)\big)^{\frac{p}{q}}}
	{1+|y_1-x_1|^2 + |y_2-x_2|^2} \,.
\end{align*}
Recall $x_1=x_2$. Let $z_1:=y_1-y_2$  and $z_2:=y_1+y_2-2x_1$, so that
$|z_1|, |z_2|\leq 2C_0\sqrt{N}$ and $|z_1|^2 +|z_2|^2 = 2(|y_1-x_1|^2 +|y_2-x_2|^2)$.
Summing over $z_2$ then leads to the bound
$$
	\sum_{|z_1|\leq 2C_0\sqrt{N}}
	\Big(\log \big(1+\frac{C_0^2 N}{1+|z_1|^2}\big)\Big)^{1+\frac{p}{q}} \leq C N
$$
by a Riemann sum approximation. Therefore the first factor in  \eqref{eq:ItoJf*} can be bounded by $C N^{\frac{1}{p}} \Vert f\Vert_{\ell^p}$. Together with \eqref{eq:gNh}, this implies \eqref{eq:ItoJ*} and concludes the proof of \eqref{eq:norm2}.
\end{proof}
\medskip

\begin{proof}[Proof of \eqref{eq:norm3}]
Note that \eqref{eq:norm3} is equivalent to showing
\begin{equation}\label{eq:hatUnormbd}
	\sum_{\bx, \by\in (\Z^2)^h_I} f(\bx)\,  \sfU^I_{\frac{\hat \lambda}{N}, N, \tilde N}
	(\bx, \by) \, g(\by) \,
	\frac{w_N^{\otimes h}(\bx)}{w_N^{\otimes h}(\by)}
	\leq \frac{c \log \tilde N}{\log \hat\lambda \frac{\tilde N}{N}}
	\Vert f\Vert_{\ell^p} \Vert g\Vert_{\ell^q}
\end{equation}
uniformly in $f\in \ell^p((\Z^2)^h_{I})$ and $g\in \ell^q((\Z^2)^h_{I})$. Without loss of generality,
we may assume $I\vdash\{1, \ldots, h\}$ consists of partition elements $\{1, 2\}, \{3\}, \ldots, \{h\}$,
so that $x_1=x_2$ and $y_1=y_2$.

Recall from \eqref{eq:QU} and \eqref{U-op} that
\begin{equation}\label{eq:sfUnsig}
	\sfU^I_{\frac{\hat \lambda}{N}, N, \tilde N} (\bx, \by)
	= \ind_{\{\bx =\by\}} + \sum_{n=1}^{2N} e^{-\hat\lambda\frac{n}{N}}
	U_{\tilde N}(n, y_1-x_1) \prod_{i=3}^{h}q_n(y_i-x_i),
\end{equation}
where $U_{\tilde N}(n,x)$ is defined in \eqref{eq:U}, with $\sigma_N^2$
replaced by $\sigma_{\tilde N}^2$.
Let us set $T := \frac{\tilde N}{N} \ge 1$ for short.
By \eqref{eq:Usingle}, \eqref{eq:asU1ub} where
$\overline{U}_{\tilde N}=\sigma_{\tilde N}^2 U_{\tilde N}$, and \eqref{eq:Gas}, we have
\begin{align}
	\sum_{\by \in (\Z^2)^h_I} \sfU^I_{\frac{\hat \lambda}{N}, N,\tilde N} (\bx, \by)
	& \leq 1+\sum_{n=1}^{2N} e^{-\hat\lambda\frac{n}{N}} U_{\tilde N}(n) \notag\\
	 \leq &  1+ C\, \frac{\log \tilde N}{\tilde N} \sum_{n=1}^{2\tilde N/T}
	 e^{-\hat\lambda \tfrac{ nT}{\tilde N}} G_\theta\big( \tfrac{n}{\tilde N}\big)
	 \leq 1+ C \log \tilde N \int_0^{\frac2T} e^{-\hat\lambda T t} G_\theta(t) \dd t \notag\\
	\leq & 1+ C e^{-\frac{\hat\lambda T}{2}}\log \tilde N + C \log \tilde N\int_0^{\frac{1}{2}
	\wedge \frac{2}{T}} \frac{e^{-\hat\lambda T t}}{t (\log \frac{1}{t})^2} \dd t
	\leq C \frac{\log \tilde N}{\log \hat\lambda T} , \label{Chatlambda}
\end{align}
where
the last inequality follows by bounding the integral separately over
$\big(0, \frac{1}{(\hat\lambda T)^{1/2}}\big)$ and
$\big(\frac{1}{(\hat\lambda T)^{1/2}}, \frac{1}{2} \wedge \frac{2}{T}\big)$,
with the dominant contribution
coming from the first interval.

On the other hand, for any $C>0$, by \eqref{eq:sfUnsig} we have
\begin{align}
	& \sum_{\by \in (\Z^2)^h_I} \sfU^I_{\frac{\hat \lambda}{N}, N, \tilde N} (\bx, \by)
	e^{C\frac{|\bx- \by|}{\sqrt N}} \notag \\
	\leq\  & 1 + \sum_{n=1}^{2N} e^{-\hat\lambda\frac{n}{N}}
	\sum_{y_1\in \Z^2} U_{\tilde N}(n, y_1-x_1) e^{C\frac{|y_1-x_1|}{\sqrt N}}
	\prod_{i=3}^{h} \Big(\sum_{y_i\in \Z^2} q_n(y_i-x_i) e^{C\frac{|x_i- y_i|}{\sqrt N}}\Big) \notag \\
	\leq \ & 1 + C \sum_{n=1}^{2N} e^{-\hat\lambda\frac{n}{N}} U_{\tilde N}(n)
	\leq C\frac{\log \tilde N}{\log \hat\lambda T} ,
\label{Chatlambda2}
\end{align}
where we applied \eqref{eq:asU3ub} and \eqref{Chatlambda}.

We can now bound the l.h.s.~ of \eqref{eq:hatUnormbd} as follows, recalling \eqref{eq:wLip}:
$$
\begin{aligned}
	&\!\!\!\!\! \sum_{\bx, \by \in (\Z^2)^h_I} \!\!\!\!\!
	f(\bx) \, \sfU^I_{\frac{\hat \lambda}{N}, N}(\bx, \by) \, g(\by)
	\frac{w_N^{\otimes h}(\bx)}{w_N^{\otimes h}(\by)}
	\leq  \sum_{\bx, \by \in (\Z^2)^h_I} \!\!\!\!
	f(\bx) \, \sfU^I_{\frac{\hat \lambda}{N}, N}(\bx, \by) e^{C\frac{|\bx -\by|}{\sqrt N}} \, g(\by)\\
	\leq\ &
	C \Big( \!\!\sum_{\bx, \by \in (\Z^2)^h_I} |f(\bx)|^p \, \sfU^I_{\frac{\hat \lambda}{N}, N}(\bx, \by)
	e^{C\frac{|\bx -\by|}{\sqrt N}}\Big)^{1/p}
	\Big(\sum_{\bx, \by \in (\Z^2)^h_I} |g(\by)|^q \,
	\sfU^I_{\frac{\hat \lambda}{N}, N}(\bx, \by)e^{C\frac{|\bx -\by|}{\sqrt N}}\Big)^{1/q}. \\
	\leq \ & C \frac{\log \tilde N}{\log \hat\lambda T} \Vert f\Vert_{\ell^p} \Vert g\Vert_{\ell^q} \,,
\end{aligned}
$$
where we applied \eqref{Chatlambda2}.
Recalling that $T = \frac{\tilde N}{N}$, this concludes the proof of \eqref{eq:hatUnormbd}.
\end{proof}

\section{Moment estimates for coarse-grained disorder}\label{Sec:MomCoarse}

In this section, we derive second and fourth moment estimates for the coarse-grained
disorder variables $\Theta^{\rm (cg)}_{N, \eps}(\vec\sfi, \vec\sfa)$ defined in \eqref{eq:Theta}.
These will be used later to bound moments of the coarse-grained model
$\mathscr{Z}_{\epsilon}^{(\cg)}(\varphi,\psi|\Theta)$
introduced in Definition~\ref{def:cg}, with $\Theta = \Theta^{\rm (cg)}_{N, \eps}$.

\subsection{Second moment estimates}

We first study the second moment of $\Theta^{\rm (cg)}_{N, \eps}(\vec\sfi, \vec\sfa)$.

\begin{lemma}\label{lem:theta}
For each time-space block $(\vec\sfi, \vec\sfa)=((\sfi, \sfi'), (\sfa, \sfa'))\in \bbT_\eps$
defined in \eqref{eq:iacond}, the coarse-grained disorder variable
$\Theta^{\rm (cg)}_{N, \eps}(\vec\sfi, \vec\sfa)$ as defined in \eqref{eq:Theta} has mean $0$
and its second moment converges to a finite limit
\begin{equation}\label{eq:sigeps}
	\sigma_\eps^2(\vec\sfi, \vec\sfa):=  \lim_{N\to\infty}
	\E\Big[\big(\Theta^{\rm (cg)}_{N, \eps}(\vec\sfi, \vec\sfa)\big)^2\Big],
\end{equation}
see  \eqref{eq:sigia1}-\eqref{eq:sigia2} below. Furthermore, there exist $c, C>0$
independent of $\eps$, $\vec\sfi$, $\vec\sfa$ such that
\begin{equation}\label{eq:siglim1}
\sigma_\eps^2(\vec\sfi, \vec\sfa) \leq \frac{C e^{-c|\vec\sfa|^2/|\vec\sfi|}\ind_{\{|\vec \sfa|\leq M_\eps\sqrt{|\vec \sfi|}\}}}{(\log \frac{1}{\eps})^{1+\ind_{\{|\vec\sfa|>0\}\cup\{|\vec\sfi|\geq 2\}}}|\vec\sfi|^2}.
\end{equation}
\end{lemma}

\begin{proof} {\bf (I) Random walk representation.} We first express $ \E\big[\big(\Theta^{\rm (cg)}_{N, \eps}(\vec\sfi, \vec\sfa)\big)^2\big]$
in terms of the time-space renewal $(\tau^{(N)}_\cdot, S^{(N)}_\cdot)$ defined in \eqref{eq:tauS}. First consider the case $|\vec\sfi|=1$, i.e., $\sfi=\sfi'$. Recall from \eqref{eq:Theta}, \eqref{eq:Zpolydiff} and \eqref{eq:Xdiff} that
\begin{equation}\notag
\begin{aligned}
& \Theta^{\rm (cg)}_{N, \eps}(\vec\sfi, \vec\sfa)  = \frac{2}{\epsilon N}
	\Bigg\{  \!\!\!\!\!\!\!\! \sum_{\substack{ d\in \cT_{\epsilon N}(\sfi) \\ x\in \cS_{\eps N}(\sfa)\cap \cS_{\eps N}(\sfa')}} \!\!\!\!\!\!\!\!\!\! \ind_{(d,x)\in \Z^3_{\rm even}}\xi_N(d, x) + \!\!\!\!\!\!\!\!\!\!\!\!\!\! \sum_{\substack{ d<f\in \cT_{\epsilon N}(\sfi) \\ x\in \cS_{\eps N}(\sfa), \, y\in \cS_{\eps N}(\sfa')}}
	 \!\!\!\!\!\!\!\!\!\! \ind_{(d,x)\in \Z^3_{\rm even}} q_{d, f}(x, y) \xi_N(d, x)\xi_N(f, y) \\
	& \ \ + \sum_{r=1}^{\infty} \!\!\!\!\! \sum_{\substack{d:=n_0< f:=n_{r+1}\in \cT_{\epsilon N}(\sfi) \\ x:=z_0\in \cS_{\eps N}(\sfa), \, y:=z_{r+1}\in \cS_{\eps N}(\sfa')}}  \!\!\!\!\!\!\!\!\!\!\!\!\!\! \ind_{(d,x)\in \Z^3_{\rm even}}\!\!\!\!\!\!\!\!\!\!\!\!
	\sum_{\substack{d < n_1 < \ldots < n_r < f  \\ z_1, \ldots, z_r \,\in\, \bigcup_{|\tilde\sfa-\sfa| \le M_\epsilon} \cS_{\eps N}(\tilde\sfa)}} \!\!\!\!\!\!\!\!\!\!
	\xi_N(d, x)  \prod_{j=1}^{r+1} q_{n_{j-1}, n_j}(z_{j-1},z_j)
	\, \xi_N(n_j,z_j)
	\Bigg\},
\end{aligned}
\end{equation}
where we note that the terms are uncorrelated because $\xi_N(\cdot, \cdot)$ are independent centred random variables with mean $0$ and variance $\sigma_N^2$ (recall \eqref{eq:xi} and \eqref{eq:sigma}). Therefore
\begin{align*}
& \E\Big[\Theta^{\rm (cg)}_{N, \eps}(\vec\sfi, \vec\sfa)^2\Big] =  \frac{4}{(\epsilon N)^2}
	\Bigg\{  \!\!\!\!\!\! \sum_{\substack{ d\in \cT_{\epsilon N}(\sfi) \\ x\in \cS_{\eps N}(\sfa)\cap \cS_{\eps N}(\sfa')}}  \!\!\!\!\!\!\!\!\!\! \ind_{(d,x)\in \Z^3_{\rm even}} \sigma_N^2 + \!\!\!\!\!\! \sum_{\substack{ d<f\in \cT_{\epsilon N}(\sfi) \\ x\in \cS_{\eps N}(\sfa), \, y\in \cS_{\eps N}(\sfa')}} \!\!\!\!\!\!\!\!\!\! \ind_{(d,x)\in \Z^3_{\rm even}} \sigma_N^4 q^2_{d, f}(x, y) 	 \\
	&\qquad  + \sum_{r=1}^{\infty}   \sigma_N^{2(r+1)} \!\!\!\! \!\!\!\! \!\!\!\!
	\sum_{\substack{d:=n_0< f:=n_{r+1}\in \cT_{\epsilon N}(\sfi) \\ x:=z_0\in \cS_{\eps N}(\sfa), \, y:=z_{r+1}\in \cS_{\eps N}(\sfa')}}
	 \!\!\!\!\!\!\!\!\!\! \ind_{(d,x)\in \Z^3_{\rm even}} \!\!\!\!\!\!\!\!\!\! \sum_{\substack{d < n_1 < \ldots < n_r < f  \\ z_1, \ldots, z_r \,\in\, \bigcup_{|\tilde\sfa -\sfa| \le M_\epsilon} \cS_{\eps N}(\tilde\sfa)}}
	\prod_{j=1}^{r+1} q^2_{n_{j-1}, n_j}(z_{j-1},z_j)
	\Bigg\}.
\end{align*}
Note that this sum admits a representation in terms of the space-time random walk $(\tau^{(N)}_k, S^{(N)}_k)_{k\geq 0}$ defined in \eqref{eq:tauS}, namely,
\begin{equation}
\begin{aligned}
& \E\Big[\Theta^{\rm (cg)}_{N, \eps}(\vec\sfi, \vec\sfa)^2\Big] = 2\sigma_N^2 \sum_{k=0}^\infty (\sigma_N^2 R_n)^k \\
& \quad \times \P^{N,\eps}_{\sfi,\sfa}
\Bigg( \frac{S^{(N)}_i}{\sqrt{N}} \in \bigcup_{|\tilde \sfa -\sfa|\leq M_\eps} \cS_\eps(\tilde \sfa)\ \forall\, 1\leq i<k; \ \Bigg(\frac{\tau_k^{(N)}}{N}, \frac{S^{(N)}_k}{\sqrt{N}}\Bigg) \in \cT_\eps(\sfi)\times \cS_\eps(\sfa')\Bigg),
\end{aligned}
\end{equation}
where $\P^{N, \eps}_{\sfi,\sfa}$ denotes probability for $(\tau^{(N)}_k, S^{(N)}_k)_{k\geq 0}$ with $(\tau^{(N)}_0, S^{(N)}_0)$ sampled uniformly from $\cT_{\eps N}(\sfi)\times \cS_{\eps N}(\sfa)\cap \Z^3_{\rm even}$. Changing variable $k=s \log N$, using $\sigma_N^2 R_N=1+ (\vartheta+o(1))/\log N$, and applying Lemma~\ref{lem:Levyconv} on the convergence of $(\tau^{(N)}_k, S^{(N)}_k)_{k\geq 0}$ to a L\'evy process $\bsY_\cdot=(Y_\cdot, V_\cdot)$, we find that the sum above converges to the Riemann integral
\begin{equation}\label{eq:sigia1}
\sigma_\eps^2(\vec\sfi, \vec\sfa)= 2\pi \int_0^\infty e^{\vartheta s} \P^\eps_{\sfi,\sfa}\Big(V_u \in \!\!\! \bigcup_{|\tilde \sfa-\sfa|\leq M_\eps} \cS_\eps(\tilde \sfa)\ \forall\, u\in (0,s); \ (Y_s, V_s) \in \cT_\eps(\sfi)\times \cS_\eps(\sfa') \Big) \, {\rm d}s,
\end{equation}
where $\P^\eps_{\sfi,\sfa}$ denotes the law of the L\'evy process $\bsY_u=(Y_u, V_u)$ with $\bsY_0$ sampled uniformly from $\cT_\eps(\sfi)\times \cS_\eps(\sfa)$.

For $|\vec\sfi|\geq 2$, $\Theta_{N, \eps}(\vec\sfi, \vec\sfa)$ is defined in \eqref{eq:Theta}.
The same argument as for the case $|\vec\sfi|=1$ gives
\begin{align}
& \sigma_\eps^2(\vec\sfi, \vec\sfa) =  \sum_{\substack{\rule{0pt}{0.8em}\sfb: \, |\sfb - \sfa| \le M_\epsilon \\
	\sfb': \, |\sfb' - \sfa'| \le M_\epsilon, \, |\sfb' - \sfb| \le M_\epsilon
	\sqrt{|\sfi'-\sfi|}}} 2\pi \int_0^\infty {\rm d}s \, e^{\vartheta s}    \notag  \\
&\ \times \P^\eps_{\sfi,\sfa}\Big(\ \exists\, t\in (0,s) \ s.t.\ (Y_{t^-}, V_{t^-}) \in \cT_\eps(\sfi)\times \cS_\eps(\sfb), \ (Y_t, V_t) \in \cT_\eps(\sfi')\times \cS_\eps(\sfb');  \label{eq:sigia2} \\
&\quad \forall\, u\in (0,t): \, V_u \in \!\!\! \bigcup_{|\tilde \sfa-\sfa|\leq M_\eps} \!\!\! \cS_\eps(\tilde \sfa); \
\forall\, v\in (t,s): \, V_v \in \!\!\! \bigcup_{|\tilde \sfb-\sfb'|\leq M_\eps} \!\!\! \cS_\eps(\tilde \sfb); \
(Y_s, V_s) \in \cT_\eps(\sfi')\times \cS_\eps(\sfa') \Big). \notag
\end{align}
Here $t$ is the time $(Y_\cdot, V_\cdot)$ jumps from $\cT_\eps(\sfi)\times \cS_\eps(\sfb)$ to $\cT_\eps(\sfi')\times \cS_\eps(\sfb')$.
\medskip

\noindent
{\bf (II) Proof of \eqref{eq:siglim1}.} First consider the case $|\vec\sfi|=1$. By translation invariance, we may assume $\sfi=\sfi'=1$ and $\sfa=0$. First note that
\begin{equation}\label{eq:sigiabd0}
\begin{aligned}
	\sigma_\eps^2(\vec\sfi, \vec\sfa) & \leq 2\pi \int_0^\infty e^{\vartheta s}
	\P^\eps_{0,0}\Big((Y_s, V_s) \in \cT_\eps(0)\times \cS_\eps(\sfa') \Big) \, {\rm d}s \\
	& = \eps^{-2} 2\pi \iint_{\substack{0<s<t<\eps \\ x\in \cS_\eps(0), y\in \cS_\eps(\sfa')}}
	G_\theta(t-s) \, g_{\frac{t-s}{4}}(y-x) {\rm d}s{\rm d}t {\rm d}x {\rm d}y,
\end{aligned}
\end{equation}
where $G_\theta(t-s,y-x) := G_\theta(t-s) \, g_{\frac{t-s}{4}}(y-x)$ is the weighted Green's function defined for the L\'evy process $\bsY_s$, see \eqref{eq:G2} and \eqref{eq:G3}.

When $\sfa'=0$, we can relax the domain of integration in \eqref{eq:sigiabd0},
use standard bounds on the Gaussian kernel $g$, and set $u:=t-s$ to obtain
\begin{equation}
\begin{aligned}
\sigma_\eps^2(\vec\sfi, \vec\sfa)  & \leq  2\pi \!\!\! \iint_{u\in (0, \eps)} \!\!\!
G_\theta(u) g_{\frac{u}{4}}(y) {\rm d}u {\rm d}y
\leq C \int_0^\eps G_\theta(u) {\rm d}u
\leq \frac{C}{\log \frac{1}{\eps}},
\end{aligned}
\end{equation}
where we applied the asymptotics for $\int_0^\eps G_\theta$ in \eqref{eq:G5}.

When $\sfa'\neq 0$, the bound for $\sigma_\eps^2(\vec\sfi, \vec\sfa)$ can be improved with an extra factor of $\frac{e^{-c|\sfa'|^2}}{\log \frac{1}{\eps}}$. Indeed, using polar coordinates (with respect to the $|\cdot|_\infty$ norm) for $x\in \cS_\eps(0)$,  we have
\begin{align}
\sigma_\eps^2(\vec\sfi, \vec\sfa) & \leq \eps^{-2} 2 \pi  \iint_{\substack{0<s<t<\eps \\ x\in \cS_\eps(0), y\in \cS_\eps(\sfa')}} G_\theta(t-s) \, g_{\frac{t-s}{4}}(y-x) {\rm d}s{\rm d}t {\rm d}x {\rm d}y \notag \\
& \leq
\eps^{-2} 2\pi \iint_{ 0<s<t<\eps} {\rm d}s{\rm d}t \, G_\theta(t-s)
\int\limits_{0<r<1} 2 \eps r {\rm d}r \int\limits_{(|\sfa'|-r)\sqrt{\eps} \leq |z| \leq (|\sfa'|+2)\sqrt{\eps}}
\!\!\!\!\!\! g_{\frac{t-s}{4}}(z)  {\rm d}z. \label{eq:sigiabd1}
\end{align}
If $|\sfa'|\geq 2$, then we can use \eqref{eq:Gas} to bound the right hand side of \eqref{eq:sigiabd1} by
\begin{align*}
& \eps^{-1} 2\pi \iint_{ 0<s<t<\eps} {\rm d}s{\rm d}t \, G_\theta(t-s)
\int\limits_{|\sfa'|-1 \leq |w|}
\!\!\!\!\!\! g_{\frac{t-s}{4\eps}}(w)  {\rm d}w \\
\leq \ & \eps^{-1} 2 \pi \iint_{ 0<s<t<\eps}  G_\theta(t-s) e^{- \frac{|\sfa'|^2}{2(t-s)}\eps} \, {\rm d}s{\rm d}t
\leq  4\pi \int_0^1  \frac{1}{u(\log \frac{1}{u} +\log \frac{1}{\eps})^2} e^{- \frac{|\sfa'|^2}{2u}} \, {\rm d}u \\
\leq \ & \frac{4\pi}{(\log \frac{1}{\eps})^2} \int_0^1 \frac{1}{u} e^{- \frac{|\sfa'|^2}{2u}} \, {\rm d}u
= \frac{4\pi}{(\log \frac{1}{\eps})^2} \int_{|\sfa'|^2}^\infty \frac{1}{v} e^{-\frac{v}{2}}{\rm d}v
\leq \frac{4\pi e^{-c|\sfa'|^2}}{(\log \frac{1}{\eps})^2}.
\end{align*}
If $1\leq |\sfa'|<2$, then we can bound the right hand side of \eqref{eq:sigiabd1} by
\begin{align*}
& \eps^{-1} 2\pi \iint_{ 0<s<t<\eps} {\rm d}s{\rm d}t \, G_\theta(t-s) \int_0^1 2r{\rm d}r
\int\limits_{1-r \leq |w|}
\!\!\!\!\!\! g_{\frac{t-s}{4\eps}}(w)  {\rm d}w \\
\leq \ & 2 \pi \int\limits_{ 0<u<\eps}  G_\theta(u) \int_0^1 2r e^{-\frac{(1-r)^2}{u}\eps} \, {\rm d}r  \, {\rm d}u
\leq  \ 4\pi \int\limits_{ 0<u<\eps}  G_\theta(u) \int_0^1 e^{-\frac{r^2}{u}\eps} \, {\rm d}r  \, {\rm d}u  \\
\leq \ & C \int_0^1 \frac{1}{v(\log \frac{1}{v} +\log \frac{1}{\eps})^2} \int_0^1 e^{-\frac{r^2}{v}} \, {\rm d}r  \, {\rm d}v \\
\leq \ & C \int_0^1 \frac{1}{\sqrt{v}(\log \frac{1}{v} +\log \frac{1}{\eps})^2} \int_0^\infty e^{-s^2} \, {\rm d}s  \, {\rm d}v
 \leq  \frac{C}{(\log \frac{1}{\eps})^2} \leq \frac{C e^{-c|\sfa'|^2}}{(\log \frac{1}{\eps})^2}.
\end{align*}
This concludes the proof of the upper bound in \eqref{eq:siglim1}.

We now bound $\sigma_\eps^2(\vec\sfi, \vec\sfa)$ for the case $|\vec\sfi|=2$. By relaxing all the constraints in \eqref{eq:sigia2} except $(Y_s, V_s) \in \cT_\eps(\sfi')\times \cS_\eps(\sfa')$, we note that except for a change of constants, the bound in \eqref{eq:siglim1} for $|\vec\sfi|=1$ also applies in this case. In particular, the bound in \eqref{eq:siglim1} holds for $|\vec\sfi|=2$ and $|\vec\sfa|\neq 0$. For $|\vec\sfi|=2$ and $|\vec\sfa|=0$, let us assume for simplicity that $\sfi=1$, $\sfi'=2$, and $\sfa=0$. Again, relaxing all constraints in \eqref{eq:sigia2} except the constraint on $(Y_s, V_s)$, we have
\begin{equation}
\begin{aligned}
\sigma_\eps^2(\vec\sfi, \vec\sfa)  & \leq  C \eps^{-2} \!\!\!\!\!\!\!\! \iint_{\substack{\rule{0pt}{0.8em} 0<s<\eps<t<2\eps \\ \rule{0pt}{0.8em} x, y\in \cS_\eps(0)}} G_\theta(t-s) g_{\frac{t-s}{4}}(y-x)  {\rm d}x {\rm d}y {\rm d}s {\rm d}t\\
& \leq C \eps^{-1} \iint_{0<s<\eps<t<2\eps} \!\!\!\! G_\theta(t-s) {\rm d}s {\rm d}t \leq C \int_0^\eps u G_\theta(u) {\rm d}u \\
& \leq C \int_0^\eps \frac{1}{(\log \frac{1}{u})^2} {\rm d}u \leq \frac{C}{(\log \frac{1}{\eps})^2}.
\end{aligned}
\end{equation}
The upper bound in \eqref{eq:siglim1} also holds.

We now consider the case $|\vec\sfi|\geq 3$. We first ignore the constraints on $V_r$ for $r\in (0,t)\cup (t,s)$ in \eqref{eq:sigia2}. Using the weighted Green's function $G_\theta$ and the L\'evy measure $\frac{\ind_{(0,1)}}{t} g_{t/4} {\rm d}t {\rm d}x$ for the L\'evy process $\bsY_s=(Y_s, V_s)$ (see \cite[Section 2]{CSZ19a}), we obtain the bound
\begin{align}
\sigma_\eps^2(\vec\sfi, \vec\sfa) \leq\ &   \!\!\!\!\!\!\!\!\!
\sum_{\substack{\rule{0pt}{0.8em}\sfb: \, |\sfb| \le M_\epsilon \\
	\sfb': \, |\sfb' - \sfa'| \le M_\epsilon, \, |\sfb' - \sfb| \le M_\epsilon
	\sqrt{|\vec \sfi|}}}
	\!\!\!\!\!\!\!\!\!
	C \eps^{-2} \iint_{\substack{\rule{0pt}{0.8em} 0<s<t<\eps \\ \rule{0pt}{0.8em} x\in \cS_\eps(0), y\in \cS_\eps(\sfb)}}
	\iint_{\substack{\rule{0pt}{0.8em} (\sfi'-1)\eps<s'<t'<\sfi'\eps \\ \rule{0pt}{0.8em} x'\in \cS_\eps(b'), y'\in \cS_\eps(\sfa')}}
	\!\!\!\!\!\! {\rm d}x {\rm d}y {\rm d}x' {\rm d}y' {\rm d}s{\rm d}t{\rm d}s' {\rm d}t' \notag  \\
	& \qquad \qquad \qquad G_\theta(t-s) g_{\frac{t-s}{4}}(y-x) \cdot \frac{g_{\frac{s'-t}{4}}(x'-y)}{s'-t}\cdot G_\theta(t'-s') \cdot g_{\frac{t'-s'}{4}}(y'-x'). \notag
\end{align}
Since $|\vec\sfi|=\sfi'-\sfi+1\geq 3$, we can bound $\frac{1}{s'-t}\leq \frac{1}{(|\vec\sfi|-2)\eps}\leq \frac{3}{|\vec\sfi|\eps}$ to obtain

\begin{align}
\sigma_\eps^2(\vec\sfi, \vec\sfa)	\leq\ & \frac{C}{|\vec\sfi|\eps^3} \!\!\!\!\!\!\!\!\! \iint_{\substack{\rule{0pt}{0.8em} 0<s<t<\eps \\ \rule{0pt}{0.8em}  (\sfi'-1)\eps<s'<t'<\sfi'\eps}}
	\iint_{\substack{\rule{0pt}{0.8em} x\in  \cS_\eps(0) \\ \rule{0pt}{0.8em} y'\in  \cS_\eps(\sfa')}}
	G_\theta(t-s) G_\theta(t'-s') g_{\frac{t'-s}{4}}(y'-x)
	\,  {\rm d}s{\rm d}t{\rm d}s' {\rm d}t' {\rm d}x {\rm d}y'  \notag \\
	\leq \ & \frac{C}{|\vec\sfi|\eps^3}\  \Big(\int_0^\eps G_\theta(u) {\rm d}u\Big)^2 \!\!\! \iint_{\substack{\rule{0pt}{0.8em} 0<s<\eps \\ \rule{0pt}{0.8em} (\sfi'-1)\eps <t'<\sfi' \eps}} \iint_{\substack{\rule{0pt}{0.8em} x\in  \cS_\eps(0) \\ \rule{0pt}{0.8em} y'\in  \cS_\eps(\sfa')}} g_{\frac{t'-s}{4}}(y'-x) 	\,  {\rm d}s {\rm d}t' {\rm d}x {\rm d}y' \notag \\
	\leq \ & \frac{C}{(\log \frac{1}{\eps})^2}  \frac{e^{-c|\vec\sfa|^2/|\vec\sfi|}}{|\vec\sfi|^2},	\label{eq:sigia6}
\end{align}
where we first relaxed the constraints on $\sfb$ and $\sfb'$, then successively integrated out $y$, $x'$, $s'$, and $t$ and applied \eqref{eq:G5}, while in the last inequality, we applied a uniform bound on the heat kernel $g_{\frac{t'-s}{4}}(y'-x)$. This concludes the proof of the upper bound in \eqref{eq:siglim1}.
\end{proof}

\subsection{Fourth moment estimates}

We next study the fourth moment. 

\begin{lemma}\label{Theta4th} Let $\Theta^{\rm (cg)}_{N, \eps}(\vec\sfi, \vec\sfa)$ be defined as in \eqref{eq:Theta}, with $(\vec\sfi, \vec\sfa)=((\sfi, \sfi'), (\sfa, \sfa'))\in \bbT_\eps$ defined in
\eqref{eq:iacond}. There exist $c, C\in (0,\infty)$ uniform in $(\vec\sfi, \vec\sfa)$, such that for all $\eps>0$ sufficiently small,
\begin{equation}\label{eq:Theta4.1}
\limsup_{N\to \infty} \, \bbE \big[ \Theta^{(\cg)}_{N, \epsilon}(\vec\sfi, \vec\sfa)^4  \big] < \frac{C e^{-c|\vec\sfa|/\sqrt{|\vec\sfi|}}\ind_{\{|\vec\sfa|\leq M_\eps \sqrt{|\vec\sfi|}\}}}{\log \frac{1}{\eps}}.
\end{equation}
\end{lemma}
\begin{proof}
We first prove \eqref{eq:Theta4.1} for $|\vec\sfi|=\sfi'-\sfi+1\leq 2$. Consider a time-space block $(\vec\sfi, \vec\sfa)=((\sfi, \sfi'), (\sfa, \sfa'))$ with $|\vec \sfa| \leq M_\eps \sqrt{|\vec\sfi|}$
and assume without loss of generality that $\sfi=\sfi'=1$ and $\sfa=0$. The case $|\vec\sfi|=2$ is similar (just replace $\eps$ by $2\eps$). We will compare $\Theta^{(\cg)}_{N, \epsilon}(\vec1, \vec\sfa)$ with an averaged partition function so that Theorem~\ref{th:mom} can be applied.

Let us recall the polynomial chaos expansion of $\Theta^{(\cg)}_{N, \epsilon}(\vec\sfi, \vec\sfa)$ from \eqref{eq:Theta}
$$
\Theta^{(\cg)}_{N, \epsilon}(\vec1, \vec\sfa) :=
	\displaystyle \frac{2}{\epsilon N} \,
	\sum_{\substack{(d,x) \in \cB_{\epsilon N}(1, 0)\\
	(f,y) \in \cB_{\epsilon N}(1, \sfa')\\\text{with }d \le f}}
	X_{d,f}^{(\diff)}(x,y),
$$
which is essentially an average of point-to-point partition functions
with average over $(d,x)$ and $(f,y)$ in the bulk instead of through boundary conditions
at time $0$ and $\eps N$ respectively. To compare with an averaged partition function as
in Theorem~\ref{th:mom}, we replace $\Theta^{(\cg)}_{N, \epsilon}(\vec 1, \vec\sfa)$ by
\begin{equation}\label{eq:Theta4bd1}
\Theta := \displaystyle \frac{2}{\epsilon N} \!\!\!\!
	\sum_{z_1, z_2\in \Z^2} \sum_{\substack{(d,x) \in \cB_{\epsilon N}(1, 0)\\
	(f,y) \in \cB_{\epsilon N}(1, \sfa')\\\text{with }d \le f}}  \!\!\!\!\!\!\!\!  \ind_{\cS_{\eps N}(0)}(z_1) q_{0, d}(z_1, x)
	X_{d,f}^{(\diff)}(x,y) q_{f, \eps N}(y, z_2) \ind_{\cS_{\eps N}(\sfa')}(z_2),
\end{equation}
where $\cS_{\eps N}(\sfa)=((\sfa-(1,1))\sqrt{\epsilon N}, \sfa\sqrt{\epsilon N}]$ and we note that uniformly in $(d,x)\in \cB_{\epsilon N}(1, 0)$ and $(f,y) \in \cB_{\epsilon N}(1, \sfa')$, we have
$$
\sum_{z_1, z_2\in \Z^2} \ind_{\cS_{\eps N}(0)}(z_1) q_{0, d}(z_1, x)  q_{f, \eps N}(y, z_2) \ind_{\cS_{\eps N}(\sfa')}(z_2) \geq C >0.
$$
Therefore $\E[\Theta^{(\cg)}_{N, \epsilon}(\vec\sfi, \vec\sfa)^4] \leq C \E[\Theta^4]$ because in the expansion for the fourth moment, all terms are non-negative if we assume $\E[\xi_N^3]\geq 0$, which we may assume from now on since our bounds are in terms of $|\E[\xi_N^k]|$ for $1\leq k\leq 4$ (see the proof of Theorem~\ref{th:mom}). In the definition of $\Theta$, we can further remove the constraint on $y$ and the summation constraints in the definition of $X_{d,f}^{(\diff)}(x,y)$ in \eqref{eq:Zpolydiff}--\eqref{eq:Xdiff}, which gives the centred partition function
$$
\cZ_{\eps N}^{\beta_N}(\varphi, \psi) - \E[\cZ_{\eps N}^{\beta_N}(\varphi, \psi)]
$$
as defined in Theorem~\ref{th:mom}, with $\varphi(x)=\ind_{\cS_1(0)}(x)$ and $\psi(x)=\ind_{\cS_1(\sfa')}(x)$. Therefore for $N$ large, we have
\begin{equation}\label{eq:Theta4bd2}
\begin{aligned}
\E\big[\Theta^{(\cg)}_{N, \epsilon}(\vec\sfi, \vec\sfa)^4\big]  & \leq C \big[(\cZ_{\eps N}^{\beta_N}(\varphi, \psi) - \E[\cZ_{\eps N}^{\beta_N}(\varphi, \psi)])^4\big]  \\
& \leq \frac{C}{\log \frac{1}{\eps}} \Big\Vert \frac{\varphi}{w}\Big\Vert_2^4 \, \Vert \psi\Vert_\infty^4 \Vert w \ind_{\cS_1(\sfa')}\Vert_2^4 \leq \frac{Ce^{-|\sfa'-\sfa|}}{\log \frac{1}{\eps}},
\end{aligned}
\end{equation}
where we applied Theorem~\ref{th:mom}
with $N$ set to $\eps N$, $T=1/\eps$, $p=q=2$, $h=4$, and $w(x) = e^{-|x|}$. This proves \eqref{eq:Theta4.1} for $|\vec\sfi|\leq 2$.

We now consider the case $|\vec\sfi|\geq 3$. Recall the definition of $\Theta^{\rm (cg)}_{N, \eps}(\vec\sfi, \vec\sfa)$ from \eqref{eq:Theta}, we can
rewrite it as
\begin{equation}\label{eq:Theta2t0}
\Theta^{\rm (cg)}_{N, \eps}(\vec \sfi, \vec \sfa) = \sum_{\sfb, \sfb'}
\Theta^{\rm (cg)}_{N, \eps}(\vec \sfi, (\sfa, \sfb), (\sfb', \sfa')),
\end{equation}
where
\begin{equation}\label{eq:Theta2texp0}
\begin{aligned}
& \Theta^{\rm (cg)}_{N, \eps}(\vec\sfi, (\sfa, \sfb), (\sfb', \sfa')) \\
=\ &
\frac{2}{\epsilon N} \!\!\!
	\sum_{\substack{(d,x) \in \cB_{\epsilon N}(\sfi, \sfa)\\
	(f',y') \in \cB_{\epsilon N}(\sfi', \sfa')}}
	\sum_{\substack{(f,y) \in \cB_{\epsilon N}(\sfi, \sfb)\\
	(d',x') \in \cB_{\epsilon N}(\sfi', \sfb')\\   d \le f, \, d' \le f'}}
	 X^{(\diff)}_{d, f}(x, y) \; q_{f, d'}(y, x') \; X^{(\diff)}_{d',f'}(x',y').
\end{aligned}
\end{equation}
For each $(\sfb, \sfb')$, because $\sfi'-\sfi\geq 2$, we can apply Lemma~\ref{lem:ker} (with $m=\eps N$) to bound
$$
	q_{f, d'}(y, x')\leq \frac{C}{ \eps N |\vec \sfi|}
	\, e^{-c|\sfb-\sfb'|^2/|\vec \sfi|}
$$
uniformly in $(f,y) \in \cB_{\eps N}(\sfi, \sfb)$ and $(d', x') \in \cB_{\eps N}(\sfi', \sfb')$. We can make this replacement in the bound for the fourth moment to obtain
\begin{align}\label{eq:Theta4bd3}
\bbE \big[ \Theta^{(\cg)}_{N, \epsilon}(\vec\sfi, \vec\sfa)^4  \big]
\leq  C\, \bbE\Big[ \Big(\sum_{\sfb, \sfb'} \Theta^{(\cg)}_{N, \epsilon}((\sfi, \sfi), (\sfa, \sfb))
\frac{e^{-c|\sfb-\sfb'|^2/|\vec\sfi|}}{|\vec\sfi|}
\Theta^{(\cg)}_{N, \epsilon}((\sfi', \sfi'), (\sfb', \sfa')) \Big)^4 \Big].
\end{align}
By triangle inequality, we can split the sum over $\sfb, \sfb'$ into three parts (with overlaps): (1) $|\sfb-\sfa|\geq |\vec\sfa|/3$; (2) $|\sfb'-\sfa'|\geq |\vec\sfa|/3$; $|\sfb'-\sfb|\geq |\vec\sfa|/3$.
It suffices to bound the fourth moment of each part.

For part (1), we can bound
\begin{equation}\label{eq:Theta4bd4}
\begin{aligned}
& \bbE\Big[ \Big(\sum_{\sfb, \sfb' \atop |\sfb-\sfa|\geq |\vec\sfa|/3} \Theta^{(\cg)}_{N, \epsilon}((\sfi, \sfi), (\sfa, \sfb))
\frac{e^{-c|\sfb-\sfb'|^2/|\vec\sfi|}}{|\vec\sfi|}
\Theta^{(\cg)}_{N, \epsilon}((\sfi', \sfi'), (\sfb', \sfa')) \Big)^4 \Big] \\
\leq\ & \bbE\Big[ \Big(\sum_{\sfb: |\sfb-\sfa|\geq |\vec\sfa|/3} \Theta^{(\cg)}_{N, \epsilon}((\sfi, \sfi), (\sfa, \sfb))\Big)^4\Big] \bbE\Big[ \Big(\sum_{\sfb'}\Theta^{(\cg)}_{N, \epsilon}((\sfi', \sfi'), (\sfb', \sfa')) \Big)^4 \Big],
\end{aligned}
\end{equation}
where the inequality can be justified if we first expand the power and take expectation and then bound
$\frac{e^{-c|\sfb-\sfb'|^2/|\vec\sfi|}}{|\vec\sfi|}<1$; we also used the independence of
$\Theta^{(\cg)}_{N, \epsilon}((\sfi, \sfi), \cdot)$ and $\Theta^{(\cg)}_{N, \epsilon}((\sfi', \sfi'), \cdot)$.
For the first factor in \eqref{eq:Theta4bd4}, we can expand the power and bound
\begin{align*}
	& \bbE\Big[ \Big(\sum_{|\sfb-\sfa|\geq |\vec\sfa|/3}
	\Theta^{(\cg)}_{N, \epsilon}((\sfi, \sfi), (\sfa, \sfb)) \Big)^4\Big]
	= \sum_{\substack{|\sfb_i-\sfa|\geq |\vec\sfa|/3\\ \mbox{\tiny{for} }
	1\leq i\leq 4}} \bbE\Big[ \prod_{i=1}^4 \Theta^{(\cg)}_{N, \epsilon}((\sfi, \sfi), (\sfa, \sfb_i)) \Big] \\
	\leq \ &\sum_{\substack{|\sfb_i-\sfa|\geq |\vec\sfa|/3\\ \mbox{\tiny{for} } 1\leq i\leq 4}}
	\prod_{i=1}^4 \bbE\Big[\Theta^{(\cg)}_{N, \epsilon}((\sfi, \sfi), (\sfa, \sfb_i))^4 \Big]^{\frac14}
	= \Big( \sum_{|\sfb-\sfa|\geq |\vec\sfa|/3}
	\bbE\Big[\Theta^{(\cg)}_{N, \epsilon}((\sfi, \sfi), (\sfa, \sfb))^4 \Big]^{\frac14}  \Big)^4
	\leq  \frac{C e^{-c|\vec\sfa|}}{\log \frac{1}{\eps}},
\end{align*}
where in the last inequality, we applied the fourth moment bound \eqref{eq:Theta4bd2}
for $\Theta^{(\cg)}_{N, \epsilon}(\vec\sfi, \cdot)$ with $|\vec\sfi|=1$.
The second factor in \eqref{eq:Theta4bd4} can be bounded the same way without the factor
$e^{-c|\vec\sfa|}$. This implies that when the sum in \eqref{eq:Theta4bd3} is restricted to
$|\sfb-\sfa|\geq |\vec\sfa|/3$, we get a fourth moment bound of
$C e^{-c|\vec\sfa|}/(\log\frac{1}{\eps})^2$. The same bound holds if the sum
is restricted to $|\sfb'-\sfa'|\geq |\vec\sfa|/3$.

When the sum in \eqref{eq:Theta4bd3} is restricted to $|\sfb-\sfb'|\geq |\vec\sfa|/3$, we can bound
$\frac{e^{-c|\sfb-\sfb'|^2/|\vec\sfi|}}{|\vec\sfi|}\leq e^{-c|\vec\sfa|^2/|\vec\sfi|}$. The rest of
the calculations is the same as before, which leads to a fourth moment bound of
$C e^{-c|\vec\sfa|^2/|\vec\sfi|}/(\log\frac{1}{\eps})^2$. Combined with the previous estimates,
it is clear that \eqref{eq:Theta4.1} holds. This concludes the proof of Lemma~\ref{Theta4th}.
\end{proof}

\section{Moment estimates for the coarse-grained model}\label{Sec:4MomCoarse}

In this section, we will prove an analogue of Theorem~\ref{th:mom} for the coarse-grained model
(defined in \eqref{eq:Zcg-gen}), that we rewrite for convenience:
\begin{equation}\label{eq:Zcg-gen2}
\begin{aligned}
	\mathscr{Z}_{\epsilon}^{(\cg)}(\varphi,\psi|\Theta)
	 & :=  \frac{1}{2} g_{\frac{1}{2}}(\varphi, \psi) \,  + \, \frac{\epsilon}{2}
	 \sum_{r=1}^{(\log\frac{1}{\epsilon})^2} \sum_{{\sfb}, {\sfc} \in \Z^2} 	
	\sum_{\substack{(\vec\sfi_1, \ldots, \vec\sfi_r) \in \bcA_{\epsilon}^{(\notri)} \\
	(\vec\sfa_1, \ldots, \vec\sfa_r) \in \bcA_{\epsilon; \, b,c}^{(\diff)}}}
	\!\!\!\!\!\!\!\!\!\!\! \varphi_\epsilon({\sfb})  g_{\frac{1}{2}\sfi_1}(\sfa_1 - \sfb)
	\Theta(\vec\sfi_1, \vec\sfa_1)  \\
	& \times
	\Bigg\{ \prod_{j=2}^r  g_{\frac{1}{2}(\sfi_j -\sfi_{j-1}')} (\sfa_j -\sfa_{j-1}')
	\, \Theta(\vec\sfi_j, \vec\sfa_j) \Bigg\}
	\, g_{\frac{1}{2}(\lfloor\tfrac{1}{\epsilon}\rfloor- \sfi_r')}({\sfc}-\sfa_r')
	 \psi_\epsilon({\sfc}) \,,
\end{aligned}
\end{equation}
with coarse-grained disorder variables
$\Theta(\vec\sfi, \vec\sfa):=\Theta^{\rm (cg)}_{N, \eps}(\vec\sfi, \vec\sfa)$
(see \eqref{eq:Theta}) indexed by time-space blocks $(\vec\sfi, \vec\sfa)=((\sfi, \sfi'), (\sfa, \sfa'))$
in the set $\bbT_\eps$ (see \eqref{eq:iacond}), while $\varphi_\epsilon, \psi_\epsilon: \Z^2 \to \R$
are defined by \eqref{eq:phieps-psieps} from $\varphi\in C_c(\R^2)$ and $\psi\in C_b(\R^2)$.

We will prove the following analogue of Theorem~\ref{th:mom} for the $4$-th
moment of the coarse-grained model.
\begin{theorem}\label{th:cgmom}
Let $\mathscr{Z}_{N, \epsilon}(\varphi,\psi)
:=\mathscr{Z}_{\epsilon}^{(\cg)}(\varphi,\psi|\Theta^{\rm (cg)}_{N, \eps})$ be the coarse-grained
model defined above. Further assume that $\Vert\psi\Vert_\infty<\infty$ and $\psi$ is supported on a
ball $B$ (possibly $B=\R^d$). Then for any $p, q \in (1,\infty)$ with $\frac{1}{p}+\frac{1}{q}=1$ and
any $w: \R^2 \to (0,\infty)$ such that $\log w$ is Lipschitz continuous, there exists $C\in (0,\infty)$
such that uniformly in $\eps\in (0,1)$,
\begin{equation}\label{eq:mombdcg}
	\limsup_{N\to\infty} \bbE\Big[\big(\mathscr{Z}_{N, \epsilon}(\varphi,\psi)
	- \bbE[\mathscr{Z}_{N, \epsilon}(\varphi,\psi)]\big)^4\Big] \leq C \eps^{\frac 4p}
	\Big\Vert \frac{\varphi_\eps}{w_\eps}\Big\Vert_{\ell^p(\Z^2)}^4 \, \Vert \psi\Vert_\infty^4
	\Vert w \ind_B\Vert_q^4,
\end{equation}
where $w_\eps:\Z^2 \to \R$ is defined from $w$ by \eqref{eq:phieps-psieps}.
\end{theorem}
\begin{proof}
We will adapt the proof of Theorem~\ref{th:mom} to the current setting. The complication is that the
coarse-grained disorder variables $\Theta^{\rm (cg)}_{N, \eps}(\vec \sfi, \vec \sfa)$ are assigned to
time-space blocks $(\vec \sfi, \vec \sfa)$ instead of individual lattice sites. We will therefore divide the
proof into two parts: first, expand the fourth moment and perform a resummation to bring it into a
similar framework as in the proof of Theorem~\ref{th:mom}; second, give the necessary bounds analogous
to those in the proof of Theorem~\ref{th:mom}.
\medskip

\noindent
{\bf Part I. Expansion.} First, as in \eqref{eq:MNh} in the proof of Theorem~\ref{th:mom}, denote
\begin{align}
\cM^{\varphi, \psi}_{N, \eps} :=& \bbE\Big[\big(\mathscr{Z}_{N, \epsilon}(\varphi,\psi) - \bbE[\mathscr{Z}_{N, \epsilon}(\varphi,\psi)]\big)^4\Big] \notag\\
= & \frac{\eps^4}{16} \,\bbE\Big[\Big(
 \sum_{r=1}^{(\log\frac{1}{\epsilon})^2} \sum_{{\sfb}, {\sfc} \in \Z^2} 	
	\sum_{\substack{(\vec\sfi_1, \ldots, \vec\sfi_r) \in \bcA_{\epsilon}^{(\notri)} \\ (\vec\sfa_1, \ldots, \vec\sfa_r) \in \bcA_{\epsilon; \, b,c}^{(\diff)}}}
	\!\!\!\!\!\!\!\!\!\!\! \varphi_\epsilon({\sfb})  g_{\frac{1}{2}\sfi_1}(\sfa_1 - \sfb)   \Theta^{\rm(cg)}_{N,\eps}(\vec\sfi_1, \vec\sfa_1) \label{eq:Mepsh0}\\
& \qquad \qquad \qquad \times \Big\{ \prod_{j=2}^r
	g_{\frac{1}{2}(\sfi_j -\sfi_{j-1}')} (\sfa_j -\sfa_{j-1}') \Theta^{\rm(cg)}_{N,\eps}(\vec\sfi_j, \vec\sfa_j) \Big\}
	g_{\frac{1}{2}(\frac{1}{\epsilon}- \sfi_r')}(\sfc-\sfa_r') \psi_{\epsilon}(\sfc)  \Big)^4\Big]. \nonumber
\end{align}
By assumption, we have $|\psi_\eps|\leq \Vert\psi\Vert_\infty\ind_{B_\eps}$, where $B_\eps=B/\sqrt{\eps}$. By the same reasoning as in the proof of Theorem~\ref{th:mom} (see the discussion leading to \eqref{eq:MNh3}), we can replace $\psi$ by $\Vert\psi\Vert_\infty\ind_{B_\eps}$ and replace $g_{\frac{1}{2}(\frac{1}{\epsilon}- \sfi_r')}(\cdot)$ by $g_{\frac{1}{2}(\tilde n- \sfi_r')}(\cdot)$ (with $\tilde n$ first summed over $[\eps^{-1}, 2\eps^{-1}]$, then extended to $[1, 2\eps^{-1}]$) to obtain the following bound
\begin{equation}\label{eq:Mepsh}
\begin{aligned}
\cM^{\varphi, \psi}_{N, \eps}   & \leq  C \Vert \psi\Vert_\infty^4\, \eps^5 \sum_{\tilde n=1}^{2/\eps} \bbE\Big[\Big(
\sum_{r=1}^{\infty}  \sum_{{\sfb}, {\sfc} \in \Z^2}
\sum_{\substack{(\vec\sfi_1, \ldots, \vec\sfi_r) \in \bcA_\epsilon^{(\notri)}  \\ (\vec\sfa_1, \ldots, \vec\sfa_r)\in \bcA_{\epsilon; \sfb, \sfc}^{(\diff)}}} \!\!\!\!\!\!\!\!\!\!
	\varphi_\eps(\sfb) g_{\frac{1}{2}\sfi_1}(\sfa_1-\sfb) \Theta^{\rm(cg)}_{N,\eps}(\vec\sfi_1, \vec\sfa_1) \\
	& \quad \qquad \quad \times \Bigg\{ \prod_{j=2}^r
	g_{\frac{1}{2}(\sfi_j -\sfi_{j-1}')} (\sfa_j -\sfa_{j-1}') \Theta^{\rm(cg)}_{N,\eps}(\vec\sfi_j, \vec\sfa_j) \Bigg\}
	g_{\frac{1}{2}(\tilde n- \sfi_r')}(\sfc-\sfa_r') \ind_{B_\epsilon}(\sfc)  \Big)^4\Big].
\end{aligned}
\end{equation}
We then expand the product in \eqref{eq:Mepsh} to obtain the sum over $4$ sequences of
time-space blocks, each time-space block contributing a $\Theta^{\rm (cg)}_{N, \eps}$ variable. Because
we will bound the sum by taking the absolute value of each summand, we can relax the
summation constraint on $r$ to obtain an upper bound. Also note that thanks to the assumption
$\varphi\in C_c(\R^2)$ and the diffusive constraint $\bcA_{\epsilon; \sfb, \sfc}^{(\diff)}$
(see \eqref{eq:bcA2}), we have a sum with finitely many terms, which allows us to pass the
limit $N\to\infty$ inside the sum later. For each $\Theta^{\rm (cg)}_{N, \eps}(\vec \sfi, \vec \sfa)$ with
$|\vec\sfi|=\sfi'-\sfi+1\geq 2$, we further expand it as
\begin{equation}\label{eq:Theta2t}
	\Theta^{\rm (cg)}_{N, \eps}(\vec \sfi, \vec \sfa) = \sum_{\substack{\sfb: \, |\sfb - \sfa|
	\le M_\epsilon \\
	\sfb': \, |\sfb' - \sfa'| \le M_\epsilon  \\ |\sfb' - \sfb| \le M_\epsilon \sqrt{|\vec \sfi'|}}}
	\Theta^{\rm (cg)}_{N, \eps}(\vec \sfi, (\sfa, \sfb), (\sfb', \sfa')),
\end{equation}
where $\Theta^{\rm (cg)}_{N, \eps}(\vec \sfi, (\sfa, \sfb), (\sfb', \sfa'))$ is defined as
$\Theta^{\rm (cg)}_{N, \eps}(\vec \sfi, \vec \sfa)$, except the sum in its definition in \eqref{eq:Theta}
is restricted to a fixed choice of $\vec \sfb:=(\sfb, \sfb')$. The expansion of the product
in \eqref{eq:Mepsh} then gives $4$ sequences of coarse-grained disorder variables
$\Theta^{\rm (cg)}_{N, \eps}$, some of which may visit two distinct mesoscopic time intervals with indices
$\sfi, \sfi'$ due to the expansion in \eqref{eq:Theta2t}. If we record the indices of the visited
mesoscopic time intervals and the mesoscopic spatial boxes of entry and exit in each time interval,
then we obtain $4$ sequences of time-space indices $(\sfi^j_1, \sfa^j_1, \sfb^j_1)$, \ldots,
$(\sfi^j_{r_j}, \sfa^j_{r_j}, \sfb^j_{r_j})$, $1\leq j\leq 4$. We will call
each such sequence a {\em mesoscopic time-space renewal sequence},
or just {\em renewal sequence} (see Figure \ref{CG-figB}).

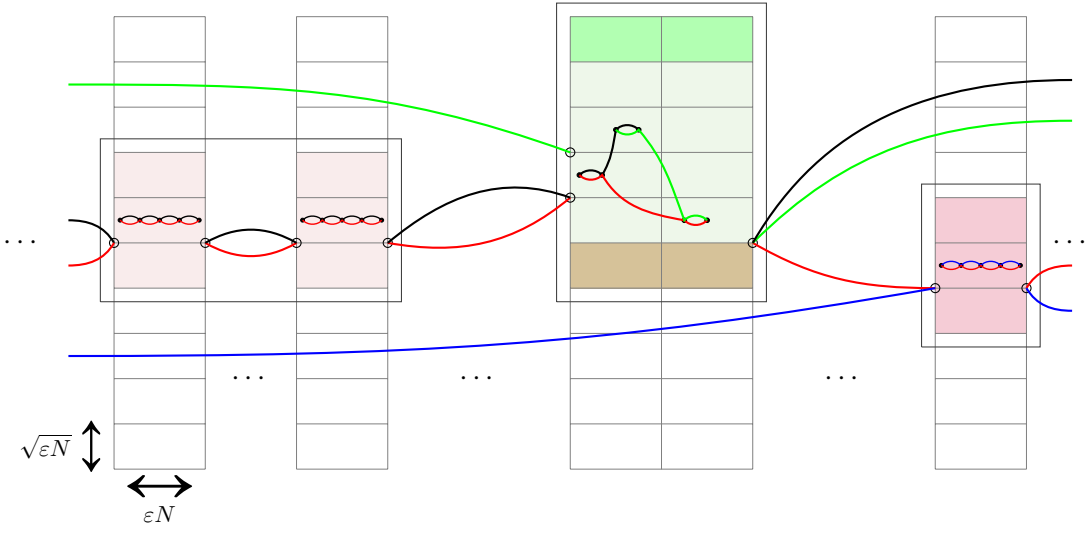
\begin{figure}
\hskip -1.5cm
\begin{tikzpicture}[scale=0.6]
\draw[ystep=1cm, xstep=2,gray,very thin] (0,0) grid (2, 10);
\draw[ystep=1cm, xstep=2,gray,very thin] (4,0) grid (6, 10);
\draw[ystep=1cm, xstep=2,gray,very thin] (10,0) grid (14, 10);
\draw[ystep=1cm, xstep=2,gray,very thin] (18,0) grid (20, 10);
\node at (3,2) {{$\cdots$}}; \node at (8,2) {{$\cdots$}}; \node at (16,2) {{$\cdots$}}; \node at (-2,5) {{$\cdots$}};
\node at (21,5) {{$\cdots$}};
\node at (1,-0.5) {\scalebox{2.5}{$\leftrightarrow$}}; \node at (-0.5, 0.5) {\scalebox{2.0}{$\updownarrow$}};
\node at (1,-1) {\scalebox{0.8}{$\epsilon N$}}; \node at (-1.5, 0.5) {\scalebox{0.8}{$\sqrt{\epsilon N}$}};
\node at (1,5.5) {\scalebox{0.65}{$\overlap$}}; \node at (5,5.5) {\scalebox{0.65}{$\overlap$}};
\node at (19,4.5) {\scalebox{0.65}{$\overlapbr$}};
\draw (0, 5)  circle [radius=0.1]; \draw (2, 5)  circle [radius=0.1];
\draw (4,5) circle [radius=0.1]; \draw (10, 6) circle [radius=0.1];  \draw (10, 7) circle [radius=0.1]; \draw (6, 5) circle [radius=0.1];
\draw (14, 5) circle [radius=0.1]; \draw (18, 4) circle [radius=0.1]; \draw (20, 4) circle [radius=0.1];
\draw[fill] (10.2,6.5)  circle [radius=0.05]; \draw[fill] (10.7,6.5)  circle [radius=0.05]; \draw[fill] (11,7.5)  circle [radius=0.05];
\draw[fill] (11.5,7.5)  circle [radius=0.05]; \draw[fill] (12.5,5.5)  circle [radius=0.05]; \draw[fill] (13,5.5)  circle [radius=0.05];
\draw[black!75] (-0.3,3.7) -- (6.3, 3.7) -- (6.3, 7.3) -- (-0.3, 7.3 ) --(-0.3, 3.7);
\draw[black!75] (9.7,3.7) -- (14.3, 3.7) -- (14.3, 10.3) -- (9.7, 10.3 ) --(9.7, 3.7);
\draw[black!75] (17.7,2.7) -- (20.3, 2.7) -- (20.3, 6.3) -- (17.7, 6.3 ) --(17.7, 2.7);
\draw[thick, green] (10,7) to [out=160, in =0] (-1,8.5);
\draw[thick, red] (2,5) to [out=-30,in=-150] (4,5); \draw[thick] (2,5) to [out=30,in=150] (4,5);
\draw[thick] (-1,5.5) to [out=0,in=120] (0,5); \draw[thick, red]  (-1,4.5) to [out=0,in=-120] (0,5) ;
\draw[thick, red] (6, 5) to [out=-10, in=220] (10, 6); \draw[thick] (6, 5) to [out=40, in=160] (10, 6);
\draw[thick, red] (14, 5) to [out=-30,in=180] (18, 4);
\draw[thick, green] (14, 5) to [out=45,in=180] (21, 7.7);
\draw[thick, blue] (-1, 2.5) to [out=0,in=190] (18,4);
\draw[thick] (14,5) to [out=60, in=180] (21, 8.6);
\draw[thick, red] (20, 4) to [out=60, in=180] (21, 4.5);
\draw[thick, blue] (20, 4) to [out=-60, in=180] (21, 3.5);
\draw[thick] (10.2,6.5)  to  [out=45, in = 135] (10.7,6.5)  to [out=60, in=-100]  (11,7.5) to [out=45, in=135]  (11.5,7.5) ;
\draw[thick, red] (10.2,6.5)  to  [out=-45, in = -135] (10.7,6.5)  to [out=-60, in=170]  (12.5,5.5) to [out=-45, in=-135]  (13,5.5) ;
\draw[thick, green] (11,7.5) to [out=-45, in=-135]  (11.5,7.5) to [out=-50, in=110] (12.5,5.5) to [out=45, in=135]  (13,5.5) ;
 \begin{pgfonlayer}{background}
 \fill[color=gray!50!red!10] (0,4) rectangle (2,7);
 \fill[color=gray!50!red!10] (4,4) rectangle (6,7);
 \fill[color=gray!70!red!50!green!10] (10,5) rectangle (14,9);
 \fill[color=blue!20!red!20] (18,3) rectangle (20,6);
 \fill[color=green!30] (10,9) rectangle (14,10);
  \fill[color=red!60!green!40] (10,4) rectangle (14,5);
 \end{pgfonlayer}
\end{tikzpicture}
\caption{A depiction of the expansion of the fourth moment of the coarse-grained model that
satisfy conditions (a)-(g). There are four time-space renewal sequences, each depicted in a different colour.
Different sequences visit different mesoscopic boxes $\cB_{\epsilon N}(\sfi,\sfa)$, but each visited box must
be visited by at least two sequences to give a non-zero contribution.  The first two time strips are visited
by the two
renewal sequences coloured black and red, which match in the disorder they sample. These two strips
are grouped
together as a  {\it block of type} $\sfU$. The third and fourth time strips are visited by three
renewal sequences, coloured
black, red, and green, which form a {\em block of type} $\sfV$ and its width cannot exceed $4$.
Within this $\sfV$ block,
the spatial boxes of entry by the three renewals are all within distance $2M_\eps$ of each other, 2
of which match exactly.
The last time strip is only visited by two renewal sequences, coloured blue and red, which also forms a
$\sfU$ block. }
\label{CG-figB}
\end{figure}

We will rearrange the expansion of \eqref{eq:Mepsh} as follows:
\begin{itemize}
\item[(1)] Sum over the set $\bigcup_{j=1}^4 \{\sfi^j_1, \ldots, \sfi^j_{r_j}\}=:\{n_1, \ldots, n_r\}$.
\item[(2)] For each time index $n_i$, $1\leq i\leq r$, sum over the set of indices
$J_i\subset \{1, \ldots, 4\}$, which determine the renewal sequences that visit time interval $n_i$.
\item[(3)] For each $j\in J_i$, i.e., a renewal sequence that visits time interval $n_i$,  sum over the
indices $(\sfa^j_i, \sfb^j_i)$ that determine the spatial boxes of entry and exit in that time interval.
\end{itemize}
Given a choice of these summation indices, the summand contains a product of coarse grained
disorder variables of either the form
$\Theta^{\rm (cg)}_{N, \eps}((n_i, n_i), (\sfa_i, \sfb_i))=: \Theta^{\rm (cg)}_{N, \eps}(n_i; \sfa_i, \sfb_i)$
or the form $\Theta^{\rm (cg)}_{N, \eps}((n_i, n_j), (\sfa_i, \sfb_i), (\sfa_j, \sfb_j))$, connected by
heat kernels $g_{\frac{1}{2}(n_k -n_l)}\, (\sfa_k -\sfb_l)$. For such a product to have non-zero
expectation, we have the following constraints (see Figure~\ref{CG-figB}):
\begin{itemize}
\item[(a)] $|J_i|\geq 2$ for each $1\leq i\leq r$;

\item[(b)] If $|J_i|=2$, say $J_i=\{k, l\}\subset \{1, \ldots, 4\}$, then we must have $\sfa^k_i=\sfa^l_i$
and $\sfb^k_i=\sfb^l_i$;

\item[(c)] If $|J_i|\geq 3$, then for each sequence $k \in J_i$, there must be another sequence $l\in J_i$
such that $|\sfa^k_i-\sfa^l_i|\leq 2M_\eps$, where $M_\eps=\log\log \frac{1}{\eps}$
as in \eqref{eq:KM}.
\end{itemize}
If (c) is violated, then by the spatial constraint in the definition of $\Theta^{\rm (cg)}_{N, \eps}$ in
\eqref{eq:Theta}, there will be a coarse-grained disorder variable
visiting time interval $n_i$, which is independent of all other coarse-grained disorder variables in the
product, and hence leads to zero expectation.

The summation constraints $\bcA_\epsilon^{(\notri)}$ and $\bcA_{\epsilon; \sfb, \sfc}^{(\diff)}$
(see \eqref{eq:bcA} and \eqref{eq:bcA2}) in the definition \eqref{eq:Zcg-gen2} of the coarse-grained
model implies the following additional constraints on the summation indices $r$, $(n_i)_{1\leq i\leq r}$,
$(J_i)_{1\leq i\leq r}$, and $(\sfa^j_i, \sfb^j_i)_{1\leq i\leq r, j\in J_i}$:
\begin{itemize}
\item[(d)] For all $1\leq i\leq r$ and each renewal sequence with index
$j\in J_i$, $|\sfb^j_i-\sfa^j_i|\leq M_\eps$;

\item[(e)] For $1\leq j\leq 4$ and $1\leq i_1 <i_2\leq r$, if $j\in J_{i_1} \cap J_{i_2}$ and $j\notin J_i$ for
all $i_1<i<i_2$ (namely renewal sequence $j$ visits the mesoscopic time intervals with indices $n_{i_1}$
and $n_{i_2}$, but nothing in between), then $|\sfa^j_{i_2}-\sfb^j_{i_1}|\leq M_\eps \sqrt{n_{i_2}-n_{i_1}}$;

\item[(f)] $K_\eps \leq n_1 <n_2 <\cdots <n_r \leq \frac{1}{\eps} - K_\eps$,
where $K_\eps=(\log \frac{1}{\eps})^6$ as in \eqref{eq:KM};

\item[(g)] $(n_1, \ldots, n_r)$ can be partitioned into consecutive stretches $\cD_1, \ldots, \cD_m$
such that each $\cD_i$ consists of consecutive integers, with a gap between $\cD_i$ and
$\cD_{i+1}$. Then each $\cD_\cdot=(n_{i}, n_{i+1}=n_i+1, \ldots, n_{j} = n_i + (j-i))$ has width
$n_{j}-n_{i}+1\leq 4$, since $|J_{n_l}|\geq 2$ for $i\leq l \leq j$ (namely the mesoscopic time
interval with index $n_l$ is visited by at least two renewal sequences), and each sequence can visit at
most two mesoscopic time intervals with indices among $n_{i}, n_{i+1}, \ldots, n_{j}$ by the constraint
$\bcA_\epsilon^{(\notri)}$.
\end{itemize}
Conditions (d)-(e) follow from the definitions of $\bcA_\epsilon^{(\diff)}$ in \eqref{eq:bcA2} and $\Theta^{\rm (cg)}_{N, \eps}$ in \eqref{eq:Theta}, while conditions (f)-(g) follow from the definition of $\bcA_\epsilon^{(\notri)}$ in \eqref{eq:bcA}.

To handle the dependency among the coarse-grained disorder variables in the expansion of \eqref{eq:Mepsh}, we perform a further resummation. First partition $(n_1, \ldots, n_r)$ into consecutive stretches $\cD_1, \ldots, \cD_m$ as in (g), so that $\{n_1, \ldots, n_r\} = \bigcup_{i=1}^m \cD_i$. For each $\cD_i$, let $\tilde J_i :=\bigcup_{j \in \cD_i} J_j$, which records which of the 4 renewal sequences visits the stretch $\cD_i$. Next we group together consecutive $\cD_{i_1}, \cD_{i_1+1}, \ldots, \cD_{i_2}$ with the same $\tilde J_i=\{k, l\}$ for some $k\neq l\in \{1, \ldots, 4\}$, and only keep track of $s:=\min \cD_{i_1}$ and $t:=\max\cD_{i_2}$, thus effectively replacing $\bigcup_{i_1\leq i\leq i_2} \cD_i$ by $[s,t]$. This allows us to identify from $(n_1, \ldots, n_r)$ a sequence of disjoint time intervals (which we call blocks) $\cI_i=[s_i, t_i]\cap\N$, $1\leq i\leq k$, each associated with a label set $\cJ_i\subset \{1, \ldots, 4\}$. Some of these intervals arise from $\bigcup_{i_1\leq i\leq i_2} \cD_i$ as above, which are visited by exactly 2 renewal sequences, the rest coincide with the original $\cD_i$'s. We can then rewrite the expansion of \eqref{eq:Mepsh} as follows:
\begin{itemize}
\item[(1')] Sum over integers $K_\eps<s_1\leq t_1 <s_2\leq t_2<\cdots <s_k\leq t_k < \tilde n\leq \frac{2}{\eps}$, with $s_{i+1}-t_i\geq 2$ for each $i$ (recall the summation index $\tilde n$ from \eqref{eq:Mepsh}). Denote $\cI_i:=[s_i, t_i]\cap \N$.
\item[(2')] For each block $\cI_i$, sum over the set of indices $\cJ_i \subset \{1, \ldots, 4\}$ with $|\cJ_i|\geq 2$. If $|\cJ_i|=2$, we call $\cI_i$ a {\em block of type} U  because it leads to contributions similar to $\sfU_s(\bz^1, \bz^2)$ in \eqref{U-op} (see also \eqref{U-diagram}); otherwise we call it a {\em block of type} V. There are no consecutive blocks $\cI_i$, $\cI_{i+1}$ of both type U with the same label set $\cJ_i=\cJ_{i+1}$, and each block $\cI_i$ of type V must have length $|\cI_i|\leq 4$.
See Figure~\ref{CG-figB} for an illustration.
\item[(3')] For each block $\cI_i$ and each renewal sequence $j\in \cJ_i$ that visits block $\cI_i$, sum over time-space indices $(\sigma^j_i, \sfa^j_i)$ and $(\tau^j_i, \sfb^j_i)$ with $s_i\leq \sigma^j_i \leq \tau^j_i \leq t_i$ and $\sfa^j_i, \sfb^j_i\in \Z^2$, which identifies the mesoscopic time-space blocks of entry and exit by the $j$-th renewal sequence in the time interval $\cI_i=[s_i, t_i]\cap\N$.
\end{itemize}
The constraints imposed in (d)-(g) carry over, so we do not repeat them here.

To rewrite the expansion of  \eqref{eq:Mepsh} in a form that fits the framework developed in the proof
of Theorem~\ref{th:mom}, we will carry out the following steps, that we describe below.
\begin{itemize}
\item[\bf (A)] To decouple different blocks, replace each coarse-grained disorder variable
$\Theta^{\rm (cg)}_{N, \eps}$ (it will arise as a summand in \eqref{eq:Theta2t}) that visits two
consecutive blocks $\cI_i$ and $\cI_{i+1}$ by the product of two coarse-grained disorder variables
of the form
$\Theta^{\rm (cg)}_{N, \eps}(\vec \sfi, \vec \sfa)$ with $|\vec \sfi|=1$, joined by a heat kernel.

\item[\bf (B)] Bound the moments of the effective disorder variable
associated with each U block and V block.

\item[\bf (C)] Modification of the heat kernels $g_{\frac{1}{2} (s_{i+1}-t_i)}(\cdot)$ connecting different blocks. In particular, carry out a Chapman-Kolmogorov type decomposition for the heat kernels as in the proof of Theorem~\ref{th:mom}, so that we sum over the spatial locations for all 4 renewal sequences at the beginning and end of each block $\cI_i$.
\end{itemize}
This rewriting will introduce a constant factor for each block, but it will not affect our final result. We now give the details for {\bf (A)-(C)}.

\smallskip

{\bf (A)} Note the technical complication that given a realisation of the summation indices in (1')--(3'), there could be coarse-grained disorder variables
$\Theta^{\rm (cg)}_{N, \eps}((\tau^j_{i_1}, \sigma^j_{i_2}), (\sfa^j_{i_1}, \sfb^j_{i_1}), (\sfa^j_{i_2}, \sfb^j_{i_2}))$ (see \eqref{eq:Theta2t}) that visit two distinct blocks $\cI_{i_1}=[s_{i_1}, t_{i_1}]\cap \N$ and $\cI_{i_2}=[s_{i_2}, t_{i_2}]\cap \N$ for some $i_1<i_2$, due to the contribution from the $j$-th renewal sequence for some $j\in \cJ_{i_1}\cap \cJ_{i_2}$. In particular, $\tau^j_{i_1} \in \cI_{i_1}$  and $\sigma^j_{i_2} \in \cI_{i_2}$. Recall from \eqref{eq:Theta2t}  and \eqref{eq:Theta} that
\begin{equation}\label{eq:Theta2texp}
\begin{aligned}
& \Theta^{\rm (cg)}_{N, \eps}((\tau^j_{i_1}, \sigma^j_{i_2}), (\sfa^j_{i_1}, \sfb^j_{i_1}), (\sfa^j_{i_2}, \sfb^j_{i_2})) \\
=\ &
\frac{2}{\epsilon N} \!\!\!
	\sum_{\substack{(d,x) \in \cB_{\epsilon N}(\tau^j_{i_1}, \sfa^j_{i_1})\\
	(f',y') \in \cB_{\epsilon N}(\sigma^j_{i_2}, \sfb^j_{i_2})}}
	\sum_{\substack{(f,y) \in \cB_{\epsilon N}(\tau^j_{i_1}, \sfb^j_{i_1})\\
	(d',x') \in \cB_{\epsilon N}(\sigma^j_{i_2}, \sfa^j_{i_2})\\   d \le f, \, d' \le f'}}
	 X^{(\diff)}_{d, f}(x, y) \; q_{f, d'}(y, x') \; X^{(\diff)}_{d',f'}(x',y').
\end{aligned}
\end{equation}
Note that by the definition of a U block, if $\cI_{i_1}$ is a block of type U, then we must have $\tau^j_{i_1}=t_{i_1}$, the last time index in the block $\cI_{i_1}$; while if $\cI_{i_1}$ is a block of type V, then we must have $\tau^j_{i_1}\geq s_{i_1}\geq t_{i_1}-3$ because V blocks of length at most $4$. Similarly, if $\cI_{i_2}$ is of type of U, then we must have $\sigma^j_{i_2}=s_{i_2}$, the first time in the block $\cI_{i_2}$;
while if $\cI_{i_2}$ is type V, then we must have $\sigma^j_{i_2}\leq t_{i_2}\leq s_{i_2} + 3$. Therefore $d'-f\leq (s_{i_2}-t_{i_1}+7)\eps N$.
On the other hand,  $\cI_{i_1}$ and $\cI_{i_2}$ are distinct blocks and hence $s_{i_2}-t_{i_1}\geq 2$ and $d'-f\geq \eps N$. We can therefore apply Lemma~\ref{lem:ker} with $m=\eps N$, $n_1=d'-f$, $n_2=s_{i_2}-t_{i_1}$ and $\rho_1=10$ to bound
\begin{equation}\label{eq:q2g1}
\sup_{\substack{(f,y) \in \cB_{\epsilon N}(\tau^j_{i_1}, \sfb^j_{i_1})\\
	(d',x') \in \cB_{\epsilon N}(\sigma^j_{i_2}, \sfa^j_{i_2})}}   q_{f, d'}(y, x')
\leq \frac{C}{\eps N}\,  g_{10(s_{i_2} -t_{i_1})}\, (\sfa^j_{i_2} -\sfb^j_{i_1}).
\end{equation}
Applying this bound in \eqref{eq:Theta2texp} then allows us to make the replacement  (recall the definition of $\Theta^{\rm (cg)}_{N, \eps}$ from \eqref{eq:Theta})
\begin{equation}\label{eq:Th2Th}
\begin{aligned}
&\Theta^{\rm (cg)}_{N, \eps}((\tau^j_{i_1}, \sigma^j_{i_2}), (\sfa^j_{i_1}, \sfb^j_{i_1}), (\sfa^j_{i_2}, \sfb^j_{i_2}))  \\
\stackrel{\leq}{\rightsquigarrow}  \ & C \Theta^{\rm (cg)}_{N, \eps}(\tau^j_{i_1}; \sfa^j_{i_1}, \sfb^j_{i_1})\  g_{10(s_{i_2} -t_{i_1})}\, (\sfa^j_{i_2} -\sfb^j_{i_1}) \ \Theta^{\rm (cg)}_{N, \eps}(\sigma^j_{i_2}; \sfa^j_{i_2}, \sfb^j_{i_2}).
\end{aligned}
\end{equation}
Of course this is not an upper bound since $\Theta^{\rm (cg)}_{N, \eps}$ could be negative. However, when we compute the moment in \eqref{eq:Mepsh}, we end up with products of the moments of $\Theta^{\rm (cg)}_{N, \eps}$'s and its constituent $\xi_N$'s, which are then be bounded by their absolute values. Applying \eqref{eq:q2g1} at this point gives a true upper bound, which has the same effect as making the replacement \eqref{eq:Th2Th} in the expansion before taking expectation, and then compute the moment as in \eqref{eq:Mepsh}.  To keep the notation simple, we will assume this replacement from now on, so that the expansion of \eqref{eq:Mepsh} now contains only $\Theta^{\rm (cg)}_{N, \eps}(\sfi; \sfa, \sfb)$ that visits a single mesoscopic time interval $\cT_{\eps N}(\sfi)$, which simplifies the expansion from \eqref{eq:Mepsh}. The cost is replacing some heat kernel $g_{\frac{\sfi}{2}}(\cdot)$ (more accurately, the associated random walk transition kernel) by $C g_{10\sfi}(\cdot)$ as in \eqref{eq:q2g1}.

\smallskip

{\bf (B)} We now consider a U block $\cI=[s, t]\cap \N$. Assuming w.l.o.g.\ $\cJ=\{1,2\}$ so that only renewal sequences 1 and 2 in the expansion of \eqref{eq:Mepsh} visit block $\cI$. Let $\sfa, \sfb\in \Z^2$ be spatial indices for the mesoscopic boxes of entry and exit in the time interval $\cI$. Then the coarse-grained disorder variables $\Theta^{\rm (cg)}_{N, \eps}$ visited by renewal sequences 1 and 2 in the time interval $\cI$ must match perfectly in order to have non-zero expectation. Taking expectation in \eqref{eq:Mepsh}, each U block in the expansion therein leads to the following quantity analogous to  $\overline{U}_N(n,x)$ defined in \eqref{eq:barU} and \eqref{U-diagram}:
\begin{equation}\label{eq:defUcg}
\overline{U}^{\rm (cg)}_{N, \eps}(t-s, \sfb-\sfa) :=  \!\!\!\!\!\!\!\!\!\!\!\!\!\! \sum_{\substack{
(\vec\sfi_1, \ldots, \vec\sfi_r) \in  \bcA_\epsilon^{(\notri)} \\
(\vec\sfa_1, \ldots, \vec\sfa_r) \in  \bcA_\epsilon^{(\diff)} \\
\sfi_1=s, \sfi_r'=t, \sfa_1=\sfa, \sfa_r'=\sfb
}} \!\!\!\!\!\!\!\!\!\!\!
        \bbE[\Theta^{\rm(cg)}_{N,\eps}(\vec\sfi_1, \vec\sfa_1)^2] \prod_{j=2}^r
	g^2_{\frac{1}{2}(\sfi_j -\sfi_{j-1}')} (\sfa_j -\sfa_{j-1}') \bbE[\Theta^{\rm(cg)}_{N,\eps}(\vec\sfi_j, \vec\sfa_j)^2],
\end{equation}
where $\bcA_\epsilon^{(\notri)}$ and $\bcA_\epsilon^{(\diff)}$ are defined in \eqref{eq:bcAbis} and \eqref{eq:bcA2bis}. Because the sum above is a sum over finitely many terms, by Lemma~\ref{lem:theta}, the following limit exists
\begin{equation}\label{eq:Ucg}
\overline{U}^{\rm(cg)}_{\infty, \eps}(t-s, \sfb-\sfa) :=\lim_{N\to\infty} \overline{U}^{\rm(cg)}_{N, \eps}(t-s, \sfb-\sfa).
\end{equation}

We next consider a V block $\cI=[s,t]\cap\N$ with size $t-s+1\leq 4$. Let $\cJ\subset \{1, 2,3,4\}$ denote the set of renewal sequences that visit $\cI$. To have non-zero expectation in the expansion of \eqref{eq:Mepsh}, we must have $|\cJ|\geq 3$ ($|\cJ|= 2$ would make it a U block instead). Each renewal sequence can visit at most two mesoscopic time intervals with indices in $\cI$. For each renewal sequence $j\in \cJ$ that visits block $\cI$, let $(\sigma^j, \sfa^j)$ and $(\tau^j, \sfb^j)$ be
the indices of the mesoscopic time-space boxes of entry and exit in $\cI$. In the expectation in \eqref{eq:Mepsh}, such a V block then leads to the following factor
\begin{align*}
V^{\rm (cg)}_{N, \eps}((\sigma^j, \tau^j, \sfa^j, \sfb^j)_{j\in\cJ}) & :=
\bbE\Big[\prod_{j\in \cJ} \Theta^{\rm(cg)}_{N, \eps}((\sigma^j, \tau^j), (\sfa^j, \sfb^j))\Big]  \\
& \leq  \prod_{j\in \cJ} \ind_{\{\exists\, j\neq k\in \cJ, |\sfa^j-\sfa^{k}|\leq 6M_\eps\}} \bbE\Big[\Theta^{\rm(cg)}_{N, \eps}((\sigma^j, \tau^j), (\sfa^j, \sfb^j))^4\Big]^{\frac14},
\end{align*}
where the indicators follow from the local dependence of $\Theta^{\rm (cg)}_{N, \eps}$ from its definition in \eqref{eq:Theta}. Applying \eqref{eq:Theta4.1}
with $|\vec \sfi|\leq 4$ (since V blocks have length at most $4$) then gives
\begin{align}
\limsup_{N\to\infty} V^{\rm (cg)}_{N, \eps}((\sigma^j, \tau^j, \sfa^j, \sfb^j)_{j\in\cJ})
& \leq \frac{C}{(\log \frac{1}{\eps})^{\frac{|\cJ|}{4}}} \prod_{j\in \cJ} \ind_{\{\exists\,  k\in \cJ, k\neq j, |\sfa^j-\sfa^{k}|\leq 6M_\eps; \atop |\sfb^j-\sfa^j|\leq M_\eps\}} e^{-|\sfb^j-\sfa^j|}  \notag \\
& =: V^{\rm (cg), \cJ}_{\infty, \eps}(\sfa, \sfb). \label{eq:Vcg}
\end{align}

\smallskip

{\bf (C)} We next modify the heat kernels connecting different blocks $\cI_i$. First, we will contract each V block $\cI_i=[s_i, t_i]\cap \N$ into a block of size $1$. Note that every heat kernel in the expansion connects two different blocks $\cI_{i_1}=[s_{i_1}, t_{i_1}]\cap \N$ and $\cI_{i_2}=[s_{i_2}, t_{i_2}]\cap\N$, $i_1<i_2$, and are of the form $g_{\frac{1}{2}(\sigma_{i_2}-\tau_{i_1})}(\sfa_{i_2}-\sfb_{i_1})$ for some $\tau_{i_1}\in\cI_{i_1}$ and $\sigma_{i_2}\in \cI_{i_2}$ with $|\sfa_{i_2}-\sfb_{i_1}|\leq M_\eps \sqrt{\sigma_{i_2}-\tau_{i_1}}$. The heat kernel $g_{\frac{1}{2}(\sigma_{i_2}-\tau_{i_1})}(\cdot)$ from time $\tau_{i_1}$ to $\sigma_{i_2}$ may jump over multiple blocks of type V. If we contract the time span $[\tau_{i_1}, \sigma_{i_2}]$ of the heat kernel by shrinking each block of type V that intersects $[\tau_{i_1}, \sigma_{i_2}]$ into a block of size 1, and let $u$ denote the length of the reduced time span for the heat kernel, then $u\geq \frac{1}{4}(\sigma_{i_2}-\tau_{i_1})$ since blocks of type V have length at most 4. Therefore
$$
g_{\frac{1}{2}(\sigma_{i_2}-\tau_{i_1})}(\sfa_{i_2}-\sfb_{i_1}) \leq 8 g_{2u}(\sfa_{i_2}-\sfb_{i_1}).
$$
The heat kernels introduced in \eqref{eq:q2g1} are of the form $g_{10(\sigma_{i_2}-\tau_{i_1})}(\sfa_{i_2}-\sfb_{i_1})$ and can similarly be bounded by $8 g_{80u}(\sfa_{i_2}-\sfb_{i_1})$. For consistency, we will further bound $g_{2u}(\cdot) \leq 20 g_{80u}(\cdot)$. This shows that at the cost of introducing a constant factor $C$ for each block $\cI_i$, we can assume from now on that all blocks $\cI_i$ of type V have length 1, namely $s_i=t_i$, and all heat kernels appearing in the expansion (as an upper bound for the expansion of \eqref{eq:Mepsh}) are of the form $g_{80(s_{i_2}-t_{s_{i_1}})}(\cdot)$.

Lastly, we perform a Chapman-Kolmogorov type decomposition for each heat kernel $g_{80(s_{i_2}-t_{i_1})}(\sfb_{i_1}, \sfa_{i_2}) :=g_{80(s_{i_2}-t_{i_1})}(\sfa_{i_2}-\sfb_{i_1})$ at each $s_j, t_j \in (t_{i_1}, s_{i_2})\cap \N$, similar to what was done in the proof of Theorem~\ref{th:mom}. To simplify notation, let $u_0=t_{i_1}, u_1, \ldots, u_{k-1}, u_k=s_{i_2}$, with $u_1, \ldots, u_{k-1}$ being the times at which we want to perform the decomposition. Let $\sfx_0:=\sfb_{i_1}$, $\sfx_k:=\sfa_{i_2}$. Then we can bound
\begin{align*}
g_{80(u_k-u_0)}(\sfx_0, \sfx_k) & = \idotsint\limits_{x_1, \ldots, x_{k-1}\in \R^2}  g_{80(u_1-u_0)}(\sfx_0, x_1)\cdots g_{80(u_k-u_{k-1})}(x_{k-1}, \sfx_k) {\rm d}x_1\cdots {\rm d}x_{k-1} \\
& \leq \sum_{\sfx_1, \ldots, \sfx_{k-1} \in \Z^2} \prod_{i=1}^k \big(2 g_{80(u_i-u_{i-1})}(\sfx_{i-1}, \sfx_i)\big),
\end{align*}
where we have discretized the spatial integral into a sum over the lattice and introduced a factor 2 for each intermediate time $u_i$, $1\leq i<k$, as a crude upper bound.

The steps {\bf (A)-(C)} we have performed so far basically allow
us to bound the expansion \eqref{eq:Mepsh} in a form that is similar to \eqref{eq:MNh3}, and ready to
lead to the analogue of \eqref{eq:Mbd1}. The U blocks we have introduced correspond to
$\sfU^I$ introduced in \eqref{U-op}, while V blocks correspond to the disorder variable $\xi$, which
even after contracting each V block into a block length 1, still has non-trivial spatial dependence. Due to
the heavy notation, we will not write down the analogue of  \eqref{eq:MNh3} here. Instead, we
explain below how the analogues of \eqref{eq:Mbd1} and \eqref{eq:Mbd2} can be derived.

\medskip
\noindent
{\bf Part II. Bounds.} Based on the considerations above, we can write down an upper bound for \eqref{eq:Mepsh} that allows us to adapt the proof of Theorem~\ref{th:mom}. First, we introduce some notation that parallels those in the proof of Theorem~\ref{th:mom}. Similar to \eqref{eq:Qdef}, for $\sfa=(\sfa^j)_{1\leq j\leq 4},  \sfb=(\sfb^j)_{1\leq j\leq 4} \in (\Z^2)^4$ and $\varphi:\Z^2\to\R$, define
\begin{gather*}
\cQ_t(\sfa, \sfb):= \prod_{j=1}^4 g_{80t}(\sfa^j, \sfb^j) , \\
\cQ_t(\varphi, \sfb) := \prod_{j=1}^4 \Big(\sum_{\sfc^j\in \Z^2}  \varphi(\sfc^j) g_{80t}(\sfb^j-\sfc^j)\Big)
, \quad \cQ_t(\sfa, \psi) := \prod_{j=1}^4 \Big(\sum_{\sfc^j\in \Z^2} g_{80t}(\sfc^j-\sfa^j)\psi(\sfc^j)\Big).
\end{gather*}
Similar to \eqref{U-op}, for $\cJ=\{k, l\}\subset \{1, 2,3,4\}$, define (recall $\overline{U}^{{\rm(cg)}}_{\infty, \eps}$ from \eqref{eq:Ucg})
\begin{equation}\label{eq:sfUcg}
\sfU^{\rm(cg), \cJ}_{\infty, \eps} (t; \sfa, \sfb) := \ind_{\{\sfa^k=\sfa^l, \sfb^k=\sfb^l\}}\overline{U}^{\rm(cg)}_{\infty, \eps} (t, \sfb^k-\sfa^k) \prod_{j\notin \cJ} g_{80t}(\sfa^j, \sfb^j),
\end{equation}
with $g_0(\sfa^j, \sfb^j):=\ind_{\{\sfa^j=\sfb^j\}}$.  Similarly, define (recall $\sfV^{{\rm(cg)}, \cJ}_{\infty, \eps}(\sfa, \sfb)$ from \eqref{eq:Vcg})
\begin{equation}\label{eq:sfVcg}
\sfV^{\rm(cg), \cJ}_{\infty, \eps} (\sfa, \sfb) := V^{{\rm(cg)}, \cJ}_{\infty, \eps} (\sfa, \sfb) \prod_{j\notin \cJ} \ind_{\{\sfa^j=\sfb^j \}}.
\end{equation}
To be consistent with the notation in the proof of Theorem~\ref{th:mom}, we will replace $\cJ\subset \{1, 2,3,4\}$, which determines which mesoscopic renewal sequences collect coarse-grained disorder variables at time $t$, by a partition $I\vdash\{1, 2,3,4\}$, which specifies which sequences interact with each other through the coarse-grained disorder variables at time $t$. In particular, in $\sfU^{\rm(cg), \cJ}_{\infty, \eps}$ corresponding to a U block, the associated partition $I$ consists of $\cJ$ and $\{j\}$ for $j\notin \cJ$, so that $|I|=3$. In $V^{\rm(cg), \cJ}_{\infty, \eps}$ corresponding to a V block, if $|\cJ|=3$, then the associated partition $I$ consists of $\cJ$ and $\{j\}$ for $j\notin\cJ$, so $|I|=2$; if $|\cJ|=4$, then the associated partition $I$ is given by the connected components of $\{1,2,3,4\}$ with an edge between $i$ and $j$ whenever $|\sfa^i-\sfa^j|\leq 6M_\eps$, and there can be no singletons in the partition to ensure that $\sfV^{\rm(cg), \cJ}_{\infty, \eps}\neq 0$ (in particular, $|I|=1$ or $2$).

From now on, we write $\sfU^{{\rm(cg)}, I}_{\infty, \eps} (t; \sfa, \sfb)$ and $\sfV^{{\rm(cg)}, I}_{\infty, \eps} (t; \sfa, \sfb)$, replacing $\cJ$ by the associated partition $I\vdash \{1, 2,3,4\}$. Define
\begin{equation}
\sfW^{{\rm(cg)}, I}_{\infty, \eps}(t, \sfa, \sfb) := \ind_{\{|I|=3\}} \sfU^{{\rm(cg)}, I}_{\infty, \eps}(t, \sfa, \sfb) + \ind_{\{|\I|<3\}} \sfV^{{\rm(cg)}, I}_{\infty, \eps}(\sfa, \sfb).
\end{equation}

We can now write down the following upper bound for \eqref{eq:Mepsh} in the limit $N\to\infty$:
\begin{equation}\label{eq:Mepsh2}
\begin{aligned}
& \cM^{\varphi, \psi}_{\infty, \eps} := \limsup_{N\to\infty} \cM^{\varphi, \psi}_{N, \eps}   \leq  \Vert \psi\Vert_\infty^4\, \eps^5  \sum_{r=1}^\infty C^r
\!\!\!\!\!\!\!\!\!\!\!\! \sum_{\substack{I_1, \ldots, I_r \vdash \{1,2,3,4\}, |I_i|\leq 3\\
                            K_\eps <s_1\leq t_1 <s_2\leq \cdots <s_r\leq t_r <s_{r+1}\leq \frac{2}{\eps} \\
                            \sfa_1, \sfb_1, \ldots, \sfa_r, \sfb_r\in (\Z^2)^4}
          }
          \\
& \qquad \qquad \cQ_{80s_1}(\varphi_{\epsilon}, \sfa_1)  \prod_{i=1}^{r} \sfW^{{\rm(cg)}, I_i}_{\infty, \eps}(t_i-s_i, \sfa_i, \sfb_i) \cQ_{80(s_{i+1}-t_i)}(\sfb_i, \sfa_{i+1}),
\end{aligned}
\end{equation}
where $C$ does not depend on $\eps$, $\sfa_{r+1}:=\ind_{B_\eps}$, the sum in \eqref{eq:Mepsh2}
contains no consecutive $I_i=I_{i+1}$ with $|I_i|=3$, and when $|I_i|\leq 2$, we must have $s_i=t_i$ thanks to the contraction of the V blocks.

For $\lambda>0$ to be chosen later, we can insert the factor
$e^{\frac{2\lambda}{\eps}} e^{-\lambda \sum_{i=1}^{r+1} (s_i-t_{i-1})-\lambda \sum_{i=1}^r (t_i-s_i)}\geq1$
into \eqref{eq:Mepsh2} to obtain a bound similar to \eqref{eq:Mbd1}:
\begin{equation}\label{eq:Mepsh3}
\big|\cM_{\infty,\eps}^{\varphi, \psi}\big| \leq e^{\frac{2\lambda}{\eps}} \Vert \psi \Vert_\infty^4 \eps^5 \sum_{r=1}^\infty C^r \sum_{I_1, \ldots, I_r} \!\!\! \langle \varphi_\eps^{\otimes 4},  \sfP^{*;I_1}_{\lambda, \eps} \sfP^{I_1;I_2}_{\lambda, \eps} \cdots \sfP^{I_{r-1};I_r}_{\lambda, \eps} \cQ^{I_r;*}_{\lambda, \eps} \ind^{\otimes 4}_{B_\eps} \rangle,
\end{equation}
where given two partitions $I, J \vdash \{1, \ldots, 4\}$, with $I=*$ denoting the partition consisting of singletons, $\sfP^{I, J}_{\lambda, \eps}$ are integral operators with kernels given by
\begin{equation}\label{eq:P1cg}
\begin{aligned}
\sfP^{I; J}_{\lambda, \eps} :=
\left\{
\begin{aligned}
\cQ_{\lambda, \eps}^{I; J} \sfV_{\lambda, \eps}^J \quad &  \mbox{if } |J| <3, \\
\cQ_{\lambda, \eps}^{I; J} \sfU_{\lambda, \eps}^J \quad & \mbox{if } |J| = 3,
\end{aligned}
\right.
\end{aligned}
\end{equation}
where for $\sfa, \sfb \in (\Z^2)^4$,
\begin{equation}\label{eq:QUVcg}
\begin{aligned}
&\cQ_{\lambda, \eps}^{I; J} (\sfb, \sfa) := \ind_{\{\sfb\sim I, \sfa\sim J \}} \sum_{n=1}^{2/\eps} e^{-\lambda n} \cQ_{80n}(\sfb, \sfa) ,    \qquad \qquad & & \\
& \sfU_{\lambda, \eps}^{I}(\sfa, \sfb) := \ind_{\{\sfa, \sfb\sim I \}} \sum_{n=0}^{2/\eps} e^{-\lambda n} \sfU^{{\rm(cg)}, I}_{\infty, \eps}(n, \sfa, \sfb),  & & |I|=3, \\
& \sfV_{\lambda, \eps}^{I}(\sfa, \sfb) := \ind_{\{\sfa, \sfb\sim I \}} \sfV^{{\rm(cg)}, I}_{\infty, \eps}(\sfa, \sfb),  & &  |I|<3.
\end{aligned}
\end{equation}
Here, given $\sfa=(\sfa^j)_{1\leq j\leq 4} \in (\Z^2)^4$ and a partition $I\vdash \{1, 2,3,4\}$, with $k\stackrel{I}{\sim} l$ denoting $k, l$ belonging to the same partition element in $I$,
\begin{equation}\label{eq:asimI}
\sfa \sim I \quad \mbox{denotes the constraint }
\left\{
\begin{aligned}
&\forall\, k\stackrel{I}{\sim} l , \ \sfa^k=\sfa^l  \qquad \quad  & \mbox{ if } |I| =3, \\
&\forall\, k\stackrel{I}{\sim} l , \ |\sfa^k-\sfa^l| \leq 20 M_\eps   & \mbox{ if } |I| \leq 2.
\end{aligned}
\right.
\end{equation}
We will denote $(\Z^2)^4_I := \{ \sfb \in (\Z^2)^4: \sfb \sim I\}$. The main difference
from \eqref{eq:Mbd1} and \eqref{eq:QU} is that the spatial constraint there are delta functions
(see \eqref{eq:xsimI}), that is, $\sfa \sim I$ with $M_\eps$ set to $0$. Here we also have the
additional operator $\sfV_{\lambda, \eps}^{J}(\sfa, \sfb)$ because we allow $\sfb\neq \sfa$. The
analogue of $\bbE[\xi^p]$ are the moments of the coarse-grained disorder variables
$\Theta^{\rm (cg)}_{N, \eps}$, which are now captured in $\sfV_{\lambda, \eps}^{J}$
and $\sfU_{\lambda, \eps}^{J}$.

As in \eqref{eq:hatQU}, for a weight function
$w_\eps:\Z^2 \to \R$, see \eqref{eq:phieps-psieps}, 
we define the weighted operators
\begin{align*}
\widehat \cQ_{\lambda, \eps}^{I; J} (\sfb, \sfa) & :=  \frac{w_\eps^{\otimes 4}(\sfb)}{w^{\otimes 4}_\eps(\sfa)} \, \cQ_{\lambda, \eps}^{I; J} (\sfb, \sfa) ,
\end{align*}
and define $\widehat \sfU_{\lambda, \eps}^{J}(\sfa, \sfb)$, $\widehat \sfV_{\lambda, \eps}^{J}(\sfa, \sfb)$, and $\widehat \sfP^{I, J}_{\lambda, \eps}$ similarly.   For $p, q>1$ with $\frac{1}{p}+\frac{1}{q}=1$, we can then bound \eqref{eq:Mepsh3} via the following analogue of \eqref{eq:Mbd2}:
\begin{align}\label{eq:Mepsh4}
	\big|\cM_{\infty, \eps}^{\varphi, \psi}\big|
	\,\leq\, e^{\frac{2\lambda}{\eps}} \Vert \psi \Vert_\infty^4 \eps^5 \sum_{r=1}^\infty C^r \!\!\!
	\sum_{I_1, \ldots, I_r} &
	\Big\Vert \frac{\varphi_\eps^{\otimes 4}}{w_\eps^{\otimes 4}} \Big\Vert_{\ell^p((\Z^2)^4_{I_1})}
	\, \Big\Vert \widehat \sfP^{*, I_1}_{\lambda, \eps}  \Big\Vert_{\ell^q \to \ell^q} \,
	\Big\Vert \widehat \sfP^{I_1;I_2}_{\lambda, \eps} \Big\Vert_{\ell^q\to \ell^q} \cdots  \\
	& \qquad \cdots \Big\Vert \widehat \sfP^{I_{r-1};I_r}_{\lambda, \eps} \Big\Vert_{\ell^q\to \ell^q} \,
	\Big\Vert \widehat \cQ^{I_r; *}_{\lambda, \eps} \Big\Vert_{\ell^q\to\ell^q}
	\, \Big\Vert \ind_{B_\eps}^{\otimes 4}
	w_\eps^{\otimes 4}\Big\Vert_{\ell^q((\Z^2)^4_{I_r})}, \notag
\end{align}
where we still have $(\Z^2)^4_I := \{ \sfb \in (\Z^2)^4: \sfb \sim I\}$, but the definition of the constraint $\sfb \sim I$ has changed as in \eqref{eq:asimI}. We still regard $\widehat \cQ_{\lambda, \eps}^{I; J} (\cdot, \cdot)$ and $\widehat \sfP_{\lambda, \eps}^{I; J} (\cdot, \cdot)$ as operators from $\ell^q((\Z^2)^4_J) \to \ell^q((\Z^2)^4_I)$, and similarly for $\widehat\sfU_{\lambda, \eps}^{J}$ and $\widehat\sfV_{\lambda, \eps}^{J}$.

We choose $\lambda:= \hat\lambda \eps$ with $\hat\lambda$ large but fixed so that $e^{\lambda \eps}$ remains bounded. We have the following analogue of Proposition~\ref{prop:opnorm},
where we again omitted $(\Z^2)^4_I$ from $\Vert \cdot \Vert_{\ell^p(\cdot)}$.
\begin{proposition}\label{p:opnormcg}
For some $c$ uniformly in $\lambda= \frac{\hat\lambda}{N}>0$, $\eps\in (0,1)$, and $I, J\subset \{1, \ldots, 4\}$ with $1\leq |I|, |J| \leq 3$ and $I\neq J$ when $|I|=|J|=3$, we have
\begin{gather}
	\Big\Vert \widehat \cQ^{I; J}_{\lambda, \eps} \Big\Vert_{\ell^q\to \ell^q} \leq c\ind_{\{|I|=|J|=3\}}
	+ cM_\eps^8 \ind_{\{|I|\wedge |J|\le2\}}; \label{eq:cgnorm1} \\
	\Big\Vert \widehat \cQ^{*, I}_{\lambda, \eps}  \Big\Vert_{\ell^q \to \ell^q}
	\leq c \eps^{-\frac{1}{q}}  \big(\ind_{\{|I|=3\}} +\ind_{\{|I|\le 2\}}M_\eps^4\big) ;
	\label{eq:cgnorm2}  \\
	\Big\Vert \widehat \cQ^{I; *}_{\lambda, \eps}
	\Big\Vert_{\ell^q}  \leq c \eps^{-\frac{1}{p}}  \big(\ind_{\{|I|=3\}}
	+\ind_{\{|I|\le 2\}}M_\eps^4\big) ;
	\label{eq:cgnorm2.5} \\
	\Big\Vert \widehat \sfU^{I}_{\hat\lambda \eps, \eps} \Big\Vert_{\ell^q\to \ell^q}
	\leq \frac{c}{\log \hat\lambda} \quad \mbox{for } |I|=3;  \label{eq:cgnorm3} \\
	\Big\Vert \widehat \sfV^{I}_{\lambda, \eps} \Big\Vert_{\ell^q\to \ell^q}
	\leq \frac{c}{(\log \frac{1}{\eps})^{\frac34}} \quad \mbox{for } |I|\leq 2.  \label{eq:cgnorm4}
\end{gather}
\end{proposition}
We now substitute these bounds into \eqref{eq:Mepsh4} and note that when $|I|=|J|=3$, each factor $\Vert \widehat \cQ^{I; J}_{\lambda, \eps}\Vert_{\ell^q\to \ell^q}$ can be controlled by $\Vert \widehat \sfU^{I}_{\hat\lambda \eps, \eps}\Vert_{\ell^q\to \ell^q}$, and when $|I|\wedge |J|\le 2$, the powers of $M_\eps=\log\log \frac{1}{\eps}$
from \eqref{eq:cgnorm1}-\eqref{eq:cgnorm2.5} can be controlled by $\Vert \widehat \sfV^{I}_{\lambda, \eps} \Vert_{\ell^q\to \ell^q}^{1/2} \Vert \widehat \sfV^{J}_{\lambda, \eps} \Vert_{\ell^q\to \ell^q}^{1/2}$ (set $\Vert \widehat \sfV^{I}_{\lambda, \eps} \Vert_{\ell^q\to \ell^q}:=1$ if $|I|=3$). This leads to a convergent geometric series similar to \eqref{eq:Mbd3}, which gives
\begin{equation}
\begin{aligned}
\big|\cM_{\infty, \eps}^{\varphi, \psi}\big| & \leq C \eps^{\frac4p} \Big\Vert \frac{\varphi_\eps}{w_\eps} \Big\Vert^4_{\ell^p} \Vert \psi \Vert_\infty^4  \big\Vert w\ind_{B}\big\Vert^4_{\ell^q}
\end{aligned}
\end{equation}
for some $C$ depending only on $\hat\lambda$. This proves Theorem~\ref{th:cgmom}.
\end{proof}
\medskip

To conclude this section, we sketch the proof of Proposition~\ref{p:opnormcg}.

\medskip
\begin{proof}[Proof of Proposition~\ref{p:opnormcg}] We will sketch how the proof of Proposition~\ref{prop:opnorm} can be adapted to the current setting.

{\em Proof of \eqref{eq:cgnorm1}.}  First note that it is equivalent to
\begin{equation}\label{eq:cQbd1}
\sum_{\bx\in (\Z^2)^4_{I}, \by\in (\Z^2)^4_{J}} \!\!\!\!\!\!\!\! f(\bx) \cQ^{I,J}_{\lambda, \eps} (\bx, \by) \frac{w_\eps^{\otimes 4}(\bx)} {w_\eps^{\otimes 4}(\by)} \, g(\by)
 \leq c \big(\ind_{\{|I|=|J|=3\}} + M_\eps^8\ind_{\{|I|\wedge|J|<3\}}\big)\Vert f\Vert_{\ell^p} \Vert g\Vert_{\ell^q}
\end{equation}
uniformly for all $f\in \ell^p((\Z^2)^4_{I})$ and $g\in \ell^q((\Z^2)^4_{J})$. We split the region of summation into $A_\eps=\{|\bx-\by|\leq C_0/\sqrt{\eps|}\}$ and $A_\eps^c$. Note that the analogue of Lemma~\ref{lem:Green} holds for $\cQ_{\lambda, \eps}$. Therefore following the same argument as in \eqref{eq:ANc}, the region $A_\eps^c$ gives the contribution
\begin{align}
& \sum_{\bx\in (\Z^2)^4_{I}, \by\in (\Z^2)^4_{J} \atop (\bx, \by) \in A_\eps^c}  f(\bx) \cQ^{I,J}_{\lambda, \eps} (\bx, \by) \frac{w_\eps^{\otimes 4}(\bx)} {w_\eps^{\otimes 4}(\by)} \, g(\by) \notag \\
\leq\ & C \eps^3 \Big( \sum_{\bx\in (\Z^2)^4_{I}, \by\in (\Z^2)^4_{J}}  |f(\bx)|^p \, e^{-|\bx-\by|\sqrt{\eps} } \Big)^{1/p}
\Big( \sum_{\bx\in (\Z^2)^4_{I}, \by\in (\Z^2)^4_{J}}  |g(\by)|^q \, e^{-|\bx-\by|\sqrt{\eps} } \Big)^{1/q} \notag \\
\leq\ & C \eps^{3-\tfrac{|J|}{p} -\tfrac{|I|}{q}} \big(\ind_{\{|I|=|J|=3\}} + M_\eps^8 \ind_{\{|I|\wedge|J|<3\}}\big) \|f\|_{\ell^p} \, \|g\|_{\ell^q} \label{eq:Aeps}\\
\leq \ & C \|f\|_{\ell^p} \, \|g\|_{\ell^q}, \notag
\end{align}
where the spatial constraints in $\bx\in (\Z^2)^4_{I}$ and $\by\in (\Z^2)^4_{J}$ (see \eqref{eq:asimI}) led to the factor in the bracket in the third line.

In the region $A_\eps$, the factor  $w_\eps^{\otimes 4}(\bx)/w_\eps^{\otimes 4}(\by)$ is bounded. By the analogue of Lemma~\ref{lem:Green} for $\cQ_{\lambda, \eps}$, it suffices to show
\begin{equation}\label{eq:cQbd2}
\sum_{\bx\in (\Z^2)^4_{I}, \by\in (\Z^2)^4_{J}}  \frac{f(\bx)  g(\by)}{(1+\sum_{i=1}^4 |x_i-y_i|^2)^3}
 \leq c \big(\ind_{\{|I|=|J|=3\}} + M_\eps^8\ind_{\{|I|\wedge|J|<3\}}\big)\Vert f\Vert_{\ell^p} \Vert g\Vert_{\ell^q}
\end{equation}
When $|I|=|J|=3$, the proof is exactly the same as that of \eqref{ItoJ}. When $|I|, |J|<3$, we can apply H\"older to bound the l.h.s.\ of \eqref{eq:cQbd2} by
\begin{align}
& \Big(\sum_{\bx\in (\Z^2)^4_{I}, \by\in (\Z^2)^4_{J}} \frac{f(\bx)^p}{(1+\sum_{i=1}^4 |x_i-y_i|^2)^3}  \Big)^{1/p} \Big(\sum_{\bx\in (\Z^2)^4_{I}, \by\in (\Z^2)^4_{J}} \frac{g(\by)^q}{(1+\sum_{i=1}^4 |x_i-y_i|^2)^3}  \Big)^{1/q} \notag \\
\leq \ & C M_\eps^8\Vert f\Vert_{\ell^p} \Vert g\Vert_{\ell^q}.
\end{align}
When $|I|<3$ and $|J|=3$ (the case $|I|=3$ and $|J|<3$ can be treated identically), we can find $k, l\in \{1,2,3,4\}$ that belong to the same partition element in $I$, but to different partition elements in $J$; in particular,  $\bx\in (\Z^2)^4_{I}$ implies $|x_k-x_l|\leq 20 M_\eps$. Fix any $a\in (0, 1/q)$. We can then apply H\"older to bound the l.h.s.\ of \eqref{eq:cQbd2} by
$$
\left.\begin{aligned}
& \Big(\sum_{\bx\in (\Z^2)^4_{I}, \by\in (\Z^2)^4_{J}} \frac{f(\bx)^p}{(1+\sum_{i=1}^4 |x_i-y_i|^2)^3} \cdot \frac{1}{(1+|y_k-y_l|^{2a})^p} \Big)^{1/p}  \\
& \quad \times \Big(\sum_{\bx\in (\Z^2)^4_{I}, \by\in (\Z^2)^4_{J}} \frac{g(\by)^q}{(1+\sum_{i=1}^4 |x_i-y_i|^2)^3}  \cdot (1+|y_k-y_l|^{2a})^q \Big)^{1/q}
\end{aligned}
\right. \leq \  \ C M_\eps^8 \Vert f\Vert_{\ell^p} \Vert g\Vert_{\ell^q},
$$
where in the first bracket, the sum over $\by$ is uniformly bounded by the same argument as in the bound for \eqref{ItoJf}, while in the second bracket, we can distinguish between two cases: either $|y_k-y_l|\leq 40M_\eps$, in which case we apply this bound and sum over $\bx$ to get a bound of $M_\eps^{\frac{4}{q}+2a}\Vert g\Vert_{\ell^q}$; or $|y_k-y_l|>40M_\eps \geq 2|x_k-x_l|$,
in which case we apply the triangle inequality
$$
|x_k-y_k| +|x_l-y_l| \geq \frac{|x_k-y_k| +|x_l-y_l|}{2} + \frac{|y_k-y_l|}{4}
$$
and follow the same argument as for \eqref{ItoJf2} to get a bound of $M_\eps^{4/q}\Vert g\Vert_{\ell^q}$, where $M_\eps^{4/q}$ comes from summing over the redundant components of $\bx$ after selecting one component of $\bx$ for each partition element of $I$.  This concludes the proof of \eqref{eq:cgnorm1}.
\medskip

\noindent
{\em Proof of \eqref{eq:cgnorm2}-\eqref{eq:cgnorm2.5}.} Similar to \eqref{eq:hatQf},
we need to show
\begin{equation}\label{eq:hatQfcg}
	\sum_{\bx\in (\Z^2)^4_{I}, \by\in (\Z^2)^4} \!\!\!\!\!\!\!\! f(\bx)  \cQ^{I,*}_{\lambda, \eps}
	(\bx, \by)  \frac{w_\eps^{\otimes 4}(\bx)}{w_\eps^{\otimes 4}(\by)} g(\by)
	 \leq c\, \eps^{-\frac{1}{p}} \big(\ind_{\{|I|=3\}} +\ind_{\{|I|<3\}}M_\eps^4\big)
	 \Vert f\Vert_{\ell^p} \Vert g\Vert_{\ell^q}
\end{equation}
uniformly in $f\in \ell^p((\Z^2)^4_{I})$ and
$g\in \ell^q((\Z^2)^4)$. Restricted to $(\bx, \by) \in A_\eps^c$, we note that the bound \eqref{eq:Aeps} can now be replaced by
$$
	C \eps^{3-\tfrac{4}{p}} \Big(\ind_{\{|I|=3\}} \eps^{-\tfrac{3}{q}} + \ind_{\{|I|<3\}}
	M_\eps^{\frac{8-2|I|}{q}}\eps^{-\frac{|I|}{q}}\Big) \|f\|_{\ell^p} \, \|g\|_{\ell^q}
	\leq C \eps^{-\frac{1}{p}}  \|f\|_{\ell^p} \, \|g\|_{\ell^q} .
$$

Restricted to $A_\eps$, it suffices to bound the following analogue of \eqref{eq:ItoJ*}:
\begin{align}\label{eq:hatQfcg2}
	\sum_{\bx \in(\Z^4)^h_{I}\,,\,\by\in(\Z^2)^4 \atop (\bx, \by)\in A_\eps}
	\frac{f(\bx) g(\by)}{\big(1+\sum_{i=1}^4  |x_i-y_i |^2\big)^3}
	\leq C \eps^{-\frac{1}{p}} \big(\ind_{\{|I|=3\}} +\ind_{\{|I|<3\}}M_\eps^4\big)
	\Vert f\Vert_{\ell^p} \Vert g\Vert_{\ell^q} .
\end{align}
When $|I|=3$, this follows by exactly the same proof as that of \eqref{eq:ItoJ*}. When $|I|\leq 2$, the
estimate is simpler and we can apply the H\"older inequality to bound the l.h.s.\ of \eqref{eq:hatQfcg2} by
\begin{equation}\label{eq:Ito*cg}
\begin{aligned}
	\Bigg(\sum_{\bx\in (\Z^2)^4_{I}, \by\in (\Z^2)^4 \atop (\bx, \by)\in A_\eps}
	\frac{f(\bx)^p}{\big(1+\sum_{i=1}^4  |x_i-y_i |^2\big)^3}\Bigg)^{\frac{1}{p}}
	\Bigg(\sum_{\bx\in (\Z^2)^4_{I}, \by\in (\Z^2)^4 \atop (\bx, \by)\in A_\eps}
	\frac{g(\by)^q}{\big(1+\sum_{i=1}^4  |x_i-y_i |^2\big)^3}\Bigg)^{\frac{1}{q}}.
\end{aligned}
\end{equation}
Since $|I|\leq 2$, the second factor can be bounded by $C M_\eps^4 \Vert g\Vert_{\ell^q}$.
For the first factor, summing over $\by$ gives
$$
\sum_{\by\in (\Z^2)^4 \atop |\by-\bx|<C_0/\sqrt{\eps}} \frac{1}{\big(1+\sum_{i=1}^4  |x_i-y_i |^2\big)^3}
\leq  C\eps^{-1}
$$
uniformly in $\bx\in (\Z^2)^4_I$. Collecting all the bounds obtained so far then gives \eqref{eq:hatQfcg}.
\medskip

\noindent
{\em Proof of \eqref{eq:cgnorm3}.}  Assume w.l.o.g.\ that $I\vdash\{1,2,3,4\}$ consists of the partition elements $\{1,2\}, \{3\}, \{4\}$.  Recall from \eqref{eq:QUVcg} and \eqref{eq:sfUcg} that for $\bx, \by\in (\Z^2)^4_I$,
$$
 \sfU^I_{\hat \lambda\eps, \eps} (\bx, \by) = \sum_{n=0}^{2/\eps} e^{-\hat\lambda \eps n} \,\overline{U}^{{\rm(cg)}}_{\infty, \eps}(n, y_1-x_1) \prod_{i=3,4}g_{80n}(x_i, y_i).
$$
We then have the following analogues of \eqref{Chatlambda} and \eqref{Chatlambda2}:
\begin{align*}
\sum_{\by\in (\Z^2)^4_I} \sfU^I_{\hat \lambda\eps, \eps} (\bx, \by) & \leq 4 \sum_{n=0}^{2/\eps} e^{-\hat\lambda \eps n} \sum_{z\in\Z^2} \overline{U}^{{\rm(cg)}}_{\infty, \eps}(n, z) \leq \frac{C}{\log \hat\lambda},
\\
\sum_{\by \in (\Z^2)^4_I} \sfU^I_{\hat \lambda\eps, \eps} (\bx, \by) e^{C|\bx- \by|\sqrt{\eps}}
& \leq C\sum_{n=0}^{2/\eps} e^{-\hat\lambda\eps n} \sum_{z\in \Z^2} \overline{U}^{{\rm(cg)}}_{\infty, \eps}(n, z) e^{C|z|\sqrt{\eps}}
\leq \frac{C}{\log \hat\lambda},
\end{align*}
both of which follows from Lemma~\ref{th:ren-cg} by Fatou's Lemma (recall $\overline{U}^{{\rm(cg)}}_{\infty, \eps}$ from \eqref{eq:Ucg}).
The rest of the proof is exactly the same as that of \eqref{eq:norm3}.
\medskip

\noindent
{\em Proof of \eqref{eq:cgnorm4}.} Given a partition $I\vdash\{1,2,3,4\}$ with $|I|\leq 2$, recall the definition of  $\widehat \sfV^{I}_{\lambda, \eps}$ from \eqref{eq:QUVcg}. We need to show that
\begin{equation}\label{eq:hatVfcg}
\sum_{\bx, \by\in (\Z^2)_I^4} f(\bx)  \sfV^{{\rm(cg)}, I}_{\infty, \eps}(\bx, \by) g(\by) \frac{w_\eps^{\otimes 4}(\bx)}{w_\eps^{\otimes 4}(\by)}
 \leq \frac{C}{(\log \frac{1}{\eps})^{\frac34}} \Vert f\Vert_{\ell^p} \Vert g\Vert_{\ell^q}
\end{equation}
uniformly for all $f\in \ell^p((\Z^2)^4_I)$ and $g\in \ell^q((\Z^2)^4_I)$. As before, we consider the sum of $\bx, \by$ over $A_\eps=\{|\bx-\by|\leq C_0/\sqrt{\eps|}\}$ and $A_\eps^c$ separately and apply H\"older's inequality. The bound \eqref{eq:cgnorm4} will follow if we show that uniformly in $\bx\in (\Z^2)^4_I$,
\begin{align}
\sum_{\by \in (\Z^2)^4_I} \sfV^{{\rm(cg)}, I}_{\infty, \eps}(\bx, \by)  \leq  \frac{C}{(\log \frac{1}{\eps})^{\frac34}}
\quad \mbox{and} \quad \sum_{\by \in (\Z^2)^4_I} \sfV^{{\rm(cg)}, I}_{\infty, \eps}(\bx, \by) e^{C|\by-\bx| \sqrt{\eps}}  \leq \frac{C}{(\log \frac{1}{\eps})^{\frac34}}.
\end{align}
These bounds hold because \eqref{eq:sfVcg} and \eqref{eq:Vcg} imply that
$$
\sfV^{{\rm(cg)}, I}_{\infty, \eps}(\bx, \by) \leq \frac{C}{(\log \frac{1}{\eps})^{\frac34}} \prod_{i=1}^4 \ind_{\{|y_i-x_i|\leq M_\eps\}} e^{-|y_i-x_i|}.
$$
This concludes the proof of Proposition~\ref{p:opnormcg}.
\end{proof}

\section{Proof of the main results: Theorems~\ref{th:main0} and~\ref{th:main1}}
\label{s:Thm1.3}

\subsection{Proof of Theorem~\ref{th:main0}} \label{s:Thm1.1}

We can rephrase Theorem~\ref{th:main0} as follows.
\begin{theorem}\label{th:main}
Assume the same setup as in Theorem~\ref{th:main0}. Let $k\in\N$. For $i=1,\ldots, k$,
assume $0 \le s_i \le t_i < \infty$, $\varphi_i \in C_c(\R^2)$ has compact support,
and $\psi_i \in C_b(\R^2)$ is bounded. Then the following
convergence in distribution holds as $N\to\infty$:
\begin{equation}\label{eq:main-statement}
	\big( \cZ_{N, s_i, t_i}^{\beta_N}(\varphi_i, \psi_i) \big)_{i=1,\ldots, k}
	\ \Longrightarrow \
	\big( \mathscr{Z}_{s,t}^\theta(\varphi_i, \psi_i) \big)_{i=1,\ldots, k}  \,,
\end{equation}
where $\cZ^{\beta_N}_{N;\, s,t}(\varphi, \psi) :=\iint \varphi(x) \,\psi(y)
\, \cZ^{\beta_N}_{N;\, s,t}(\dd x, \dd y)$.
\end{theorem}
We will prove Theorem~\ref{th:main} by showing that
the random vector in the l.h.s.\ of \eqref{eq:main-statement}
converges in distribution as $N\to\infty$ to a unique
random limit. This in turn implies that $(\cZ^{\beta_N}_{N;\, s,t}(\dd x, \dd y))_{0\le s \le t < \infty}$
converges to a unique limit, denoted
$\mathscr{Z}^\theta = (\mathscr{Z}_{s,t}^\theta(\dd x, \dd y))_{0 \le s \le t < \infty}$.

\medskip

The convergence of the one point distribution in Theorem~\ref{th:main} follows from the following result. We will explain how this can be adapted to finite dimensional distributions in Remark \ref{rem:fdd}.

\begin{proposition}\label{p:Thm1.3}
Given $\varphi\in C_c(\R^2)$ and $\psi\in C_b(\R^2)$, let
$\cZ_N(\varphi, \psi) := \cZ^{\beta_N}_{N,0,1}(\varphi, \psi) =
\iint \varphi(x) \psi(y) \cZ^{\beta_N}_N({\rm d}x, {\rm d}y)$ be as in Theorem~\ref{th:main}
with $N\in 2\N$. Then $\cZ_N(\varphi, \psi)$ converges in distribution to a unique limit as $N\to\infty$.
\end{proposition}
\begin{proof} Since
$$
	\E[|\cZ_N(\varphi, \psi)|] \leq \frac{1}{N} \sum_{x, y\in\Z^2}
	\big|\varphi\big(\tfrac{x}{\sqrt N}\big)\big| \big|\psi\big(\tfrac{y}{\sqrt N}\big)
	\big| q_N(y-x) \to \iint |\varphi(x)| |\psi(y)| g_{\frac12}(y-x) {\rm d}x {\rm d}y \,,
$$
it follows that $\E[|\cZ_N(\varphi, \psi)|]$ is uniformly bounded in $N$ and hence
$(\cZ_N(\varphi, \psi))_{N\in\N}$ is a tight family and admits subsequential weak limits.

To show that the limit is unique, it then suffices to show that for every bounded $f:\R\to\R$,
with uniformly bounded first three derivatives, the limit
\begin{equation*}
	\lim_{N\to\infty} \E[f(\cZ_N(\varphi, \psi))]
\end{equation*}
exists. To this end, we will show that $(\E[f(\cZ_N(\varphi, \psi))])_{N\in\N}$ is a Cauchy sequence.

Theorem~\ref{th:cg-main} allows us to approximate $\cZ_N(\varphi, \psi)$ by the coarse-grained model
$ \mathscr{Z}_{\epsilon}^{(\cg)}(\varphi,\psi|\Theta)$
with coarse-grained disorder variables
$\Theta = \Theta_{N,\epsilon}^{(\cg)}$, with an
$L^2$ error that is arbitrarily small, uniformly in large $N$, if $\eps>0$
is chosen sufficiently small.  Therefore it only remains to show
\begin{equation}\label{eq:fZlind}
	\lim_{\eps\downarrow 0} \lim_{N\to\infty}
	\sup_{m,n\geq N}\big|\E[f( \mathscr{Z}_{\epsilon}^{(\cg)}(\varphi,\psi|\Theta_{m,\epsilon}^{(\cg)}))]
	- \E[f( \mathscr{Z}_{\epsilon}^{(\cg)}(\varphi,\psi|\Theta_{n,\epsilon}^{(\cg)}))]\big| = 0.
\end{equation}
We will prove \eqref{eq:fZlind} by applying the Lindeberg principle for multilinear polynomials
of dependent random variables formulated in Lemmas~\ref{lem:lind1}--\ref{lem:lind3}.

\smallskip

Let us set
$\Phi(\Theta):= \mathscr{Z}_{\epsilon}^{(\cg)}(\varphi,\psi|\Theta)$,
and note from its definition in \eqref{eq:Zcg-gen} that $\Phi(\Theta)$ is a multilinear polynomial
in the variables $\Theta :=(\Theta(\vec\sfi, \vec\sfa))_{(\vec\sfi, \vec\sfa) \in \bbT_\eps}$,
where recall from \eqref{eq:iacond} that
\begin{equation*}
	\bbT_\eps:=\Big\{ (\vec \sfi, \vec\sfa)= ((\sfi, \sfi'), (\sfa, \sfa')):
	\ |\vec\sfi|=\sfi'-\sfi+1 \leq K_\eps, \ |\vec\sfa|=|\sfa-\sfa'| \le M_\epsilon \sqrt{|\vec\sfi|} \Big\}.
\end{equation*}
We write $\Theta_n$ for the coarse-grained disorder variables $\Theta^{\rm (cg)}_{n, \eps}
:=(\Theta^{\rm (cg)}_{n, \eps}(\vec\sfi, \vec\sfa))_{(\vec\sfi, \vec\sfa) \in \bbT_\eps}$,
see \eqref{eq:Theta}.
These satisfy Assumption~\ref{ass:dep} with the following dependency neighborhoods:
\begin{itemize}
\item for each $\sfz_1:=(\vec\sfi_1, \vec\sfa_1)\in \bbT_\eps$,
its dependency neighbourhood is given by
\begin{align*}
	A_{\sfz_1} =\Big\{\sfz_2=(\vec \sfi_2, \vec\sfa_2)\in \bbT_\eps & : \{\sfi_2, \sfi_2'\}
	\cap \{\sfi_1, \sfi_1'\}\neq \emptyset, \\
	& \quad {\rm dist}(\{\sfa_2\}, \{\sfa_1, \sfa_1'\})
	\wedge {\rm dist}(\{\sfa_2'\}, \{\sfa_1, \sfa_1'\}) \leq 2M_\eps \Big\} ;
\end{align*}

\item given $\sfz_1 \in \bbT_\eps$ and $\sfz_2\in A_{\sfz_1}$, the
dependency neighbourhood of $\{\sfz_1, \sfz_2\}$ is given by
$$
A_{\sfz_1 \sfz_2} = A_{\sfz_1}\cup A_{\sfz_2}.
$$
\end{itemize}
Recalling the definition of $\bbT_\eps$,  we see that,
uniformly in $\epsilon > 0$ and $\sfz \in \bbT_\epsilon$,
\begin{equation}\label{eq:Az1}
	|A_{\sfz}|\leq C \, M_\eps^2 \, K_\eps \, (M_\epsilon \sqrt{K_\epsilon})^2
	= C \, M_\epsilon^4 \, K_\epsilon^2 \,.
\end{equation}

In order to apply Lemma~\ref{lem:lind3},
we first verify that condition \eqref{eq:vanish} is satisfied by $\Phi(\Theta)$.

\begin{lemma}
The multilinear polynomial $\Phi(\Theta)
:= \mathscr{Z}_{\epsilon}^{(\cg)}(\varphi,\psi|\Theta)$ satisfies condition \eqref{eq:vanish}.
\end{lemma}

\begin{proof} Condition \eqref{eq:vanish} reads as
$$
	\forall\, \sfz_1 \in \bbT_\eps, \ \forall\,\sfz_2\in A_{\sfz_1}, \ \forall\,\sfz_3\in A_{\sfz_1\sfz_2}
	=A_{\sfz_1} \cup A_{\sfz_2}:
	\qquad 	\partial^2_{\sfz_i \sfz_j} \Phi = 0 \mbox{ for all } 1\leq i, j\leq 3,
$$
where $\partial_{\sfz}$ denotes derivative w.r.t.\ $\Theta(\sfz)$.
Since $\Phi$ is multilinear in $(\Theta(\sfz))_{\sfz \in \bbT_\eps}$, this condition is equivalent to the claim
that no term in the expansion of $\Phi$ (recall its definition from \eqref{eq:Zcg-gen}) contains more than one
of the factors $\Theta(\sfz_i)$, $1\leq i\leq 3$. From the definition of $\Phi$, clearly the product
$\Theta(\sfz_1)\Theta(\sfz_2)$ cannot appear because $\sfz_1=((\sfi_1, \sfi_1'), (\sfa_1, \sfa_1'))$ and
$\sfz_2=((\sfi_2, \sfi_2'), (\sfa_2, \sfa_2'))$ have an overlapping time index. Similarly, if
$\sfz_3\in A_{\sfz_i}$ for either $i=1$ or $2$, then the factor $\Theta(\sfz_i)\Theta(\sfz_3)$ cannot appear.
The last case is if $\sfz_3\in A_{\sfz_2}$, but $\sfz_3\notin A_{\sfz_1}$ (the case
$\sfz_3\in A_{\sfz_1}$, $\sfz_3\notin A_{\sfz_2}$ is the same by symmetry, since $\sfz_2\in A_{\sfz_1}$ if
and only if $\sfz_1 \in A_{\sfz_2}$): let us show that $\Theta(\sfz_1)\Theta(\sfz_3)$
does not appear in $\Phi(\Theta)$.
Both $\vec\sfi_1$ and $\vec\sfi_3$ have an overlapping time index with $\vec\sfi_2$, hence
${\rm dist}(\vec\sfi_1, \vec\sfi_3)\leq \sfi_2'-\sfi_2= |\vec \sfi_2|-1\leq K_\eps-1$, which contradicts
the constraint imposed by $\bcA_\epsilon^{(\notri)}$ in \eqref{eq:bcA}, that for
$\Theta(\sfz_1)\Theta(\sfz_3)$ to appear, we must have ${\rm dist}(\vec\sfi_1, \vec\sfi_3)\geq K_\eps$.
This verifies condition \eqref{eq:vanish}.
\end{proof}

We can then apply Lemmas~\ref{lem:lind1}--\ref{lem:lind3} to bound
\begin{equation}\label{eq:I123}
	\big|\E[f(\Phi(\Theta_m))]- \E[f(\Phi(\Theta_n))]\big|
	\leq I_1^{(m)} + I_2^{(m)} + I_1^{(n)} + I_2^{(n)} + I_3^{(m,n)},
\end{equation}
where $I_1^{(m)}$ and $I_2^{(m)}$ are the terms from applying Lemma~\ref{lem:lind1} to
$h(\cdot) = f(\Phi(\cdot))$ and $X = \Theta_m$, see \eqref{eq:I1} and \eqref{eq:I2},
similarly for $I_1^{(n)}$ and $I_2^{(n)}$, while $I_3^{(m,n)}$ is the term from
applying Lemma~\ref{lem:lind2} to two Gaussian families
$Z = \Theta^{({\rm G)}}_m$ and $\tilde Z = \Theta^{({\rm G)}}_n$
with the same mean and covariance structure as
$\Theta_m$ and $\Theta_n$, respectively, but independent of them,
see \eqref{eq:I3}.\footnote{Since  $(\Theta_n(\sfz))_{\sfz \in \bbT_\eps}$ are uncorrelated,
$(\Theta^{({\rm G)}}_n(\sfz))_{\sfz \in \bbT_\eps}$ are in fact independent Gaussian random variables.}

We are now ready to prove \eqref{eq:fZlind} exploiting \eqref{eq:I123}.
It suffices to prove that
\begin{equation}\label{eq:lasttoprove}
	\lim_{\epsilon\downarrow 0} \, \limsup_{n\to\infty} \, I_1^{(n)} = 0 \,, \qquad
	\lim_{\epsilon\downarrow 0} \, \limsup_{n\to\infty} \, I_2^{(n)} = 0 \,, \qquad
	\lim_{\epsilon\downarrow 0} \, \limsup_{n,m\to\infty} \, I_3^{(m,n)} = 0 \,,
\end{equation}
We will prove these relations separately,
exploiting \eqref{eq:I1bd}, \eqref{eq:I2bd} and \eqref{eq:I3bd} from Lemma~\ref{lem:lind3}.
This will conclude the proof of Proposition~\ref{p:Thm1.3}.

\medskip
\noindent
{\bf Bound on $\limsup\limits_{n\to\infty} I_1^{(n)}$.} By \eqref{eq:I1bd}, we have
\begin{equation}\label{eq:I1bdcg}
\begin{aligned}
|I_1^{(n)}| \, &\leq\,  \frac{1}{2} \|f'''\|_\infty \sup_{\sfz_1 \in \bbT_\eps} \E\big[|\Theta_n(\sfz_1)|^3\big] \sum_{\sfz_1 \in \bbT_\eps,\, \sfz_2 \in A_{\sfz_1}, \, \sfz_3 \in A_{\sfz_1\sfz_2}}  \\
& \quad    \sup_{s,t,u} \E \Big[ \big|\partial_{\sfz_1} \Phi \big(
	W_{s,t,u}^{\sfz_1, \sfz_2} \big) \big|^3 \Big]^{\frac13}
	\sup_{s,t,u} \E \Big[ \big|\partial_{\sfz_2} \Phi \big(
	W_{s,t,u}^{\sfz_1, \sfz_2} \big) \big|^3 \Big]^{\frac13}
	\sup_{s,t,u} \E \Big[ \big|\partial_{\sfz_3} \Phi \big(
	W_{s,t,u}^{\sfz_1, \sfz_2} \big) \big|^3 \Big]^{\frac13},
\end{aligned}
\end{equation}
where for $s,t,u\in [0,1]$,
\begin{equation}\label{eq:Wmix}
	W_{s,t,u}^{\sfz_1, \sfz_2} := s u\sqrt{t}\Theta_n^{A_{\sfz_1}} + u\sqrt{t}
	\Theta_n^{A_{\sfz_1\sfz_2}\backslash A_{\sfz_1}} +\sqrt{t} \Theta_n^{A_{\sfz_1\sfz_2}^c}
	+\sqrt{1-t}\Theta^{({\rm G)}}_n,
\end{equation}
with $\Theta_n^A(\sfz) := \Theta_n(\sfz) \ind_{\{\sfz\in A\}}$.

First note that by the assumption $\varphi\in C_c(\R^2)$ and the definition of the multilinear polynomial $\Phi(\Theta)$ in \eqref{eq:Zcg-gen}, $\Phi(\Theta)$  depends only on $\Theta(\sfz)$ for a finite set of $\sfz \in \bbT_\eps$. In particular, the sum in \eqref{eq:I1bdcg} is finite, and we can pass $\limsup_{n\to\infty}$ inside the sum.

Note that $\Vert f'''\Vert_\infty$ is bounded by assumption, and by \eqref{eq:Theta4.1},
\begin{equation}\label{eq:Thetanz1}
\limsup_{n\to\infty} \sup_{\sfz_1 \in \bbT_\eps} \E\big[|\Theta_n(\sfz_1)|^3\big] \leq \limsup_{n\to\infty} \sup_{\sfz_1 \in \bbT_\eps} \E\big[|\Theta_n(\sfz_1)|^4\big]^{3/4} \leq \frac{C}{(\log \frac{1}{\eps})^{3/4}}.
\end{equation}

The sum in \eqref{eq:I1bdcg} can be bounded by
\begin{equation}\label{eq:Thetanz2}
	\sum_{\sfz_1 \in \bbT_\eps,\, \sfz_2 \in A_{\sfz_1}, \, \sfz_3 \in A_{\sfz_1\sfz_2}}
	\frac{1}{3} \sum_{i=1}^3 \sup_{s,t,u} \E \Big[ \big|\partial_{\sfz_i} \Phi \big(
	W_{s,t,u}^{\sfz_1, \sfz_2} \big) \big|^4 \Big]^{3/4}.
\end{equation}
Given $\sfz=(\vec\sfi, \vec \sfa)=((\sfi, \sfi'), (\sfa, \sfa'))\in \bbT_\eps$, by the definition of
$\Phi(\Theta)=\mathscr{Z}_{\epsilon}^{(\cg)}(\varphi,\psi|\Theta)$ in \eqref{eq:Zcg-gen},
\begin{equation}\label{eq:Thetanz2.5}
	\partial_\sfz \Phi(\Theta) = \partial_{\Theta(\sfz)} \mathscr{Z}_{\epsilon}^{(\cg)}(\varphi,\psi|\Theta)
	= \frac{2}{\eps} \,\bar{\mathscr{Z}}^{(\cg)}_{[0, \sfi]}(\varphi,\ind_{\cS_\eps(\sfa)}|\Theta) 	
	\bar{\mathscr{Z}}^{(\cg)}_{[\sfi', 1/\eps]}(\ind_{\cS_\eps(\sfa')},\psi|\Theta),
\end{equation}
where $\bar{\mathscr{Z}}^{(\cg)}_{[0, \sfi]}(\varphi,\ind_{\{\cS_\eps(\sfa)\}}|\Theta)$ denotes the
centred coarse-grained model with initial condition $\varphi$ at time $0$
and terminal condition $\ind_{\cS_\eps(\sfa)}$  at time $\sfi$: this is just the original
coarse-grained model in \eqref{eq:Zcg-gen} with time horizon $\sfi$ instead of
$\lfloor 1/\epsilon \rfloor$ and with the constant term $\frac{1}{2} \, g_{\frac{1}{2}}(\varphi, \psi)$
omitted.
(We recall that $\cS_\eps(\sfa)$ is a square of side length $\sqrt{\eps}$ defined as in \eqref{eq:cI2}.)
The definition of $\bar{\mathscr{Z}}^{(\cg)}_{[\sfi', 1/\eps]}(\ind_{\cS_\eps(\sfa')},\psi|\Theta)$ is
similar, which is independent of
$ \bar{\mathscr{Z}}^{(\cg)}_{[0, \sfi]}(\varphi,\ind_{\{\cS_\eps(\sfa)\}}|\Theta)$. Each of
$\partial_{\Theta(\sfz)} \mathscr{Z}_{\epsilon}^{(\cg)}$, $\bar{\mathscr{Z}}^{(\cg)}_{[0, \sfi]}$ and
$\bar{\mathscr{Z}}^{(\cg)}_{[\sfi', 1/\eps]}$ contains a factor of $\eps/2$ by the defintion of the
coarse-grained model in \eqref{eq:Zcg-gen}, which is why there is a prefactor of $2/\eps$
in \eqref{eq:Thetanz2.5}. We then have
\begin{align}
	&  \E \Big[ \big|\partial_{\sfz} \Phi \big(W_{s,t,u}^{\sfz_1, \sfz_2} \big) \big|^4 \Big]
	\label{eq:Thetanz3} \\
	\!\!\!\!\!\! = & \frac{16}{\eps^4}\, \E \Big[\! \Big(\!\bar{\mathscr{Z}}^{(\cg)}_{[0, \sfi]}	
	(\ind_{\cS_\eps(0)}, \varphi(\cdot -\sqrt{\eps}\sfa)| W_{s,t,u}^{\sfz_1, \sfz_2})\Big)^4 \Big]
	\E \Big[\! \Big(\!\bar{\mathscr{Z}}^{(\cg)}_{[0, 1/\eps-\sfi']}
	(\ind_{\cS_\eps(0)},\psi(\cdot -\sqrt{\eps}\sfa')| W_{s,t,u}^{\sfz_1, \sfz_2})\Big)^4 \Big], \notag
\end{align}
where we interchanged initial and terminal conditions by symmetry and used translation invariance
and independence.

\smallskip

We can bound the two factors in the r.h.s.\ of \eqref{eq:Thetanz3}
by applying slight variants of Theorem~\ref{th:cgmom}, which was formulated
for the original coarse-grained model.
Let us focus on the first factor in the r.h.s.\ of \eqref{eq:Thetanz3},
for which we need to take into account two differences:
the time range $[0,\sfi]$ instead of $[0,\epsilon^{-1}]$ and the disorder $W_{s,t,u}^{\sfz_1, \sfz_2}$
instead of $\Theta_n$.

The first difference is immaterial, because our moment estimate Theorem~\ref{th:cgmom}
is monotone increasing in the time length $\sfi\in [1, \eps^{-1}]$
(see the argument leading to \eqref{eq:Mepsh}). As a consequence,
we can apply the bound in Theorem~\ref{th:cgmom} with $w(x)=e^{-|x|}$
to $\bar{\mathscr{Z}}^{(\cg)}_{[0, \sfi]}(\varphi,\psi|\Theta_n)$,
which yields
\begin{equation}\label{eq:Thetanz4}
\begin{aligned}
\E \Big[ \Big(\bar{\mathscr{Z}}^{(\cg)}_{[0, \sfi]}(\ind_{\cS_\eps(0)}, \varphi(\cdot -\sqrt{\eps}\sfa)| \Theta_n)\Big)^4 \Big]
& \leq C \eps^{\frac4p} \Big\Vert \frac{\ind_{\cS_1(0)}}{w} \Big\Vert^4_{\ell^p} \Vert \varphi\Vert_\infty^4 \Vert w\ind_{\{B_\varphi-\sqrt{\eps}\sfa\}}\Vert_q^4 \\
& \leq C \Vert \varphi\Vert_\infty^4  e^{-{\rm dist}(\sqrt{\eps} \sfa, B_\varphi)} \eps^{\frac4p},
\end{aligned}
\end{equation}
where $C$ depends on the choice of $p,q$, but is uniform in $\sfi$ and $\sfa$ and in $n$ sufficiently large,
while $B_\varphi$ is a ball that contains the support of $\varphi$.

The second difference is also immaterial, that is we can replace $\Theta_n$ by
$W_{s,t,u}^{\sfz_1, \sfz_2}$ in \eqref{eq:Thetanz4}.
This is recorded in the following result, which we prove later.

\begin{lemma}\label{th:switch}
The bound \eqref{eq:Thetanz4} also holds if $\Theta_n$ is replaced by
$W_{s,t,u}^{\sfz_1, \sfz_2}$, uniformly in $\sfz_1, \sfz_2 \in \bbT_\epsilon$ and $s,t,u\in [0,1]$,
and $n$ large.
\end{lemma}

\noindent
Similarly, we can also bound the second factor in the r.h.s. \eqref{eq:Thetanz3} by
\begin{equation}\label{eq:Thetanz5}
	\E \Big[\! \Big(\!\bar{\mathscr{Z}}^{(\cg)}_{[0, 1/\eps-\sfi']}(\ind_{\cS_\eps(0)},
	\psi(\cdot -\sqrt{\eps}\sfa')| W_{s,t,u}^{\sfz_1, \sfz_2})\!\Big)^4 \Big]
	\leq C \eps^{\frac4p} \Big\Vert \frac{\ind_{\cS_1(0)}}{w} \Big\Vert^4_{\ell^p}
	\Vert \psi\Vert_\infty^4 \Vert w\Vert_q^4
	\leq C \Vert \psi\Vert_\infty^4 \eps^{\frac4p}.
\end{equation}
Substituting these bounds into  \eqref{eq:Thetanz3}  and then \eqref{eq:Thetanz2} gives
(for $n$ sufficiently large)
\begin{equation*}
\begin{aligned}
	\sum_{\sfz_1 \in \bbT_\eps,\, \sfz_2 \in A_{\sfz_1} \atop  \sfz_3 \in A_{\sfz_1\sfz_2}}
	\!\!\!\!\!  \frac{1}{3} \sum_{i=1}^3 \sup_{s,t,u} \E \Big[ \big|\partial_{\sfz_i} \Phi \big(
	W_{s,t,u}^{\sfz_1, \sfz_2} \big) \big|^4 \Big]^{3/4}
	& \leq  C\Vert \varphi\Vert_\infty^3 \Vert \psi\Vert_\infty^3  \eps^{\frac6p-3}
	\!\!\!\!\!\! \sum_{\sfz_1=(\vec\sfi_1, \vec\sfa_1) \in \bbT_\eps,\, \sfz_2 \in
	A_{\sfz_1} \atop \sfz_3 \in A_{\sfz_1\sfz_2}}  \!\!\!\!\!\!\!\!
	e^{-\frac{3}{4}{\rm dist}(\sqrt{\eps} \sfa_1, B_\varphi)}  \\
	& \leq C \Vert \varphi\Vert_\infty^3 \Vert \psi\Vert_\infty^3  \eps^{\frac6p-5}
	M_\eps^{10} K_\eps^6,
\end{aligned}
\end{equation*}
where we used the symmetry in the dependency structure between $\sfz_1, \sfz_2, \sfz_3$
($\sfz_2 \in A_{\sfz_1}$ if and only if $\sfz_1\in A_{\sfz_2}$) to reduce $\sum_{i=1}^3$ to the case $i=1$,
and in the last inequality, we first summed out $\sfz_2, \sfz_3$ and applied the bounds
$|A_{z_1}| , |A_{\sfz_1\sfz_2}| \leq C M_\eps^4 K_\eps^2$ from \eqref{eq:Az1},
then summed out $(\sfi_1', \sfa_1')$ in $\sfz_1=((\sfi_1, \sfi_1'), (\sfa_1, \sfa_1'))$ where the sum over
$\sfi_1'$ gives another factor of $K_\eps$ and the sum over $\sfa_1'$ gives another factor of
$M_\eps^2 \, K_\epsilon$,
and lastly we summed out $(\sfi_1, \sfa_1)$, noting that
$\sfi_1 \in \{1,\ldots, \epsilon^{-1}\}$
while the sum over $\sfa_1$ gives a factor $O((\sqrt{\epsilon})^{-2}) = O(\epsilon^{-1})$
because of the exponential decay on the scale $(\sqrt{\epsilon})^{-1}$
(we recall that $\varphi$ has compact support).

We recall from \eqref{eq:KM}
that $K_\eps=(\log \frac{1}{\eps})^6$ and $M_\eps=\log \log \frac{1}{\eps}$.
Choose $p>1$ sufficiently close to $1$ and substitute the above bound into \eqref{eq:Thetanz2} and
then \eqref{eq:I1bdcg}, together with \eqref{eq:Thetanz1}, we then obtain that
for any $\delta\in (0,1)$, there exists $C_\delta$ such that
$$
	\limsup_{n\to\infty} |I^{(n)}_1| \leq C_\delta \, \eps^{\delta} \,.
$$
This proves the first relation in \eqref{eq:lasttoprove}.

\medskip
\noindent
{\bf Bound on $\limsup\limits_{n\to\infty} I_2^{(n)}$.} By \eqref{eq:I2bd},
\begin{equation}\notag
\begin{aligned}
	|I_2^{(n)}| \, &\leq\,  \frac{1}{2} \|f'''\|_\infty \sup_{\sfz_1 \in \bbT_\eps}
	\E\big[|\Theta_n(\sfz_1)|^3\big] \sum_{\sfz_1 \in \bbT_\eps,\, \sfz_2 \in
	A_{\sfz_1}, \, \sfz_3 \in A_{\sfz_1\sfz_2}}  \\
	& \quad    \sup_{t,u} \E \Big[ \big|\partial_{\sfz_1} \Phi \big(
	W_{t,u}^{\sfz_1, \sfz_2} \big) \big|^3 \Big]^{\frac13}
	\sup_{t,u} \E \Big[ \big|\partial_{\sfz_2} \Phi \big(
	W_{t,u}^{\sfz_1, \sfz_2} \big) \big|^3 \Big]^{\frac13}
	\sup_{t,u} \E \Big[ \big|\partial_{\sfz_3} \Phi \big(
	W_{t,u}^{\sfz_1, \sfz_2} \big) \big|^3 \Big]^{\frac13},
\end{aligned}
\end{equation}
where  $W_{t,u}^{\sfz_1, \sfz_2} := u\sqrt{t}\Theta_n^{A_{\sfz_1\sfz_2}}
+ \sqrt{t} \Theta_n^{A_{\sfz_1\sfz_2}^c} +\sqrt{1-t}\Theta^{({\rm G)}}_n$ for $t,u\in [0,1]$.
The bounds are exactly the same as for $|I^{(1)}_n|$, which gives
$\limsup_{n\to\infty} |I^{(n)}_2| \leq C_\delta \eps^{\delta}$ for any $\delta \in (0,1)$.
This proves the second relation in \eqref{eq:lasttoprove}.

\medskip
\noindent
{\bf Bound on $\limsup\limits_{m,n\to\infty}I_3^{(m,n)}$.} By \eqref{eq:I3bd} and the fact that
$\Theta_n$ is a family of mean zero uncorrelated random variables,
\begin{equation}\label{eq:I3bdcg}
\begin{aligned}
	|I_3^{(m,n)}| \, &\leq\,  \frac{1}{2} \|f''\|_\infty \sum_{\sfz\in \bbT_\eps}
	\big(\bbE\big[\Theta^2_n(\sfz) \big]- \bbE\big[\Theta^2_m(\sfz) \big]\big)
	\sup_{t\in [0,1]}
	\E \big[ \big|\partial_{\sfz} \Phi(W_t)\big|^2\big]
\end{aligned}
\end{equation}
where  $W_t:= \sqrt{t}\Theta^{({\rm G)}}_n + \sqrt{1-t}\Theta^{({\rm G)}}_m$ for $t\in [0,1]$. Note that by
definition, $\Phi(\Theta)$ depends only on a finite set of $\Theta_\sfz$, $\sfz\in \bbT_\eps$. Therefore the
sum in \eqref{eq:I3bdcg} contains finitely many terms. For each $\sfz\in \bbT_\eps$,
$$
\lim_{m,n\to\infty} \big(\bbE\big[\Theta^2_n(\sfz) \big]- \bbE\big[\Theta^2_m(\sfz) \big]\big) = 0
$$
by Lemma~\ref{lem:theta}. On the other hand, uniformly in $t\in [0,1]$,
$\E \big[ \big|\partial_{\sfz} \Phi(W_t)\big|^2\big]$ converges to a finite limit as $m,n\to\infty$ because
$\partial_{\sfz} \Phi(W_t)$ is a multilinear polynomial in $W_t(\sfz)$ for finitely many $\sfz\in \bbT_\eps$,
while its second moment is a multilinear polynomial of $\E[W^2_t(\sfz)]$, $\sfz\in \bbT_\eps$, each of
which converges by Lemma~\ref{lem:theta}. It follows that $\limsup_{m,n\to\infty} |I_3^{(m,n)}|=0$,
which is stronger than the third relation in \eqref{eq:lasttoprove}.

\medskip

\noindent
\textbf{Conclusion.} Assuming Lemma~\ref{th:switch}, we have proved \eqref{eq:lasttoprove}. This implies \eqref{eq:fZlind} and
finally completes the proof of Proposition~\ref{p:Thm1.3}.
\end{proof}

\medskip

\begin{remark}[Extension to finite-dimensional distribution]\label{rem:fdd}
Finally, to prove the finite-dimensional distribution convergence in Theorem~\ref{th:main},
we argue as in the proof of Proposition~\ref{p:Thm1.3}.
First we approximate the components $\cZ_{N, s_i, t_i}^{\beta_N}(\varphi_i, \psi_i)$ of
the random vector in the l.h.s.\ of \eqref{eq:main-statement} by coarse-grained models
$\mathscr{Z}_{\epsilon, s_i, t_i}^{(\cg)}(\varphi_i,\psi_i|\Theta)$,
with the same coarse-grained disorder variables $\Theta = \Theta_{N,\epsilon}^{(\cg)}$,
which we can do with a small $L^2$ error, uniformly in large~$N$, provided
we choose $\epsilon > 0$ small enough, by Theorem~\ref{th:cg-main}.

It remains to apply a Lindeberg principle for a vector of multilinear polynomials,
which is given in Remark~\ref{rem:multiLind}.
The estimates needed are exactly the same as in the Lindeberg principle for a single multilinear
polynomial. This concludes the proof of Theorem~\ref{th:main}.
\end{remark}

\medskip

\begin{proof}[Proof of Lemma~\ref{th:switch}]
We re-examine the proof of Theorem~\ref{th:cgmom}. First note that, similar to the $L^2$
orthogonal decomposition of $\Theta_n(\sfi, \sfa)$ with $|\vec\sfi|\geq 1$ as in \eqref{eq:Theta2t},
we can write
\begin{equation}\label{eq:W2t}
W_{s,t,u}^{\sfz_1, \sfz_2}(\vec \sfi, \vec \sfa) = \sum_{\substack{\sfb: \, |\sfb - \sfa| \le M_\epsilon, \,
	\sfb': \, |\sfb' - \sfa'| \le M_\epsilon  \\ |\sfb' - \sfb| \le M_\epsilon \sqrt{|\vec\sfi|}}}
W_{s,t,u}^{\sfz_1, \sfz_2}(\vec \sfi, (\sfa, \sfb), (\sfb', \sfa')),
\end{equation}
where $W_{s,t,u}^{\sfz_1, \sfz_2}(\vec \sfi, (\sfa, \sfb), (\sfb', \sfa'))$ are defined through the same mixture as in \eqref{eq:Wmix} between $\Theta_n(\vec \sfi, (\sfa, \sfb), (\sfb', \sfa'))$ and $\Theta^{({\rm G)}}_n(\vec \sfi, (\sfa, \sfb), (\sfb', \sfa'))$, with the latter being independent normals with mean $0$ and the same variance as $\Theta_n(\vec \sfi, (\sfa, \sfb), (\sfb', \sfa'))$. We can then carry out the same expansion as for \eqref{eq:Mepsh}, and note that: whenever a product of coarse-grained disorder variables $\Theta_n$ has zero expectation because of the presence of some $\Theta_n(\sfi, \sfa)$ with either $(\sfi, \sfa)$ or $(\sfi', \sfa')$ unmatched by any other $\Theta_n$ in the product, the same is true if the family $\Theta_n$ is replaced by $W_{s,t,u}^{\sfz_1, \sfz_2}$; similarly, whenever two collections of $\Theta_n$ variables are independent of each other, the same is true if $\Theta_n$ is replaced by $W_{s,t,u}^{\sfz_1, \sfz_2}$. This implies that the expansion and re-summation carried out for the r.h.s.\ of \eqref{eq:Mepsh}, as well as the accompanying constraints on summation indices, also apply when $\Theta_n$ is replaced by $W_{s,t,u}^{\sfz_1, \sfz_2}$.

Next we claim that, for $W_{s,t,u}^{\sfz_1, \sfz_2}((\sfi_1, \sfi_2), (\sfa_1, \sfb_1), (\sfa_2, \sfb_2)) $ that visits two distinct blocks $\cI_1, \cI_2$ (see the exposition leading to \eqref{eq:Theta2texp}), although we no longer have a chaos expansion representation as in \eqref{eq:Theta2texp} due to the Gaussian component of $W_{s,t,u}^{\sfz_1, \sfz_2}$, we can still make a  replacement similar to \eqref{eq:Th2Th} in order to bound the r.h.s.\ of \eqref{eq:Mepsh}, with $\Theta_n$ replaced by $W_{s,t,u}^{\sfz_1, \sfz_2}$:
\begin{equation}\label{eq:W2W}
\begin{aligned}
& W_{s,t,u}^{\sfz_1, \sfz_2}((\sfi_1, \sfi_2), (\sfa_1, \sfb_1), (\sfa_2, \sfb_2))  \\
\stackrel{\leq}{\rightsquigarrow} \quad &   C W_{1,1/2,1}^{\sfz_1, \sfz_2}(\sfi_1; \sfa_1, \sfb_1)\  g_{10(\sfi_2 - \sfi_1)}\, (\sfa_2 -\sfb_1) \ W_{s,t,u}^{\sfz_1, \sfz_2}(\sfi_2; \sfa_2, \sfb_2).
\end{aligned}
\end{equation}
To see this, note that $W_{s,t,u}^{\sfz_1, \sfz_2}$ is a mixture of $\Theta_n$ and $\Theta^{({\rm G)}}_n$ with mixture coefficients given in \eqref{eq:Wmix}. When we expand the r.h.s.\ of \eqref{eq:Mepsh} with $W_{s,t,u}^{\sfz_1, \sfz_2}$ in place of $\Theta^{\rm(cg)}_{N, \eps}$, we can further expand $W_{s,t,u}^{\sfz_1, \sfz_2}$ in terms of its mixture. Each term in the expansion then consists of a product of $\Theta_n$ variables and $\Theta^{({\rm G)}}_n$ variables, whose expectation factorises due to the independence between $\Theta_n$ and $\Theta^{({\rm G)}}_n$. For terms that contain the factor $\Theta_n((\sfi_1, \sfi_2), (\sfa_1, \sfb_1), (\sfa_2, \sfb_2))$, the decomposition \eqref{eq:Theta2texp} applies, and the same argument justifying the replacement \eqref{eq:Th2Th} can be applied here. For terms that contain the Gaussian factor $\Theta^{({\rm G)}}_n((\sfi_1, \sfi_2), (\sfa_1, \sfb_1), (\sfa_2, \sfb_2))$, the expectation will be non-zero only if this factor appears exactly twice or four times. The resulting contribution will be either the second or the fourth moment of $\Theta^{({\rm G)}}_n((\sfi_1, \sfi_2), (\sfa_1, \sfb_1), (\sfa_2, \sfb_2))$. Its second moment coincide with that of $\Theta_n((\sfi_1, \sfi_2), (\sfa_1, \sfb_1), (\sfa_2, \sfb_2))$, while its fourth moment can be bounded in terms of its second moment by Gaussianity. Then as in \eqref{eq:Th2Th}, we make the following replacement in the expansion:
\begin{equation}\label{eq:hThetafac}
\Theta^{({\rm G)}}_n((\sfi_1, \sfi_2), (\sfa_1, \sfb_1), (\sfa_2, \sfb_2)) \stackrel{\leq}{\rightsquigarrow} C \Theta^{({\rm G)}}_n (\sfi_1; \sfa_1, \sfb_1)\  g_{10(\sfi_2 - \sfi_1)}\, (\sfa_2 -\sfb_1) \ \Theta^{({\rm G)}}_n(\sfi_2; \sfa_2, \sfb_2).
\end{equation}
Therefore in the mixture of $W_{s,t,u}^{\sfz_1, \sfz_2}((\sfi_1, \sfi_2), (\sfa_1, \sfb_1), (\sfa_2, \sfb_2))$ (recall \eqref{eq:Wmix}), we can replace the $\Theta_n$ and $\Theta^{({\rm G)}}_n$ components each by its factorisation as in \eqref{eq:hThetafac}. The mixture coefficient of the term $C\Theta_n (\sfi_1; \sfa_1, \sfb_1)g_{10(\sfi_2 - \sfi_1)}\, (\sfa_2 -\sfb_1)\Theta_n(\sfi_2; \sfa_2, \sfb_2)$ is equal to $\alpha \sqrt{t}$ with either $\alpha=su$, $u$ or 1, while the mixture coefficient of the term $C\Theta^{(\rm G)}_n (\sfi_1; \sfa_1, \sfb_1)g_{10(\sfi_2 - \sfi_1)}\, (\sfa_2 -\sfb_1)\Theta_n^{(\rm G)}(\sfi_2; \sfa_2, \sfb_2)$ equals $\sqrt{1-t}$.
This mixture can be further replaced by the r.h.s.\ of \eqref{eq:W2W}, which just contains extra terms and larger coefficients, where the choice of $s=u=1$ and $t=1/2$ in the first factor $W_{1,1/2,1}^{\sfz_1, \sfz_2}(\sfi_1; \sfa_1, \sfb_1)$ helps to bound the mixture coefficients.

The remaining parts of the proof of Theorem~\ref{th:cgmom} depends on the coarse-grained disorder variables $\Theta_n=\Theta^{\rm (cg)}_{n,\eps}$ only through their second and fourth moments. Note that uniformly in $s,t,u\in [0,1]$ and $\sfz_1, \sfz_2$, the second moment of $W_{s,t,u}^{\sfz_1, \sfz_2}$ is bounded by that of $\Theta_n$, and modulo a constant multiple, the fourth moment of $W_{s,t,u}^{\sfz_1, \sfz_2}$ can be bounded by that of $\Theta_n$. Therefore the remaining parts of the proof of Theorem~\ref{th:cgmom} carries through without change if $\Theta_n$ is replaced by $W_{s,t,u}^{\sfz_1, \sfz_2}$. In particular, the bound on $U^{\rm(cg)}_{N, \eps}$ in \eqref{eq:Ucg} still holds since it only depends on the second moment of $\Theta_n$, and the bound on $V^{\rm (cg)}_{N,\eps}$ in \eqref{eq:Vcg} still holds because it only depends on the fourth moment of $\Theta_n$.
\end{proof}

\subsection{Proof of Theorem~\ref{th:main1}}
The translation invariance of the law of $\mathscr{Z}^\theta$ is obvious.

To prove the scaling relation \eqref{eq:scaling}, let us write $\beta_N(\theta)$ to emphasize that $\beta_N$ as specified in \eqref{intro:sigma} depends on a parameter $\theta$. Given $\sfa>0$, let $\tilde N=N/\sfa$.
Then using \eqref{intro:sigma} and the fact that $R_N=\frac{1}{\pi}(\log N+ \alpha +o(1))$ given later in \eqref{eq:RN}, we have
\begin{equation}\label{eq:betachange}
\sigma_N^2 =  \frac{1}{R_N} \bigg(1 + \frac{\theta + o(1)}{\log N}\bigg) = \frac{1}{R_{\tilde N}} \bigg(1 + \frac{\theta -\log \sfa + o(1)}{\log \tilde N}\bigg),
\end{equation}
so that $\beta_N(\theta) = \beta_{\tilde N}(\theta - \log \sfa)$, or equivalently, $\beta_N(\theta+\log a) = \beta_{\tilde N}(\theta)$.

By \eqref{eq:main-statement}, for $\varphi, \psi \in C_c(\R^2)$, we have
\begin{equation}\label{eq:sca1}
\begin{aligned}
\iint \varphi\Big(\frac{x}{\sqrt a}\Big) \psi\Big(\frac{y}{\sqrt a}\Big) \cZ^{\beta_{\tilde N}(\theta), \omega}_{\tilde N;\, \sfa s, \sfa t}(\dd x, \dd y) \Asto{\tilde N} &
\iint \varphi\Big(\frac{x}{\sqrt a}\Big) \psi\Big(\frac{y}{\sqrt a}\Big) \mathscr{Z}_{\sfa s, \sfa t}^\theta(\dd x , \dd y)  \\
= & \iint \varphi(x) \psi(y) \mathscr{Z}_{\sfa s, \sfa t}^\theta(\dd (\sqrt{\sfa} x) , \dd (\sqrt{\sfa} y)).
\end{aligned}
\end{equation}
On the other hand, recall \eqref{eq:rescZmeas}, we can rewrite the l.h.s.\ of \eqref{eq:sca1} as
\begin{align*}
 &	\tilde N \iint \varphi\Big(\frac{x}{\sqrt a}\Big)   \psi\Big(\frac{y}{\sqrt a}\Big) Z_{\ev{\tilde N\sfa s}, \ev{\tilde N\sfa t}}^{\beta_{\tilde N}(\theta),\,\omega}(\ev{\sqrt{\tilde N} x},
	\ev{\sqrt{\tilde N} y}) \, \dd x \, \dd y \\
	=\ \  & \frac{1}{\tilde N}\sum_{\tilde x, \tilde y\in \Z^2_{\rm even}} Z_{\ev{\tilde N\sfa s}, \ev{\tilde N\sfa t}}^{\beta_{\tilde N}(\theta),\,\omega}(\tilde x, \tilde y)
	\!\!\!\!\! \int\limits_{|u-\tilde x|_1\leq 1}\!\!\!\!\! \varphi\Big(\frac{u}{\sqrt {\sfa \tilde N}}\Big) \dd u  \!\!\!\!\! \int\limits_{|v-\tilde y|_1\leq 1}\!\!\!\!\! \psi\Big(\frac{v}{\sqrt {\sfa \tilde N}}\Big) \dd v\\
	=\ \  & \frac{\sfa }{N}\sum_{\tilde x, \tilde y\in \Z^2_{\rm even}} Z_{\ev{Ns}, \ev{Nt}}^{\beta_N(\theta+ \log \sfa),\,\omega}(\tilde x, \tilde y)
	\!\!\!\!\! \int\limits_{|u-\tilde x|_1\leq 1}\!\!\!\!\! \varphi\Big(\frac{u}{\sqrt {N}}\Big) \dd u  \!\!\!\!\! \int\limits_{|v-\tilde y|_1\leq 1}\!\!\!\!\! \psi\Big(\frac{v}{\sqrt {N}}\Big) \dd v\\
	\Asto{N} &  \iint \varphi(x) \psi(y) \, \sfa \mathscr{Z}_{s,t}^{\theta+\log \sfa}(\dd x, \dd y),
\end{align*}
where we again applied \eqref{eq:main-statement}. Combined with \eqref{eq:sca1},
this implies \eqref{eq:scaling}.

The first and second moments of $\mathscr{Z}^\theta$ can be identified from the limits in \eqref{eq:m1-lim0} and \eqref{eq:m2-lim0}, since for $\varphi\in C_c(\R^2)$ and $\psi\in C_b(\R^2)$, the averaged partition function $\mathcal{Z}^{\beta_N}_{N, t-s}(\varphi, \psi)$ has a finite fourth moment that is uniformly bounded in $N$, see Theorem~\ref{th:mom}.

\appendix

\section{Enhanced Lindeberg principle}
\label{app:Lind}

In this appendix, we prove a Lindeberg principle for multilinear polynomials of dependent random variables
with a local form of dependence. This extends \cite{R13}, which requires the function to have bounded first three derivatives and is not applicable to
multilinear polynomials, and it extends \cite{MOO10}, which considers multilinear polynomials of
independent random variables. We first introduce the necessary setup.
Let $\bbT$ be a finite index set.

\begin{assumption}[Local dependence] \label{ass:dep}
Let $X = (X_i)_{i\in \bbT}$ be random variables satisfying:
\begin{itemize}
\item $\E[X_i] = 0$ and $\E[X_i X_j] = \sigma_{ij}$;

\item for every $k \in \bbT$ there is $A_k \subseteq \bbT$ such that
$X_k$ is independent of $(X_i)_{i\in A_k^c}$;

\item for all $k \in \bbT$, $l \in A_k$ there is $ A_{kl} \subseteq \bbT$ such that
$(X_k, X_l)$ is independent of $(X_i)_{i\in A_{kl}^c}$.
\end{itemize}
The sets $(A_k)_{k\in\T}$ and $(A_{kl})_{k\in\T, l\in A_k}$\!
will be called \emph{dependency neighbourhoods}
of $X = (X_i)_{i\in \bbT}$.
\end{assumption}

Let $Z = (Z_i)_{i\in\bbT} \sim N(0, \sigma_\cdot)$ be a Gaussian vector independent of $X$,
but with the same covariance matrix as $X$.
For $u, t \in [0,1]$, $k\in \bbT$ and $l\in A_k$, we define
\begin{equation} \label{eq:Wtilde}
	W_{t,u}^{k, l} := u\sqrt{t} X^{A_{kl}} +  \sqrt{t} X^{A_{kl}^c} +\sqrt{1-t}Z,
\end{equation}
where for any subset of indices $B \subseteq \bbT$ we write
\begin{equation}
	X^B_i := X_i \, \ind_{\{i \in B\}} \,.
\end{equation}
For $s, u, t\in [0,1]$, $k\in \bbT$ and $l\in A_k$, we define
\begin{equation} \label{eq:W}
\begin{aligned}
	W_{s,t,u}^{k, l} &:=
	s u\sqrt{t}X^{A_k} + u\sqrt{t} X^{A_{kl}\backslash A_k} +\sqrt{t}X^{A_{kl}^c} +\sqrt{1-t}Z.
\end{aligned}
\end{equation}

We have the following Lindeberg type result, which controls the distributional distance between $X$ and $Z$ through smooth test functions.

\begin{lemma}[Lindeberg principle for dependent random variables]\label{lem:lind1}
Let $X$, $Z$, $W_{t,u}^{k, l}$ and $W_{s,t,u}^{k, l}$ be defined as above. Let $h: \R^{|\bbT|} \to \R$ be bounded and thrice differentiable. Then
\begin{equation*}
	\E[h(X)] - \E[h(Z)] = I_1 + I_2
\end{equation*}
where
\begin{align} \label{eq:I1}
	I_1 & := \frac{1}{2} \int_{[0,1]^3}
	\sum_{k \in \bbT,\, l \in A_k, \, m \in A_{kl}}
	\!\!\!\!\!\!\!\! \E \Big[ X_k \, X_l \,
	X_m \big(s \ind_{\{m \in A_k\}}+ \ind_{\{m\in A_{kl}\backslash A_k \}}\big) \, \sqrt{t} \,
	\partial^3_{kl m} h\big(
	W_{s,t,u}^{k, l} \big) \Big] \, \dd s \, \dd t \, \dd u \,, \\
	\label{eq:I2}
	I_2 &:= -\frac{1}{2} \int_{[0,1]^2}
	\sum_{k \in \bbT,\, l \in A_k, \, m \in A_{kl}} \sigma_{kl} \,
	\E \Big[ X_m \,
	\sqrt{t} \,
	\partial^3_{kl m} h\big(W_{t,u}^{k, l} \big) \Big] \dd t  \, \dd u \,,
\end{align}
assuming that the integrals and expectations above are all finite.
\end{lemma}

To control the distributional distance between $X$ and another random vector
$\widetilde X=(\widetilde X_i)_{i\in \bbT}$ with slightly perturbed covariance matrix
$(\tilde \sigma_{ij})_{i, j\in \bbT}$ and dependency neighbourhoods $\tilde A_k, \tilde A_{kl}$, we
will apply Lemma~\ref{lem:lind1} to $X$ and $\widetilde X$ separately,
plus the following result that compares two Gaussian vectors.

\begin{lemma}\label{lem:lind2} Let $Z=(Z_i)_{i\in \bbT}\sim N(0, \sigma_\cdot)$ and $\widetilde Z=(\widetilde Z_i)_{i\in \bbT}\sim N(0, \widetilde \sigma_\cdot)$ be centred Gaussian random vectors with covariance matrices $(\sigma_{ij})_{i,j\in \bbT}$ and $(\tilde \sigma_{ij})_{i,j\in \bbT}$ respectively. Let $h: \R^{|\bbT|} \to \R$ be bounded and twice differentiable. Denote $W_t:=\sqrt{t}\widetilde Z +\sqrt{1-t}Z$. Then we have
\begin{equation}\label{eq:I3}
	\E[h(\widetilde Z )] - \E[h(Z)] =: I_3 =  \frac{1}{2} \sum_{k, l \in \bbT} (\tilde \sigma_{kl}-\sigma_{kl}) \int_0^1 \E\Big[\partial^2_{kl} h(W_t) \Big] \dd t,
\end{equation}
assuming that the integrals and expectations above are all finite.
\end{lemma}

We now specialise Lemmas~\ref{lem:lind1} and~\ref{lem:lind2} to our case of interest, where $h$ is a function of a multi-linear polynomial $\Phi(X)$.

\begin{lemma}\label{lem:lind3}
Assume that $h(X) := f(\Phi(X))$ for some bounded $f:\R\to\R$ with bounded first three derivatives, and
\begin{equation} \label{eq:Phi}
	\Phi(X) = \sum_{I \subseteq \bbT} c_I \, \prod_{i\in I} X_i
\end{equation}
for some fixed coefficients $c_I \in \R$. Furthermore, assume that
\begin{equation}\label{eq:vanish}
\forall\, k\in \bbT, \, l\in A_k, \, m\in A_{kl}, \qquad 	\partial^2_{k m} \Phi = \partial^2_{l m} \Phi = \partial^2_{k l} \Phi = 0 \,.
\end{equation}
Then for $I_1$, $I_2$ and $I_3$ as in \eqref{eq:I1}, \eqref{eq:I2} and \eqref{eq:I3}, we have
\begin{align}
|I_1| \, &\leq\,  \frac{1}{2} \|f'''\|_\infty \sup_{k \in \bbT} \E\big[|X_k|^3\big] \sum_{k \in \bbT,\, l \in A_k, \, m \in A_{kl}}  \notag\\
& \quad \qquad    \sup_{s,t,u} \E \Big[ \big|\partial_k \Phi \big(
	W_{s,t,u}^{k, l} \big) \big|^3 \Big]^{\frac13}
	\sup_{s,t,u} \E \Big[ \big|\partial_l \Phi \big(
	W_{s,t,u}^{k, l} \big) \big|^3 \Big]^{\frac13}
	\sup_{s,t,u} \E \Big[ \big|\partial_m \Phi \big(
	W_{s,t,u}^{k, l} \big) \big|^3 \Big]^{\frac13},  \label{eq:I1bd} \\
|I_2| \, &\leq\,  \frac{1}{2} \|f'''\|_\infty \sup_{k \in \bbT} \E\big[|X_k|^3\big] \sum_{k \in \bbT,\, l \in A_k, \, m \in A_{kl}} \notag \\
& \quad \qquad    \sup_{t,u} \E \Big[ \big|\partial_k \Phi \big(
	W_{t,u}^{k, l} \big) \big|^3 \Big]^{\frac13}
	\sup_{t,u} \E \Big[ \big|\partial_l \Phi \big(
	W_{t,u}^{k, l} \big) \big|^3 \Big]^{\frac13}
	\sup_{t,u} \E \Big[ \big|\partial_m \Phi \big(
	W_{t,u}^{k, l} \big) \big|^3 \Big]^{\frac13},  \label{eq:I2bd} \\
	|I_3| \, &\leq\,  \frac{1}{2} \|f''\|_\infty 	\sum_{k\in \bbT,\, l \in A_k}
	(\tilde \sigma_{kl}-\sigma_{kl}) \sup_t \E\Big[(\partial_k \Phi (W_t))^2\Big]^{\frac12}
	\sup_t \E\Big[(\partial_l \Phi (W_t))^2\Big]^{\frac12}.  \label{eq:I3bd}
\end{align}
\end{lemma}

\begin{remark}\label{rem:multiLind}
We can extend Lemma~\ref{lem:lind3} to
the vector setting, i.e.\ for a function $h(X) = f(\Phi^{(1)}(X), \ldots, \Phi^{(k)}(X))$
of a finite number~$k$ of multi-linear polynomials $\Phi^{(i)}(X)$
as in \eqref{eq:Phi}, each satisfying \eqref{eq:vanish}, where $f: \R^k \to \R$
has bounded partial derivatives of order up to three. The bounds
\eqref{eq:I1bd}-\eqref{eq:I2bd} are simply modified replacing $\|f'''\|_\infty$ by
$\max_{1\le i,j,l \le k} \|\partial_{i,j,l} f\|_\infty$ and the three occurrences of $\Phi$
by $\Phi^{(i)}$, $\Phi^{(j)}$, $\Phi^{(l)}$, and then summing over $1 \le i,j,l \le k$;
similar modifications apply to the bound~\eqref{eq:I3bd}.
The adaptation of the proof is straightforward.
\end{remark}

\smallskip

We now give the proofs of Lemmas~\ref{lem:lind1}--\ref{lem:lind3}.
\begin{proof}[Proof of Lemma~\ref{lem:lind1}]
Let $Y_t:=\sqrt{t}X +\sqrt{1-t}Z$, $t\in [0,1]$. Then we can write
\begin{equation*}
\begin{split}
	\E[h(X)] - \E[h(Z)] &= \int_0^1 \frac{\dd}{\dd t} \E[h(Y_t)] \, \dd t \\
	&= \frac{1}{2} \int_0^1 \E \bigg[ \sum_{k \in \bbT}  \frac{X_k}{\sqrt{t}} \,  \partial_k h(Y_t) -
	\sum_{k \in \bbT} \frac{Z_k}{\sqrt{1-t}} \,  \partial_k h(Y_t) \bigg] \dd t  \,.
\end{split}
\end{equation*}

Given $C \subseteq \bbT$, let us denote
\begin{equation*}
	U^C_t := \sqrt{t} X^C + \sqrt{1-t} Z \,, \qquad
	\text{where} \qquad X^C_i := X_i \, \ind_{\{i \in C\}} \,.
\end{equation*}
In particular, $Y_t=U_t^\bbT$. Observe that $\E[X_k \, \partial_k h(U_t^{A_k^c}) ]
= \E[X_k] \E[\partial_k h(U_t^{A_k^c}) ] = 0$ by independence.
We recall Gaussian integration by parts: for smooth functions $\varphi$,
\begin{equation*}
	\E [ Z_k \, \varphi(Z) ] = \sum_{l \in \bbT}
	\sigma_{kl} \, \E [ \partial_l\varphi(Z)] = \sum_{l \in A_k}
	\sigma_{kl} \, \E [ \partial_l\varphi(Z)] \,,
\end{equation*}
since $\sigma_{kl} = 0$ for $l \not\in A_k$. Then
\begin{equation*}
\begin{split}
	\E[h(X)] - \E[h(Z)] &=
	\frac{1}{2} \int_0^1 \E \bigg[ \sum_{k \in \bbT}  \frac{X_k}{\sqrt{t}} \,
	\big( \partial_k h(U^\bbT_t) - \partial_k h(U^{A_k^c}_t) \big) -
	\sum_{k \in \bbT, \, l \in A_k} \sigma_{kl} \, \partial^2_{kl} h(U^\bbT_t) \bigg] \dd t  \,.
\end{split}
\end{equation*}

Let us expand the first term. We can interpolate between $U^\bbT_t$
and $U_t^{A_k^c}$ by
\begin{equation*}
	 U^{A_k^c}_t + s \sqrt{t} X^{A_k} \qquad \text{for } s \in [0,1] \,.
\end{equation*}
Then
\begin{equation*}
\begin{split}
	\partial_k h(U_t^\bbT) - \partial_k h(U^{A_k^c}_t)
	&= \int_0^1 \frac{\dd\ }{\dd s} \partial_k h(U^{A_k^c}_t + s \sqrt{t} X^{A_k}) \, \dd s \\
	&= \sqrt{t} \int_0^1 \sum_{l \in \bbT} X^{A_k}_l \,
	\partial^2_{kl} h\big( U^{A_k^c}_t + s \sqrt{t} X^{A_k} \big) \, \dd s \,.
\end{split}
\end{equation*}
Note that we can restrict the sum to $l \in A_k$ because $X^{A_k}_l = 0$ otherwise.
This leads to
\begin{equation*}
\begin{split}
	\E[h(X)] - \E[h(Z)] &=
	\frac{1}{2} \int_0^1 \int_0^1 \E \bigg[ \sum_{k \in \bbT,\, l \in A_k} X_k \, X_l \,
	\partial^2_{kl} h\big( U^{A_k^c}_t + s \sqrt{t} X^{A_k}\big) \\
	& \qquad\qquad\qquad\qquad -
	\sum_{k \in \bbT, \, l \in A_k} \sigma_{kl} \, \partial^2_{kl} h(U^\bbT_t) \bigg]
	\dd s \, \dd t  \,.
\end{split}
\end{equation*}

Note that $\E[X_k \, X_l \,\partial^2_{kl} h\big( U^{A_{kl}^c}_t \big) ]
= \E[X_k \, X_l] \E[\partial^2_{kl} h\big( U^{A_{kl}^c}_t \big) ]
= \sigma_{kl}
\E[\partial^2_{kl} h\big( U^{A_{kl}^c}_t \big) ]$
by the independence assumption.
By adding and subtracting this term, we get
\begin{equation*}
	\E[h(X)] - \E[h(Z)] = I_1 + I_2
\end{equation*}
where
\begin{equation*}
\begin{split}
	I_1 & := \frac{1}{2} \int_0^1 \int_0^1 \E \bigg[
	\sum_{k \in \bbT,\, l \in A_k}  X_k \, X_l \,
	\Big( \partial^2_{kl} h\big( U^{A_k^c}_t + s \sqrt{t} X^{A_k}\big) -
	\partial^2_{kl} h\big( U^{A_{kl}^c}_t \big) \Big) \bigg] \, \dd s \, \dd t \\
	I_2 &:= \frac{1}{2} \int_0^1 \int_0^1 \E \bigg[
	\sum_{k \in \bbT,\, l \in A_k} \sigma_{kl} \,
	\Big( \partial^2_{kl} h\big( U^{A_{kl}^c}_t \big) -  \partial^2_{kl} h(U^\bbT_t)
	\Big) \bigg] \dd s \, \dd t  \\
	&= \frac{1}{2} \int_0^1 \E \bigg[
	\sum_{k \in \bbT,\, l \in A_k} \sigma_{kl} \,
	\Big( \partial^2_{kl} h\big( U^{A_{kl}^c}_t \big) -  \partial^2_{kl} h(U^\bbT_t)
	\Big) \bigg] \dd t  \,,
\end{split}
\end{equation*}
where we performed the integration on $s$ in $I_2$
(whose integrand does not depend on $s$).

\smallskip

Let us deal with $I_1$. Note that
$$
U^{A_k^c}_t + s \sqrt{t} X^{A_k} = s \sqrt{t} X^{A_k} + \sqrt{t} X^{A_k^c} + \sqrt{1-t}Z \quad \mbox{and} \quad
U^{A_{kl}^c}_t = \sqrt{t} X^{A_{kl}^c} +\sqrt{1-t}Z.
$$
We can therefore interpolate between $U^{A_{kl}^c}_t$ and $U^{A_k^c}_t + s \sqrt{t} X^{A_k}$ by
$$
	W_{s,t,u}^{k, l} := U^{A_{kl}^c}_t + u\sqrt{t}(sX^{A_k} + X^{A_{kl}\backslash A_k}) \,, \qquad
	u \in [0,1] \,.
$$
Note that $(sX^{A_k} + X^{A_{kl}\backslash A_k})_m=0$ for $m \not\in A_{kl}$. We can then write
\begin{equation*}
\begin{split}
&	\partial^2_{kl} h\big( U^{A_k^c}_t + s \sqrt{t} X^{A_k} \big) -
	\partial^2_{kl} h\big( U^{A_{kl}^c}_t \big) =
	\int_0^1 \frac{\dd}{\dd u} \partial^2_{kl} h\big(
	W_{s,t,u}^{k, l} \big) \, \dd u \\
	= \ & \sqrt{t} \int_0^1 \sum_{m \in A_{kl}} X_m(s \ind_{\{m \in A_k\}}+ \ind_{\{m\in A_{kl}\backslash A_k \}}) \,
	\partial^3_{kl m} h\big( W_{s,t,u}^{k, l}\big) \, \dd u \, .
	\end{split}
\end{equation*}
This yields the final form of $I_1$:
\begin{equation*}
\begin{split}
	I_1 & := \frac{1}{2} \int_{[0,1]^3} \E \bigg[
	\sum_{k \in \bbT,\, l \in A_k, \, m \in A_{kl}}  \!\!\!\!\!\!\!\!\!\! X_k \, X_l \,
	X_m (s \ind_{\{m \in A_k\}}+ \ind_{\{m\in A_{kl}\backslash A_k \}}) \, \sqrt{t} \,
	\partial^3_{kl m} h\big(
	W_{s,t,u}^{k, l} \big) \bigg] \, \dd s \, \dd t \, \dd u.
\end{split}
\end{equation*}

Let us now deal with $I_2$. We can interpolate between $U^{A^c_{kl}}_t$ and $U^\bbT_t$ by
\begin{equation*}
	W_{t,u}^{k, l} := U^{A^c_{kl}}_t + u \, \sqrt{t} X^{A_{kl}} \,, \qquad
	u \in [0,1] \,.
\end{equation*}
Then
\begin{equation*}
\begin{split}
	\partial^2_{kl} h\big( U^{A_{kl}^c}_t \big) -  \partial^2_{kl} h(U^\bbT_t) &=
	- \int_0^1 \frac{\dd}{\dd u} \partial^2_{kl} h\big( W_{t,u}^{k, l}  \big)
	\, \dd u
	=  - \sqrt{t} \int_0^1 \sum_{m \in A_{kl}}
	X_m \,
	\partial^3_{kl m} h\big( W_{t,u}^{k, l}  \big) \, \dd u \,.
\end{split}
\end{equation*}
This yields the final form of $I_2$:
\begin{equation*}
\begin{split}
	I_2 &:= -\frac{1}{2} \int_{[0,1]^2} \E \bigg[
	\sum_{k \in \bbT,\, l \in A_k, \, m \in B_{kl}} \sigma_{kl} \, X_m \,
	\sqrt{t} \,
	\partial^3_{kl m} h\big(W_{t,u}^{k, l}  \big) \bigg] \dd t  \, \dd u \,.
\end{split}
\end{equation*}
This concludes the proof of Lemma~\ref{lem:lind1}.
\end{proof}
\medskip

\begin{proof}[Proof of Lemma~\ref{lem:lind2}]
Let $W_t:=\sqrt{t}\widetilde Z +\sqrt{1-t}Z$, $t\in [0,1]$. Using Gaussian integration by parts as in the proof of Lemma~\ref{lem:lind1}, we have
\begin{equation*}
\begin{split}
	\E[h(\widetilde Z)] - \E[h(Z)] &= \int_0^1 \frac{\dd}{\dd t} \E[h(W_t)] \, \dd t \\
	&= \frac{1}{2} \int_0^1 \E \bigg[ \sum_{k \in \bbT}  \frac{\widetilde Z_k}{\sqrt{t}} \,  \partial_k h(W_t) -
	\sum_{k \in \bbT} \frac{Z_k}{\sqrt{1-t}} \,  \partial_k h(W_t) \bigg] \dd t  \\
	& = \frac{1}{2} \sum_{k, l \in \bbT} (\tilde \sigma_{kl}-\sigma_{kl}) \int_0^1 \E\Big[\partial^2_{kl} h(W_t) \Big] \dd t  \,,
\end{split}
\end{equation*}
which proves the lemma.
\end{proof}
\medskip

\begin{proof}[Proof of Lemma~\ref{lem:lind3}]
We can easily compute
\begin{align*}
	\partial_k h(x) & = f'(\Phi(x)) \, \partial_k \Phi(x) \,, \\
	\partial^2_{k l} h(x) & = f''(\Phi(x)) \, \partial_k \Phi(x) \, \partial_l \Phi(x)
	+ f'(\Phi(x)) \, \partial^2_{kl} \Phi(x) \,, \\
	\partial^3_{k l m} h(x) & = f'''(\Phi(x)) \, \partial_k \Phi(x) \, \partial_l \Phi(x)
	\, \partial_m \Phi(x)  \\
	& \quad + f''(\Phi(x)) \big\{ \partial^2_{k m} \Phi(x) \, \partial_{l} \Phi(x)
	+ \partial^2_{l m} \Phi(x) \, \partial_{k} \Phi(x)
	+ \partial^2_{k l} \Phi(x) \, \partial_{m} \Phi(x) \big\} \\
	& \quad +   f'(\Phi(x)) \, \partial^3_{k l m} \Phi(x) \,.
\end{align*}
which by assumption \eqref{eq:vanish}, gives
\begin{equation}\label{eq:ourthird}
	\partial^3_{k l m} h(x) = f'''(\Phi(x)) \, \partial_k \Phi(x) \, \partial_l \Phi(x)
	\, \partial_m \Phi(x).
\end{equation}
We can then substitute this into \eqref{eq:I1} to bound
\begin{equation*}
\begin{split}
	& \Big| \E \Big[ X_k \, X_l \,
	X_m (1 - s \ind_{\{m \in A_k\}}) \, \sqrt{t} \,
	\partial^3_{kl m} h\big(
	W_{s,t,u}^{A_k, B_{kl}} \big) \Big] \Big| \\
	& \le \|f'''\|_\infty \,
	\bigg( \E \Big[ \big|X_k \,\partial_k \Phi \big(
	W_{s,t,u}^{k, l} \big) \big|^3 \Big]
	\, \E \Big[ \big|X_l \,\partial_l \Phi \big(
	W_{s,t,u}^{k, l} \big) \big|^3 \Big]
	\, \E \Big[ \big|X_m \,\partial_m \Phi \big(
	W_{s,t,u}^{k, l} \big) \big|^3 \Big] \bigg)^{1/3} \\
	& \le \|f'''\|_\infty \, \sup_{k \in \bbT} \E[|X_k|^3] \,
	\bigg( \E \Big[ \big|\partial_k \Phi \big(
	W_{s,t,u}^{k, l} \big) \big|^3 \Big]
	\, \E \Big[ \big|\partial_l \Phi \big(
	W_{s,t,u}^{k, l} \big) \big|^3 \Big]
	\, \E \Big[ \big|\partial_m \Phi \big(
	W_{s,t,u}^{k, l} \big) \big|^3 \Big] \bigg)^{1/3} \,,
\end{split}
\end{equation*}
where we used the fact that $\partial_k\Phi(\ldots)$ is independent of $X_k$ since assumption \eqref{eq:vanish} implies that
$\partial_k \Phi(\ldots)$ does not depend on $(X_m)_{m\in A_{kl}}$. This immediately implies \eqref{eq:I1bd}.

The proof of \eqref{eq:I2bd} is the same if in \eqref{eq:I2}, we write $\sigma_{kl}= \E[\widetilde X_k \widetilde X_l]$ for $(\widetilde X_k, \widetilde X_l)$ with the same distribution as $(X_k, X_l)$ but independent of $X$. The proof of \eqref{eq:I3bd} is even simpler.
\end{proof}

\section*{Acknowledgements}
We especially thank Te LI for showing us how Hardy-Littlewood-Sobolev
type inequalities can be proved without using Fourier transform.
F.C. is supported by INdAM/GNAMPA.
R.S.~is supported by NUS grant R-146-000-288-114. N.Z.~is supported
by EPRSC through grant EP/R024456/1.


\begin{thebibliography}{AGH+88}


%
\bibitem[AKQ14a]{AKQ14a}
T.~Alberts, K.~Khanin, J.~Quastel.
The intermediate disorder regime for directed polymers in dimension $1+1$.
{\em Ann. Probab.} 42, 1212--1256, 2014.

\bibitem[AKQ14b]{AKQ14b}
T.~Alberts, K.~Khanin, J.~Quastel.
The continuum directed random polymer.
{\em J. Stat. Phys.} 154, 305-326, 2014.


\bibitem[AFH+92]{AFHKKL92}
S. Albeverio, J.E. Fenstad, R. Hoegh-Krohn, W. Karwowski, T. Lindstrom.
Schr\"odinger operators with potentials supported by null sets.
{\em Ideas and methods in quantum and statistical physics}, 63-95,
Cambridge Univeristy Press, 1992.

\bibitem[AGH+05]{AGHKH05}
S. Albeverio, F. Gesztesy, R. Hoegh-Krohn, H. Holden.
{\em Solvable models in quantum mechanics.}
AMS Chelsea Publishing, 2005.

%
%
%

\bibitem[Ba21]{Ba21}
E. Bates.
Full-path localization of directed polymers.
{\em Electron. J. Probab.} 26, 1-24, 2021.

\bibitem[Bo89]{B89}
E. Bolthausen.
A note on the diffusion of directed polymers in a random environment.
{\em Comm. Math. Phys.} 123, 529-534, 1989.


\bibitem[BC98]{BC98}
L. Bertini, N. Cancrini.
The two-dimensional stochastic heat equation: renormalizing a multiplicative noise.
{\em J. Phys. A: Math. Gen.} 31, 615, 1998.

\bibitem[BC20a]{BC20a}
E. Bates, S. Chatterjee.
The endpoint distribution of directed polymers.
{\em Ann. Probab.} 48, 817-871, 2020.

\bibitem[BC20b]{BC20b}
E. Bates, S. Chatterjee.
Localization in Gaussian disordered systems at low temperature.
{\em Ann. Probab.} 48, 2755-2806, 2020.


\bibitem[BL17]{BL17}
Q. Berger, H. Lacoin.
The high-temperature behavior for the directed polymer in dimension $1+2$.
{\em Ann. Institut Henri Poincar\'e, Prob. et Statistiques} 53, 430-450, 2017.

\bibitem[CES21]{CES21}
G. Cannizzaro,  D. Erhard, P. Sch\"onbauer.
$2D$ anisotropic KPZ at stationarity: Scaling, tightness and nontriviality.
{\em Ann. Prob.} 49, 122-156, 2021.

\bibitem[CET20a]{CET20a}
G. Cannizzaro,  D. Erhard, F.L. Toninelli.
Logarithmic superdiffusivity of the $2-$dimensional anisotropic KPZ equation.
ArXiv:2009.12934, 2020.

\bibitem[CET20b]{CET20b}
G. Cannizzaro,  D. Erhard, F.L. Toninelli.
The stationary AKPZ equation: logarithmic superdiffusivity.
ArXiv:2007.12203, 2020.

\bibitem[CET21]{CET21}
G. Cannizzaro,  D. Erhard, F.L. Toninelli.
Weak coupling limit of the Anisotropic KPZ equation.
ArXiv:2108.09046, 2021.

\bibitem[CSZ16]{CSZ16}
F.~Caravenna, R.~Sun,  N.~Zygouras.
The continuum disordered pinning model.
{\em Probab. Theory Related Fields} 164, 17--59, 2016.

\bibitem[CSZ17a]{CSZ17a}
F. Caravenna, R. Sun, N. Zygouras.
Polynomial chaos and scaling limits of disordered systems.
{\em J. Eur. Math. Soc.} 19, 1-65, 2017.

\bibitem[CSZ17b]{CSZ17b}
F. Caravenna, R. Sun, N. Zygouras.
Universality in marginally relevant disordered systems.
{\em Ann. Appl. Probab.}  27, 3050-3112, 2017.

\bibitem[CSZ19a]{CSZ19a}
F. Caravenna, R. Sun, N. Zygouras.
The Dickman subordinator, renewal theorems, and disordered systems.
{\em Electron. J. Probab.} 24, paper no. 101, 2019.

\bibitem[CSZ19b]{CSZ19b}
F. Caravenna, R. Sun, N. Zygouras.
On the moments of the (2+1)-dimensional directed polymer and
stochastic heat equation in the critical window.
{\em Commun. Math. Phys.} 372, 385-440, 2019.

\bibitem[CSZ20]{CSZ20}
F. Caravenna, R. Sun, N. Zygouras.
The two-dimensional KPZ equation in the entire subcritical regime.
{\em Ann. Prob.} 48, 1086-1127, 2020.

\bibitem[CSZ22]{CSZ22}
F. Caravenna, R. Sun, N. Zygouras.
The critical 2d Stochastic Heat Flow is not a Gaussian Multiplicative Chaos.
ArXiv:2206.08766, 2022.

\bibitem[CH02]{CH02}
P. Carmona, Y. Hu,
On the partition function of a directed polymer in a Gaussian random environment.
{\em Prob. Th. Rel. Fields}, 124(3), 431-457, (2002).

\bibitem[Cha06]{Cha06}
S.~Chatterjee.
A generalization of the Lindeberg principle.
{\em Ann. Probab.} 34, 2061-2076, 2006.

\bibitem[Cha19]{Cha19}
S.~Chatterjee.
Proof of the path localization conjecture for directed polymers.
{\em Comm. Math. Phys.} 370, 703-717, 2019.

\bibitem[CD20]{CD20}
S. Chatterjee, A. Dunlap.
Constructing a solution of the $(2+1)$-dimensional KPZ equation.
{\em Ann. Prob.} 48, 1014-1055, 2020.

\bibitem[Che21]{C21}
Y.-T. Chen.
The critical 2D delta-Bose gas as mixed-order asymptotics of planar Brownian motion.
arXiv:2105.05154, 2021.

	
\bibitem[Cla19b]{Cla19b}
J.T. Clark.
The conditional Gaussian multiplicative chaos structure underlying a critical continuum random polymer
model on a diamond fractal.
ArXiv:1908.08192, 2019.

\bibitem[Cla21]{Cla21}
J. T. Clark.
Weak-disorder limit at criticality for directed polymers on hierarchical graphs.
{\em Comm. Math. Phys.} 386, 651-710, 2021.

\bibitem[Cla22]{Cla19a}
J.T. Clark.
Continuum models of directed polymers on disordered diamond fractals in the critical case.
{\em Ann. Appl. Prob.} 32, 4186-4250, 2022.

\bibitem[Com17]{C17}
F. Comets.
{\em Directed Polymers in Random Environments}.
Lecture Notes in Mathematics, 2175. Springer, Cham, 2017.

\bibitem[CCM20]{CCM20}
F. Comets, C. Cosco, C.Mukherjee.
Renormalizing the Kardar-Parisi-Zhang Equation in $d\ge 3$ in weak disorder.
{\em J. Stat. Physics} 179, 713-728, 2020.

\bibitem[CSY03]{CSY03}
F. Comets, T. Shiga, N. Yoshida,
Directed polymers in a random environment: path localization and strong disorder.
{\em Bernoulli} 9, 705-723, 2003.

\bibitem[CY06]{CY06}
F. Comets, N. Yoshida.
Directed polymers in random environment are diffusive at weak disorder.
{\em Ann. Probab.} 34, 1746–1770, 2006.


\bibitem[Cor12]{Cor12}
I.~Corwin.
The Kardar-Parisi-Zhang equation and universality class.
{\em Random Matrices Theory Appl.} 1, 1130001, 2012.

\bibitem[Cor16]{Cor16}
I.~Corwin.
Kardar-Parisi-Zhang universality.
{\em Notices of the AMS} 63, 230-239, 2016.

\bibitem[CH16]{CH}
I. Corwin, A. Hammond.
KPZ Line Ensemble.
{\em Probab. Theory Relat. Fields} 166, 67--185, 2016.

\bibitem[CN21]{CN21}
C. Cosco, S. Nakajima.
Gaussian fluctuations for the directed polymer partition function in dimension $d\geq 3$ and in the whole
$L^2$-region.
{\em  Ann. Inst. H. Poincar\'e Prob. Stat.} 57, 872-889, 2021.

\bibitem[CNN22]{CNN20}
C. Cosco, S. Nakajima, M. Nakashima.
Law of large numbers and fluctuations in the sub-critical and $L^2$ regions for SHE and KPZ equation in
dimension $ d\geq 3$.
{\em  Stochastic Process. Appl.} 151, 127-173, 2022.

\bibitem[DFT94]{DFT94}
G. F. Dell'Antonio, R. Figari, A. Teta.
Hamiltonians for systems of $N$ particles interacting through point interactions.
{\em  Ann. Inst. H. Poincar\'e Phys. Th\'eor.} 60, 253-290, 1994.


\bibitem[DR04]{DR04}
J. Dimock, S. Rajeev.
Multi-particle Schr\"odinger operators with point interactions in the plane.
{\em J. Phys. A: Math. Gen.} 37(39):9157, 2004.

\bibitem[DGRZ20]{DGRZ20}
A. Dunlap, Y. Gu, L. Ryzhik, O. Zeitouni.
Fluctuations of the solutions to the KPZ equation in dimensions three and higher.
{\em Probab. Th. Rel. Fields} 176, 1217-1258, 2020.

\bibitem[ET60]{ET60}
P. Erd\"os, S. J. Taylor.
Some problems concerning the structure of random walk paths.
{\em Acta Math. Acad. Sci. Hungar.} 11, 137-162, 1960.

\bibitem[F16]{Feng}
Z.~S. Feng.
Rescaled Directed Random Polymer in Random Environment in Dimension 1+2.
Ph.D.\ thesis, Ann Arbor, MI, 2016.
Available at \url{https://www.proquest.com/docview/1820736587}.

\bibitem[Ga21]{G21+}
S. Gabriel.
Central limit theorems for the (2+1)-dimensional directed polymer in the weak disorder limit.
ArXiv:2104.07755, 2021.

\bibitem[GaSt12]{GS12}
C. Garban, J. Steif.
Lectures on noise sensitivity and percolation.
Proceedings of the Clay Mathematics Institute Summer School (Buzios, Brazil),
Clay Mathematics Proceedings 15, 49-154, 2012.

\bibitem[GaSu09]{GS09}
J. G\"artner, R. Sun.
A quenched limit theorem for the local time of random walks on $\Z^2$.
{\em  Stochastic Process. Appl.} 119 , 1198-1215, 2009.

\bibitem[Gi11]{Gi11}
G.~Giacomin.
{\em Disorder and critical phenomena through basic probability models.}
Lecture Notes in Mathematics, 2025, Springer, Heidelberg, 2011.

\bibitem[GLT10]{GLT10}
G. Giacomin, H. Lacoin, F.L. Toninelli.
Marginal relevance of disorder for pinning models.
{\em Comm. Pure Appl. Math}  63, 233-265, 2010.

\bibitem[GJ14]{GJ14}
P. Goncalves, M. Jara,
Nonlinear fluctuations of weakly asymmetric interacting particle systems.
{\em  Arch. Ration. Mech. Anal.} 212, 597-644, 2014.


\bibitem[Gu20]{Gu20}
Y. Gu.
Gaussian fluctuations from the 2D KPZ equation.
{\em Stoch. Partial Differ. Equ. Anal. Comput.} 8, 150-185, 2020.

\bibitem[GQT21]{GQT21}
Y. Gu, J. Quastel, L.-C. Tsai.
Moments of the 2D SHE at criticality.
{\em Prob. Math. Phys.} 2, 179-219, 2021.

\bibitem[GRZ18]{GRZ18}
Y. Gu, L. Ryzhik, O. Zeitouni.
The Edwards-Wilkinson limit of the random heat equation in dimensions three and higher.
{\em Comm. Math. Phys.} 363, 351–388, 2018.

\bibitem[GIP15]{GIP15}
M. Gubinelli, P. Imkeller, N. Perkowski.
Paracontrolled distributions and singular PDEs.
\emph{Forum Math. Pi} 3, e6, 2015.

\bibitem[GP17]{GP17}
M. Gubinelli, N. Perkowski.
KPZ reloaded.
{\em Comm. Math.Phys.} 349, 165-269, 2017.

\bibitem[H13]{H13}
M. Hairer.
Solving the KPZ equation.
{\em Ann. of Math.} 178, 559-664, 2013.

\bibitem[H14]{H14}
M. Hairer.
A theory of regularity structures.
{\em Inventiones Math.} 198, 269-504, 2014.

\bibitem[HH12]{HH12}
T. Halpin-Healy.
$(2+1)$-dimensional directed polymer in a random medium: scaling phenomena and universal distributions.
{\em Phys. Rev. Lett.} 109, 170602, 2012.

\bibitem[HH13]{HH13}
T. Halpin-Healy.
Extremal paths, the stochastic heat equation, and the three-dimensional
Kardar-Parisi-Zhang universality class.
{\em Physical Review E} 88, 042118, 2013.

\bibitem[HH85]{HH85}
D.A.~Huse, C.L.~Henley.
Pinning and roughening of domain walls in Ising systems due to random impurities.
{\em Phys. Rev. Lett.} 54, 2708-2711, 1985.

\bibitem[IS88]{IS88}
J. Z. Imbrie,  T. Spencer.
Diffusion of directed polymers in a random environment.
{\em J. Stat. Physics} 52, 609-626, 1988.

\bibitem[J00]{J00}
K. Johansson.
Transversal fluctuations for increasing subsequences on the plane.
{\em Probab. Theory Related Fields} 116, 445-456, 2000.


\bibitem[Kal97]{Kal97}
O. Kallenberg.
{\em Foundations of modern probability}.
Springer, 1997.


\bibitem[Koz07]{Koz07}
G. Kozma.
The scaling limit of loop-erased random walk in three dimensions.
{\em Acta Math.} 199, 29-152, 2007.


\bibitem[Kup14]{Kup14}
A. Kupiainen.
Renormalization Group and Stochastic PDEs.
{\em Ann. H. Poincar\'e} 17, 497-535, 2016.

\bibitem[L10]{L10}
H. Lacoin.
New bounds for the free energy of directed polymers in dimension $1+1$ and $1+2$.
{\em Comm. Math. Physics} 294, 471-503, 2010.

\bibitem[LaLi10]{LaLi10}
G.F. Lawler, V. Limic.
{\em Random walk: a modern introduction}.
Cambridge University Press, 2010.

\bibitem[LiLo01]{LiLo01}
E.H. Lieb, M. Loss.
{\em Analysis}, 2nd ed.
Graduate Studies in Mathematics 14, American Mathematical Society, 2001.

\bibitem[LZ22]{LZ20}
D. Lygkonis, N. Zygouras.
Edwards-Wilkinson fluctuations for the directed polymer in the full $L^2$-regime for dimensions $d\geq 3$.
{\em  Ann. Inst. Henri Poincar\'e Probab. Stat.}  58, 65-104, 2022.

\bibitem[MU18]{MU18}
J. Magnen, J. Unterberger.
The scaling limit of the KPZ equation in space dimension 3 and higher.
{\em  J. Stat. Physics} 171, 543-598, 2018.

\bibitem[MOO10]{MOO10}
E.~Mossel, R.~O'Donnell, K.~Oleszkiewicz.
Noise stability of functions with low influences: Invariance and optimality.
{\em Ann.\ Math.} 171, 295-341, 2010.

\bibitem[MSZ16]{MSZ16}
C. Mukherjee, A. Shamov, O. Zeitouni.
Weak and strong disorder for the stochastic heat equation and continuous directed polymers in $d\geq 3$.
{\em Electron. Comm. Prob.} 21, 1-12, 2016.

\bibitem[N19]{N19}
M. Nakashima.
Free energy of directed polymers in random environment in $1+1$-dimension at high temperature.
{\em Electron. J. Probab.} 24, 1-43, 2019.

\bibitem[QS15]{QS15}
J.~Quastel, H.~Spohn.
The One-dimensional KPZ equation and tts universality class.
{\em J. Stat. Physics} 160, 965-984, 2015.

\bibitem[Raj99]{R99}
S. G. Rajeev,
A condensation of interacting Bosons in two dimensional space,
arXiv preprint hep-th/9905120, (1999).

\bibitem[Rol13]{R13}
A. R\"ollin.
Stein's method in high dimensions with applications.
{\em  Ann. Inst. Henri Poincar\'e Probab. Stat.} 49, 529-549, 2013.

\bibitem[Rot79]{R79}
V.I. Rotar.
Limit theorems for polylinear forms.
{\em J. Multivariate Anal.} 9,  511-530, 1979.


\bibitem[V07]{V07}
V. Vargas.
Strong localization and macroscopic atoms for directed polymers.
{\em Probab. Theory Related Fields} 138, 391-410, 2007.


\end{thebibliography}
\end{document}